\documentclass[12pt,a4paper,twoside]{book}
\usepackage{amsmath, amsfonts, amssymb, amsthm}
\usepackage{latexsym}
\usepackage{graphicx}
\usepackage{graphics}

\usepackage{epsfig}
\usepackage{mathrsfs}
\usepackage{pb-diagram}
\usepackage[nottoc]{tocbibind}
\usepackage[english]{babel}
\usepackage[bookmarks=false]{hyperref}
\graphicspath{{./Figure/}}

\theoremstyle{plain}
\newtheorem{theorem}{Theorem}[section]

\newtheorem{lemma}[theorem]{Lemma}
\newtheorem{corollary}[theorem]{Corollary}

\theoremstyle{definition}
\newtheorem{definition}[theorem]{Definition}
\newtheorem{example}[theorem]{Example}

\theoremstyle{remark}
\newtheorem{remark}[theorem]{Remark}

\numberwithin{equation}{section}
\numberwithin{figure}{chapter}

\DeclareMathOperator{\dist}{dist}

\DeclareMathOperator{\rot}{rot}

\DeclareMathOperator{\Id}{\mbox{Id}}

\def\eop{\hfill$\square$}

\def\stretchcontm{{\stretchcont^{\!\!\!\!^{m}}}{\;}}
\def\stretchkkym{{\stretchkky^{\!\!\!\!^{m}}}{\;}}

\def\stretchxtwo{\stretchx^{\!\!\!\!\!^{2}}}
\def\stretchxm{\stretchx^{\!\!\!\!\!^{m}}}

\def\stretchcont{\;\;\;{^{\smallfrown}_{\smallsmile}{\!\!\!\!\!\!\!{\longrightarrow}}}\;}
\def\stretchkky{\;\;\;{^{\smallfrown}{\!{_{\smallsmile}}}{\!\!\!\!\!\!\!{\longrightarrow}}}\;}

\def\stretchx{\Bumpeq{\!\!\!\!\!\!\!\!{\longrightarrow}}}

\begin{document}

\title{\Huge{\textbf{Fixed points and\\ chaotic dynamics for expansive-contractive maps\\ in Euclidean spaces,\\ with some applications}}}
  \author{\Huge{\rm{\textbf{Marina Pireddu}}}
\\{}\\{}\\
{{Dipartimento di Matematica e Informatica, Universit\`{a} di Udine}}\\{}\\
{{Corso di Dottorato di Ricerca in Matematica e Fisica, Ciclo XXI}}\\{}\\
 \bigskip
{{Anno Accademico: 2008-2009}}\\{}\\
{{Ph.D. Thesis}}
\footnote{\,This is the final thesis for the Ph.D. program in Mathematics and Physics - Section: Mathematics. The final defence has
been held on June 4, 2009.}
\footnote{\,Sinceri ringraziamenti al Prof. Zanolin per l'infinita disponibilit\`{a} ed il prezioso e costante aiuto.}}

\date{}

\frontmatter
\maketitle

\chapter{Introduction}
In this work we introduce a topological method for the search of fixed points and periodic points for continuous maps defined on generalized rectangles in finite dimensional Euclidean spaces. We name our technique ``Stretching Along the Paths'' method, since we deal with maps that expand the arcs along one direction. Such theory was developed in the planar case by Papini and Zanolin in \cite{PaZa-04a, PaZa-04b} and it has been extended to the $N$-dimensional framework by the author and Zanolin in \cite{PiZa-07}. In the bidimensional setting, elementary theorems from plane topology suffice, while in the higher dimension some results from degree theory are needed, leading to the study of the so-called ``Cutting Surfaces'' \cite{PiZa-07}.
In the past decade a big effort has been addressed towards the search of fixed and periodic points. In this respect several authors \cite{BaCs-07,BaCs-08,GiRo-03, MiMr-95a,PoSzMI-01,Sr-00,SrWo-97,Sz-96,WoZg-00,Zg-96,ZgGi-04} have studied maps expansive along some directions and compressive along other ones, obtaining results that are also significant from a dynamical point of view, in particular in relation to Markov partitions. These achievements usually require sophisticated algebraic tools, like the Conley index or the Lefschetz number. Our approach, even if confined to a special configuration, looks instead more elementary. The description of the Stretching Along the Paths method and suitable variants of it can be found in Chapter \ref{ch-me}.\\
In Chapter \ref{ch-ch} we discuss which are the chaotic features that can be obtained for a given map when our technique applies. In particular, we are able to
prove the semi-conjugacy to the Bernoulli shift and thus the positivity of the topological entropy, the presence of topological transitivity and sensitivity with respect to initial conditions, the density of periodic points. Moreover we show the mutual relationships among various classical notions of chaos (such as that by Devaney, Li-Yorke, etc.). We also introduce an alternative geometrical framework related to the so-called ``Linked Twist Maps'', where it is possible to employ our method in order to detect complex dynamics. Chapter \ref{ch-ch} is as self-contained as possible and hence accessible also to the reader who is not familiar with the concept of chaos. Due to the large quantity of definitions available in the literature, our exposition is thought as a brief survey, where a comparison with the notion of chaos in the sense of coin-tossing is presented, too. Such characterization is indeed natural when considered in relation with the approach in Chapter \ref{ch-me}.  \\
The theoretical results obtained so far find an application to discrete- and continuous-time systems in Chapters \ref{ch-de} and \ref{ch-ode}. Namely, in Chapter \ref{ch-de} we deal with some one-dimensional and planar discrete economic models, both of the Overlapping Generations and of the Duopoly Games classes. The bidimensional models are taken from \cite{Me-92, Re-86} and \cite{AgEl-04}, respectively. On the other hand, in Chapter \ref{ch-ode} we analyze some nonlinear ODEs with periodic coefficients. In more details, we consider a modified version of the Volterra predator-prey model, in which a periodic harvesting is included, as well as a simplification of the Lazer-McKenna suspension bridges model \cite{LaMK-87, LaMK-90} from \cite{PaPiZa-08,PaZa-08}.
When dealing with ODEs with periodic coefficients, our method is applied to the associated Poincar\'{e} map and thus we are led back to work with discrete dynamical systems.\\
The above summary is just meant to mention the main aspects of this work. Each chapter is indeed equipped with a more detailed introduction, including the corresponding bibliography.\\
The contents of the present thesis are based on the papers \cite{MePiZa->,PaPiZa-08,PiZa-07,PiZa-07Ts,PiZa-08}
and partially on \cite{PiZa-05}, where maps expansive along several directions were considered.

\chapter{Notations}
For the reader's convenience, we introduce some basic notations that will be used throughout this work.

\medskip
\noindent
We denote by $\mathbb N,\,\mathbb Z,\,\mathbb Q,\,\mathbb R$ and $\mathbb C$ the sets of \textit{natural}, \textit{integer}, \textit{rational}, \textit{real} and \textit{complex numbers}, respectively. In particular ${\mathbb N}$ is the set of nonnegative integers, while ${\mathbb N}_0$ denotes the set of the positive integers. The sets of nonnegative and positive real numbers will be indicated with $\mathbb R^+$ and $\mathbb R_0^+.$ Accordingly, the first quadrant and the open first quadrant will be denoted by $(\mathbb R^+)^2$ and $(\mathbb R_0^+)^2,$ respectively.

\medskip
\noindent
For a subset $M$ of the topological space $W,$ we denote by $\overline M,\,{\mbox {Int}}(M)$ and $\partial M$ the \textit{closure}, the \textit{interior} and the \textit{boundary} of $M$ in $W,$ respectively. The set $M\subseteq W$ is said to be \textit{dense} in $W$ if $\overline M=W.$ If $M\subseteq W,$ we indicate with $W\setminus M$ the \textit{complement} of $M$ in $W.$ In the case that $A, B\subseteq W,$ we denote by $A\setminus B$ the \textit{relative complement} of $B$ in $A.$ By $|A|$ we mean the \textit{cardinality} of the set $A.$

\medskip
\noindent
We denote by $\Id$ the \textit{identity map} and by $\Id_W$ the identity on the space $W,$ when we need to specify it. Given a function $f,$ we sometimes indicate with $D_f$ its \textit{domain}. By $f\restriction_M$ we mean the \textit{restriction} of $f$ to a subset $M$ of its domain. For a function $f:W\supseteq D_f\to Z$ between the topological spaces $W$ and $Z,$ we define the \textit{preimage} of $z\in Z$ as $f^{-1}(z):=\{w\in D_f:f(w)=z\}.$

\medskip
\noindent
A function $f:W\to W$ is called a \textit{self-map} of the space $W.$ The \textit{iterates} of $f$ are defined recursively with the convention $f^0 = \Id_W,$ $f^1 =f$ and $f^n=f\circ f^{n-1},\,\forall n\ge 2.$ We say that $\overline w\in W$ is a \textit{fixed point}\index{point, fixed} for $f$ if $f(\overline w)=\overline w.$ We say that $\overline w\in W$ is a \textit{periodic point}\index{point, periodic} for $f$ if there exists an integer $l\ge 1$ such that $f^l(\overline w)=\overline w.$ The minimal $l$ such that $f^l(\overline w)=\overline w$ is called \textit{period} of $\overline w.$

\medskip
\noindent
The $N$-dimensional Euclidean space $\mathbb R^N$ is endowed with the usual {scalar product} $\langle\cdot\,,\cdot\rangle,$ {norm} $\|\cdot\|$ and {distance} $\dist(\cdot\,,\cdot).$
In the case of $\mathbb R,$ the norm $\|\cdot\|$ will be replaced with the \textit{absolute value} $|\cdot|\,.$ If $X$ is a metric space different from $\mathbb R^N,$ we denote by $d_X$ the distance defined on it. The distance between $M_1,\,M_2\subseteq X$ is indicated with $d_X(M_1,\,M_2):=\inf\{d_X(x,y):x\in M_1,\, y\in M_2\}.$

\medskip
\noindent
In a normed space $(X,\|\cdot\|_X),$ we denote by $B(x_0,r)$ and $B[x_0,r]$ the \textit{open} and \textit{closed balls} centered in $x_0\in X$ with radius $r>0,$ i.e. $B(x_0,r):=\{x\in X:\|x-x_0\|_X<r\}$ and $B[x_0,r]:=\{x\in X:\|x-x_0\|_X\le r\}.$ For $M\subseteq X,$ we set $B(M,r):=\{x\in X: \exists\, w\in M \mbox{ with } \|x-w\|_X< r\}.$ The set $B[M,r]$ is defined accordingly. If we wish to specify the dimension $m$ of the ball, when it is contained in $\mathbb R^m,$ we write $B_m(x_0,r)$ and $B_m[x_0,r]$ in place of $B(x_0,r)$ and $B[x_0,r].$ We denote by $S^{N-1}$ the $N-1$-dimensional \textit{unit sphere} embedded in $\mathbb R^N,$ that is, $S^{N-1}:=\partial B_N[0,1].$

\medskip
\noindent
Given an open bounded subset $\Omega$ of $\mathbb R^N,$ an element $p$ of $\mathbb R^N$ and a continuous function $f:\overline{\Omega} \to \mathbb R^N,$ we indicate with $\deg(f,\Omega,p)$
the \textit{topological degree} of the map $f$ with respect to $\Omega$ and $p.$ We say that $\deg(f,\Omega,p)$ is defined if $f(x)\ne p,\,\forall x \in \partial\Omega$ \cite{Ll-78}.

\medskip
\noindent
We denote by $\mathcal C(D,\mathbb R^N)$ the set of the continuous maps $f:D\to\mathbb R^N$  on a compact set $D,$ endowed with the infinity norm $|f|_\infty:=\max_{x\in D}\|f(x)\|.$

\medskip
\noindent
Given a compact interval $[a,b]\subset \mathbb R,$ we set $L^1([a,b]):=\{f:[a,b]\to\mathbb R: f \mbox{ is Lebesgue measurable and } \int_{a}^b|f(t)| dt <+\infty\},$ where the integral is the Lebesgue integral. This space is endowed with the norm $||f||_1:= \int_{a}^b|f(t)| dt\,.$

\tableofcontents

\mainmatter

\part{Expansive-contractive maps in Euclidean spaces}\label{pa-ec}
\chapter[Maps expansive along one direction]{Search of fixed points for maps expansive along one direction}\label{ch-me}
One of the main areas in general topology concerns the search of fixed points for continuous maps defined on arbitrary topological spaces. Given a topological space $W\supseteq A\ne\emptyset$ and a continuous map $f: A\to W,$ a \textit{fixed point} for $f$ is simply a point $\bar x\in A$ that is not moved by the map, i.e. $f(\bar x)=\bar x.$ In spite of the simplicity of the definition, the concept of fixed point turns out to be central in the study of dynamical systems, since, for instance, fixed points of suitable operators correspond to periodic solutions (cf. Chapter \ref{ch-ode}). A concept related to that of fixed point is the one of \textit{periodic point}. Namely, the periodic points of a continuous map $f$ are the fixed points of the iterates of the map, that is, $\bar x$ is a periodic point for $f$ if there exists an integer $l\ge 1$ such that $f^l(\bar x)=\bar x.$\\
As often happens, the more general are the spaces we consider, the more sophisticated are the tools to be employed:
the search of fixed points and periodic points can indeed require the use of advanced theories, like that of the
Conley or fixed point index \cite{MiMr-95a, MiMr-02, Ro-02, Sz-96, WoZg-00, Zg-96}, Lefschetz number \cite{LWSr-02, Sr-97, Sr-00, SrWo-97} and many other geometric or algebraic methods \cite{BaMr-07, Ea-75, SrWoZg-05, Wo-95}. In the past decades several efforts have been made in order to find elementary tools to deal with such a problem. For example Poincar{\'e}-Miranda Theorem or some other equivalent versions of Brouwer fixed point Theorem have turned out to be very useful when dealing with $N$-dimensional Euclidean spaces \cite{BaCsGa-07, Ma-06}.\\
It is among these not too sophisticated approaches that the theory of the Cutting Surfaces for the search of fixed point and periodic points, introduced by the author and Zanolin in \cite{PiZa-07}, has to be placed. More precisely, it concerns continuous maps defined on generalized $N$-dimensional rectangles and having an expansive direction. Our main tools are a modified version of the
classical Hurewicz-Wallman Intersection Lemma \cite[p.72]{En-78}, \cite[D), p.40]{HuWa-41}
and the Fundamental Theorem of Leray-Schauder \cite[Th\'{e}or\`{e}me Fondamental]{LeSh-34}.
The former result (also referred to Eilenberg-Otto \cite[Th.3]{EiOt-38},
according to \cite[Theorem on partitions, p.100]{DuGr-03}) is one of the basic lemmas
in dimension theory and it is known to be one of the equivalent versions of
Brouwer fixed point Theorem. It may be interesting to observe that extensions of
the Intersection Lemma led to generalizations of Brouwer fixed point Theorem to
some classes of possibly noncontinuous functions \cite{Wh-67}.
On the other hand, the Leray-Schauder Continuation Theorem concerns
topological degree theory and has found important applications in
nonlinear analysis and differential equations (see \cite{Ma-97}).
Actually, in Theorem \ref{th-es} of Section \ref{sec-nd}, we will
employ a more general version of such result
due to Fitzpatrick, Massab\'{o} and Pejsachowicz \cite{FiMaPe-86}.
However, we point out that for our applications to the study of periodic
points and chaotic-like dynamics we rely on the classical
Leray-Schauder Fundamental Theorem.
Another particular feature of our approach consists in a combination of Poincar\'{e}-Miranda
Theorem \cite{Ku-97,Ma-00}, which is an $N$-dimensional version of the
intermediate value theorem, with the properties of topological surfaces cutting the
arcs between two given sets.
The theory of Cutting Surfaces is explained in Section \ref{sec-nd}.\\
When we confine ourselves to the bidimensional setting, suitable results from plane topology can be applied \cite{PaPiZa-08}, like the Crossing Lemma \ref{lem-cr}, and no topological degree arguments are needed. In this case we enter the context of the ``Stretching Along the Paths'' method developed by Papini and Zanolin in \cite{PaZa-04a, PaZa-04b}. The name comes from the fact that one deals with maps that expand the arcs, in a sense that will be clarified in Section \ref{sec-sap}, where the interested reader can find all the related details. We stress that, with respect to \cite{PaZa-04a, PaZa-04b}, the terminology and notation are slightly different and some proofs are new.\\
In any case, our approach based on the Cutting Surfaces represents a possible extension of the planar theory in \cite{PaZa-04a, PaZa-04b} to the $N$-dimensional setting, for any $N\ge 2.$ By the similarity between the concepts employed and the results obtained in the two frameworks, we will keep the ``stretching'' terminology also for the higher dimensional case, without distinguishing between them. Indeed, the main difference between the planar setting and the $N$-dimensional one resides in the tools to be employed in order to get some results, like for instance the fixed point Theorems \ref{th-fp} and \ref{th-fpn}, on which the two theories are respectively based. Namely, as already mentioned, in the plane elementary results from general topology suffices, while in the higher dimension, one needs more sophisticated tools, such as topological degree.

\section{The planar case}\label{sec-sap}
\subsection{Stretching along the paths and variants}\label{sub-sap}

Before introducing the main concepts of the ``Stretching Along the Paths'' method, let us recall some facts about paths, arcs and continua.\\
Let $W$ be a topological space. By a \textit{path}\index{path, subpath} $\gamma$ in $W$ we mean a
continuous map $\gamma : {\mathbb R}\supseteq [a,b]\to W.$ Its range will be denoted by
${\overline{\gamma}},$ that is,
${\overline{\gamma}}:=\gamma([a,b]).$
A \textit{subpath} $\omega$ of $\gamma$ is the restriction of $\gamma$ to a closed subinterval of its
domain and hence it is defined as
$\omega:= \gamma\restriction_{[c,d]},$ for some $[c,d]\subseteq [a,b].$
If $W,Z$ are topological spaces and $\psi:W\supseteq
D_{\psi}\to Z$ is a map which is continuous on a set ${\mathcal
M}\subseteq D_{\psi}\,,$ then for any path $\gamma$ in $W$ with
$\overline{\gamma}\subseteq {\mathcal M},$ it follows that
$\psi\circ\gamma$ is a path in $Z$ with range
$\psi({\overline{\gamma}}).$ Notice that there is no loss of generality in assuming the
paths to be defined on $[0,1].$ Indeed, if $\theta_1: [a_1,b_1]\to
W$ and $\theta_2:[a_2,b_2]\to W,$ with $a_i<b_i,\, i=1,2,$ are two
paths in $W,$ we define the equivalence relation `` $\sim$ '' between $\theta_1$ and $\theta_2$ by setting $\theta_1\sim \theta_2$ if there is a
homeomorphism $h$ of $[a_1,b_1]$ onto $[a_2,b_2]$ such that $\theta_2(h(t)) =
\theta_1(t),\,$ $\forall\, t\in [a_1,b_1].$ It is easy to check that if $\theta_1\sim\theta_2\,,$
then the ranges of $\theta_1$ and $\theta_2$
coincide. Hence, for any path $\gamma$
there exists an equivalent path defined on $[0,1].$ In view of this fact we will usually deal with paths defined on $[0,1],$ but sometimes we will also consider paths defined on an arbitrary interval $[a,b],$ when this can help to simplify the exposition.
A concept similar to the one of path is that of arc.
More precisely an \textit{arc}\index{arc} is the homeomorphic image of the
compact interval $[0,1],$ while an \textit{open arc} is an arc without its
end-points. A \textit{continuum}\index{continuum, subcontinuum} of $W$ is a compact connected subset of $W$ and a \textit{subcontinuum} is a  subset of a continuum which is itself a continuum.\\
Now we can start with our definitions. Given a metric space $X,$
we call \textit{generalized rectangle}\index{rectangle! generalized} any set
${\mathcal R}\subseteq X$ homeomorphic to the unit square ${\mathcal Q}:=[0,1]^2$ of ${\mathbb R}^2.$
If ${\mathcal R}$ is a generalized rectangle and $h: {\mathcal Q}\to h({\mathcal Q})={\mathcal R}$
is a homeomorphism defining it, we call \textit{contour}\index{contour} $\vartheta{\mathcal R}$ of
${\mathcal R}$ the set
$$\vartheta {\mathcal R}:= h(\partial{\mathcal Q}),$$
where $\partial{\mathcal Q}$ is the usual boundary of the unit square. Notice that the contour $\vartheta{\mathcal R}$
is well-defined as it does not depend on the choice of the homeomorphism $h.$ In fact, $\vartheta{\mathcal R}$
is also a homeomorphic image of $S^1,$ that is, a Jordan curve.\\
By an \textit{oriented rectangle}\index{rectangle! oriented} we mean a pair
$${\widetilde{\mathcal R}}:=
({\mathcal R},{\mathcal R}^-),$$
where ${\mathcal R} \subseteq X$ is a generalized rectangle and
$${\mathcal R}^- := {\mathcal R}^-_{\ell}\cup {\mathcal R}^-_{r}$$
is the union of two disjoint arcs
${\mathcal R}^-_{\ell}\,,{\mathcal R}^-_{r}\,\subseteq \vartheta{\mathcal R},$ that we
call the \textit{left} and the \textit{right sides} of ${\mathcal R}^-.$\index{$[\,\cdot\,]^-$-set}
Since $\vartheta{\mathcal R}$ is a Jordan curve
it follows that $\vartheta{\mathcal R} \setminus(\,{\mathcal R}^-_{\ell}\cup{\mathcal R}^-_{r}\,)$
consists of two open arcs. We denote by ${\mathcal R}^+$ the closure of such open arcs, that we name
${\mathcal R}^+_{d}$ and ${\mathcal R}^+_{u}\,$ (the \textit{down} and \textit{up sides} of ${\mathcal R}^+$).
It is important to notice that we can always label the arcs
${\mathcal R}^{-}_{\ell}\,,$ ${\mathcal R}^{+}_{d}\,,$ ${\mathcal R}^{-}_{r}$ and
${\mathcal R}^{+}_{u},$ following the cyclic order $\ell-d-r-u-\ell,$
and take a homeomorphism $g: {\mathcal Q}\to g({\mathcal Q})={\mathcal R}$
so that
\begin{equation}\label{eq-or}
\begin{array}{lll}
& g(\{0\}\times [0,1]) = {\mathcal R}^-_{\ell}\,,\quad
& g(\{1\}\times [0,1]) = {\mathcal R}^-_{r}\,,\\
& g([0,1]\times\{0\}) = {\mathcal R}^+_{d}\,,\quad
& g([0,1]\times\{1\}) = {\mathcal R}^+_{u}\,.
\end{array}
\end{equation}

\noindent
Both the term ``generalized rectangle'' for ${\mathcal R}$ and the decomposition of the contour
$\vartheta{\mathcal R}$ into ${\mathcal R}^-$ and ${\mathcal R}^+$ are somehow inspired to the construction
of rectangular domains around hyperbolic sets arising in the theory of Markov partitions \cite[p.291]{HaKa-03},
as well as by the Conley-Wa\.{z}ewski theory \cite{Co-78,Sr-04}. Roughly speaking,
in such frameworks, the sets labeled as $[\cdot]^-,$
or as $[\cdot]^+,$ are made by those points which are moved by the flow outward, respectively inward,
with respect to ${\mathcal R}.$ As we shall see in the next definition of \textit{stretching along the paths} property, also in our case the $[\cdot]^{-}$-set is
loosely related to the expansive direction. Indeed we have:

\begin{definition}\label{def-sap}
{\rm{Let $X$ be a metric
space and let $\psi: X \supseteq D_{\psi}\to X$ be a map
defined on a set $D_{\psi}.$ Assume that ${\widetilde{\mathcal A}}:=
({\mathcal A},{\mathcal A}^-)$ and ${\widetilde{\mathcal B}}:=
({\mathcal B},{\mathcal B}^-)$ are oriented rectangles of
$X$ and let ${\mathcal K}\subseteq {\mathcal A}\cap D_{\psi}$
be a compact set. We say that \textit{$({\mathcal K},\psi)$ stretches
${\widetilde{\mathcal A}}$ to ${\widetilde{\mathcal B}}$ along the
paths}\index{stretching along the paths@\textsl{stretching along the paths} $\stretchx$} and write
\begin{equation}\label{eq-sap}
({\mathcal K},\psi): {\widetilde{\mathcal A}} \stretchx {\widetilde{\mathcal B}},
\end{equation}
if the following conditions hold:
\begin{itemize}
\item{} \; $\psi$ is continuous on ${\mathcal K}\,;$ \item{} \;
For every path $\gamma: [0,1]\to {\mathcal A}$ such that
$\gamma(0)\in {\mathcal A}^-_{\ell}$ and $\gamma(1)\in {\mathcal
A}^-_{r}$ (or $\gamma(0)\in {\mathcal A}^-_{r}$ and $\gamma(1)\in
{\mathcal A}^-_{\ell}$), there exists a subinterval
$[t',t'']\subseteq [0,1]$ with
$$\gamma(t)\in {\mathcal K},\quad \psi(\gamma(t))\in {\mathcal B}\,,\;\;
\forall\, t\in [t',t'']$$ and, moreover, $\psi(\gamma(t'))$ and
$\psi(\gamma(t''))$ belong to different sides of ${\mathcal B}^-.$
\end{itemize}

\medskip

\noindent In the special case in which ${\mathcal K} = {\mathcal
A},$ we simply write
$$\psi: {\widetilde{\mathcal A}} \stretchx {\widetilde{\mathcal B}}.$$
}}
\end{definition}

\noindent
The role of the compact set ${\mathcal K}$ is crucial in the results which use
Definition \ref{def-sap}. For instance, if \eqref{eq-sap} is satisfied with $\widetilde{\mathcal B}=\widetilde{\mathcal A},$
that is,
\begin{equation}\label{eq-aa}
({\mathcal K},\psi): {\widetilde{\mathcal A}} \stretchx {\widetilde{\mathcal A}},
\end{equation}
we are able to prove the existence of a fixed point for $\psi$ in ${\mathcal K}$ (see Theorem \ref{th-fp}).\\
Notice that when \eqref{eq-sap} holds, the stretching condition
$$({\mathcal K}',\psi): {\widetilde{\mathcal A}}\stretchx {\widetilde{\mathcal B}}$$
is fulfilled for any compact set ${\mathcal K}',$ with
${\mathcal K}\subseteq {\mathcal K}'\subseteq {\mathcal A}\cap D_{\psi},$ on which $\psi$ is
continuous. On the other hand, since $\mathcal K$ provides a localization of the fixed point, it is convenient to choose it as small as possible.

\medskip

\noindent
We point out that in Definition \ref{def-sap} (as well as in its variant in Definition \ref{def-cn} below) we don't require $\psi(\mathcal A)\subseteq \mathcal B.$ In fact, we will refer to $\mathcal B$ as to a ``target set'' and not as to a codomain.\\
The relation introduced in Definition \ref{def-sap} is fundamental in our ``stretching along the paths'' method.
Indeed, such strategy consists in checking property \eqref{eq-aa}, so that, in view of Theorem \ref{th-fp}, there exists at least a fixed point for $\psi$ in $\mathcal K.$ If in particular \eqref{eq-aa} is satisfied with respect to two or more pairwise disjoint compact sets $\mathcal K_i$'s we get a multiplicity of fixed points. On the other hand, when \eqref{eq-aa} holds for some iterate of $\psi,$ the existence of periodic points is ensured. Since the stretching property in \eqref{eq-sap} is preserved under composition of mappings (cf. Lemma \ref{lem-comp}), combining such arguments, also the presence of chaotic dynamics follows.
A detailed treatment of such topic can be found in Section \ref{sec-de}.

\noindent
\begin{definition}\label{def-cn}
{\rm{Let $X$ be a metric
space and let $\psi: X \supseteq D_{\psi}\to X$ be a map
defined on a set $D_{\psi}.$ Assume that ${\widetilde{\mathcal A}}:=
({\mathcal A},{\mathcal A}^-)$ and ${\widetilde{\mathcal B}}:=
({\mathcal B},{\mathcal B}^-)$ are oriented rectangles of
$X$ and let ${\mathcal K}\subseteq {\mathcal A}\cap D_{\psi}$
be a compact set. Let also $m\geq 2$ be an integer.
We say that \textit{$({\mathcal D},\psi)$ stretches ${\widetilde{\mathcal A}}$
to ${\widetilde{\mathcal B}}$ along the paths with crossing number $m$}\index{stretching along the paths @\textsl{stretching along the paths} $\stretchxm$} and write
$$({\mathcal D},\psi): {\widetilde{\mathcal A}} \stretchxm \,{\widetilde{\mathcal B}},$$
if there exist $m$ pairwise disjoint compact sets
$${\mathcal K}_0\,,\dots, {\mathcal K}_{m-1}\,\subseteq {\mathcal D}$$ such that
$$({\mathcal K}_i,\psi): {\widetilde{\mathcal A}} \stretchx {\widetilde{\mathcal B}},\quad i=0,\dots,m-1.$$
When ${\mathcal D} = {\mathcal A}\,,$
we simply write
$$\psi: {\widetilde{\mathcal A}} \stretchxm \, {\widetilde{\mathcal B}}.$$
}}
\end{definition}

\begin{figure}[htbp]
\centering
\includegraphics[scale=0.3]{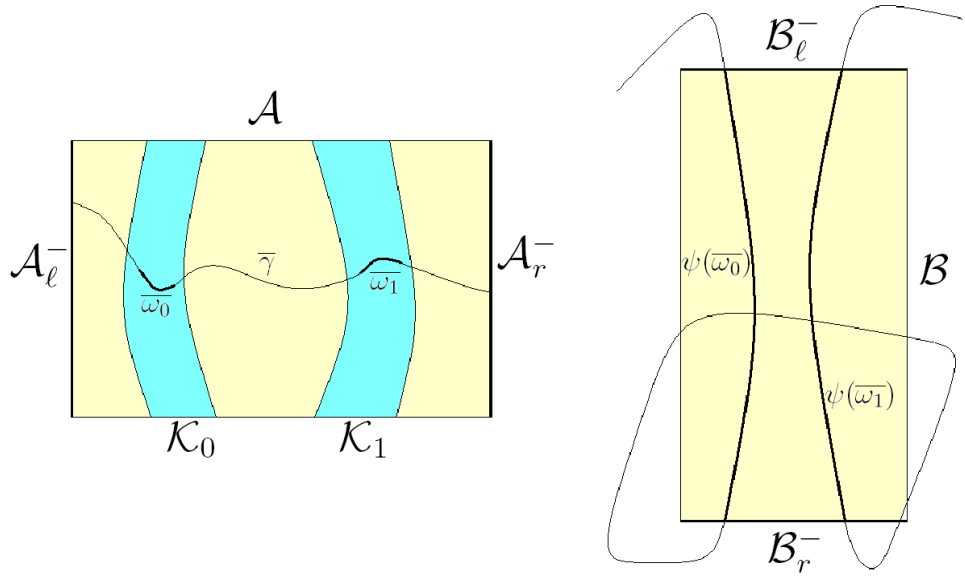}
\caption{\footnotesize{A pictorial comment to Definitions \ref{def-sap} and \ref{def-cn}.
The rectangles ${\mathcal A}$ and ${\mathcal B}$
have been oriented by selecting the sets ${\mathcal A}^-$
and ${\mathcal B}^-$ (drawn with thicker lines), respectively.
We represent a case in which the relation
$({\mathcal K_i},\psi): {\widetilde{\mathcal A}} {\Bumpeq{\!\!\!\!\!\!\!\!{\longrightarrow}}} {\widetilde{\mathcal B}},\, i=0,1,$ is satisfied for a map $\psi:\mathbb R^2\supseteq \mathcal A\to \mathbb R^2$ and for the two darker compact subsets ${\mathcal K_0}$ and ${\mathcal K_1}$ of
${\mathcal A},$ on which $\psi$ is continuous.
For a generic path $\gamma: [0,1]\to {\mathcal A}$ with $\gamma(0)$ and $\gamma(1)$ belonging to
different components of ${\mathcal A}^-,$ we have highlighted two subpaths
$\omega_0$ and $\omega_1$ with range in ${\mathcal K_0}$ and ${\mathcal K_1},$
respectively, such that their composition with $\psi$ determines two new paths
(drawn by bolder vertical lines)
with values in ${\mathcal B}$ and joining the two sides of ${\mathcal B}^-.$
In this framework, according to Definition \ref{def-cn}, we could also write $\psi: {\widetilde{\mathcal A}} \stretchxtwo\, {\widetilde{\mathcal B}}.$
}}
\end{figure}

\noindent
Definition \ref{def-cn} is an extension of Definition \ref{def-sap}, as they are coincident when $m=1.$ The concept of ``crossing number''\index{crossing number} is borrowed from Kennedy and Yorke \cite{KeYo-01} and adapted to our framework. Indeed in \cite{KeYo-01} the authors deal with a very general setting, that looks closely related to ours when restricted to the planar case. More precisely, in
\cite[horseshoe hypotheses $\Omega$]{KeYo-01}
a locally connected and compact subset $Q$ of the separable
metric space $X$ is considered, on which two disjoint and compact
sets $\hbox{\it end}_0\,, \hbox{\it end}_1\subseteq Q$ are selected, so that
any component of $Q$ intersects both of them. On $Q$ a continuous
map $f:Q\to X$ is defined in such a way that every continuum $\Gamma\subseteq Q$ joining $\hbox{\it end}_0$
and $\hbox{\it end}_1$ (i.e. a \textit{connection} according to
\cite{KeYo-01}) admits at least $m\geq 2$ pairwise disjoint compact and
connected subsets, whose images under $f$ are again connections.
Such subcontinua are named \textit{preconnections} and $m$ is the so-called
\textit{crossing number}.\\
Clearly, the
arcs ${\mathcal R}^-_{\ell}$ and ${\mathcal R}^-_{r}$ in our
definition of oriented rectangle are a particular case of the sets
$\hbox{\it end}_0$ and $\hbox{\it end}_1$ considered by Kennedy
and Yorke. Moreover, any path $\gamma$ with values in ${\mathcal
R}$ and joining ${\mathcal R}^-_{\ell}$ with ${\mathcal R}^-_{r}$
determines a connection (according to \cite{KeYo-01})
via its image $\overline{\gamma}.$ Similarly, any subpath
$\omega$ of $\gamma,$ with $\omega = \gamma\restriction_{[t',t'']}\,$ like in
Definition \ref{def-sap}, makes $\overline{\omega}$ a preconnection according
to Kennedy and Yorke. In fact, the idea of considering paths to detect a sort of expansion was independently developed by such authors in \cite{KeYo-02,KeYo-03}, where a terminology analogous to ours was employed.
A special request in our approach, not explicitly assumed in \cite{KeYo-01},
is the one concerning the compact sets like ${\mathcal K}.$ Indeed, in
\cite{KeYo-01} there are no sets playing the role of the
$\mathcal K_i$'s in Definition \ref{def-cn}.\\
On the other hand, our stretching condition is strong enough to ensure the
existence of fixed points and periodic points for the map $\psi$ inside
the $\mathcal K_i$'s, as we shall see in Theorems \ref{th-fp} and \ref{th-per} below, while
in \cite[Example 10]{KeYo-01} a fixed point free map defined on
$\mathbb R^2\times S^1$ and satisfying the horseshoe hypotheses
$\Omega$ is presented. Indeed, we point out that Theorem \ref{th-fp} allows not only to infer the existence of fixed points, but also to localize them. This turns out to be of particular importance when multiple coverings in the sense of Definition \ref{def-cn} occur.
\medskip

\noindent
In the proof of Theorems \ref{th-fp} and \ref{th-per}, we employ a classical result from plane topology (cf. Crossing Lemma \ref{lem-cr}), that we recall in Subsection \ref{sub-cl} for the reader's convenience. See also \cite{PaZa-04a, PaZa-04b}.

\begin{theorem}\label{th-fp}
Let $X$ be a metric space and let $\psi: X \supseteq D_{\psi}\to X$ be a map defined on a set $D_{\psi}.$ Assume that
${\widetilde{\mathcal R}}:= ({\mathcal R},{\mathcal R}^-)$
is an oriented rectangle of $X.$ If ${\mathcal K}\subseteq {\mathcal R}\cap D_{\psi}$ is a compact set for which it holds that
\begin{equation}\label{eq-ora}
({\mathcal K},\psi): {\widetilde{\mathcal R}} \stretchx {\widetilde{\mathcal R}},
\end{equation}
then there exists at least one point $z\in {\mathcal K}$ with $\psi(z) = z.$
\end{theorem}
\begin{proof}
By the definition of oriented rectangle,
there exists a homeomorphism
$h: {\mathbb R}^2 \supseteq {\mathcal Q}\to h({\mathcal Q}) = {\mathcal R}\subseteq X,$
mapping in a correct way (i.e. as in \eqref{eq-or})
the sides of ${\mathcal Q}= [0,1]^2$ into the arcs that
compose the sets ${\mathcal R}^-$ and ${\mathcal R}^+.$
Then, passing to the planar map $\phi:= h^{-1}\circ \psi \circ h$ defined on
$D_{\phi}:= h^{-1}(D_{\psi})\subseteq {\mathcal Q},$
we can confine ourselves
to the search of a fixed point for $\phi$ in the compact set
${\mathcal H}:=h^{-1}({\mathcal K})\subseteq {\mathcal Q}.$
The stretching assumption on $\psi$ is now translated to
$$({\mathcal H},\phi): {\widetilde{\mathcal Q}}\stretchx {\widetilde{\mathcal Q}}\,.$$
On ${\widetilde{\mathcal Q}}$ we consider the natural ``left-right'' orientation, choosing
$${\mathcal Q}^-=(\{0\}\times [0,1])\cup(\{1\}\times [0,1]).$$
A fixed point for $\phi$ in ${\mathcal H}$ corresponds to a fixed point for
$\psi$ in ${\mathcal K}.$
\\
For $\phi = (\phi_1,\phi_2)$ and $x= (x_1,x_2),$ we define the compact set
$$V:=\{x\in {\mathcal H}: 0\leq \phi_2(x)\leq 1,\; x_1 - \phi_1(x) = 0\}.$$
The proof consists in showing that $V$ contains a continuum
${\mathcal C}$ which joins in ${\mathcal Q}$
the lower side $[0,1]\times\{0\}$ to the upper side $[0,1]\times\{1\}.$
To this end, in view of Lemma \ref{lem-cr}, it is sufficient to prove that $V$ acts as a ``cutting surface''
(in the sense of Definition \ref{def-cut})
between the left and the right sides of ${\mathcal Q},$ that is,
$V$ intersects any path in ${\mathcal Q}$ joining the left side $\{0\}\times[0,1]$ to the right side
$\{1\}\times[0,1].$
Such cutting property can be checked via the intermediate value theorem by observing that if
$\gamma = (\gamma_1,\gamma_2):[0,1]\to\mathcal Q$ is a continuous map
with $\gamma(0)\in \{0\} \times [0,1]$ and $\gamma(1)\in \{1\} \times [0,1]$
then, the stretching hypothesis $({\mathcal H},\phi): {\widetilde{\mathcal Q}} \stretchx {\widetilde{\mathcal Q}}$ implies that
there exists an interval
$[t',t'']\subseteq [0,1]$ such that
$\gamma(t)\in {\mathcal H},$  $\phi(\gamma(t))\in {\mathcal Q},\,\forall\, t\in [t',t'']$
and
$\gamma_1(t') - \phi_1(\gamma(t')) \geq 0 \geq \gamma_1(t'') - \phi_1(\gamma(t''))$
or
$\gamma_1(t') - \phi_1(\gamma(t')) \leq 0 \leq \gamma_1(t'') - \phi_1(\gamma(t'')).$
Notice that, by the definition of $V$ it follows that $\phi_2(z)\in [0,1],\forall\, z \in {\mathcal C}.$
Hence, for every point $p = (p_1,p_2)\in
{\mathcal C}\cap ([0,1]\times\{0\})$ we have $p_2 - \phi_2(p) \leq 0$
and, similarly, $p_2 - \phi_2(p) \geq 0$ for every
$p = (p_1,p_2)\in {\mathcal C}\cap ([0,1]\times\{1\}).$
Applying Bolzano Theorem we obtain the existence of
at least a point $v=(v_1,v_2)\in {\mathcal C}\subseteq V \subseteq {\mathcal H}$ such that
$v_2 - \phi_2(v) =0.$ Hence $v$ is a fixed point of $\phi$ in ${\mathcal H}$ and $z:=h(v)$ is a fixed point for $\psi$ in ${\mathcal K}\subseteq {\mathcal R}.$
\end{proof}

\begin{remark}\label{rem-ora}
{\rm{
Notice that, for the validity of Theorem \ref{th-fp}, it is fundamental that the orientation of the generalized rectangle ${\mathcal R}$ in \eqref{eq-ora} remains the same for $\mathcal R$ considered as ``starting set'' and ``target set'' of the map $\psi.$ Indeed, if one chooses two different orientations for $\mathcal R,$ in general the above result does not hold anymore and the existence of fixed points for $\psi$ is no longer ensured, not only in $\mathcal K,$ but even in  $\mathcal R,$ as shown by the example depicted in Figure \ref{fig-nf}.
}}

\hfill$\lhd$\\
\end{remark}

\begin{figure}[htbp]
\centering
\includegraphics[scale=0.345]{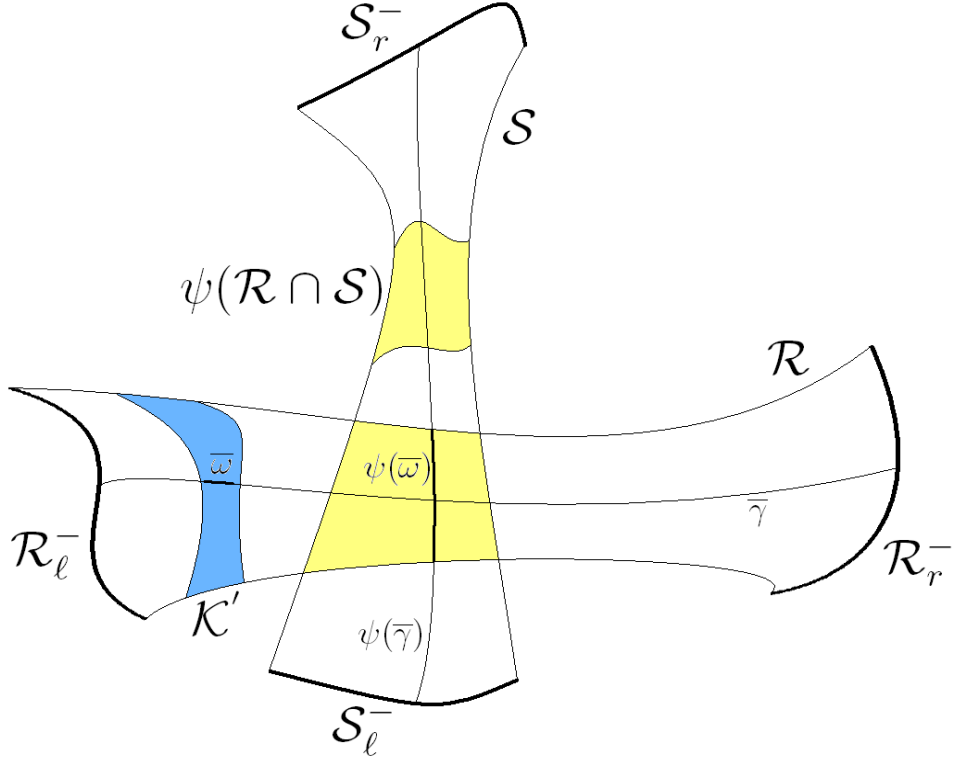}
\caption{\footnotesize{The generalized rectangle $\mathcal R$ is transformed by a continuous planar map $\psi$ onto the generalized rectangle $\mathcal S=\psi(\mathcal R),$ so that, in particular, ${\mathcal S}^-_{\ell}=\psi({\mathcal R}^-_{\ell})$ and ${\mathcal S}^-_{r}=\psi({\mathcal R}^-_{r}).$ The boundary sets ${\mathcal R}^-={\mathcal R}^-_{\ell}\cup {\mathcal R}^-_{r}$ and ${\mathcal S}^-={\mathcal S}^-_{\ell}\cup {\mathcal S}^-_{r}$ are drawn with thicker lines. As it is immediate to verify, for ${\widetilde{\mathcal R}}=(\mathcal R, {\mathcal R}^-)$ and ${\widetilde{\mathcal S}}=(\mathcal S, {\mathcal S}^-),$ it holds that $\psi: {\widetilde{\mathcal R}} \stretchx {\widetilde{\mathcal S}}.$ On the other hand, calling $\,{\widetilde{\!\widetilde{\mathcal R}}}$ the generalized rectangle $\mathcal R$ oriented by choosing $\overline{\vartheta{\mathcal R}\setminus {\mathcal R}^-}$ as $[\cdot]^-$-set, it also holds that
$({\mathcal K}^{'},\psi): {\widetilde{\mathcal R}} \stretchx \,{\widetilde{\!\widetilde{\mathcal R}}},$ where ${\mathcal K}^{'}$ is the subset of $\mathcal R$ depicted with a darker color. However, since $\mathcal R\cap \mathcal S$ is mapped by $\psi$ outside $\mathcal R$ (both $\mathcal R\cap \mathcal S$ and $\psi(\mathcal R\cap \mathcal S)$ are drawn with the same light color), there cannot exist fixed points for $\psi$ in $\mathcal R$ and, a fortiori, neither in ${\mathcal K}^{'}.$ Notice that Theorem \ref{th-fp} does not apply because we have taken two different orientations for $\mathcal R.$
A similar geometrical framework was already considered in \cite[Figure 8]{KeKoYo-01} and \cite[Figure 4]{PiZa-05}.
}}
\label{fig-nf}
\end{figure}

\noindent
As a comment to Theorem \ref{th-fp}, we discuss here its main differences with respect to the more classical Brouwer fixed point Theorem (recalled in the two-dimensional case as Theorem \ref{th-br} in Subsection \ref{sub-cl}).
It is a well-known fact that the fixed point property for continuous maps is preserved by homeomorphisms.
Therefore, it is straightforward to prove the existence of a fixed point for a continuous map $\psi$ if $\psi({\mathcal R})\subseteq {\mathcal R},$ with
${\mathcal R}$ a generalized rectangle of a metric space $X.$
The framework depicted in Theorem \ref{th-fp} is quite different.
Indeed, first of all, the stretching assumption $({\mathcal K},\psi): {\widetilde{\mathcal R}} \stretchx {\widetilde{\mathcal R}}$ does not imply $\psi({\mathcal R})\subseteq {\mathcal R}$ and, secondly, we need $\psi$ to be continuous only on ${\mathcal K}$
and not on the whole set ${\mathcal R}.$ Finally, as already noticed,
we stress that our result also localizes the presence of a
fixed point in the subset ${\mathcal K}.$ From the point of view of the applications, this means that we
are able to obtain a multiplicity of fixed points provided that the stretching property is fulfilled
with respect to pairwise disjoint compact subsets of ${\mathcal R}.$
Indeed, if the condition $(\mathcal D,\psi): {\widetilde{\mathcal R}} \stretchxm {\widetilde{\mathcal R}}$ is satisfied with a crossing number $m\ge 2,$ in view of Theorem \ref{th-fp}, the map $\psi$ has at least a fixed point in each of the compact sets $\mathcal K_i$'s, $i=0,\dots,m-1,$ from Definition \ref{def-cn} and therefore there are at least $m$ fixed points for $\psi$ in $\mathcal R$ (see Theorem \ref{th-per}).

\smallskip

\noindent
A different case, related to the stretching relation from Definition \ref{def-sap}, in which it is possible to find fixed points is when the special geometric configuration in Definition \ref{def-hv} gets realized, as stated in Theorem \ref{th-fpt} below.

\begin{definition}\label{def-hv}
{\rm{
Let ${\widetilde{\mathcal A}}:=({\mathcal A},{\mathcal A}^-)$ and ${\widetilde{\mathcal B}}:=({\mathcal B},{\mathcal B}^-)$
be two oriented rectangles of a metric space $X.$ We say that
${\widetilde{\mathcal A}}$ is a \textit{horizontal slab}\index{slab, horizontal} of ${\widetilde{\mathcal B}}$
and write
$${\widetilde{\mathcal A}} \subseteq_{\,h} \, {\widetilde{\mathcal B}}$$
if ${{\mathcal A}}\subseteq {{\mathcal B}}$ and, either
$${{\mathcal A}}^-_{\ell} \subseteq {{\mathcal B}}^-_{\ell}\,\quad\mbox{and }\quad
{{\mathcal A}}^-_{r} \subseteq {{\mathcal B}}^-_{r}\,,$$
or
$${{\mathcal A}}^-_{\ell} \subseteq {{\mathcal B}}^-_{r}\,\quad\mbox{and }\quad
{{\mathcal A}}^-_{r} \subseteq {{\mathcal B}}^-_{\ell}\,,$$
so that any path in ${\mathcal A}$ joining the two sides of
${\mathcal A}^-$ is also a path in ${\mathcal B}$ and joins the two opposite
sides of ${\mathcal B}^-.$
\\
We say that
${\widetilde{\mathcal A}}$ is a \textit{vertical slab}\index{slab, vertical} of ${\widetilde{\mathcal B}}$
and write
$${\widetilde{\mathcal A}} \subseteq_{\,v} \, {\widetilde{\mathcal B}}$$
if ${{\mathcal A}}\subseteq {{\mathcal B}}$ and
every path in ${\mathcal B}$ joining the two sides of
${\mathcal B}^-$ admits a subpath in ${\mathcal A}$
that joins the two opposite
sides of ${\mathcal A}^-.$
\\
Given three oriented rectangles ${\widetilde{\mathcal A}}:=({\mathcal A},{\mathcal A}^-),$
${\widetilde{\mathcal B}}:=({\mathcal B},{\mathcal B}^-)$ and ${\widetilde{\mathcal E}}:=({\mathcal E},{\mathcal E}^-)$ of the metric space $X,$
with ${\mathcal E}\subseteq {\mathcal A} \cap {\mathcal B},$
we say that ${\widetilde{\mathcal B}}$ \textit{crosses}
${\widetilde{\mathcal A}}$ \textit{in} ${\widetilde{\mathcal E}}$\index{crossing@\textsl{crossing} relation $\pitchfork$} and write
$${\widetilde{\mathcal E}}\in \{ {\widetilde{\mathcal A}}\pitchfork {\widetilde{\mathcal B}} \},$$
if
$${\widetilde{\mathcal E}}\subseteq_{\,h} \, {\widetilde{\mathcal A}}\quad\mbox{and} \quad
{\widetilde{\mathcal E}}\subseteq_{\,v} \, {\widetilde{\mathcal B}}.$$
}}
\end{definition}

\medskip

\noindent
The above definitions, which are adapted from the concept of ``slice'' in \cite{PaZa-04b, PaZa-07} and imitate the classical terminology in \cite[Ch.2.3]{Wi-88},
are topological in nature and therefore do not necessitate any metric assumption (like
smoothness, lipschitzeanity, or similar properties often required in the literature).
We also notice that the terms ``horizontal'' and ``vertical'' are employed
in a purely conventional manner, as it looks clear from Figure \ref{fig-inttv}: the horizontal is the expansive
direction and the vertical is the contractive one (in a quite broad sense). For instance, in \cite{PiZa-07} the terms ``vertical'' and ``horizontal'' were interchanged in regard to the $N$-dimensional setting, but this did not make any difference with respect to the meaning of the results obtained.

\smallskip

\noindent
The next theorem depicts a situation where the ``starting set'' and the ``target set''
of the mapping $\psi$ are two intersecting oriented rectangles. A graphical
illustration of it can be found in Figure \ref{fig-inttv}.

\begin{theorem}\label{th-fpt}
Let $X$ be a metric space and let $\psi: X \supseteq D_{\psi}\to X$ be a map defined on a set $D_{\psi}.$ Assume that
${\widetilde{\mathcal A}}:=({\mathcal A},{\mathcal A}^-)$ and ${\widetilde{\mathcal B}}:=({\mathcal B},{\mathcal B}^-)$ are oriented rectangles
of $X$ and let
${\mathcal K}\subseteq {\mathcal A}\cap D_{\psi}$ be a compact set such that
\begin{equation}\label{eq-intab}
({\mathcal K},\psi): {\widetilde{{\mathcal A}}}\stretchx {\widetilde{{\mathcal B}}}.
\end{equation}
If there exists an oriented rectangle
${\widetilde{\mathcal E}}:=({\mathcal E},{\mathcal E}^-)$ with
${\widetilde{\mathcal E}} \,\in \{\,{\widetilde{\mathcal A}} \pitchfork
{\widetilde{\mathcal B}}\,\},$
then $\psi$ has at least a fixed point in ${\mathcal K}\cap {\mathcal E}.$
\end{theorem}
\begin{proof}
In order to achieve the thesis, we show that
\begin{equation}\label{eq-stt}
(\mathcal K\cap{\mathcal E} ,\psi): {\widetilde{{\mathcal B}}}
\stretchx {\widetilde{{\mathcal B}}}.
\end{equation}
Indeed, let $\gamma$ be a path with ${\overline{\gamma}}\subseteq
\mathcal B$ and ${\overline{\gamma}}\cap {\mathcal
B}^-_{\ell}\ne\emptyset,\; {\overline{\gamma}}\cap {\mathcal
B}^-_{r}\ne\emptyset.$  Then, since ${\widetilde{\mathcal E}}
\subseteq_{\,v} \, {\widetilde{\mathcal B}},$ there exists a subpath $\omega$ of
$\gamma$ such that ${\overline{\omega}}\subseteq {\mathcal E}$ and
${\overline{\omega}}\cap {\mathcal E}^-_{\ell}\ne\emptyset,$
${\overline{\omega}}\cap {\mathcal E}^-_{r}\ne\emptyset.$ Recalling now
that ${\widetilde{\mathcal E}} \subseteq_{\,h} \,
{\widetilde{\mathcal A}},$ it holds that
${\overline{\omega}}\subseteq {\mathcal E} \subseteq {\mathcal A}$
and ${\overline{\omega}}\cap {\mathcal A}^-_{\ell}\ne\emptyset,\;
{\overline{\omega}}\cap {\mathcal A}^-_{r}\ne\emptyset.$ Finally, since
$({\mathcal K},\psi): {\widetilde{\mathcal A}}\stretchx
{\widetilde{\mathcal B}},$ there is a subpath $\eta$ of
$\omega$ such that $\overline{\eta}\subseteq \mathcal K\cap{\mathcal
E} ,\,\psi({\overline{\eta}})\subseteq \mathcal B,$ with
$\psi({\overline{\eta}})\cap {\mathcal B}^-_{\ell}\ne\emptyset,\;
\psi({\overline{\eta}})\cap {\mathcal B}^-_{r}\ne\emptyset.$ In this way
we have proved that any path $\gamma$ with
${\overline{\gamma}}\subseteq \mathcal B$ and ${\overline{\gamma}}\cap
{\mathcal B}^-_{\ell}\ne\emptyset,$ ${\overline{\gamma}}\cap {\mathcal
B}^-_{r}\ne\emptyset$ admits a subpath $\eta$ such that
$\overline{\eta}\subseteq {\mathcal K}\cap{\mathcal E}$ and
$\psi(\overline{\eta})\subseteq {\mathcal B}$ with
$\psi(\overline{\eta})\cap
{\mathcal B}^-_{\ell}\ne\emptyset,\; \psi(\overline{\eta})\cap {\mathcal
B}^-_{r}\ne\emptyset.$
Therefore condition \eqref{eq-stt} has been checked
and the existence of at least a fixed point for $\psi$ in ${\mathcal
K\cap{\mathcal E}}$ follows from
Theorem \ref{th-fp}. Notice that $\psi$ is continuous on $\mathcal K\cap{\mathcal E}$ as, by \eqref{eq-intab}, it is continuous on $\mathcal K.$
\end{proof}

\begin{figure}[htbp]
\centering
\includegraphics[scale=0.4]{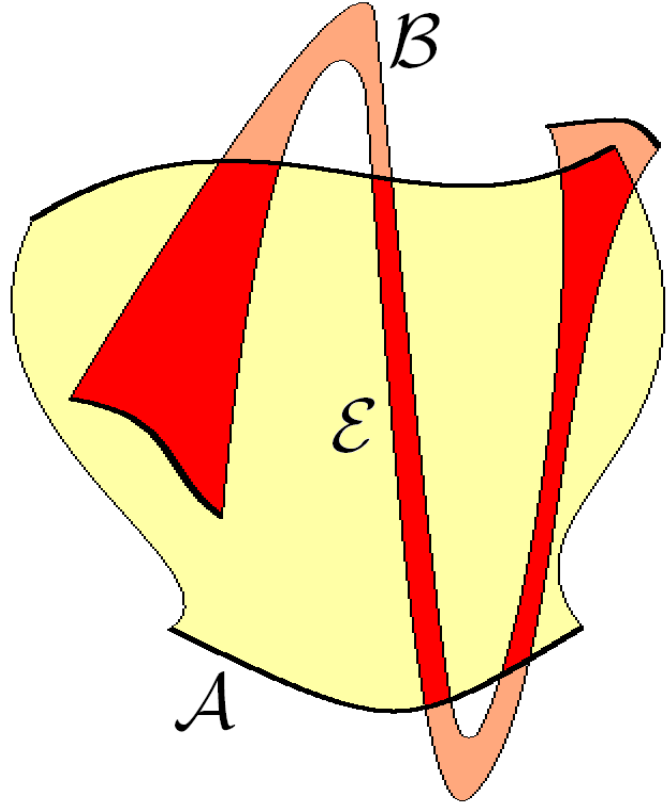}
\caption{\footnotesize{The continuous planar map $\psi$ transforms the generalized rectangle $\mathcal A$ onto the snake-like generalized rectangle $\mathcal B\supseteq\psi(\mathcal A),$ so that ${\mathcal B}^-_{\ell}\supseteq\psi({\mathcal A}^-_{\ell})$ and ${\mathcal B}^-_{r}\supseteq\psi({\mathcal A}^-_{r}).$ The boundary sets ${\mathcal A}^-={\mathcal A}^-_{\ell}\cup {\mathcal A}^-_{r}$ and ${\mathcal B}^-={\mathcal B}^-_{\ell}\cup {\mathcal B}^-_{r}$ have been drawn with thicker lines. Clearly, for ${\widetilde{\mathcal A}}=(\mathcal A, {\mathcal A}^-)$ and ${\widetilde{\mathcal B}}=(\mathcal B, {\mathcal B}^-),$ it holds that $\psi: {\widetilde{\mathcal A}} \stretchx {\widetilde{\mathcal B}}.$ Notice that we do not require that the end sets ${\mathcal B}_{\ell}$ and ${\mathcal B}_{r}$ of ${\mathcal B}$ lie outside ${\mathcal A},$ differently from the approaches based on degree theory, as discussed in \cite{PiZa-07Ts}.
Among the three intersections between
${\mathcal A}$ and ${\mathcal B},$ only the central one, that we call ${\mathcal E},$ corresponds to a crossing in the sense of
Definition \ref{def-hv}, since ${\widetilde{\mathcal E}} \subseteq_{\,h} \,
{\widetilde{\mathcal A}}$ and ${\widetilde{\mathcal E}} \subseteq_{\,v} \,
{\widetilde{\mathcal B}}.$ Therefore, Theorem \ref{th-fpt}
ensures the existence of at least a fixed point for $\psi$ in ${\mathcal E}.$ Even if the
intersection between ${\mathcal A}$ and ${\mathcal B}$ on the right is ``not far'' from belonging to  $\{{\widetilde{\mathcal A}}\pitchfork {\widetilde{\mathcal B}}\},$ it is possible to define $\psi$ so that it has no fixed points therein. This shows that our fixed point theorems are, in some sense, ``sharp''. See \cite{PiZa-07Ts} for more details and corresponding examples.
}}
\label{fig-inttv}
\end{figure}

\clearpage

\noindent
In regard to Theorem \ref{th-fpt}, we stress that if \eqref{eq-intab} holds and there exist $m\ge 2$ pairwise disjoint oriented rectangles
\begin{equation}\label{eq-intt}
{\widetilde{\mathcal E}}_0,\dots,{\widetilde{\mathcal E}}_{m-1} \,\in \{\,{\widetilde{\mathcal A}} \pitchfork {\widetilde{\mathcal B}}\,\},
\end{equation}
then the map $\psi$ has at least a fixed point in $\mathcal K\cap{\mathcal E}_i\,,\,\forall\, i=0,\dots,m-1.$ Hence, similarly to Definition \ref{def-cn}, also in this framework it is possible to find a multiplicity of fixed points, obtaining conclusions analogous to the ones in Theorem \ref{th-per} below.

\medskip

\noindent
If in place of fixed points we are concerned with the search of periodic points of any period (as in the following Theorems \ref{th-per}--\ref{th-comp}), then, in order to apply our stretching along the paths method, we need to check the stretching relation to be preserved under composition of maps. This fact can be easily proved by induction: the basic step is the content of the next lemma.
\begin{lemma}\label{lem-comp}
Let $X$ be a metric space and let $\varphi: X \supseteq D_{\varphi}\to X$ and $\psi: X \supseteq D_{\psi}\to X$ be maps defined on the sets $D_{\varphi}$ and $D_{\psi},$ respectively. Assume that
${\widetilde{\mathcal A}}:= ({\mathcal A},{\mathcal A}^-),\,
{\widetilde{\mathcal B}}:= ({\mathcal B},{\mathcal B}^-)$ and ${\widetilde{\mathcal C}}:= ({\mathcal C},{\mathcal C}^-)$
are oriented rectangles of $X.$
If
${\mathcal H}\subseteq {\mathcal A}\cap D_{\varphi}$ and ${\mathcal K}\subseteq {\mathcal B}\cap D_{\psi}$ are compact sets such that
$$({\mathcal H},\varphi): {\widetilde{\mathcal A}} \stretchx {\widetilde{\mathcal B}}\quad  \mbox{ and } \quad ({\mathcal K},\psi): {\widetilde{\mathcal B}}
\stretchx {\widetilde{\mathcal C}},$$
then it follows that
$$\left({\mathcal H}\cap \varphi^{-1}(\mathcal K),\psi\circ\varphi\right): {\widetilde{\mathcal A}} \stretchx {\widetilde{\mathcal C}}.$$
\end{lemma}
\begin{proof}
Let $\gamma: [0,1]\to \mathcal A$ be a path such that $\gamma(0)$ and $\gamma(1)$ belong to the different sides of ${\mathcal A}^-.$
Then, since $({\mathcal H},\varphi): {\widetilde{\mathcal A}} \stretchx {\widetilde{\mathcal B}},$
there exists a subinterval $[t',t'']\subseteq [0,1]$ such that
$$\gamma(t)\in {\mathcal H},\quad \varphi(\gamma(t))\in {\mathcal B}\,,\;\;\forall\, t\in [t',t'']$$
and, moreover, $\varphi(\gamma(t'))$ and $\varphi(\gamma(t''))$ belong to different components of
${\mathcal B}^-.$ Let us call $\omega$ the restriction of $\gamma$ to $[t',t'']$ and define $ \nu:[t',t'']\to \mathcal B$ as $\nu:=\varphi\circ\omega.$ Notice that $\nu(t')$ and $\nu(t'')$ belong to the different sides of ${\mathcal B}^-$ and so, by the stretching hypothesis $({\mathcal K},\psi): {\widetilde{\mathcal B}} \stretchx {\widetilde{\mathcal C}},$ there is a subinterval $[s',s'']\subseteq[t',t'']$ such that
$$\nu(t)\in {\mathcal K},\quad \psi(\nu(t))\in {\mathcal C}\,,\;\;\forall\, t\in [s',s'']\,$$
with $\psi(\nu(s'))$ and $\psi(\nu(s''))$ belonging to different components of ${\mathcal C}^-.$ Rewriting all in terms of $\gamma,$ this means that we have found a subinterval $[s',s'']\subseteq [0,1]$ such that $$\gamma(t)\in {\mathcal H}\cap \varphi^{-1}(\mathcal K),\quad  \psi(\varphi(\gamma(t)))\in {\mathcal C}\,,\;\;\forall\, t\in [s',s'']$$ and $\psi(\varphi(\gamma(s')))$ and $\psi(\varphi(\gamma(s'')))$ belong to the different sides of ${\mathcal C}^-.$ By the arbitrariness of the path $\gamma,$ the stretching property $$\left({\mathcal H}\cap \varphi^{-1}(\mathcal K),\psi\circ\varphi\right): {\widetilde{\mathcal A}} \stretchx {\widetilde{\mathcal C}}$$
is thus fulfilled. We just point out that the continuity of the composite mapping $\psi\circ\varphi$ on the compact set ${\mathcal H}\cap \varphi^{-1}(\mathcal K)$ follows from the continuity of $\varphi$ on $\mathcal H$ and of $\psi$ on $\mathcal K,$ respectively.
\end{proof}

\smallskip
\noindent
\begin{theorem}\label{th-per}
Let $X$ be a metric space and $\psi: X \supseteq D_{\psi}\to X$ be a map defined on a set $D_{\psi}.$ Assume that
${\widetilde{\mathcal R}}:= ({\mathcal R},{\mathcal R}^-)$
is an oriented rectangle $X.$ If ${\mathcal K_0},\dots,{\mathcal K_{m-1}} $
are $m\ge 2$ pairwise disjoint compact subsets of ${\mathcal R}\cap D_{\psi}$ and
$$({\mathcal K}_i,\psi): {\widetilde{\mathcal R}} \stretchx {\widetilde{\mathcal R}}, \mbox{ for } i=0,\dots,m-1,$$
then the following conclusions hold:
\begin{itemize}
\item The map $\psi$ has at least a fixed point in ${\mathcal K}_i,\,i=0,\dots,m-1;$
\item For each two-sided sequence
$(s_{h})_{h\in{\mathbb Z}}\in \{0,\dots,m-1\}^{\mathbb Z},$
there exists a sequence of points $(x_{h})_{h\in {\mathbb Z}}$
such that
$\psi(x_{h-1}) = x_{h}\in {\mathcal K}_{s_{h}}\,,\,\forall\, h\in {\mathbb Z}\,;$
\item For each sequence \textbf{s} $=(s_n)_{n}\in \{0,\dots,m-1\}^{\mathbb N},$
there exists a compact connected
set ${\mathcal C}_{\mbox{\textbf{s}}}\subseteq {\mathcal K}_{s_0}$ satisfying
$${\mathcal C}_{\mbox{\textbf{s}}}\cap {\mathcal R}^+_{d}\ne\emptyset,
\quad
{\mathcal C}_{\mbox{\textbf{s}}}\cap {\mathcal R}^+_{u}\ne\emptyset
$$
and such that
$\psi^{i}(x)\in {\mathcal K}_{s_i}\,,\;\forall\, i\geq 1,\;\;\forall \, x\in {\mathcal C}_{\mbox{\textbf{s}}}\,;$
\item Given an integer $j\ge 2$ and a $j+1$-uple $(s_0,\dots,s_j),\, s_i\in \{0,\dots,m-1\},$ for $i=0,\dots,j,$ and $s_0=s_j,$
then there exists a point $w\in{\mathcal K_{s_0}}$ such that
$${\psi}^{i}(w)\in {\mathcal K}_{s_i},\,\forall i=1,\dots,j \quad \mbox{ and } \quad {\psi}^{j}(w)=w.$$
\end{itemize}
\end{theorem}

\noindent
As we shall see in Section \ref{sec-de}, the previous result turns out to be our fundamental tool in the proof of Theorem \ref{th-ch} about chaotic dynamics. On the other hand, Theorem \ref{th-per} can be viewed as a particular case of Theorem \ref{th-comp} below. The proof of the former result is thus postponed since it comes as a corollary of the latter more general theorem. \\
Before stating Theorem \ref{th-comp}, we just make an observation that will reveal its significance in Section \ref{sec-sd} when dealing with symbolic dynamics.

\begin{remark}\label{rem-mp}
{\rm{We observe that in the hypotheses of Theorem \ref{th-per}, or equivalently when we enter the framework of Definition \ref{def-cn} with $\widetilde{\mathcal A}=\widetilde{\mathcal B},$ i.e., when there exist $m\ge 2$ pairwise disjoint compact subsets ${\mathcal K}_0,\dots,{\mathcal K}_{m-1}$ of an oriented rectangle $\mathcal A\subseteq X,$ for which $(\mathcal K_i,\psi):\widetilde{\mathcal A}\stretchx \widetilde{\mathcal A}$ holds, then it is possible to find $m$ pairwise disjoint vertical slabs $\widetilde{\mathcal R}_i$ of $\widetilde{\mathcal A}$ such that $\mathcal R_i\supseteq\mathcal K_i,$ for $i=0,\dots,m-1.$ Indeed, by Theorem \ref{th-per}, we know that any $\mathcal K_i$ contains a compact connected set $\mathcal C_i$ (actually, infinitely many) joining ${\mathcal A}^+_{d}$ and ${\mathcal A}^+_{u}.$ Thus, the idea is to ``fill'' and ``fatten'' each $\mathcal C_i$ up, in order to obtain a compact set $\mathcal R_i$ that still joins ${\mathcal A}^+_{d}$ and ${\mathcal A}^+_{u},$ but homeomorphic to the unit square of $\mathbb R^2$ and containing only $\mathcal K_i$ among all the $\mathcal K_j$'s. Moreover we require ${\mathcal A}^+_{d}\cap\mathcal R_i$ and ${\mathcal A}^+_{d}\cap\mathcal R_i$ to be arcs. Notice that this is possible because the $\mathcal K_i$'s are compact and disjoint. To such generalized rectangles $\mathcal R_i$'s we give the orientation ``inherited'' from $\mathcal A,$ that is, we set ${\mathcal R}^i_{d}:={\mathcal A}^+_{d}\cap\mathcal R_i$ and ${\mathcal R}^i_{u}:={\mathcal A}^+_{u}\cap\mathcal R_i,$ for $i=0,\dots,m-1,$
where we have denoted by
${\mathcal R}^{i}_{d}$ and ${\mathcal R}^{i}_{u}$ the two sides of ${\mathcal R}_i^+.$ Indicating with ${\mathcal R}^{i}_{\ell}$ and ${\mathcal R}^{i}_{r}$ the two parts of ${\mathcal R}_i^-\,,$ we find that they coincide with
the two components of $\overline{\vartheta\mathcal R_i\setminus ({\mathcal R}^i_{d}\cup{\mathcal R}^i_{u})}.$ In particular, we can name such sets following the cyclic order $\ell-d-r-u-\ell.$ As usual we put $\widetilde{\mathcal R}_i=(\mathcal R_i, \mathcal R_i^-)$ and, by construction, these are the desired vertical slabs of $\widetilde{\mathcal A}.$ See Figure \ref{fig-fat} for a graphical illustration.\\
If in addition the map $\psi$ is continuous on $\mathcal A,$ we claim that
\begin{equation*}
\psi:\widetilde{\mathcal R}_i\stretchx \widetilde{\mathcal R}_j,\,\forall\, i,j\in \{0,\dots,m-1\}.
\end{equation*}
At first we notice that $\psi$ is continuous on each $\mathcal R_i$ because $\bigcup_{i=0}^{m-1}\mathcal R_i\subseteq\mathcal A.$ Moreover, for any $i=0,\dots,m-1$ and for every path $\gamma: [a,b]\to {\mathcal R}_i,$ with
$\gamma(a)$ and $\gamma(b)$ belonging to different components of ${\mathcal R}_i^-,$
there exists a subinterval
$[t',t'']\subseteq [a,b]$ such that
$\gamma(t)\in {\mathcal K}_i$ and $\psi(\gamma(t))\in {\mathcal A},\,\forall\, t\in [t',t''],$ with $\psi(\gamma(t'))$ and
$\psi(\gamma(t''))$ belonging to different sides of ${\mathcal A}^-.$ This follows from the fact that $\gamma$ can be extended to a path $\gamma^{*}: [a',b']\to {\mathcal A},$ with $[a',b']\supseteq [a,b],$ such that $\gamma^{*}\restriction_{[a,b]}=\gamma$ and $\gamma^{*}(a'),$ $\gamma^{*}(b')$ belong to different sides of ${\mathcal A}^-,$ and by
recalling that $(\mathcal K_i,\psi):\widetilde{\mathcal A}\stretchx \widetilde{\mathcal A}.$
Then, since $\bigcup_{i=0}^{m-1}\mathcal R_i\subseteq\mathcal A,$ for any fixed $j= 0,\dots,m-1,$ there exists a subinterval $[s',s'']\subseteq [t',t'']$ such that $\psi(\gamma(t))\in {\mathcal R}_j,$ $\forall\, t\in [s',s''],$ with $\psi(\gamma(s'))$ and
$\psi(\gamma(s''))$ belonging to different sides of ${\mathcal R}_j^-.$ But this means that
$\psi:\widetilde{\mathcal R}_i\stretchx \widetilde{\mathcal R}_j.$ By the arbitrariness of $i,j\in\{0,\dots,m-1\},$ the proof of our claim is complete.
}}

\hfill$\lhd$\\
\end{remark}

\clearpage

\begin{figure}[htbp]
\centering
\includegraphics[scale=0.35]{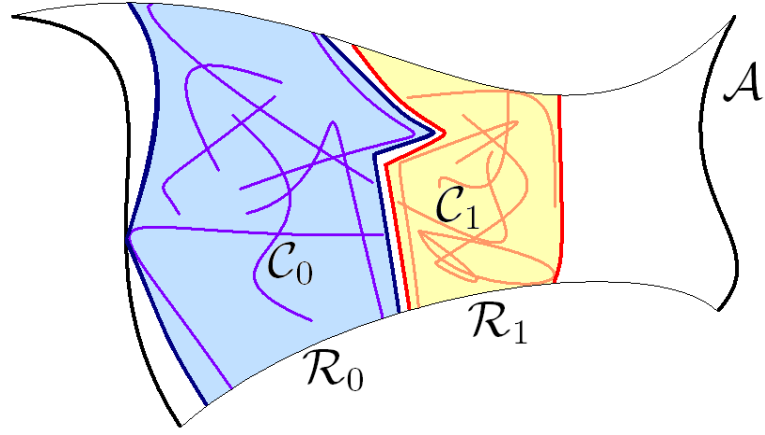}
\caption{\footnotesize{ With reference to the situation described in Remark \ref{rem-mp} for $m=2,$ we have depicted the generalized rectangle $\mathcal A,$ that we orientate by choosing as ${\mathcal A}^-$ the boundary set drawn with thicker lines. The continua $\mathcal C_0$ and $\mathcal C_1,$ contained respectively in the disjoint compact sets $\mathcal K_0$ and $\mathcal K_1$ (not represented in the picture, in order not to overburden it) and joining the two sides of ${\mathcal A}^+$, are embedded (together with the $\mathcal K_i$'s) in the disjoint generalized rectangles ${\mathcal R}_0$ and ${\mathcal R}_1,$ whose $[\cdot]^{-}$-sets have been indicated with thicker lines. With this choice, $\widetilde{\mathcal R}_0$ and $\widetilde{\mathcal R}_1$ are vertical slabs of $\widetilde{\mathcal A}\,.$
}}
\label{fig-fat}
\end{figure}

\begin{theorem}\label{th-comp}
Assume there is a double sequence of oriented rectangles $({\widetilde{\mathcal R}_i})_{i\in {\mathbb Z}}\,\,($with ${\widetilde{\mathcal R}_i} = ({\mathcal R}_i,{\mathcal R}^-_i)\,)$ of a metric space $X$ and a sequence $(({\mathcal K}_i,\psi_i))_{i\in {\mathbb Z}}\,,$
with ${\mathcal K}_i \subseteq {\mathcal R}_i$ compact sets, such that
$$({\mathcal K}_i,\psi_i): {\widetilde{\mathcal R}_i}  \stretchx {\widetilde{\mathcal R}_{i+1}}\,,\quad\forall\, i\in {\mathbb Z}.$$
Let us denote by
${\mathcal R}^{i}_{\ell}$ and ${\mathcal R}^{i}_{r}$ the two components of ${\mathcal R}_i^-$ and by ${\mathcal R}^{i}_{d}$ and ${\mathcal R}^{i}_{u}$ the two components of ${\mathcal R}_i^+\,.$
Then the following conclusions hold:
\begin{itemize}
\item There is a sequence $(w_k)_{k\in {\mathbb Z}}$ such that $w_k\in {\mathcal K}_k$
and $\psi_k(w_k) = w_{k+1}\,,$ for all ${k\in {\mathbb Z}}\,;$
\item For each $j\in {\mathbb Z}$ there exists a compact connected set ${\mathcal C}_j\subseteq {\mathcal K}_j$
 satisfying
$${\mathcal C}_j\cap {\mathcal R}^j_{d}\ne\emptyset,
\quad
{\mathcal C}_j\cap {\mathcal R}^j_{u}\ne\emptyset
$$
and such that, for every $w\in {\mathcal C}_j\,,$
there is a sequence $(y_i)_{i\geq j}$ with $y_j = w$ and
$$y_i \in {\mathcal K}_i\,,\quad \psi_{i}(y_i) = y_{i+1}\,,\;\forall\, i\geq j\,;$$
\item If there are integers $h$ and $l,$ with $h < l,$ such that ${\widetilde{\mathcal R}_h} = {\widetilde{\mathcal R}_{l}}\,,$
then there exists a finite sequence $(z_i)_{h\leq i\leq l-1}\,,$ with $z_i\in {\mathcal K}_i$ and
$\psi_i(z_i) = z_{i+1}$ for each $i=h,\dots,l-1,$ such that $z_{l} = z_h\,,$ that is,
$z_h$ is a fixed point of $\psi_{l-1}\circ\dots\circ\psi_{h}$ in ${\mathcal K}_h.$
\end{itemize}
\end{theorem}
\begin{proof}
We prove the conclusions of the theorem in the reverse order. So,
let us start with the verification of the last assertion. By the assumptions
and by Definition \ref{def-sap}, it
is easy to check that
\begin{equation}\label{eq-comp}
({\mathcal H},\psi_{l-1}\circ\dots\circ\psi_h): {\widetilde{\mathcal R}_h}
\stretchx {\widetilde{\mathcal R}_{l}},
\end{equation}
where
$$\mathcal H:=\{z\in \mathcal K_h:\psi_{i}\circ\dots\circ\psi_h(z)\in\mathcal K_{i+1},\,\forall i=h,\dots,l-1\}.$$
With the positions $\widetilde {\mathcal R}={\widetilde{\mathcal R}_h} =
{\widetilde{\mathcal R}_{l}}$ and $\phi=\psi_{l-1}\circ\dots\circ\psi_h,$
we read condition \eqref{eq-comp} as $({\mathcal H},\phi): {\widetilde{\mathcal R}}
\stretchx {\widetilde{\mathcal R}}$ and therefore the thesis follows immediately by Theorem \ref{th-fp}.\\
As regards the second conclusion, without loss of generality, we can assume
$j=0.$ Let us define
the closed set
\begin{equation*}
\mathcal S:=\{z\in {\mathcal K}_0: \psi_j\circ\dots\circ\psi_0(z)\in
{\mathcal K}_{j+1},\,\forall j\geq 0\}
\end{equation*}
and fix a path $\gamma_0:[0,1]\to \mathcal R_0$ such that $\gamma_0(0)$ and $\gamma_0(1)$ belong to the different components of ${\mathcal R}_0^-.$ Then, since
$({\mathcal K}_0,\psi_0): {\widetilde{\mathcal R}_0} \stretchx
{\widetilde{\mathcal R}_{1}},$ there exists a
subinterval
$$[t'_1,t''_1]\subseteq [t'_0,t''_0]:=[0,1]$$
such that
$$\gamma_0(t)\in {\mathcal K}_{0}\quad\mbox{and } \;
\gamma_1(t):=\psi_0(\gamma_0(t))\in {\mathcal R}_1,\;\;\forall\,t\in [t'_1,t''_1].$$
By the same assumption, we also have that
$\psi_0(\gamma_0(t'_1))$ and $\psi_0(\gamma_0(t''_1))$ belong to different components of
${\mathcal R}^-_1.$ Similarly,
there exists a
subinterval
$$[t'_2,t''_2]\subseteq [t'_1,t''_1]$$
such that
$$\gamma_1(t)\in {\mathcal K}_{1}\quad\mbox{and } \;
\gamma_2(t):=\psi_1(\gamma_1(t))\in {\mathcal R}_2,\;\;\forall\,t\in [t'_2,t''_2],$$
with
$\psi_1(\gamma_1(t'_2))$ and $\psi_1(\gamma_1(t''_2))$ belonging to the different components of
${\mathcal R}^-_2.$
Defining
$$\Gamma_2:=\{x\in{\gamma}_0([t'_1,t''_1]):\psi_0(x)\in{\gamma}_1([t'_2,t''_2])\}\subseteq \{z\in {\mathcal K}_0:\psi_0(z)\in
{\mathcal K}_{1}\}$$
and proceeding by induction, we can find a decreasing sequence of
nonempty compact sets
$$\Gamma_0:={\gamma}_0([t'_0,t''_0])\supseteq\Gamma_1:={\gamma}_0([t'_1,t''_1])
\supseteq\Gamma_2\supseteq\dots\supseteq\Gamma_n\supseteq\Gamma_{n+1}\supseteq\dots$$
such that
$\psi_j\circ\dots\circ\psi_0(\Gamma_{j+1})\subseteq
{\mathcal R}_{j+1},$ with  $\psi_j\circ\dots\circ\psi_0(\Gamma_{j+1})\cap
{\mathcal R}_{\ell}^{j+1}\ne\emptyset$ and
$\psi_j\circ\dots\circ\psi_0(\Gamma_{j+1})\cap
{\mathcal R}_{r}^{j+1}\ne\emptyset,$ for $j\geq 0.$
Moreover, for every $i\geq 1,$ we have that
$$\Gamma_{i+1}\subseteq \{z\in {\mathcal K}_0:
\psi_{j-1}\circ\dots\circ\psi_0(z)\in
{\mathcal K}_{j},\,\forall \,j=1,\dots, i\}.$$
As it is straightforward to see,
${\displaystyle{\cap_{j=0}^{+\infty}\,\Gamma_j\ne\emptyset}}$ and for any
${\displaystyle{z\in\cap_{j=0}^{+\infty}\,\Gamma_j}}$ it holds that
$\psi_n\circ\dots\circ\psi_0(z)\in\mathcal K_{n+1},\,\forall
n\in\mathbb N.$ In this way we have shown that any path
$\gamma_0$ joining in $\mathcal R_0$ the two sides of
${\mathcal R}_0^{-}$ intersects $\mathcal S.$
The existence of the connected compact set ${\mathcal C}_0
\subseteq \mathcal S\subseteq {\mathcal K}_0$ joining the two components of $\mathcal R_0^{+}$ comes from Lemma \ref{lem-cr}. By the definition of $\mathcal S$ it is obvious that any point of $\mathcal C_0\subseteq\mathcal S$ generates a sequence as required in the statement of the theorem.\\
The first conclusion follows now by a standard diagonal argument (see, e.g., \cite[Proposition 5]{KeKoYo-01} and
\cite[Theorem 2.2]{PaZa-04a}), which allows to extend the result to bi-infinite sequences once it has been proved for one-sided sequences.
\end{proof}

\noindent {\textit{Proof of Theorem \ref{th-per}.}}
The first conclusion easily follows from Theorem \ref{th-fp}. As regards the remaining ones, given a two-sided sequence
$(s_{h})_{h\in{\mathbb Z}}\in \{0,\dots,m-1\}^{\mathbb Z},$ they can be achieved by applying Theorem \ref{th-comp} with the positions $\mathcal R_i=\mathcal R,$ $\psi_i=\psi$ and $\mathcal K_i=\mathcal K_{s_i}\,,\,\forall i\in \mathbb Z.$ The details are omitted since they are of straightforward verification.
\eop

\bigskip

We end this subsection with the presentation and discussion of some stretching relations alternative to the one in Definition \ref{def-sap}. In particular we will ask ourselves if the results on the existence and localization of fixed points are still valid with respect to these new concepts.

\smallskip

\noindent
For all the next definitions, the basic setting concerns the following framework:
{\em
Let $X$ be a metric
space. Assume $\psi: X \supseteq D_{\psi}\to X$ is a map
defined on a set $D_{\psi}$ and let ${\widetilde{\mathcal A}}:=
({\mathcal A},{\mathcal A}^-)$ and ${\widetilde{\mathcal B}}:=
({\mathcal B},{\mathcal B}^-)$ be oriented rectangles of $X.$ Let also $\mathcal K\subseteq \mathcal A\cap D_{\psi}$ be a compact set.
}
\begin{definition}\label{def-sac}
{\rm{
We say that $({\mathcal K},\psi)$ \textit{stretches} ${\widetilde{\mathcal A}}$
\textit{to} ${\widetilde{\mathcal B}}$ \textit{along the continua}\index{stretching along the continua@\textsl{stretching along the continua}} and write
$$({\mathcal K},\psi): {\widetilde{\mathcal A}}\stretchcont {\widetilde{\mathcal B}},$$
if the following conditions hold:
\begin{itemize}
\item{} $\psi$ is continuous on ${\mathcal K}\, ;$
\item{} For every continuum $\Gamma\subseteq {\mathcal A}$
with $\Gamma\cap {\mathcal A}^-_{\ell}\ne\emptyset$ and $\Gamma\cap {\mathcal A}^-_{r}\ne\emptyset,$
there exists a continuum $\Gamma'\subseteq \Gamma\cap {\mathcal K}$ such that
$\psi(\Gamma')\subseteq {\mathcal B}$
and
$$\psi(\Gamma')\cap {\mathcal B}^-_{\ell}\ne\emptyset\,,
\quad \psi(\Gamma')\cap {\mathcal B}^-_{r}\ne\emptyset.$$
\end{itemize}
}}
\end{definition}
\noindent

\smallskip

\noindent
While the above definition is based on the one given by Kennedy and Yorke in \cite{KeYo-01},
the following bears some resemblances to that of ``family of expanders'' considered
in \cite{KeKoYo-01}.

\begin{definition}\label{def-ea}
{\rm{
We say that $({\mathcal K},\psi)$ \textit{expands} ${\widetilde{\mathcal A}}$
\textit{across} ${\widetilde{\mathcal B}}$ \index{expanding across@\textsl{expanding across}$\stretchkky$} and write
$$({\mathcal K},\psi): {\widetilde{\mathcal A}}\stretchkky {\widetilde{\mathcal B}},$$
if the following conditions hold:
\begin{itemize}
\item{} $\psi$ is continuous on ${\mathcal K}\,;$
\item{} For every continuum $\Gamma\subseteq {\mathcal A}$
with $\Gamma\cap {\mathcal A}^-_{\ell}\ne\emptyset$ and $\Gamma\cap {\mathcal A}^-_{r}\ne\emptyset,$
there exists a nonempty compact set $P\subseteq \Gamma\cap{\mathcal K}$ such that
$\psi(P)$ is a continuum contained in ${\mathcal B}$ and
$$\psi(P)\cap {\mathcal B}^-_{\ell}\ne\emptyset\,,
\quad \psi(P)\cap {\mathcal B}^-_{r}\ne\emptyset.$$
\end{itemize}
}}
\end{definition}

\noindent
When it is possible to take ${\mathcal K} = {\mathcal A},$
we simply write
$\psi: {\widetilde{\mathcal A}}\stretchcont {\widetilde{\mathcal B}}$ and $\psi: {\widetilde{\mathcal A}}\stretchkky {\widetilde{\mathcal B}}$
in place of
$({\mathcal A},\psi): {\widetilde{\mathcal A}}\stretchcont {\widetilde{\mathcal B}}$ and $({\mathcal A},\psi): {\widetilde{\mathcal A}}\stretchkky {\widetilde{\mathcal B}},$ respectively.\\
In analogy to Definition \ref{def-cn}, one can define the variants $\stretchcontm$ and $\stretchkkym,$ when multiple coverings occur.

\smallskip

\noindent
As we shall see in a moment, the above relations behave in a different way with respect to the possibility of detecting fixed points. Indeed, since the ranges of paths are a particular kind of continua, one can follow the same steps as in the proof of Theorem \ref{th-fp} and show that an analogous result still holds when the property of stretching along the paths is replaced with the one of stretching along the continua (cf. \cite[Theorem 2.10]{PiZa-07Ts}). On the other hand, this is no more true when the relation in Definition \ref{def-ea} is fulfilled, since in general it guarantees neither the existence of fixed points nor their localization (in case that fixed points do exist).\\
Concerning the existence of fixed points,
a possible counterexample is described in Figure
\ref{fig-bh}
and is inspired to the
bulging horseshoe in \cite[Fig. 4]{KeKoYo-01}. For sake of conciseness,
we prefer to present it by means of a series of graphical illustrations in Figures \ref{fig-bh}--\ref{fig-lr2}:
we point out, however, that it is
based on a concrete definition of a planar map (whose form, although complicated,
can be explicitly given in analytical terms).\\
With respect to the localization of fixed points, a counterexample can instead be obtained
by suitably adapting a one-dimensional map to the planar case. Indeed,
if $f: {\mathbb R}\supseteq [0,1]\to {\mathbb R}$
is any continuous function, we can set
\begin{equation}\label{eq-psi}
\psi(x_1,x_2):= (f(x_1),x_2)
\end{equation}
and have a continuous planar map defined on the unit square $[0,1]^2$ of $\mathbb R^2,$
inheriting all the interesting properties of $f.$
Notice that, in this special case, any fixed point $x^*$ for $f$ generates
a vertical line $(x^*,s)$ (with $s\in [0,1]$) of fixed points for $\psi.$
The more general framework of a map $\psi$ defined as
$$\psi(x_1,x_2):= (f(x_1),g(x_2)),$$
for $g:[0,1]\to[0,1]$ a continuous function,
could be considered as well.
\\
In view of the above discussion, we define a continuous map
$f: [0,1]\to [0,1]$ of the form
\begin{equation}\label{eq-f}
f(s):=\; \left\{
\begin{array}{lll}
&\frac{1 - c}{a}\,s + c\quad &0\leq s < a,\\ \\
&\frac{1}{a-b}(s-b)\quad &a\leq s \leq b,\\ \\
&\frac{d}{1-b}(s-b)\quad &b < s \leq 1,\\
\end{array}
\right.
\end{equation}
where $a,b,c,d$ are fixed constants such that $0 < a < b < 1$
and $0 < c < d < 1.$

\smallskip

\noindent
For ${\widetilde{\mathcal A}}$ the unit square oriented in the standard
left-right manner and $\psi$ as in \eqref{eq-psi},
it holds that $\psi: {\widetilde{\mathcal A}}\stretchkky {\widetilde{\mathcal A}}.$ In particular, we
can write both
$$({\mathcal K}_0,\psi): {\widetilde{\mathcal A}}\stretchkky {\widetilde{\mathcal A}},
\quad\mbox{for } \, {\mathcal K}_0 = [a,b]$$
and
$$({\mathcal K}_1,\psi): {\widetilde{\mathcal A}}\stretchkky {\widetilde{\mathcal A}},
\quad\mbox{for } \, {\mathcal K}_1 = [0,a]\cup [b,1].$$
In the former case, we also have
$$({\mathcal K}_0,\psi): {\widetilde{\mathcal A}}\stretchx {\widetilde{\mathcal A}}$$
and therefore, consistently with Theorem \ref{th-fp}, there exists at least a fixed point for $\psi$ in ${\mathcal K}_0\,.$
On the other hand, as it is clear from Figure
\ref{fig-gr},
there are no fixed points for $\psi$ in ${\mathcal K}_1\,.$ Thus, the localization of the
fixed points is not guaranteed when only the relation $\stretchkky$ is satisfied.
The same example could be slightly modified in order to have that $\psi: {\widetilde{\mathcal A}}\stretchkkym {\widetilde{\mathcal A}}$ for an arbitrary $m\geq 2,$
but just one fixed point does exist.\\
We finally observe that, playing with the coefficients $a,b,c,d$ and choosing a suitable compact set
${\mathcal K}_2\subset [0,1],$ it is possible to have
$$({\mathcal K}_2,\psi): {\widetilde{\mathcal A}}\stretchkky {\widetilde{\mathcal A}},$$
for a case in which neither $\psi$ nor $\psi^2$ possess fixed points in ${\mathcal K}_2\,.$
The set ${\mathcal K}_2$ will consist of the union of some compact subintervals of ${\mathcal K}_0$
and ${\mathcal K}_1\,.$ A graphical representation of $f$ and $f^2$ is given in Figure \ref{fig-gr}.

\bigskip

\begin{figure}[ht]
\centering
\includegraphics[scale=0.22]{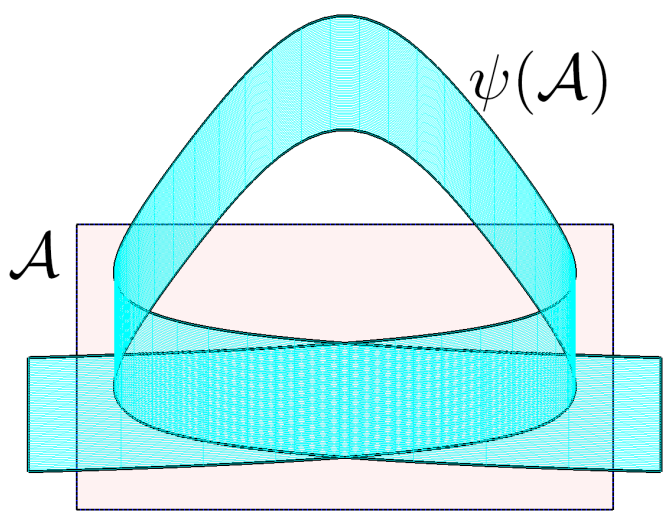}
\caption{\footnotesize {A rectangle $\mathcal A,$ oriented by taking as
$[\cdot]^-$-set the union of its left and right vertical segments, is deformed by a continuous planar map $\psi$
onto $\psi(\mathcal A),$ a ribbon bent across the rectangle itself. The left and
right sides of the domain are homeomorphically transformed onto
the two endings of the ribbon. This is an example of 
$\psi:{\widetilde{\mathcal A}} \stretchkky {\widetilde{\mathcal A}}\,.$
}}
\label{fig-bh}
\end{figure}

\begin{figure}[ht]
\centering
\includegraphics[scale=0.35]{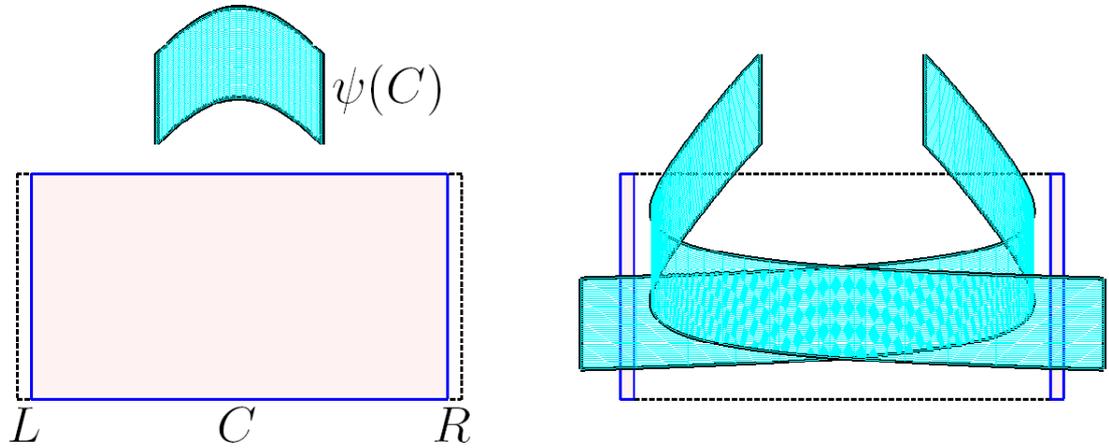}
\caption{\footnotesize{In regard to the example portrayed in the previous picture, we emphasize
how some sub-rectangles of $\mathcal A$ are transformed by the map $\psi.$ In particular,
the main central part $C$ of the rectangle is pushed out of $\mathcal A$ (picture at the left),
while the two narrow rectangles $L$ and $R$ near the left and right sides of it, respectively, are stirred onto two overlapping bent strips
(picture at the right). This can happen in two different ways, as shown in Figures \ref{fig-lr1}--\ref{fig-lr2}.
One can pass from a configuration to the other
by considering, instead of the map $\psi,$ the related map
$(x_1,x_2)\mapsto \psi(-x_1,x_2).$
}}
\label{fig-c}
\end{figure}

\begin{figure}[ht]
\centering
\includegraphics[scale=0.35]{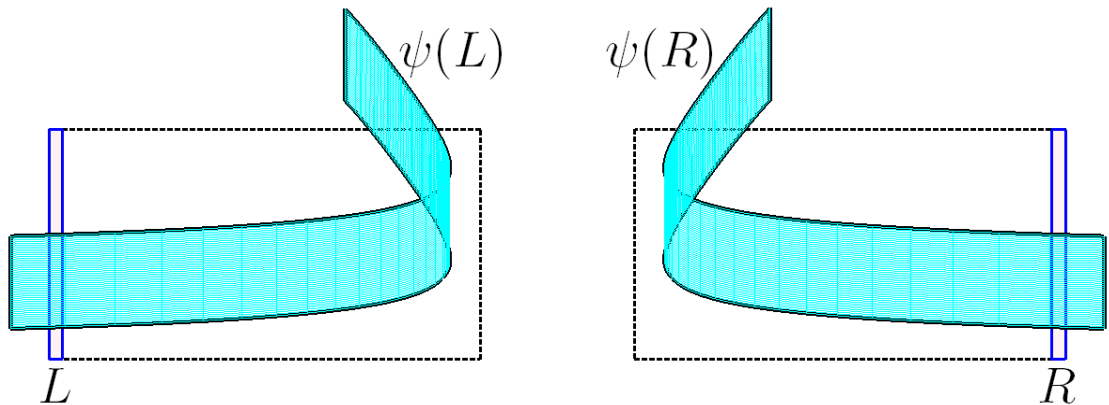}
\caption{\footnotesize{In this case, we enter the setting of Theorem \ref{th-fp}.
By applying our result to the narrow rectangles $L$ and $R$ (provided they are suitably oriented in an obvious
left-right manner), we can prove the existence of at least two fixed points for
the planar map $\psi$ in ${\mathcal A}$ (one fixed point lies in $L$ and the other in $R$).
}}
\label{fig-lr1}
\end{figure}

\begin{figure}[ht]
\centering
\includegraphics[scale=0.35]{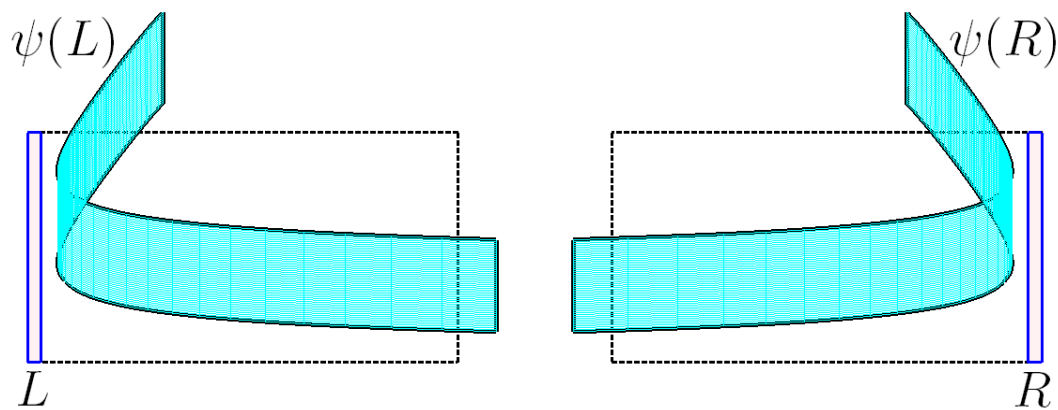}
\caption{\footnotesize{In the framework depicted above, we do not have any fixed point. In fact,
the central main part $C$ of the rectangle $\mathcal A$ is mapped by $\psi$ onto a hat-like figure outside the domain
(see the left of Figure
\ref{fig-c}) and, at the same time,
each of the two narrow rectangles $L$ and $R,$ which constitute the remaining part of the domain, is
mapped onto a set which is disjoint from itself.
}}
\label{fig-lr2}
\end{figure}

\begin{figure}[ht]
\centering
\includegraphics[scale=0.41]{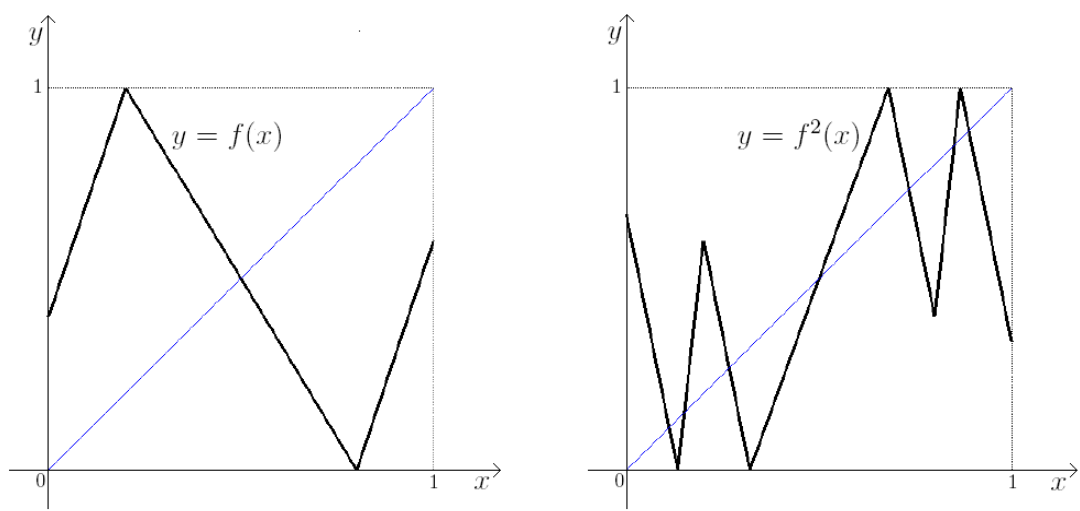}
\caption{\footnotesize{Example of the graph of a function $f$ (left)
and of its iterate $f^2$ (right), with $f$ defined as in
\eqref{eq-f}. We have drawn
the line $y=x$ in order to show the fixed points and the points of
period $2.$
}}
\label{fig-gr}
\end{figure}

\clearpage

\subsection{The Crossing Lemma}\label{sub-cl}

The fundamental result from plane topology that we have employed in the proof of the main theorems of Subsection \ref{sub-sap} is the Crossing Lemma (cf. Lemma \ref{lem-cr} below), which establishes the existence of a continuum ${\mathcal C}$
joining two opposite sides of a rectangle and contained in a given compact set ${\mathcal S}.$ More precisely, the existence of $\mathcal C$ is guaranteed by the \textit{cutting property} $(CP)$ for ${\mathcal S}$ in Lemma \ref{lem-cr},
that is, a condition which asserts that ${\mathcal S}$ intersects any path
contained in the rectangle and joining the other two sides. In the higher dimensional
case the set ${\mathcal S}$ looks like a \textit{cutting surface} (see Section \ref{sec-nd}\,). \\
In the next pages we furnish a verification of the Crossing Lemma, presenting a few topological facts that allow to achieve it. Even if, for the reader's convenience, we try to make the treatment as self-contained as possible, some theorems are stated below without proof.\\
The first among such results is Whyburn Lemma (see Lemma \ref{lem-w}). It is a central tool in bifurcation theory
and in the application of continuation methods to nonlinear equations
\cite{Al-81, Ma-97, Ra-70}. We quote here one of its most used versions, as taken from
\cite[Lemma 1.33]{Ra-70}. The interested reader can find a proof as well as further details in Chapter V of
Kuratowski's book \cite{Ku-68}. We refer also to
\cite{Al-81} and \cite{Ma-97} for interesting surveys about the role of
connectedness in the study of fixed point theory and operator equations.

\begin{lemma}[Whyburn Lemma]\index{Whyburn Lemma}\label{lem-w}
If $A$ and $B$ are (nonempty) disjoint closed subsets of a compact metric
space $K$ such that no connected components of $K$ intersect both
$A$ and $B,$ then $K = K_{A}\cup K_{B}\,,$ where $K_{A}$ and
$K_{B}$ are disjoint compact sets containing $A$ and $B,$
respectively.
\end{lemma}

\noindent
In the applications, Whyburn Lemma is often employed in the form of the following alternative:
\begin{flushleft}
\textit{Either there exists a continuum ${\mathcal C}\subseteq K$
with ${\mathcal C}\cap A\ne\emptyset$ and ${\mathcal C}\cap B\ne\emptyset,$
or $A$ and $B$ are separated, in the sense that there exists $K_{A}$ and $K_{B}$
as in Lemma \\ \ref{lem-w}.}
\end{flushleft}

\smallskip
\noindent
The other result that we state without proof is a two-dimensional version of the well-known Brouwer fixed point Theorem. In such a simplified form, it can be proved by using the winding number \cite{ChSt-66}.

\begin{theorem}\label{th-br}
Let $\phi: D\to D$ be a continuous self-mapping of a set $D\subseteq {\mathbb R}^2,$
where $D$ is homeomorphic to a closed disc. Then $\phi$ has at least one fixed point in $D.$
\end{theorem}
\noindent
In the sequel, the set $D$ will be the product of two nondegenerate compact intervals $[a,b]$ and $[c,d].$

\bigskip
\noindent
The next lemma is a particular case of the
celebrated {Leray-Schauder Continuation Principle} \cite[Th\'{e}or\`{e}me Fondamental]{LeSh-34}.
Although its intuitiveness, the following result plays a useful role in various contexts:
for a different application, see for instance \cite{KuPoSoTu-05}, while the interested reader can find more corresponding details in \cite{Ma-97,Ze-86} and in the bibliography therein.
By its significance in our approach, we furnish two proofs of Lemma \ref{lem-ls}: the first one is
less direct,
but elementary, while the second one is based on degree theory and thus it is immediately extendable to an higher dimensional setting
\cite{PiZa-07}.

\smallskip

\noindent
Introducing the projections of the plane onto the coordinate axes
$$\pi_{1}, \pi_{2}\,:\, {\mathbb R}^2\to {\mathbb R},\quad \pi_{1}(x,y): = x,\;\; \pi_{2}(x,y) := y,\;\;
\forall\, (x,y)\in {\mathbb R}^2\,,$$
it can be stated as follows:

\begin{lemma}\label{lem-ls}
Let $f: [a,b]\times[c,d]\to {\mathbb R}$ be a continuous map such that
\begin{equation*}
\begin{array}{c}
f(x,y) \leq 0,\quad\forall\, (x,y)\in [a,b]\times \{c\},\\{}\\
f(x,y) \geq 0,\quad\forall\, (x,y)\in [a,b]\times \{d\}.
\end{array}
\end{equation*}
Then there exists a compact connected set
$${\mathcal S}\subseteq \{(x,y)\in [a,b]\times [c,d]\,: f(x,y) =0\},$$
such that
$$\pi_{1}({\mathcal S}) = [a,b].$$
\end{lemma}
\begin{proof}
Let
$$K:=\{(x,y)\in [a,b]\times [c,d]\,: f(x,y) =0\}$$
be the zero-set for the map $f$
and let
$$A:= K\cap (\, \{a\}\times [c,d] \,),\quad B:= K\cap (\, \{b\}\times [c,d] \,).$$
By the intermediate value theorem, $A$ and $B$ are both nonempty. They are also
disjoint compact subsets of the compact set $K.$
\\
Assume, by contradiction, that there are no compact connected subsets of $K$
intersecting both $A$ and $B.$
Then, according to Whyburn Lemma \ref{lem-w}, there exist compact sets $K_A$ and $K_B$ with
$$
K = K_A \cup K_B\,,\;\; K_A\supseteq A,\;\; K_B\supseteq B,\;\; K_A \cap K_B = \emptyset.
$$
We can thus define the continuous function (see \cite[Lemma on Interpolation]{NeSt-60})
$$
g(x,y) := \frac{\dist(p,K_A) - \dist(p,K_B)}{\dist(p,K_A) + \dist(p,K_B)}\,\quad \mbox{ for }\; p=(x,y)\in [a,b]\times [c,d],
$$
which has $[-1,1]$ as range and attains the value $-1$ on $K_A\,$ and the value $1$ on $K_B\,.$
Notice that the denominator never vanishes as the compact sets $K_A$ and $K_B$ are disjoint.
\\
We also introduce the projections of the real line onto the intervals $[a,b]$ and $[c,d],$ respectively.
For $s\in\mathbb R$ they are defined as
\begin{equation}\label{eq-pr}
P_{[a,b]}(s):= \max\{a,\min\{s,b\} \},\quad P_{[c,d]}(s):= \max\{c,\min\{s,d\} \}.
\end{equation}
Finally, we  consider the continuous map
$$\phi = (\phi_1,\phi_2): [a,b]\times [c,d] \to [a,b]\times [c,d],$$
with components
$$\phi_1(x,y):= P_{[a,b]}(x - g(x,y)),\quad \phi_2(x,y):= P_{[c,d]}(y - f(x,y)).$$
Since $D:=[a,b]\times [c,d]$
is homeomorphic to the closed unit disc, Brouwer Theorem \ref{th-br} ensures the existence of a
fixed point $z = (z_1,z_2)\in D$ for the map $\phi,$ that is,
\begin{equation}\label{eq-fixp}
z_1 = P_{[a,b]}(z_1 - g(z_1,z_2)),\quad z_2 = P_{[c,d]}(z_2 - f(z_1,z_2)).
\end{equation}
We claim that
\begin{equation}\label{eq-cd}
z_2 - f(z_1,z_2)\in [c,d].
\end{equation}
Indeed, if $z_2 - f(z_1,z_2)\not\in [c,d],$ then either $z_2 - f(z_1,z_2) < c$ or $z_2 - f(z_1,z_2) > d.$
In the former case, $P_{[c,d]}(z_2 - f(z_1,z_2)) = c$ and thus, from the second relation in \eqref{eq-fixp},
we find $z_2 = c.$ By the sign condition on $f(x,y)$ for $(x,y)\in [a,b]\times\{c\},$
it follows that $f(z) = f(z_1,c) \leq 0$ and therefore
$z_2 - f(z_1,z_2) = c - f(z_1,c) \geq c.$ This contradicts the assumption $z_2 - f(z_1,z_2) < c.$
On the other hand, if $z_2 - f(z_1,z_2) > d,$ then $P_{[c,d]}(z_2 - f(z_1,z_2)) = d$ and thus the second relation in \eqref{eq-fixp} implies $z_2 = d.$ By the sign condition on $f(x,y)$ for $(x,y)\in [a,b]\times\{d\},$
it follows that $f(z) = f(z_1,d) \geq 0$ and therefore
$z_2 - f(z_1,z_2) = d - f(z_1,d) \leq d.$ This contradicts the assumption $z_2 - f(z_1,z_2) > d.$ Condition \eqref{eq-cd} is thus proved.

\noindent From \eqref{eq-cd} and the second relation in \eqref{eq-fixp}, we obtain
$z_2 = P_{[c,d]}(z_2 - f(z_1,z_2))$ $= z_2 - f(z_1,z_2).$ Hence $f(z)=0,$ that is, $z\in K = K_A \cup K_B$
and, recalling the definition of the interpolation map $g,$ we conclude that
\begin{equation}\label{eq-pm1}
g(z) =\pm 1.
\end{equation}

\noindent
Arguing as in the proof of \eqref{eq-cd}, we can also show that
$$z_1 - g(z_1,z_2)\in [a,b].$$
Then, the first relation in \eqref{eq-fixp} implies that
$z_1 = P_{[a,b]}(z_1 - g(z_1,z_2)) = z_1 - g(z_1,z_2).$ Hence,
$$g(z)=0,$$
in contradiction with \eqref{eq-pm1}.
Thus, by Whyburn Lemma \ref{lem-w}, the existence of a continuum
${\mathcal S}\subseteq K = \{(x,y)\in [a,b]\times [c,d]\,: f(x,y) =0\},$
with ${\mathcal S}\cap A\ne\emptyset$ and ${\mathcal S}\cap B\ne\emptyset,$ follows.
This means that the set ${\mathcal S}$ contains points of the form $(a,u)$ and $(b,v),$
with $u,v\in [c,d],$ and consequently the image of ${\mathcal S}$ under
$\pi_{1}$ covers the interval $[a,b].$
The verification of the lemma is complete.

\medskip

\noindent
We present also an alternative proof via degree theory.
\\
Consider the auxiliary function $\widehat{f}: [a,b]\times {\mathbb R}\to {\mathbb R}$ defined by
$$\widehat{f}(x,y):= f(x,P_{[c,d]}(y)) + \min\{y-c,\max\{0,y-d\} \}.$$
The map $\widehat{f}$ is continuous and such that
$$\widehat{f}(x,y) = f(x,y),\quad\forall\, (x,y)\in [a,b]\times[c,d].$$
Moreover, it satisfies the following inequalities:
\begin{equation*}
\begin{array}{c}
\widehat{f}(x,y) < 0,\quad\forall\, (x,y)\in [a,b]\times (-\infty,c\,[\,,\\
\!\!\widehat{f}(x,y) > 0,\quad\forall\, (x,y)\in [a,b]\times \,]\,d,+\infty).
\end{array}
\end{equation*}
If we treat the variable $x\in [a,b]$ as a parameter and consider the open set
$\Omega:= \,]c-1,d+1[\,,$ we have that
$\widehat{f}(x,y)\ne 0,$ for every $x\in [a,b]$ and $y\in\partial\Omega.$
The Brouwer degree $\mbox{deg}(\widehat{f}(a,\cdot),\Omega,0)$ is well-defined and nontrivial,
since, by the sign condition $\widehat{f}(a,c-1)< 0 < \widehat{f}(a,d+1),$ it holds that
$$\mbox{deg}\left(\widehat{f}(a,\cdot),\Omega,0\right)=1.$$
The Leray-Schauder Continuation Theorem (see \cite{LeSh-34,Ma-97,Ze-86}) ensures that
the set of solution pairs
$$\widehat{K}:=\left\{(x,y)\in [a,b]\times \Omega\,: \,\widehat{f}(x,y) =0\right\}$$
contains a continuum along which $x$ assumes all the values in $[a,b].$
By the definition of $\widehat{f},$ it is clear that $\widehat{K} = K$ and this gives
another proof of the existence of ${\mathcal S}.$
\end{proof}

\smallskip
\noindent
The next result shows that the continuum ${\mathcal S}$
found in Lemma \ref{lem-ls} can be $\varepsilon$-approximated
by paths.
\begin{lemma}\label{lem-lsp}
Let $f: [a,b]\times[c,d]\to {\mathbb R}$ be a continuous map such that
\begin{eqnarray*}
f(x,y) \leq -1,&\quad\forall\, (x,y)\in [a,b]\times \{c\},\\
f(x,y) \geq 1,&\quad\forall\, (x,y)\in [a,b]\times \{d\}.
\end{eqnarray*}
Then, for each
$\varepsilon > 0,$ there exists a continuous map
$$\gamma= \gamma_{\varepsilon}\,:[0,1]\to [a,b]\times [c,d],$$
such that
$$\gamma(0)\in \{a\}\times [c,d],\;\;\; \gamma(1)\in \{b\}\times [c,d]$$
and
$$|f(\gamma(t))|< \varepsilon,\;\;\forall\, t\in [0,1].$$
\end{lemma}
\begin{proof}
Without loss of generality, we can assume
$0 < \varepsilon < 1.$
\\
First of all, we introduce the continuous function
$${\tilde{f}}: {\mathbb R}^2 \to {\mathbb R},\quad
{\tilde{f}}(x,y) := f(P_{[a,b]}(x),P_{[c,d]}(y) ),$$
which extends $f$ to the whole plane by means of the projections in \eqref{eq-pr}.

\noindent From the sign conditions on $f,$ we easily see that
\begin{eqnarray*}
{\tilde{f}}(x,y) \leq -1,&\quad\forall\, (x,y)\in {\mathbb R}^2: y\leq c\,,\\
{\tilde{f}}(x,y) \geq 1,&\quad\forall\, (x,y)\in {\mathbb R}^2: y\geq d\,.
\end{eqnarray*}
We also define the set
$$T_{\varepsilon}:=\left\{(x,y)\in {\mathbb R}^2: \, |{\tilde{f}}(x,y)| < \varepsilon\right\},$$
which is an open subset of the strip ${\mathbb R}\times\, ]c,d[\,.$
\\
By Lemma \ref{lem-ls}, there exists a compact connected set
$${\mathcal S}\subseteq \{(x,y)\in [a,b]\times [c,d]\,: f(x,y) =0\},$$
such that
$\pi_{1}({\mathcal S}) = [a,b].$ Notice that
$${\mathcal S}\subseteq T_{\varepsilon}\,.$$
For every $p\in {\mathcal S},$ there exists an open disc $B(p,\delta_p)$ with center in $p$
and radius $\delta_p > 0,$ such that $B(p,\delta_p)\subseteq T_{\varepsilon}\,.$
By compactness, we can find a finite number of points $p_1\,,\dots,p_n\,\in {\mathcal S}$
such that
$${\mathcal S}\subseteq B:=\bigcup_{i=1}^n B(p_i,\delta_i) \subseteq T_{\varepsilon}\,,$$
where we have set $\delta_i:= \delta_{p_i}\,.$
Without loss of generality (adding two further discs, if necessary),
we can suppose that
$$p_1\in (\{a\}\times [c,d])\cap T_{\varepsilon}\,\;\;\mbox{ and }\;\; p_n\in (\{b\}\times [c,d])\cap T_{\varepsilon}\,.$$
The open set $B$ is connected. Indeed, it is union of connected sets
(the open discs $B(p_i,\delta_i)$) and each of such connected sets has nonempty intersection
with the connected set ${\mathcal S}.$
\\
Since every open connected set in the plane is also
arcwise connected, there exists a continuous map $\omega_{\varepsilon}: [s_0,s_1]\to B$
with $\omega_{\varepsilon}(s_0)=p_1$ and $\omega_{\varepsilon}(s_1)=p_n\,.$
Then, we define
$$s'\,:= \max\{s\in [s_0,s_1]\,:\,\omega_{\varepsilon}(s)\in\, \{a\}\times [c,d]\}$$
and
$$s''\,:= \min\{s\in [s',s_1]\,:\,\omega_{\varepsilon}(s)\in\, \{b\}\times [c,d]\}.$$
Hence, the path
$$\gamma_{\varepsilon}(t):= \omega_{\varepsilon}(s'+ t \cdot(s'' - s')),\;\; \mbox{for } t\in [0,1],$$
is continuous and has the desired properties. Indeed,
$\gamma_{\varepsilon}(t)\in ([a,b]\times {\mathbb R})\cap T_{\varepsilon}$ $\subseteq \, [a,b]\times \,]c,d[\,$
and therefore $|{\tilde{f}}(\gamma_{\varepsilon}(t))| = |f(\gamma_{\varepsilon}(t))| < \varepsilon,$
for all $t\in [0,1].$ Moreover,
$\gamma_{\varepsilon}(0) = \omega_{\varepsilon}(s')\in \{a\}\times [c,d]$ and
$\gamma_{\varepsilon}(1) = \omega_{\varepsilon}(s'')\in \{b\}\times [c,d].$ The proof is complete.
\end{proof}

\smallskip
\noindent
We are now in position to prove the following\index{Crossing Lemma}:
\begin{lemma}[Crossing Lemma]\label{lem-cr}
Let $K\subseteq [0,1]^2$ be a compact set which satisfies the {\em cutting property}\footnote{In view of Definition \ref{def-cut} in Section \ref{sec-nd}, we could also say that \textit{$K$ cuts the arcs between $\{0\}\times [0,1]$ and $\{1\}\times [0,1].$}}{\em :}
\begin{equation*}
\begin{array}{c}
K\cap \gamma([0,1])\ne\emptyset,\\
\mbox{for each continuous map} \; \gamma: [0,1]\to [0,1]^2,\\
\mbox{with }
\gamma(0)\in \{0\}\times [0,1] \;\mbox{and } \gamma(1)\in \{1\}\times [0,1].
\end{array}
\leqno{(CP)}
\end{equation*}
Then there exists a continuum  ${\mathcal C}\subseteq K$
with
$${\mathcal C}\cap ([0,1]\times\{0\})\ne\emptyset\quad\mbox{and }\;\; {\mathcal C}\cap ([0,1]\times\{1\})\ne\emptyset.$$
\end{lemma}
\begin{proof}
By the assumptions, the sets
$$A:= K\cap ([0,1]\times \{0\}),\quad B:= K\cap ([0,1]\times \{1\})$$
are compact, nonempty and disjoint. The thesis is achieved if we prove that there exists
a continuum ${\mathcal C}\subseteq K$ with ${\mathcal C}\cap A\ne\emptyset$ and
${\mathcal C}\cap B\ne\emptyset.$ If, by contradiction, there is no continuum of this kind,
Whyburn Lemma \ref{lem-w} implies that the set $K$ can be decomposed as
$$K= K_A\cup K_B\,,\;\; K_A\cap K_B=\emptyset,\;\; K_A\supseteq A,\;\; K_B\supseteq B,$$
with $K_A$ and $K_B$ compact sets.
Let us then define the compact sets
$${K'_A}:=K_A\cup ([0,1]\times \{0\})\,,\quad {K'_B}:= K_B\cup ([0,1]\times \{1\})\,.$$
We claim that they are disjoint. Indeed, since $K_A\cap K_B=\emptyset,$ if there exists $\bar x\in {K'_A}\cap{K'_B},$ then $\bar x\in {K_A}\cap ([0,1]\times \{1\})$ or $\bar x\in {K_B}\cap ([0,1]\times \{0\}).$ In the former case, recalling that $K_A\subseteq K,$ we find that $\bar x\in B\subseteq K_B$ and thus $\bar x\in K_A\cap K_B=\emptyset.$ The same contradiction can be achieved by assuming $\bar x\in {K_B}\cap ([0,1]\times \{0\}).$ The claim on the disjointness of ${K'_A}$ and ${K'_B}$ is thus checked.\\
Arguing like in the proof of Lemma \ref{lem-ls}, we introduce the continuous interpolation function
$$
f(x,y) := \frac{\dist(p,{K'_A}) - \dist(p,{K'_B})}{\dist(p,{K'_A}) + \dist(p,{K'_B})}\,\quad \mbox{ for }\; p=(x,y)\in [0,1]^2
$$
which has $[-1,1]$ as range and attains the value $-1$ on ${K'_A}\,$ and the value $1$ on ${K'_B}\,.$
By Lemma \ref{lem-lsp} there exists a continuous map
$\gamma: [0,1]\to [0,1]^2$ with $\gamma(0)\in \{0\}\times [0,1]$ and
$\gamma(1)\in \{1\}\times [0,1]$ such that
\begin{equation}\label{eq-1-2}
|f(\gamma(t))|< \frac1 2\,,\quad\forall\, t\in [0,1].
\end{equation}
The cutting property $(CP)$ ensures the existence of $t^*\in [0,1]$ such that
$$\gamma(t^*)\in K.$$
By the splitting of $K= K_A\cup K_B$ and the definition of $f$ we find that
$$f(\gamma(t^*))= \pm 1,$$
in contradiction with \eqref{eq-1-2}. This concludes the proof.
\end{proof}

\smallskip
\noindent
The above Crossing Lemma may be used to give a simple proof of Poincar\'{e}-Miranda Theorem
in dimension two. We recall that Poincar\'{e}-Miranda Theorem asserts the existence of a
zero for a continuous vector field $F = (F_1,F_2)$ defined on a rectangle
$[a_1,b_1]\times [a_2,b_2]\subseteq {\mathbb R}^2,$
such that $F_1(a_1,x_2)\cdot F_1(b_1,x_2)\le 0,$ for every $x_2\in [a_2,b_2]$ and
$F_2(x_1,a_2)\cdot F_2(x_1,b_2)\le 0,$ for every $x_1\in [a_1,b_1].$
The result holds also for the standard hypercube of ${\mathbb R}^N$ and
for $N$-dimensional rectangles, as we shall see in Section \ref{sec-nd} (cf. Theorem \ref{th-pm}).
Such theorem is usually referred to Carlo Miranda, who in 1940
noticed its equivalence to Brouwer fixed point Theorem \cite{Mi-40}. On the other hand, as remarked in
\cite{Ma-00}, Henry Poincar\'{e} in \cite{Po-83,Po-84} had already announced this result
with a suggestion of a correct proof using the Kronecker's index. In \cite{Po-84},
with regard to the two-dimensional case, Poincar\'{e} (assuming strict inequalities for the components of the vector field
on the boundary of the rectangle) described also an heuristic proof as follows:
the ``curve'' $F_2=0$ departs from a point of the side $x_1 = b_1$ and ends at some point of $x_1 = a_1\,;$
in the same manner,
the curve $F_1=0,$ departing from a point of  $x_2 = b_2$ and ending at some point of $x_2 = a_2\,,$
must necessarily meet the first ``curve'' in the interior of the rectangle.
Further information and historical remarks about Poincar\'{e}-Miranda Theorem can be found in
\cite{Ku-97, Ma-00}.\\
Following Poincar\'{e}'s heuristic argument, one could adapt his proof in the following manner:
if $\gamma(t)$ is a path contained in the rectangle $[a_1,b_1]\times [a_2,b_2]$ and
joining the left and the right sides, by Bolzano Theorem there exists at least a zero of
$F_1(\gamma(t))$ and this, in turns, means that any path as above meets the set ${\mathcal S}:=F_1^{-1}(0)\,.$
The Crossing Lemma then implies the existence of a compact connected set ${\mathcal C}_1\subseteq
F_1^{-1}(0)$ which intersects the lower and the upper sides of the rectangle.
At this point one can easily achieve the conclusion in various different ways.
For instance, one could just repeat the same argument on $F_2$ in order to obtain a compact connected set
${\mathcal C}_2\subseteq F_2^{-1}(0)$ which intersects the left and the right sides of the rectangle
and thus prove the existence of a zero of the vector field $F$ using the fact that
${\mathcal C}_1\cap{\mathcal C}_2\ne\emptyset.$ Alternatively, one could apply
Bolzano Theorem and find a zero for $F_2$ restricted to ${\mathcal C}_1$ (see also
\cite{ReZa-00} for a similar use of a variant of the Crossing Lemma and \cite{PiZa-07}
for extensions to the $N$-dimensional setting).
Conversely, it is possible to provide a proof
of the Crossing Lemma via Poincar\'{e}-Miranda Theorem \cite{PaZa-04b}.

\section{The $N$-dimensional setting}\label{sec-nd}
\subsection{Cutting surfaces and topological results}\label{sub-cs}
In addition to the notions on paths introduced in Section \ref{sec-sap}, we need some further concepts in order to present a generalization to higher dimensional settings of the planar theory previously explained.\\
If $\gamma_1,\,\gamma_2:[0,1]\to W$ are paths in the topological space $W$ with
$\gamma_1(1)=\gamma_2(0),$  we define the \textit{gluing of
$\gamma_1$ with $\gamma_2$}\index{gluing of paths} as the path
$\gamma_1\star\gamma_2:[0,1]\to W$ such that
\begin{displaymath}
\gamma_1\star\gamma_2(t):=\left\{ \begin{array}{ll}
\gamma_1(2t) & \textrm{for $0\le t\le \frac 1 2 \,,$}\\
\gamma_2(2t-1) & \textrm{for $\frac 1 2\le t\le 1\,.$ }
\end{array} \right.
\end{displaymath}
Moreover, given a path $\gamma:[0,1]\to W,$ we denote by
$\gamma^-:[0,1]\to W$ the path having $\overline{\gamma}$ as support,
but run with reverse orientation with respect to $\gamma,$ i.e. $\gamma^-(t):=\gamma(1-t),$
for all $t\in [0,1].$ We also recall that a topological
space $W$ is said \textit{arcwise connected}\index{arcwise connectedness} if,
for any couple of distinct points $p,q\in W,$ there is a path $\gamma:
[0,1]\to W$ such that $\gamma(0) = p$ and $\gamma(1) = q.$ In the
case of a Hausdorff topological space $W,$ the range
${\overline{\gamma}}$ of $\gamma$ turns out to be a locally connected
metric continuum (a Peano space according to \cite{HoYo-61}).
Thus, if $W$ is a metric space, the above definition of arcwise connectedness is equivalent
to the fact that, given any two points $p,q\in W$ with $p\ne q,$
there exists an arc (i.e. the homeomorphic image of
[0,1]) contained in $W$ and having $p$ and $q$ as
extreme points (see \cite[pp.115-131]{HoYo-61}).

\medskip

We start our explanation by presenting some topological lemmas concerning
the relationship between particular surfaces \index{cutting surfaces} and zero sets of
continuous real valued functions. Analogous results can be found, often in a more implicit form, in different contexts. However, since for our
applications we need a specific version of the statements, we
give an independent proof with all the details. \\
At first, let us introduce the central concept of ``cutting set''\index{cutting the arcs@\textsl{cutting the arcs} property}.

\begin{definition}\label{def-cut}
{\rm{
Let $X$ be an arcwise connected metric space and let $A, B,
C$ be closed nonempty subsets of $X$ with $A\cap
B=\emptyset.$ We say that \textit{$C$ cuts the arcs between $A$
and $B$} if for any path $\gamma: [0,1]\to X,$
with ${\overline{\gamma}}\cap A\neq\emptyset$ and ${\overline{\gamma}}\cap
B\neq\emptyset,$ it follows that ${\overline{\gamma}}\cap
C\neq\emptyset.$ In the sequel, if $X$ is a subspace of a larger
metric space $Z$ and we wish to stress that we consider
only paths contained in $X,$ we make more precise our definition
by saying that \textit{$C$ cuts the arcs between $A$ and $B$ in $X$}.}}
\end{definition}

\noindent
Such definition is a modification of the classical one regarding the cutting
of a space between two points in \cite{Ku-68}.
See \cite{BeDiPe-02} for a more general concept concerning a set $C$ that
intersects every connected set meeting two nonempty sets $A$ and $B.$
In the special case that $A$ and $B$ are the opposite faces of an $N$-dimensional cube,
J. Kampen \cite[p.512]{Ka-00} says that $C$ separates $A$ and $B.$ We prefer to use
the ``cutting'' terminology in order to avoid any misunderstanding with other notions
of separation which are more common in topology.
In particular we remark that
our definition agrees with the usual one of cut
when $A,B,C$ are pairwise disjoint \cite{FeLeSh-99}.

\smallskip

\noindent
In the sequel, even when not explicitly mentioned, we assume
the basic space $X$ to be arcwise connected. In some of the next results
the local arcwise connectedness\index{arcwise connectedness, local} of $X$ is required, too, i.e. for any $p\in X$ it holds that each neighborhood of $p$ contains an arcwise connected neighborhood of $p.$ With this respect, we
recall that
any connected and locally arcwise connected metric space is arcwise connected (see \cite[Th.2, p.253]{Ku-68}).

\begin{lemma}\label{lem-cut}
Let $X$ be a connected and locally arcwise connected metric space and let
$A, B, C\subseteq X$ be closed and nonempty sets with $A\cap
B=\emptyset.$ Then $C$ cuts the arcs between $A$ and $B$ if and
only if there exists a continuous function $f: X\to {\mathbb R}$ such that
\begin{equation}\label{eq-ab}
f(x)\leq 0,\, \forall\, x\in A,\qquad f(x) \geq 0,\, \forall\, x\in B
\end{equation}
and
\begin{equation}\label{eq-zs}
C = \{x\in X: f(x) = 0\}.
\end{equation}
\end{lemma}
\begin{proof}
Assume there exists a continuous function $f: X\to {\mathbb R}$ satisfying
\eqref{eq-ab} and \eqref{eq-zs}. Let $\gamma: [0,1]\to X$ be a path
such that $\gamma(0)\in A$ and $\gamma(1)\in B.$ We want to prove that
${\overline{\gamma}}\cap C\neq\emptyset.$ Indeed, for
the composite continuous function
$\theta:=f\circ\gamma: [0,1]\to {\mathbb R},$ we have that $\theta(0)\leq 0 \leq \theta(1)$ and so
Bolzano Theorem ensures the existence of $t^*\in [0,1]$ with $\theta(t^*)=0.$ This means that
$\gamma(t^*)\in C$ and therefore ${\overline{\gamma}}\cap C\neq\emptyset.$
Hence, we have proved that $C$ cuts the arcs between $A$ and $B.$\\
Conversely, let us assume that $C$ cuts the arcs between $A$ and $B.$
We introduce the auxiliary functions
\begin{equation*}
\rho: X\to {{\mathbb R}}^+\,,
\end{equation*}
\begin{equation}\label{eq-ro}
\rho(x):= \mbox{dist}(x,C),\;\forall\, x\in X
\end{equation}
and
\begin{equation*}
\mu: X\to \{-1,0,1\},
\end{equation*}
\begin{equation}\label{eq-mu}
\mu(x):=
\quad
\left
\{
\begin{array}{ll}
~\,0&\; \mbox{ if } \, x\in C,\\
\!-1 &\; \mbox{ if } \, x\not\in C \;
\mbox{and $\exists$ a path $\gamma_x:[0,1]\to X\setminus C$} \\
~\quad &\mbox{
\,such that $\gamma_x(0)\in A$ and $\gamma_x(1) = x,$}\\
~\,1&\; \mbox{ elsewhere.}\\
\end{array}
\right.
\end{equation}
Observe that $\rho$ is a continuous function with $\rho(x) = 0$ if and only if $x\in C$
and also $\mu(x) = 0$ if and only if $x\in C.$ Moreover, $\mu$ is bounded.\\
Let $x_0\not\in C.$ We claim that $\mu$ is continuous in $x_0\,.$ Actually, $\mu$
is locally constant on $X\setminus C.$ Indeed, since $x_0\in X\setminus C$ (an open set) and $X$
is locally arcwise connected, there is a neighborhood $U_{x_0}$ of $x_0$ with
$U_{x_0}\subseteq X\setminus C,$ such that for each $x\in U_{x_0}$ there exists a path $\omega_{x_0,x}$
joining $x_0$ to $x$ in $U_{x_0}\,.$ Clearly, if there is a path $\gamma_{a,x_0}$ in $X\setminus C$
joining some point $a\in A$ with $x_0\,,$ then the path $\gamma_{a,x_0}\star\omega_{x_0,x}$
connects $a$ to $x$ in $X\setminus C.$ This proves that if $\mu(x_0) = -1,$ then $\mu(x) = -1$
for every $x\in U_{x_0}\,.$ On the other hand, if there is a path $\gamma_{a,x}$ in
$X\setminus C$
which connects some point $a\in A$ to $x\in U_{x_0}\,,$ then the path
$\gamma_{a,x}\star\omega^-_{x_0,x}$ connects $a$ to $x_0\,$ in $X\setminus C.$
This shows that if $\mu(x_0) = 1$ (that is, it is not possible to connect $x_0$ to any point
of $A$ in $X\setminus C$ using a path), then $\mu(x) = 1$ for every $x\in U_{x_0}$
(that is, it is not possible to connect any point $x\in U_{x_0}$ to any point
of $A$ in $X\setminus C$ using a path).\\
We can now define
\begin{equation}\label{eq-frm}
f: X\to {\mathbb R}, \qquad f(x):= \rho(x) \mu(x).
\end{equation}
Clearly, $f(x)= 0$ if and only if $x\in C$ and, moreover, $f$ is continuous.
Indeed, if $x_0\not \in C,$ we have that $f$ is continuous in $x_0$ because both $\rho$ and $\mu$ are
continuous in $x_0\,.$ If $x_0\in C$ and $x_n\to x_0$ (as $n\to \infty$), then $\rho(x_n)\to 0$
and $|\mu(x_n)|\leq 1,$ so that $f(x_n)\to 0 = f(x_0).$
Finally, by the definition of $\mu$ in \eqref{eq-mu}, it holds that $\mu(a) = -1,$ for every $a\in A\setminus C,$ and hence, for such an $a,$ it holds that $f(a) < 0.$ On the contrary, if we suppose that $b\in B\setminus C,$ we must have $\mu(b) =1.$ In fact, by the cutting condition,
there is no path connecting in $X\setminus C$ the point $b$ to any point of $A.$ Therefore, in this case we have $f(b) > 0.$
The proof is complete.
\end{proof}

\noindent Considering the functions $\mu$ and $f$ as in
\eqref{eq-mu}--\eqref{eq-frm}, we notice that their definition,
although adequate from the point of view of the proof of Lemma
\ref{lem-cut}, perhaps does not represent
an optimal choice. For instance, we would like the sign
of $f$ to coincide for all the points located ``at the same side''
of $A$ (respectively of $B$) with respect to $C$ (see Figure \ref{fig-si}). Having such
request in mind, we will propose a different definition for the
function $\mu$ in \eqref{eq-mu2} below.
With this respect, we introduce
the following sets that we call \textit{the side of $A$ in $X$}
and \textit{the side of $B$ in $X$}\index{side of a set ${\mathcal S}(\,\cdot)$}, respectively:
\begin{equation*}
{\mathcal S}(A):=\{x\in X: \, {\overline{\gamma}}\cap A\ne\emptyset,\,\forall\, \mbox{path }\,\gamma: [0,1]\to X,
\mbox{with } \gamma(0) = x, \gamma(1)\in B\},
\end{equation*}
\begin{equation*}
{\mathcal S}(B):=\{x\in X: \, {\overline{\gamma}}\cap B\ne\emptyset,\,\forall\, \mbox{path }\,\gamma: [0,1]\to X,
\mbox{with } \gamma(0) = x, \gamma(1)\in A\},
\end{equation*}that is, a point $x$ belongs to the side of $A$ (to the side of $B$)
if, whenever we try to connect $x$ to $B$ (to $A$) by a path,
we first meet $A$ (we first meet $B$). By definition, it follows that
\begin{equation*}
A\subseteq {\mathcal S}(A),\quad B\subseteq {\mathcal S}(B).
\end{equation*}
In addition it is possible to check that ${\mathcal S}(A)$ and ${\mathcal S}(B)$ are closed and disjoint. This is the content of the next lemma.

\begin{lemma}\label{lem-dis}
Let $X$ be a connected and locally arcwise connected metric space and let
$A, B\subseteq X$ be closed and nonempty sets with $A\cap
B=\emptyset.$ Then ${\mathcal S}(A)$ and ${\mathcal S}(B)$ are closed and, moreover,
\begin{equation*}
{\mathcal S}(A)\cap {\mathcal S}(B) = \emptyset.
\end{equation*}
\end{lemma}
\begin{proof}
First of all, we show that ${\mathcal S}(A)$ is closed by
proving that if $w\not\in {\mathcal S}(A)$ then there exists a
neighborhood $U_w$ of $w$ with $U_w\subseteq X\setminus {\mathcal
S}(A).$ Indeed, if $w\not\in {\mathcal S}(A),$ there is a path
$\gamma:[0,1]\to X$ with $\gamma(0) = w$ and $\gamma(1)=b\in B$
and such that $\gamma(t)\not\in A,$ for every $t\in [0,1].$ Since
$A$ is closed and $X$ is locally arcwise connected, there
exists an arcwise connected set $V_w$ with $w\in V_w\subseteq
X\setminus A.$ Hence, for every $x\in V_w,$ there is a path
$\omega_x$ connecting $x$ to $w$ in $V_w\,.$ As a consequence, we
find that the path $\gamma_x:=\omega_x\star\gamma$ connects $x\in
V_w$ to $b\in B$ with $\overline{\gamma_x}\subseteq X\setminus A.$
Clearly, the neighborhood $U_w:= V_w$ satisfies our
requirement and this proves that $X\setminus {\mathcal S}(A)$ is
open. The same argument ensures that ${\mathcal S}(B)$ is closed.\\
It remains to show that ${\mathcal S}(A)$ and ${\mathcal S}(B)$
are disjoint. Since $A\cap B =\emptyset,$
it follows immediately from the definition of side that
$$A\cap {\mathcal S}(B)= \emptyset,\quad
B\cap {\mathcal S}(A)= \emptyset.$$
Assume by contradiction that there exists
$$x\in {\mathcal S}(A) \cap {\mathcal S}(B),$$
with $x\not\in A\cup B.$
Since $X$ is arcwise connected, there is a path $\gamma: [0,1]\to X$ such that $\gamma(0) = x$
and $\gamma(1) = b\in B.$ The fact that $x\in {\mathcal S}(A)\setminus A$
implies that there is $t_1\in \, ]0,1[\,$ such that $\gamma(t_1) = a_1\in A.$
On the other hand, since
$x\in {\mathcal S}(B)\setminus B,$
there exists $s_1\in \, ]0,t_1[\,$ such that $\gamma(s_1) = b_1\in B.$
Proceeding by induction and using repeatedly the definition of ${\mathcal S}(A)$
and ${\mathcal S}(B),$ we obtain a sequence
$$t_1 > s_1 > t_2 > \dots > t_{j} > s_{j} > t_{j+1} > \dots > 0$$
with $\gamma(t_i) = a_i \in A$ and $\gamma(s_i) = b_i\in B.$
For $t^*= \inf t_n = \inf s_n \in [0,1[\,,$ it holds that
$\gamma(t^*) = \lim a_n = \lim b_n \in A\cap B,$ a contradiction.
This concludes the proof.
\end{proof}

\noindent
Thanks to Lemma \ref{lem-dis}, it is possible to consider the cutting property with the sets ${\mathcal S}(A)$ and ${\mathcal S}(B)$ playing the role of $A$ and $B$ in Definition \ref{def-cut}. In particular, in connection with Lemma \ref{lem-cut}, the following results hold:

\begin{lemma}\label{lem-cuteq}
Let $X$ be a connected and locally arcwise connected metric space and let
$A, B, C\subseteq X$ be closed and nonempty sets with $A\cap
B=\emptyset.$ Then $C$ cuts the arcs between $A$ and $B$ if and
only if $C$ cuts the arcs between ${\mathcal S}(A)$ and ${\mathcal S}(B).$
\end{lemma}
\begin{proof}
One direction of the inference is obvious. In fact, every path joining $A$ to $B$
is also a path joining ${\mathcal S}(A)$ to ${\mathcal S}(B).$ Thus, if
$C$ cuts the arcs between ${\mathcal S}(A)$ and ${\mathcal S}(B),$
then it also cuts the arcs between $A$ and $B.$\\
Conversely, let us assume that $C$ cuts the arcs between $A$ and $B.$
We want to prove that $C$ cuts the arcs between ${\mathcal S}(A)$ and ${\mathcal S}(B).$
By the definition of ${\mathcal S}(A)$ and ${\mathcal S}(B),$ it is
straightforward to check that
$C$ cuts the arcs between $A$ and ${\mathcal S}(B),$ as well as that
it cuts the arcs between ${\mathcal S}(A)$ and $B.$
Suppose by contradiction that there exists a path
$$\gamma: \left[\tfrac 1 2,1\right] \to X\setminus C$$
such that
$\gamma(\tfrac 1 2) = a\in {\mathcal S}(A)\setminus A$ and
$\gamma(1) = b\in {\mathcal S}(B)\setminus B.$
We choose a point $a_0 \in A$ and connect it to $a\in {\mathcal S}(A)$ by a path
$\omega: [0,\tfrac 1 2] \to X$ with $\omega(0) = a_0$ and $\omega(\tfrac 1 2) = a.$
Introduce the new path $\gamma_0:= \omega\star\gamma : [0,1]\to X,$
with $\gamma_0(0) = a_0\in A$ and $\gamma_0(1) = b\in {\mathcal S}(B).$
By the definition of ${\mathcal S}(B),$ there exists $s_1\in \, ]0,1[\,$
such that $\gamma_0(s_1) = b_1\in B.$ Since $B\cap {\mathcal S}(A) = \emptyset,$ it follows that $b_1\not\in {\mathcal S}(A).$
Notice also that $0 < s_1 < \tfrac 1 2:$ indeed,
if $\tfrac 1 2 < s_1 < 1,$ recalling that $\gamma_0(\tfrac 1 2) =
\gamma(\tfrac 1 2) = a \in {\mathcal S}(A)$ and $\gamma_0(s_1) =
\gamma(s_1) = b_1 \in B,$ there must be a ${\tilde{t}}\in [\tfrac 1 2, s_1]$
such that $\gamma({\tilde{t}}\,) \in C,$ in contrast to the assumption on $\gamma.$
The restriction of the path
$\gamma_0$ to the interval $[s_1, \tfrac 1 2]$
defines a path joining $b_1\in B$ to $a\in {\mathcal S}(A).$ Therefore there exists
$t_1\in \,]s_1,\tfrac 1 2[\,$ such that
$\gamma_0(t_1) = a_1\in A.$ The restriction of the path
$\gamma_0$ to the interval $[t_1, 1]$
defines a path joining $a_1\in A$ to $b\in {\mathcal S}(B).$ Hence there exists
$s_2\in \,]t_1, 1[\,$ with
$\gamma_0(s_2) = b_2\in B.$ As above, we can also observe that
$b_2\not\in {\mathcal S}(A)$ and  $t_1 < s_2 < \tfrac 1 2.$
Proceeding by induction, we obtain a sequence
$$s_1 < t_1 < s_2 < \dots < s_{j} < t_{j} < s_{j+1} < \frac 1 2$$
with $\gamma(t_i) = a_i \in A$ and $\gamma(s_i) = b_i\in B.$
For $t^*= \sup t_n = \sup s_n \in \,]0,\tfrac 1 2],$ we have that
$\gamma_0(t^*)= \omega(t^*)= \lim a_n = \lim b_n \in A\cap B,$ a contradiction.
The proof is complete.
\end{proof}

\begin{lemma}\label{lem-cuts}
Let $X$ be a connected and locally arcwise connected metric space and let
$A, B, C\subseteq X$ be closed and nonempty sets with $A\cap
B=\emptyset.$ Then $C$ cuts the arcs between $A$ and $B$ if and
only if there exists a continuous function $f: X\to {\mathbb R}$ such that
\begin{equation}\label{eq-ab2}
f(x)\leq 0,\, \forall\, x\in {\mathcal S}(A),\qquad f(x) \geq 0,\, \forall\, x\in {\mathcal S}(B)
\end{equation}
and
\begin{equation}\label{eq-zs2}
C = \{x\in X: f(x) = 0\}.
\end{equation}
\end{lemma}
\begin{proof}
Clearly, if there exists a continuous function $f: X\to {\mathbb R}$ satisfying \eqref{eq-ab2} and \eqref{eq-zs2}, then \eqref{eq-ab} and \eqref{eq-zs} hold, too. Hence Lemma \ref{lem-cut}
implies that $C$ cuts the arcs between $A$ and $B.$\\
Conversely, if $C$ cuts the arcs between $A$ and $B,$
by Lemma \ref{lem-cuteq} it follows that
$C$ cuts the arcs between ${\mathcal S}(A)$ and ${\mathcal S}(B).$
Thus we can apply Lemma \ref{lem-cut} with respect to the triple
$({\mathcal S}(A),{\mathcal S}(B),C).$ In particular, the function
$f$ is defined as in \eqref{eq-frm},
with $\rho$ like in \eqref{eq-ro} and $\mu: X \to \{-1,0,1\}$ as follows:
\begin{equation}\label{eq-mu2}
\mu(x):=
\quad
\left
\{
\begin{array}{ll}
~\,0&\; \mbox{ if } \, x\in C,\\
\!-1 &\; \mbox{ if } \, x\not\in C \;
\mbox{and $\exists$ a path $\gamma_x:[0,1]\to X\setminus C$} \\
~\quad &\mbox{
\,such that $\gamma_x(0)\in {\mathcal S}(A)$ and $\gamma_x(1) = x,$}\\
~\,1&\; \mbox{ elsewhere.}\\
\end{array}
\right.
\end{equation}
\end{proof}

\begin{figure}[ht]
{
\includegraphics[width=2.6in,height=2in]{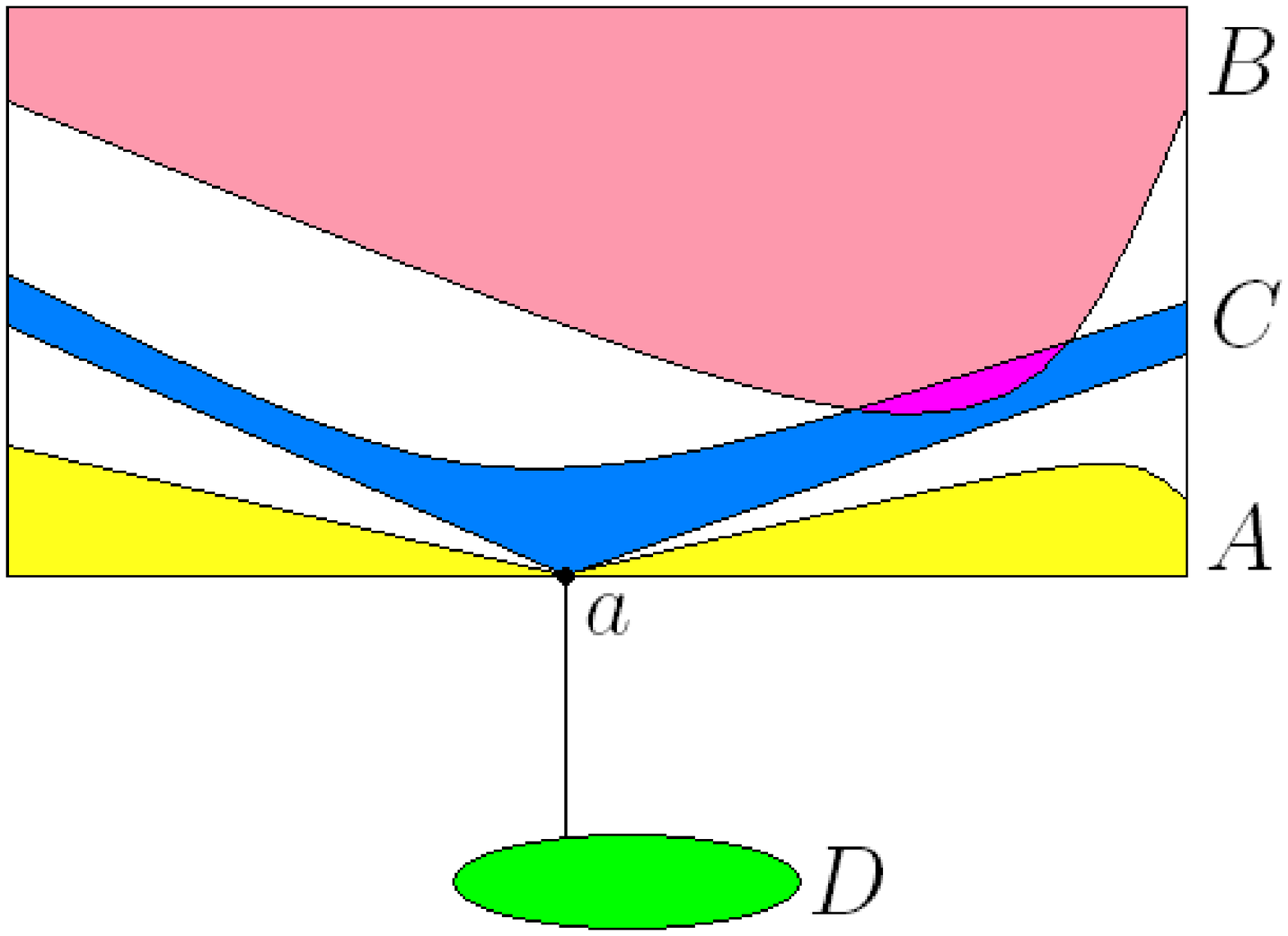}
\includegraphics[width=2.6in,height=2in]{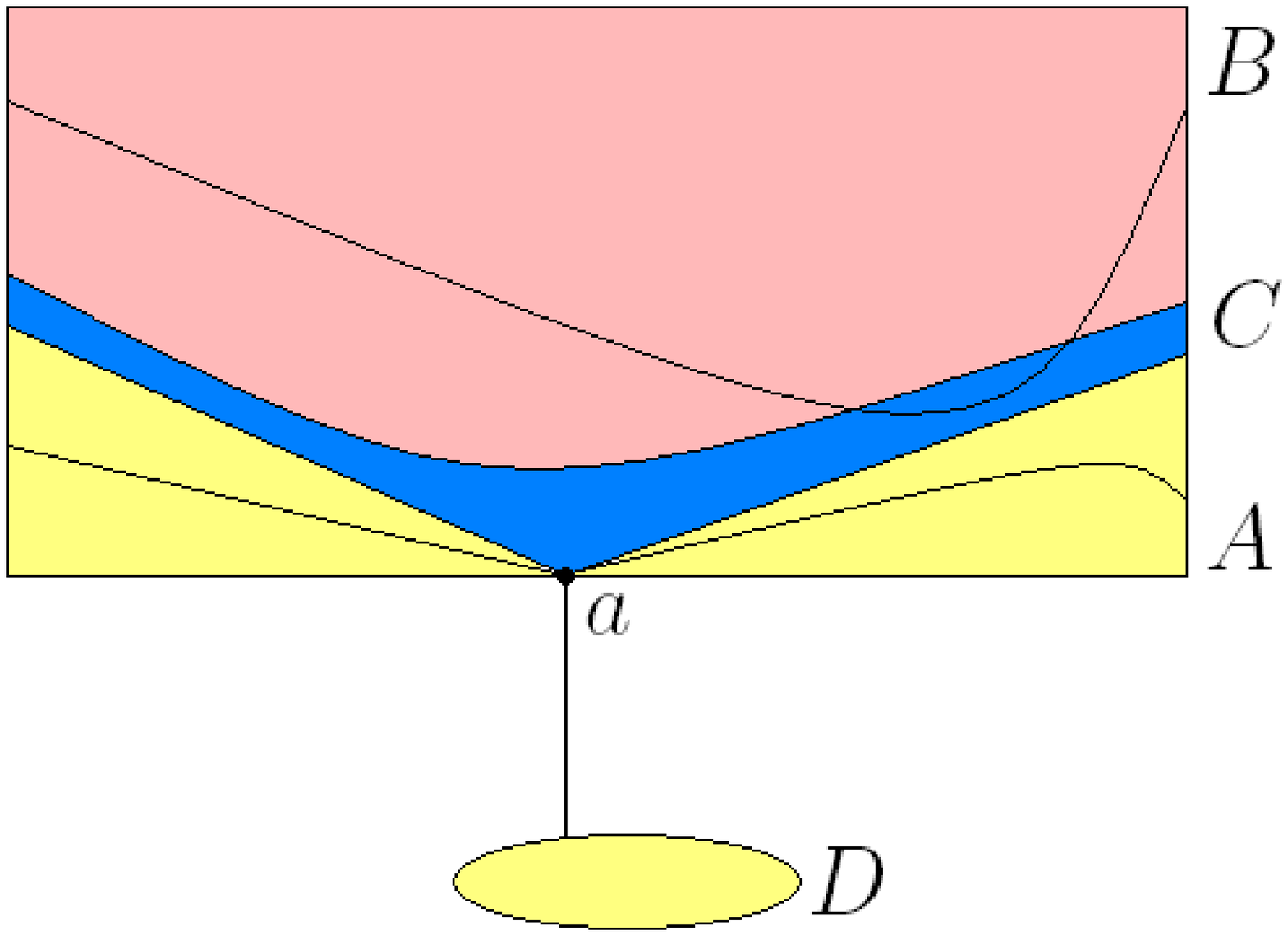}
}
\caption{\footnotesize {The picture on the left-hand side gives an illustration
of the framework described in Definition \ref{def-cut}.
The space $X$ is the figure itself as a subset of the plane and the set $C$ (darker region) cuts the arcs between $A$
and $B.$
We allow a nonempty intersection between $B$ and $C$
as well as between $A$ and $C$ (the singleton $\{a\}$).
Notice that the only manner to connect with a path
the points of $A$ to the points of the ``appendix'' $D$ is passing through the point $a\in A\cap C.$
Therefore on the set $D$ the map $f$ in Lemma \ref{lem-cut} with $\mu$ as in \eqref{eq-mu} would be positive, while it is
negative on $A\setminus C.$
The picture on the right-hand side provides an interpretation of Lemma \ref{lem-cuts}.
In regard to a function $f$ having its factor $\mu$
defined like in \eqref{eq-mu2}, we have painted
with a light color (yellow, in the colored version) the points where $f<0$ and in grey (pink) the points where
$f > 0.$ Notice that the region $D$ has been painted in light color, because
$D\subseteq {\mathcal S}(A).$
}}
\label{fig-si}
\end{figure}

So far we have considered only the case of paths connecting two disjoint sets $A$ and $B.$
This choice is motivated
by the foregoing examples for subsets of $N$-dimensional spaces.
For the sake of completeness we end this subsection by discussing
the framework in which $A$ and $B$ are joined by a continuum. We confine ourselves to the
following lemma which will find an application in Theorem \ref{th-fpn}.

\begin{lemma}\label{lem-appr}
Let $X$ be a connected and locally arcwise connected metric space and let
$A, B\subseteq X$ be closed and nonempty sets with $A\cap
B=\emptyset.$ Let $\Gamma\subseteq X$ be a compact connected set
such that
$$\Gamma\cap A\ne\emptyset,\quad \Gamma\cap B\ne\emptyset.$$
Then, for every $\varepsilon > 0,$ there exists a path
$\gamma=\gamma_{\varepsilon}\,: [0,1]\to X$ with
$\gamma(0)\in A,$ $\gamma(1)\in B$ and
$${\overline{\gamma}} \subseteq B(\Gamma,\varepsilon).$$
Moreover, if $X$ is locally compact and $C\subseteq X$
is a closed set which cuts the arcs between $A$ and $B,$
then
$$\Gamma\cap C\ne\emptyset.$$
\end{lemma}
\begin{proof}
Let $\varepsilon > 0$ be fixed and consider, for every $p\in \Gamma,$
a radius $\delta_p\in\, ]0,\varepsilon[$ such that any two points
in $B(p,\delta_p)$ can be joined by a path in $B(p,\varepsilon).$
Since $\Gamma$ is compact, we can find a finite number of points
$$p_1\,,p_2\,,\dots, p_k\,\in \Gamma,$$
such that
$$\Gamma\subseteq
{\mathcal B}:=\bigcup_{i=1}^k B(p_i,\delta_i)
\subseteq B(\Gamma,\varepsilon), \; \mbox{ where } \;\delta_i:= \delta_{p_i}\,.$$
By the connectedness of $\Gamma,$
for any partition of $\{1,\dots,k\}$ into two nonempty disjoint subsets $J_1$
and $J_2\,,$ there exist $i\in J_1$ and $j\in J_2$ such that
$B(p_i,\delta_i)\cap B(p_j,\delta_j)\ne\emptyset.$
This, in turns, implies that we can rearrange the $p_i$'s (possibly changing their
order in the labelling) so that
$$B(p_i,\delta_i)\cap B(p_{i+1},\delta_{i+1})\ne\emptyset,\quad\forall\, i=1,\dots, k-1.$$
Hence we can conclude that for any pair of points $w,z\in
{\mathcal B},$ with $w\ne z,$ there is a path $\gamma=
\gamma_{w,z}$ joining $w$ with $z$ and such that
$\overline{\gamma}\subseteq B(\Gamma,\varepsilon).$ In particular,
taking $a\in A\cap \Gamma$ and $b\in B\cap \Gamma,$ we have that
there exists a path $\gamma = \gamma_{\varepsilon}\,: [0,1]\to
B(\Gamma,\varepsilon),$ with $\gamma(0) = a$ and $\gamma(1) = b$
and this proves the first part of the statement.\\
Assume now that $X$ is locally compact (i.e., for any $p\in X$ and
$\eta >0,$ there exists $0<\mu_p \leq\eta$ such that
${\overline{B(p,\mu_p)}}$ is compact). By the compactness of
$\Gamma$ we can find a finite number of points
$q_1\,,q_2\,,\dots, q_{l}\,\in \Gamma$
and corresponding radii $\mu_i := \mu_{q_i}$ such that
$$\Gamma\subseteq
{\mathcal A}:=\bigcup_{i=1}^l B(q_i,\mu_i)$$
and
${\overline {\mathcal A}} = \bigcup_{i=1}^l {\overline{B(q_i,\mu_i)}}$
is compact. Since ${\mathcal A}$ is an open neighborhood of the compact set $\Gamma,$
there exists $\varepsilon_0 > 0$ such that
$B(\Gamma,\varepsilon_0) \subseteq {\mathcal A}.$ Hence, for each
$0 < \varepsilon \leq \varepsilon_0$ we have that the set
${\overline{B(\Gamma,\varepsilon)}}$ is compact. Taking $\varepsilon=\frac 1 n,$ we know that for every $n\in\mathbb N_0$ there exists a path $\gamma_{n}\,: [0,1]\to X,$ with $\gamma_n(0)\in A,$ $\gamma_n(1)\in B$ and $\overline{\gamma_n}
\subseteq B(\Gamma,\tfrac 1 n).$ But, since $C$ cuts the arcs
between $A$ and $B,$ it follows that for every $n\in \mathbb N_0,$
there is a point $c_n\in C\cap B(\Gamma,\frac 1 n).$
For $n\geq \hat n$ large enough ($\hat n > 1/\varepsilon_0$),
the sequence $(c_n)_{n\geq\hat n}$ is
contained in the compact set ${\overline{B(\Gamma,\varepsilon_0)}}$
and therefore it admits a
converging subsequence
$c_{n_k}\to c^{*}\in {\overline{B(\Gamma,\varepsilon_0)}}.$
Since $d_X(c_{n_k},\Gamma)<\frac 1 {n_k}$ and
the sets $C,\,\Gamma$ are closed, the limit point $c^{*}\in
\Gamma\cap C.$ This concludes the proof.
\end{proof}

\subsection{The Intersection Lemma}\label{sub-hw}

We continue our exposition by presenting some
applications of the topological lemmas just proved to the
intersection of generalized surfaces which
separate the opposite edges of an $N$-dimensional cube. Such
generalized surfaces (cf. Definition \ref{def-gr}) will be described as zero-sets of continuous
scalar functions and therefore a nonempty intersection will be
obtained as a zero of a suitably defined vector field. To this
aim, we recall a classical result about the existence of zeros for
continuous maps in ${\mathbb R}^N,$ the Poincar\'{e}-Miranda
Theorem\index{Poincar\'{e}-Miranda
Theorem}.

\begin{theorem}\label{th-pm}
Let $I^N:= [0,1]^N$ be the $N$-dimensional unit cube of ${\mathbb R}^N,$
for which we denote by
$[x_i = k]:= \{x=(x_1,\dots,x_N)\in I^N: \, x_i = k\}.$
Let
$F = (F_1,\dots, F_N): I^N\to {\mathbb R}^N$
be a continuous mapping such that, for each $i\in\{1,\dots,N\},$
$$F_i(x)\leq 0,\;\forall\, x\in [x_i = 0]\,\mbox{ and }\;
F_i(x)\geq 0,\;\forall\, x\in [x_i = 1]$$
or
$$F_i(x)\geq 0,\;\forall\, x\in [x_i = 0]\,\mbox{ and }\;
F_i(x)\leq 0,\;\forall\, x\in [x_i = 1].$$
Then there exists $\bar{x}\in I^N$ such that $F(\bar{x}) = 0.$
\end{theorem}

\medskip

\noindent
Let us introduce now the spaces we are going to consider.

\begin{definition}\label{def-gr}
{\rm{
Let $Z$ be a metric space and
$$h: {\mathbb R}^N\supseteq I^N \to X\subseteq Z$$
be a homeomorphism of $I^N:=[0,1]^N$ onto its image $X.$
We call the pair
$${\widehat{X}}:= ({X},h)$$
a \textit{generalized $N$-dimensional rectangle of $Z$}\index{rectangle! generalized $N$-dim.}.
We also set
$${X}_{i}^{\ell}:= h([x_i = 0]),\quad
{X}_{i}^{r}:= h([x_i = 1])$$ and name them the \textit{left} and
the \textit{right} $i$-faces of $X.$
\\
Finally, we define
$$\vartheta X:= h\left(\partial I^N\right)$$
and call it the \textit{contour}\index{contour} of $X.$}}
\end{definition}
\noindent
As it is immediate to see, the generalized $N$-dimensional rectangles
are a natural extension to the higher dimension of the generalized rectangles in Section \ref{sec-sap}. In a similar
way also the concept of oriented rectangle will be transposed to the $N$-dimensional framework (cf. Definition \ref{def-or}).

\medskip

\noindent
Our main result on the intersection of generalized $N$-dimensional rectangles is Theorem \ref{th-hw} below, which can be considered as a variant of
Hurewicz-Wallman Lemma about dimension \cite{HuWa-41}. The statements of the two results are in fact very similar, but the lemma in \cite{HuWa-41} concerns, instead of our concept of cutting, the stronger notion of separation and requires the sets  $A, B, C$ in Definition \ref{def-cut} to be pairwise disjoint (see \cite{Ku-68}). Furthermore, with reference to Definition \ref{def-gr}, in the statement of Hurewicz-Wallman Lemma only the very special case in which $Z= {\mathbb R}^N,$ $X= I^N$ and $h= {\mbox {Id}}_{{\mathbb R}^N}$ is considered.
For such reasons we have chosen to provide all the details.

\begin{theorem}[Intersection Lemma]\index{Hurewicz-Wallman Lemma}\label{th-hw}
Let ${\widehat{X}}:= ({X},h)$ be a generalized $N$-dimensional rectangle of a metric space $Z.$
Assume that, $\forall i\in\{1,\dots,N\},$ there exists a compact set
${\mathcal S}_i\subseteq {X}$
that cuts the arcs between ${X}_{i}^{\ell}$
and ${X}_{i}^{r}$ in $X.$ Then
$$\bigcap_{i=1}^N {\mathcal S}_i \ne\emptyset.$$
\end{theorem}

\begin{proof}
Through the inverse of the homeomorphism
$h: {\mathbb R}^N\supseteq I^N \to X\subseteq Z$
we can define the compact sets
$$C_i:= h^{-1}({\mathcal S}_i),$$
which cut the arcs between
$[x_i=0]$ and $[x_i=1]$ in $I^N$ (for $i=1,\dots, N$).
Clearly, it is sufficient to prove that
$$\bigcap_{i=1}^N C_i \ne\emptyset.$$
By Lemma \ref{lem-cut}, for every $i=1,\dots,N,$ there exists
a continuous function $f_i: I^N\to {\mathbb R}$ such that
$f_i \leq 0$ on $[x_i =0]$ and $f_i \geq 0$ on $[x_i=1].$ Moreover,
$$C_i =\left\{x\in I^N:\, f_i(x) = 0\right\}.$$
The continuous vector field
$f^{^{^{{\!\!\!\!\rightarrow}}}}:= (f_1,\dots,f_N):I^N\to {\mathbb R}^N$
satisfies the assumptions of Poincar\'{e}-Miranda Theorem \ref{th-pm} and therefore
there exists ${\bar{x}}\in I^N$ such that
$$f_i({\bar{x}})=0,\;\;\forall\, i=1,\dots,N.$$
Hence ${\bar{x}}\in \bigcap_{i=1}^N C_i$ and the proof is complete.
\end{proof}

\noindent
We recall that the previous result was applied in \cite{PiZa-07} in order to extend
some recent theorems about fixed points and periodic points for
continuous mappings in Euclidean spaces. In particular, in \cite[Corollary 3.5]{PiZa-07}
we generalized, via a simplified proof, a theorem by Kampen \cite{Ka-00},
while in \cite[Theorem 3.8]{PiZa-07}
we obtained an extension of a result by Zgliczy\'nski \cite{Zg-01}
about periodic points associated to Markov partitions.

\subsection{Zero sets of maps depending on parameters}\label{sub-cont}

As a next step we deal with the intersection of generalized surfaces
which separate the opposite edges of an $N$-dimensional cube,
in the case that the number of
cutting surfaces is smaller than the dimension of the space.
Our main tool is a result by
Fitzpatrick, Massab\'{o} and  Pejsachowicz (see \cite[Theorem 1.1]{FiMaPe-86})
on the covering dimension of the
zero set of an operator depending on parameters.
For the reader's convenience, we recall the concept of covering dimension \cite{En-78} in Definition \ref{def-dim} below, where by {\textit{order}} of a family ${\mathfrak A}$
of subsets of the metric space $Z$
we mean the largest integer $n$ such that the family ${\mathfrak A}$ contains
$n+1$ sets with a non-empty intersection; if no such integer exists,
the family ${\mathfrak A}$ has order infinity.

\begin{definition}\label{def-dim}
{\rm{\cite[p.54, p.208]{En-78}
Let $Z$ be a metric space. We say that $\mbox{\rm dim \!} Z \leq n$
if every finite open cover of the space $Z$ has a finite open [closed]
refinement of order $\leq n.$
The object $\mbox{\rm dim \!} Z\in {\mathbb N}\cup\{\infty\}$
is called the \textit{covering dimension} or the
{\textit{\v Cech-Lebesgue dimension}}\index{dimension, covering} of the metric space $Z.$
According to \cite{FiMaPe-86}, if $z_0\in Z,$ we also say that
{\textit{$\mbox{\rm dim \!} Z \geq j$ at $z_0$}} if each neighborhood
of $z_0$ has dimension at least $j.$}}
\end{definition}

\noindent
By a classical result from topology (cf. \cite[The coincidence theorem]{En-78}),
in separable metric spaces
the covering dimension coincides with the {inductive dimension}
\cite[p.3]{En-78}.

\smallskip

\noindent
In view of the following results we also recall that,
given an open bounded set ${\mathcal O}\subseteq {\mathbb R}^N$ and $n\in \{1,\dots, N-1\},$ a continuous map
$\pi: {\overline{\mathcal O}}\to {\mathbb R}^{N-n}$ is a \textit{complement}\index{complementing map} for the continuous map
$F: {\overline{\mathcal O}}\to {\mathbb R}^{n}$ if
the topological degree
$\mbox{\rm deg}((\pi,F),{\mathcal O},0)$ is defined and nonzero \cite{FiMaPe-86}.
According to \cite{HoYo-61}, a mapping $f$ of a space $X$ into a space $Y$ is
said to be \textit{inessential} if $f$ is homotopic to a constant; otherwise $f$ is
\textit{essential}\index{essential map}.\\
At last we introduce a further notation. Given
 an $N$-dimensional rectangle ${\mathcal R}:= \prod_{i=1}^N [a_i,b_i],$ we
denote its opposite $i$-faces by
$${\mathcal R}^{\ell}_i := \{x\in {\mathcal R}: \, x_i = a_i\} \quad \mbox{and} \quad
{\mathcal R}^{r}_i := \{x\in {\mathcal R}: \, x_i = b_i\}\,.$$

\begin{theorem}\label{th-es}
Let ${\mathcal R}:= \prod_{i=1}^N [a_i,b_i]$ be an $N$-dimensional
rectangle and let $P = (p_1,\dots,p_N)$ be any interior point of
${\mathcal R}.$ Let $n\in \{1,\dots, N-1\}$ be fixed. Suppose that
$F = (F_1,\dots, F_n): {\mathcal R}\to {\mathbb R}^n$ is a
continuous mapping such that, for each $i\in\{1,\dots,n\},$
$$F_i(x) < 0,\;\forall\, x\in {\mathcal R}^{\ell}_i\quad\mbox{and }\;\,
F_i(x) > 0,\;\forall\, x\in {\mathcal R}^{r}_i$$
or
$$F_i(x) > 0,\;\forall\, x\in {\mathcal R}^{\ell}_i\quad\mbox{and }\;\,
F_i(x) < 0,\;\forall\, x\in {\mathcal R}^{r}_i\,.$$ Define also
the affine map
\begin{equation}\label{eq-prt}
\pi: {\mathbb R}^N\to {\mathbb R}^{N-n},\quad
\pi_j(x_1,\dots,x_{N}):= x_j - p_j\,,\; j = n + 1,\dots, N.
\end{equation}
Then there exists a connected subset ${\mathcal Z}$ of
$$F^{-1}(0) = \{x\in {\mathcal R}:\, F_i(x) =0,\,\forall\, i=1,\dots,n\}$$
whose dimension at each point is at least
$N-n.$
Moreover,
$$\mbox{\rm dim}({\mathcal Z}\cap\partial {\mathcal R})\geq N - n - 1$$
and
$$\pi: {\mathcal Z}\cap \partial {\mathcal R} \to {\mathbb R}^{N-n}\setminus\{0\}$$
is essential.
\end{theorem}
\begin{proof}
We define the continuous mapping
$$H := (F,\pi):{\mathcal R}\to {\mathbb R}^N.$$
By the assumptions on $F$ and $\pi$ we have
$$\mbox{\rm deg}(H, {\mbox {Int}}(\mathcal R),0) = (-1)^d\ne 0,$$
where $d$ is the number of components $i\in\{1,\dots,n\}$ such
that $F_i(x) > 0$ for $x \in {\mathcal R}^{\ell}_i$ and
$F_i(x) < 0$ for $x\in {\mathcal R}^{r}_i.$ Hence $\pi$ turns
out to be a complementing map for $F$ (according to
\cite{FiMaPe-86}). A direct application of \cite[Theorem
1.1]{FiMaPe-86}  gives the thesis (observe that the dimension $m$ in
\cite[Theorem 1.1]{FiMaPe-86} corresponds to our $N-n$).
\end{proof}

\noindent
Notice that if $a_i < 0 < b_i\,,$ for $i=1,\dots,N,$ then we can take $P=0,$ so
that the complementing map is just the projection $\pi: {\mathbb
R}^N\to {\mathbb R}^{N-n}\,.$

\bigskip

\noindent
A more elementary version of Theorem \ref{th-es} can be
given for the zero set of a vector field with range in ${\mathbb R}^{N-1}\,.$
In this case it is possible to achieve the thesis by a direct use of the classical
Leray-Schauder Continuation Theorem \cite{LeSh-34},
instead of the more sophisticated tools
in \cite{FiMaPe-86}.
Namely, we have:

\begin{theorem}\label{th-z}
Let ${\mathcal R}:= \prod_{i=1}^N [a_i,b_i]$ be an $N$-dimensional rectangle
and let
$F = (F_1,\dots, F_{N-1}): {\mathcal R}\to {\mathbb R}^{N-1}$
be a continuous mapping such that, for each $i\in\{1,\dots,N-1\},$
$$F_i(x) < 0,\;\forall\, x\in {\mathcal R}^{\ell}_i\quad\mbox{and }\;
F_i(x) > 0,\;\forall\, x\in {\mathcal R}^{r}_i$$
or
$$F_i(x) > 0,\;\forall\, x\in {\mathcal R}^{\ell}_i\quad\mbox{and }\;
F_i(x) < 0,\;\forall\, x\in {\mathcal R}^{r}_i\,.$$
Then there exists a closed connected subset ${\mathcal Z}$ of
$$F^{-1}(0) = \{x\in {\mathcal R}:\, F_i(x) =0,\,\forall\, i=1,\dots,N-1\}$$
such that
$${\mathcal Z}\cap {\mathcal R}^{\ell}_N\ne\emptyset,\quad
{\mathcal Z}\cap {\mathcal R}^{r}_N\ne\emptyset.$$
\end{theorem}
\begin{proof}
We split $x = (x_1,\dots,x_{N-1},x_N)\in {\mathcal R}\subseteq {\mathbb R}^N$ as
$x = (y,\lambda)$ with
$$y= (x_1,\dots,x_{N-1})
\in {\mathcal M}:= \prod_{i=1}^{N-1} [a_i,b_i],\quad
\lambda = x_N\in [a_N,b_N]$$
and define
$$f = f(y,\lambda): {\mathcal M}\times [a_N,b_N]\to {\mathbb R}^{N-1}\,,\quad
f(y,\lambda):= F(x_1,\dots,x_{N-1},\lambda),$$
treating the variable $x_N =\lambda$ as a parameter for the
$(N-1)$-dimensional vector field
$$f_{\lambda}(\cdot) = f(\cdot,\lambda).$$
By the assumptions on $F$ we have
$$\mbox{\rm deg}(f_{\lambda}, {\mbox {Int}}(\mathcal M),0) = (-1)^d\ne 0,
\quad\forall\, \lambda\in [a_N,b_N],$$ where $d$ is the number of
components $i\in\{1,\dots,N-1\}$ such that $F_i(x) > 0$ for $x
\in {\mathcal R}^{\ell}_i$ and $F_i(x) < 0$ for $x\in
{\mathcal R}^{r}_i.$ The Leray-Schauder Continuation Theorem
\cite[Th\'{e}or\`{e}me Fondamental]{LeSh-34} ensures the existence of a closed
connected set
$$\mathcal Z\subseteq \left\{(y,\lambda)\in {\mathcal M}\times [a_N,b_N]\,:\, f(y,\lambda) =
0\in {\mathbb R}^{N-1}\right\},$$
whose projection onto the $\lambda$-component covers the interval $[a_N,b_N].$
By the above positions the thesis immediately follows.
\end{proof}

\noindent
For the interested reader, we mention that Theorem \ref{th-z} can be found in \cite{KuSoTu-00} and that it was then applied in \cite{KuPoSoTu-05}.

\bigskip

\noindent
In the next lemma we take the unit cube
$I^N:=[0,1]^N$ as $N$-dimensional rectangle and choose
the interior point $P = \left(\tfrac 1 2,\tfrac 1 2,\dots, \tfrac 1 2\right),$ in order to apply Theorem \ref{th-es}. Obviously,
any other point interior to $I^N$ could be chosen as well.

\begin{lemma}\label{lem-es}
Let $n\in \{1,\dots, N-1\}$ be fixed.
Assume that, $\forall \,i\in\{1,\dots,n\},$ there is a compact set
${\mathcal S}_i\subseteq I^N$
that cuts the arcs between $[x_i = 0]$
and $[x_i = 1]$ in $I^N.$ Then
there exists a connected subset ${\mathcal Z}$ of
$\,\displaystyle{\bigcap_{i=1}^n {\mathcal S}_i \ne \emptyset},$
whose dimension at each point is at least
$N-n.$ Moreover,
$$\mbox{\rm dim}\left({\mathcal Z}\cap\partial I^N\right)\geq N - n - 1$$
and
$$\pi: {\mathcal Z}\cap \partial I^N \to {\mathbb R}^{N-n}\setminus\{0\}$$
is essential, where $\pi$ is the affine map defined in \eqref{eq-prt}.
\end{lemma}
\begin{proof}
For any fixed index $i^*\in\{1,\dots,n\}$ we introduce the {{tunnel set}}
$$T_{i^*}\,:= \prod_{i=1}^{i^*-1}\, [0, 1]\times {\mathbb R}\times
\prod_{i=i^*+1}^N [0,1].$$
It is immediate to check that ${\mathcal S}_{i^*}$ cuts the arcs between
$[x_{i^*} = 0]$
and $[x_{i^*} = 1]$ in $T_{i^*}\,.$
\\
By Lemma \ref{lem-cuts} there exists a continuous function
$f_{i^*}\,: T_{i^*}\to {\mathbb R}$ such that
$$f_{i^*}(x) \leq 0\,,\;\forall\, x\in T_{i^*}\,\;\mbox{with } x_{i^*}\leq 0\quad
\mbox{and } \;
f_{i^*}(x) \geq 0\,,\;\forall\, x\in T_{i^*}\,\;\mbox{with } x_{i^*}\geq 1\,,$$
and moreover
$${\mathcal S}_{i^*}\,= \{ x\in T_{i^*}\,: \, f_{i^*}(x) = 0\}.$$
By this latter property and the fact that ${\mathcal S}_{i^*}\subseteq I^N$
it follows that
$$f_{i^*}(x) < 0\,,\;\forall\, x\in T_{i^*}\,\;\mbox{with } x_{i^*}< 0\quad
\mbox{and } \;
f_{i^*}(x) > 0\,,\;\forall\, x\in T_{i^*}\,\;\mbox{with } x_{i^*} >  1\,.$$
Now we define, for $x = (x_1,\dots,x_{i^*-1},x_{i^*},x_{i^*+1},\dots, x_N)\in {\mathbb R}^N,$
the continuous function
$$F_{i^*}(x):= f_{i^*}\bigl(\,\eta_{[0,1]}(x_1),\dots,
\eta_{[0,1]}(x_{i^* -1}), x_{i^*},
\eta_{[0,1]}(x_{i^* +1}),\dots,
\eta_{[0,1]}(x_{N})\,\bigr ),$$
where
\begin{equation}\label{eq-pri}
\eta:\mathbb R\to [0,1],\,\,\eta_{[0,1]}(s):= \max\{0,\min\{s,1\} \}
\end{equation}
is the projection of ${\mathbb R}$ onto the interval $[0,1].$
As a consequence of the above positions we find that
$$F_{i^*}(x) < 0,\;\forall\, x\in{\mathbb R}^N\,: \, x_{i^*} < 0\quad\mbox{and }\;
F_{i^*}(x) > 0,\;\forall\, x\in{\mathbb R}^N\,: \, x_{i^*} >  1\,.$$
We can thus apply Theorem \ref{th-es} to the map
$F = (F_1,\dots,F_n)$ restricted to the $N$-dimensional rectangle
$${\mathcal R}:= \prod_{i=1}^{n}\,\,[-1,2] \times \prod_{i=n+1}^{N}[0,1].$$
Clearly,
$$\bigl(F\restriction_{\mathcal R}\bigr)^{-1}(0) = \bigcap_{i=1}^n {\mathcal S}_i\subseteq I^N$$
and the proof is complete.
\end{proof}

\begin{remark}\label{rem-conv}
{\rm{
Both in Theorem \ref{th-es} and in Lemma \ref{lem-es} the fact that we have privileged
the first $n$ components is purely conventional. It is evident that the results are valid
for any finite sequence of indexes $i_1 < i_2 < \dots < i_n$ in $\{1,\dots,N\}.$
The same observation applies systematically to all the other results
(preceding and subsequent) in which some directions are conventionally chosen.
Moreover, notice that Lemma \ref{lem-es} is invariant under homeomorphisms in a sense that is described
in Theorem \ref{th-invh} below.}}

\hfill$\lhd$\\
\end{remark}

\begin{theorem}\label{th-invh}
Let ${\widehat{X}}:= ({X},h)$ be a generalized $N$-dimensional rectangle of a metric space $Z.$ Let $i_1 < i_2 < \dots < i_n$ be a finite sequence of $n\ge 1$ indexes
in $\{1,\dots,N\}.$
Assume that, for each $j\in\{i_1,\dots,i_n\},$ there is a compact set
${\mathcal S}_j\subseteq {X}$
that cuts the arcs between ${X}_{j}^{\ell}$
and ${X}_{j}^{r}$ in $X.$
Then there exists a compact connected subset ${\mathcal Z}$ of
$\,\displaystyle{\bigcap_{k=1}^n {\mathcal S}_{i_k}\ne\emptyset},$
whose dimension at each point is at least
$N-n.$ Moreover,
$$\mbox{\rm dim}({\mathcal Z}\cap\vartheta X)\geq N - n - 1$$
and
$$\pi: h^{-1}({\mathcal Z})\cap \partial I^N \to {\mathbb R}^{N-n}\setminus\{0\}$$
is essential, where $\pi$ is defined as in \eqref{eq-prt} for $P = \left(\tfrac 1 2,\tfrac 1 2,\dots, \tfrac 1 2\right).$
\end{theorem}
\begin{proof}
The result easily follows by moving to the setting of Lemma \ref{lem-es}
through the homeomorphism $h^{-1}$ and repeating the arguments used therein.
\end{proof}

\noindent
At last we present a result (Corollary \ref{cor-jo}) which plays a crucial role
in the subsequent proofs.
It concerns the case $n=N-1$ and could be obtained by
suitably adapting the arguments employed in Lemma \ref{lem-es}.
However, due to its significance for our applications, we prefer to provide a detailed proof
using Theorem \ref{th-z}
(which requires only the knowledge of the Leray-Schauder principle
and therefore, in some sense, is more elementary).
Corollary \ref{cor-jo} extends to an arbitrary dimension some results
in \cite[Appendix]{ReZa-00} which were there proved only for $N=2$ using \cite{Sa-80}.

\begin{corollary}\label{cor-jo}
Let ${\widehat{X}}:= ({X},h)$ be a generalized $N$-dimensional rectangle of a metric space $Z.$
Let $k\in \{1,\dots,N\}$ be fixed.
Assume that, for each $j\in\{1,\dots,N\}$ with $j\ne k,$
there exists a compact set
${\mathcal S}_j\subseteq {X}$
that cuts the arcs between ${X}_{j}^{\ell}$
and ${X}_{j}^{r}$ in $X.$
Then there exists a compact connected subset ${\mathcal C}$ of
$\,\displaystyle{\bigcap_{i\ne k} {\mathcal S}_{i}\ne \emptyset},$
such that
$${\mathcal C}\cap {X}_{k}^{\ell}\ne \emptyset,\quad
{\mathcal C}\cap {X}_{k}^{r}\ne\emptyset.$$
\end{corollary}
\begin{proof}
Without loss of generality (if necessary, by a permutation of the
coordinates), we assume $k=N.$ In this manner,
using the homeomorphism $h^{-1}: Z\supseteq X= h(I^N)\to I^N,$
we can confine ourselves to the
following framework:
\\
For each $j\in\{1,\dots,N-1\},$ there exists a compact set
$${\mathcal S}\,'_j \,:= h^{-1}({\mathcal S}_j)\subseteq I^N$$
that cuts the arcs between $[x_j=0]$ and $[x_j=1]$
in $I^N\,.$\\
Proceeding as in the proof of Lemma \ref{lem-es},
for any fixed $i^*\in\{1,\dots,N-1\},$ we introduce the {{tunnel set}}
$$T_{i^*}\,:= \prod_{i=1}^{i^*-1}\,\, [0,1]\times {\mathbb R}\times
\prod_{i=i^*+1}^N [0, 1]$$
and find that ${\mathcal S}\,'_{i^*}$ cuts the arcs between
$[x_{i^*} = 0]$
and $[x_{i^*} = 1]$ in $T_{i^*}\,.$
\\
Hence, by Lemma \ref{lem-cuts} there exists a continuous function
$f_{i^*}\,: T_{i^*}\to {\mathbb R}$ such that
$$f_{i^*}(x) \leq 0\,,\;\forall\, x\in T_{i^*}\,\;\mbox{with } x_{i^*}\leq 0\quad
\mbox{and } \;
f_{i^*}(x) \geq 0\,,\;\forall\, x\in T_{i^*}\,\;\mbox{with } x_{i^*}\geq  1\,,$$
as well as
$${\mathcal S}\,'_{i^*}\,= \{ x\in T_{i^*}\,: \, f_{i^*}(x) = 0\}.$$
More precisely it holds that
$$f_{i^*}(x) < 0\,,\;\forall\, x\in T_{i^*}\,\;\mbox{with } x_{i^*}< 0\quad
\mbox{and } \;
f_{i^*}(x) > 0\,,\;\forall\, x\in T_{i^*}\,\;\mbox{with } x_{i^*} >  1\,.$$
We define, for $x = (x_1,\dots,x_{i^*-1},x_{i^*},x_{i^*+1},\dots, x_N)\in {\mathbb R}^N,$
the continuous function
$$F_{i^*}(x):= f_{i^*}\bigl(\,\eta_{[0,1]}(x_1),\dots,
\eta_{[0, 1]}(x_{i^* -1}), x_{i^*},
\eta_{[0, 1]}(x_{i^* +1}),\dots,
\eta_{[0, 1]}(x_{N})\,\bigr ),$$
where
$\eta_{[0, 1]}$
is the projection of ${\mathbb R}$ onto the interval $[0, 1]$
defined as in \eqref{eq-pri}. Then we have
$$F_{i^*}(x) < 0,\;\forall\, x\in{\mathbb R}^N\,: \, x_{i^*} < 0\quad\mbox{and }\;
F_{i^*}(x) > 0,\;\forall\, x\in{\mathbb R}^N\,: \, x_{i^*} >  1\,.$$
Now we consider the map
$F = (F_1,\dots,F_{N-1})$ restricted to the $N$-dimensional rectangle
$${\mathcal R}:= \prod_{i=1}^{N-1}[-\varepsilon,1+\varepsilon]
\times [0,1],$$
for any fixed $\varepsilon > 0.$
Since
$$\bigl(F\restriction_{\mathcal R}\bigr)^{-1}(0) = \bigcap_{i=1}^{N-1} {\mathcal S}\,'_i\subseteq I^N,$$
we can set
$${\mathcal C}:= h({\mathcal Z}),$$
where ${\mathcal Z} \subseteq \bigl(F\restriction_{\mathcal R}\bigr)^{-1}(0)$
comes from the statement of Theorem \ref{th-z}, and this concludes the proof.
\end{proof}

\noindent
As we shall see, the latter result turns out to be our main
ingredient in the next subsection for the study of the dynamics of continuous maps defined on topological $N$-dimensional rectangles
which possess, in a very broad sense, a one-dimensional expansive
direction. Indeed, it allows to prove the crucial fixed point Theorem \ref{th-fpn} (similar to Theorem \ref{th-fp})
and subsequent results on the existence of multiple periodic points. Notice that the
role of Corollary \ref{cor-jo} in the proof of Theorem \ref{th-fpn} is analogous to the one of the Crossing Lemma \ref{lem-cr} in the verification of Theorem \ref{th-fp}.

\subsection{Stretching along the paths in the $N$-dimensional case}\label{sub-strn}
Finally, we are in position to provide an extension to $N$-dimensional spaces of the results
obtained in Section \ref{sec-sap} for the planar maps which expand the paths along a certain direction. To this aim, we reconsider Definition \ref{def-gr} in order to focus our attention on
generalized $N$-dimensional rectangles in which we have fixed once for all
the left and right sides. In the applications, these opposite sides
give a sort of orientation to the generalized $N$-dimensional rectangles
and in fact they are usually related to the expansive direction.

\begin{definition}\label{def-or}
\rm{
Let $Z$ be a metric space and let
${\widehat{X}}:= ({X},h)$ be a
generalized $N$-dimensional rectangle of $Z.$
We set
$$X_{\ell}:= h([x_N=0]),\quad X_{r}:= h([x_N=1])$$
and
$$X^-:= X_{\ell}\cup X_{r}\,.$$
The pair
$${\widetilde{X}}:= (X, X^-)$$
is called an {\textit{oriented $N$-dimensional rectangle}}\index{rectangle! oriented $N$-dim.}
\textit{of $Z$}.
}
\end{definition}

\begin{remark}\label{rem-convn}
\rm{A comparison between Definitions \ref{def-gr} and \ref{def-or}
shows that an oriented $N$-dimensional rectangle is just a generalized $N$-dimensional rectangle in which
we privilege the two subsets of its contour which correspond to the
opposite faces for some fixed component (namely, the $x_N$-component).
In analogy with Remark \ref{rem-conv}, we point out that the choice of the
$N$-th component is purely conventional.
For example, in some other papers (see \cite{GiRo-03, PiZa-05, ZgGi-04}) and also in the planar examples from Section \ref{sec-sap},
the first component was selected.
Clearly, there is no substantial difference as the homeomorphism $h: {\mathbb R}^N\supseteq I^N \to X\subseteq Z$
could be composed with a permutation matrix, yielding to a new
homeomorphism with the same image set. From this point of view, our
definition fits to the one of $h$-set of $(1,N-1)$-type,
given by Zgliczy\'{n}ski and Gidea in \cite{ZgGi-04}
for a subset of ${\mathbb R}^N$ which is obtained as the inverse image
of the unit cube through a homeomorphism of ${\mathbb R}^N$ onto itself.
The similar concept of $(1,N-1)$-window is then considered by Gidea and Robinson in \cite{GiRo-03}:
it is defined as a homeomorphic copy of the unit cube $I^N$ of $\mathbb R^N$
through a homeomorphism whose domain is an open neighborhood of $I^N.$
}

\hfill$\lhd$\\
\end{remark}

\noindent
We now adapt to maps
between oriented $N$-dimensional rectangles the concept of stretching along the paths,
already considered in Definition \ref{def-sap} in regard to the planar case.

\begin{definition}\label{def-sapn}
{\rm{Let $Z$ be a metric
space and let $\psi: Z \supseteq D_{\psi}\to Z$ be a map
defined on a set $D_{\psi}.$ Assume that ${\widetilde{X}}:=
({X},{X}^-)$ and ${\widetilde{Y}}:=
({Y},{Y}^-)$ are oriented rectangles of
$Z$ and let ${\mathcal K}\subseteq {X}\cap D_{\psi}$
be a compact set. We say that \textit{$({\mathcal K},\psi)$ stretches
${\widetilde{X}}$ to ${\widetilde{Y}}$ along the
paths} \index{stretching along the paths@\textsl{stretching along the paths} $\stretchx$} and write
\begin{equation*}
({\mathcal K},\psi): {\widetilde{X}} \stretchx {\widetilde{Y}},
\end{equation*}
if the following conditions hold:
\begin{itemize}
\item{} \; $\psi$ is continuous on ${\mathcal K}\,;$ \item{} \;
For every path $\gamma: [0,1]\to {X}$ such that
$\gamma(0)\in {X}^-_{\ell}$ and $\gamma(1)\in {
X}^-_{r}$ (or $\gamma(0)\in {X}^-_{r}$ and $\gamma(1)\in
{X}^-_{\ell}$), there exists a subinterval
$[t',t'']\subseteq [0,1]$ such that
$$\gamma(t)\in {\mathcal K},\quad \psi(\gamma(t))\in {Y}\,,\;\;
\forall\, t\in [t',t'']$$ and, moreover, $\psi(\gamma(t'))$ and
$\psi(\gamma(t''))$ belong to different sides of ${Y}^-.$
\end{itemize}
}}
\end{definition}

\noindent
By the similarity between Definitions \ref{def-sap} and \ref{def-sapn}, it is
immediate to see that the remarks in Section \ref{sec-sap} about the stretching along the path relation for planar
maps remain valid also in the higher dimensional setting. In particular the comments on the relationship with the Brouwer fixed point Theorem or with the theory
of the topological horseshoes in \cite{KeKoYo-01,KeYo-01} hold true.
Suitable modifications of Definitions \ref{def-cn}--\ref{def-ea}
could be presented as well. For the sake of conciseness we prefer to omit them.
In view of its significance for the following treatment, we present however the next fixed point theorem, which shows that the set $\mathcal K$ still plays a crucial role.

\begin{theorem}\label{th-fpn}
Let $Z$ be a metric space and let $\psi: Z \supseteq D_{\psi}\to Z$ be a map defined on a set $D_{\psi}.$ Assume that
${\widetilde{X}}:= ({X},{X}^-)$
is an oriented $N$-dimensional rectangle of $Z.$ If ${\mathcal K}\subseteq {X}\cap D_{\psi}$ is a compact set for which it holds that
\begin{equation}\label{eq-fpn}
({\mathcal K},\psi): {\widetilde{X}} \stretchx {\widetilde{X}},
\end{equation}
then there exists at least one point $z\in {\mathcal K}$ with $\psi(z) = z.$
\end{theorem}
\begin{proof}
Let $h: I^N=[0,1]^N\to h\left(I^N\right)=X\subseteq Z$ be a homeomorphism such that $X_{\ell}= h([x_N=0]),\, X_{r}= h([x_N=1])$ and consider the compact subset of $I^N$
$${\mathcal W}:= h^{-1}({\mathcal K}\cap \psi^{-1}(X))$$
as well as the continuous mapping $\phi = (\phi_1,\dots,\phi_N): {\mathcal W}\to I^N$ defined by
$$\phi(x):= h^{-1}\bigl(\psi( h(x))\bigr ),\quad\forall\, x\in {\mathcal W}.$$
By Tietze-Urysohn Theorem \cite[p.87]{En-78} there exists a continuous map
$$\varphi = (\varphi_1,\dots,\varphi_N): I^N\to I^N\,,\quad \varphi\restriction_{\mathcal W}= \phi.$$
Let us introduce, for every $i = 1,\dots, N-1,$ the closed sets
$${\mathcal S}_i:= \left\{x=(x_1,\dots,x_{N-1},x_N)\in I^N\,: \, x_i = \varphi_i(x)\right\}\subseteq I^N.$$
Since $\varphi(I^N)\subseteq I^N,$ by the continuity of the $\varphi_i$'s, it is straightforward
to check that ${\mathcal S}_i$ cuts the arcs between $[x_i = 0]$ and $[x_i=1]$ in $I^N$
for $i=1,\dots,N-1.$ Indeed, if $\gamma: [0,1]\to I^N$ is a path with $\gamma_i(0) = 0$ and
$\gamma_i(1)=1,$ then, for the auxiliary function
$g: [0,1]\ni t\mapsto \gamma_i(t) - \varphi_i(\gamma(t)),$ we have $g(0) \leq 0 \leq g(1)$
and therefore Bolzano Theorem implies the existence of $s\in [0,1]$ such that
$\gamma_i(s) = \varphi_i(\gamma(s)),$ that is, ${\overline{\gamma}}\cap {\mathcal S}_i\ne\emptyset.$
The cutting property is thus proved.\\
Now Corollary \ref{cor-jo} guarantees that there is a continuum
\begin{equation}\label{eq-int}
{\mathcal C}\subseteq \bigcap_{i=1}^{N-1}{\mathcal S}_i
\end{equation}
such that
$${\mathcal C}\cap [x_N=0]\ne\emptyset,\quad {\mathcal C}\cap [x_N=1]\ne\emptyset.$$
Lemma \ref{lem-appr} implies that for every $\varepsilon > 0$ there exists a path
$\gamma_{\varepsilon}: [0,1]\to I^N$ such that
$$\gamma_{\varepsilon}(0) \in [x_N=0],\quad \gamma_{\varepsilon}(1) \in [x_N=1]\quad
\mbox{and }\; \gamma_{\varepsilon}(t)\in B({\mathcal C},\varepsilon)\cap I^N,\;\forall\, t\in [0,1].$$
By the stretching assumption \eqref{eq-fpn} and the definition of ${\mathcal W}$ and $\phi,$ there is a subpath
$\omega_{\varepsilon}$ of $\gamma_{\varepsilon}$ such that
$${\overline{\omega_{\varepsilon}}}\subseteq {\mathcal W}\;\mbox{ and }\;
\phi({\overline{\omega_{\varepsilon}}})\subseteq I^N,\; \mbox{ with }\;
\phi({\overline{\omega_{\varepsilon}}})\cap [x_N=0]\ne\emptyset,\;\;
\phi({\overline{\omega_{\varepsilon}}})\cap [x_N=1]\ne\emptyset.$$

\noindent
Bolzano Theorem applied to the continuous mapping $x\mapsto x_N - \varphi_N(x)$
on ${\overline{\omega_{\varepsilon}}}$ ensures the existence of a point
$${\tilde{x}\,}^{\varepsilon}= ({\tilde{x}\,}^{\varepsilon}_1,\dots,
{\tilde{x}\,}^{\varepsilon}_N)\in {\overline{\omega_{\varepsilon}}} \subseteq {\mathcal W}$$
such that
$${\tilde{x}\,}^{\varepsilon}_N = \varphi_N({\tilde{x}\,}^{\varepsilon}).$$
Taking $\varepsilon = \tfrac 1 n$ (for $n\in \mathbb N_0$) and letting $n\to \infty,$ by a standard compactness argument
we find a point
$${\tilde{x}} = ({\tilde{x}}_1,\dots,{\tilde{x}}_N)\in {\mathcal C} \cap {\mathcal W}$$
such that
$${\tilde{x}}_N = \varphi_N({\tilde{x}}).$$
By \eqref{eq-int}, recalling also the definition of the ${\mathcal S}_i$'s, we get
$${\tilde{x}} = \varphi({\tilde{x}})\in {\mathcal W}.$$
Then, since $\varphi\restriction_{\mathcal W}= \phi,$ by the relation
$$h( \phi(x) ) = \psi( h(x) ),\quad\forall\, x\in {\mathcal W},$$
we have that $h({\tilde{x}}) = \psi( h({\tilde{x}}) )\in h({\mathcal W})$ and therefore
$$z:= h({\tilde{x}})\in {\mathcal K}\cap \psi^{-1}(X)$$
is the desired fixed point for $\psi.$
\end{proof}

\noindent
Having proved Theorems \ref{th-invh} and \ref{th-fpn}, we have available the tools for
extending to any finite dimension the results obtained in Section \ref{sec-sap}
for the planar case \cite{PaZa-04a, PaZa-04b}. For brevity's sake we focus our attention only on
a few of them, that we present below in the more general setting. Notice that now we can complement the previous results
adding some information on the dimension of the cutting surfaces.

\clearpage

\begin{figure}[htbp]
\centering
\includegraphics[scale=0.37]{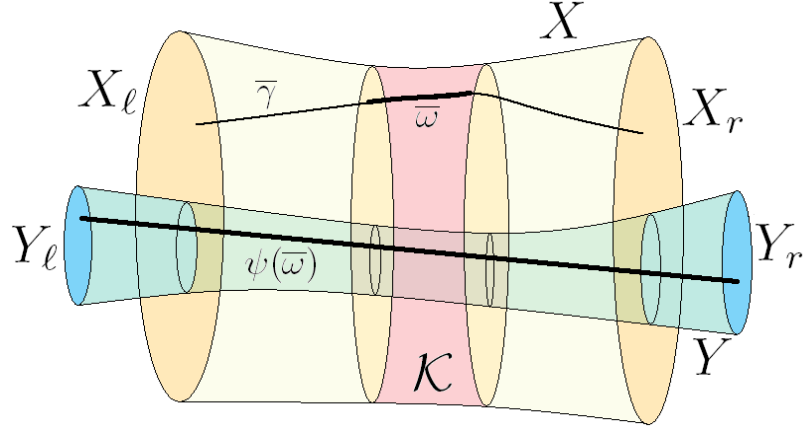}
\caption{\footnotesize {The tubular sets $X$ and $Y$
represent two generalized 3-dimensional rectangles, in which we
have indicated the compact set $\mathcal K$ and the boundary
sets $X_{\ell}$ and $X_{r},$ as well as $Y_{\ell}$ and $Y_{r}.$ The function $\psi$ is continuous on $\mathcal K$ and maps $X$ onto $Y=\psi(X),$ so that $Y_{\ell}=\psi(X_{\ell})$ and $Y_{r}=\psi(X_{r}).$ However, in this particular
case the map $\psi$ stretches the paths of $X$ not only across $Y,$
but also across $X$ itself and therefore the existence of a fixed
point for $\psi$ inside $\mathcal K$ is ensured by Theorem
\ref{th-fpn}.
}}
\end{figure}

\begin{theorem}\label{th-pern}
Let $Z$ be a metric space and $\psi: Z \supseteq D_{\psi}\to Z$ be a map defined on a set $D_{\psi}.$ Assume that
${\widetilde{X}}:= ({X},{X}^-)$
is an oriented $N$-dimensional rectangle of $Z.$ If ${\mathcal K_0},\dots,{\mathcal K_{m-1}}$
are $m\ge 2$ pairwise disjoint compact subsets of $X\cap D_{\psi}$ and
$$({\mathcal K}_i,\psi): {\widetilde{X}} \stretchx {\widetilde{X}}, \mbox{ for } i=0,\dots,m-1,$$
then the following conclusions hold:
\begin{itemize}
\item The map $\psi$ has at least a fixed point in ${\mathcal K}_i\,,\,i=0,\dots,m-1;$
\item For each two-sided sequence $(s_{h})_{h\in{\mathbb Z}}\in \{0,\dots,m-1\}^{\mathbb Z},$
there exists a sequence of points $(x_{h})_{h\in {\mathbb Z}}$ such that
$\psi(x_{h-1}) = x_{h}\in {\mathcal K}_{s_{h}}\,,\,\forall\, h\in {\mathbb Z}\,;$
\item For each sequence of $m$ symbols $\textbf{s}=(s_n)_n\in \{0,1,\dots,m-1\}^{\mathbb N},$
there exists a compact connected set ${\mathcal C}_{\textbf{s}}\subseteq {\mathcal K}_{s_0}$
which cuts the arcs between $X_{\ell}$ and $X_{r}$ in $X$ and such that,
for every $w\in {\mathcal C}_{\textbf{s}}\,,$
there is a sequence $(y_n)_{n}$ with $y_0 = w$ and
$$y_n \in {\mathcal K}_{s_n}\,,\;\; \psi(y_n) = y_{n+1}\,,\;\forall\, n\geq 0.$$
The dimension of ${\mathcal C}_{\textbf{s}}$ at any point is at least $N-1.$
Moreover, $\mbox{\rm dim}({\mathcal C}_{\textbf{s}}\cap \vartheta X)$ $\geq N - 2$ and
$\pi: h^{-1}({\mathcal C}_{\textbf{s}})\cap \partial I^N \to {\mathbb R}^{N-1}\setminus\{0\}$
is essential (where $\pi$ is defined as in \eqref{eq-prt} for $p_i=\tfrac 1 2,\,\forall\, i$ and $h$ is the homeomorphism defining $X$)\,;
\item Given an integer $j\ge 2$ and a $j+1$-uple $(s_0,\dots,s_j),\, s_i\in \{0,\dots,m-1\},$ for $i=0,\dots,j,$ and $s_0=s_j,$
then there exists a point $w\in{\mathcal K_{s_0}}$ such that
$${\psi}^{i}(w)\in {\mathcal K}_{s_i}\,,\,\forall i=1,\dots,j, \,\mbox{ and } \,\, {\psi}^{j}(w)=w.$$
\end{itemize}
\end{theorem}

\noindent
Like in the two-dimensional framework, such result can be easily obtained as a corollary of the more general:
\begin{theorem}\label{th-compn}
Let $({\widetilde{X}_i})_{i\in {\mathbb Z}}\,\,($with ${\widetilde{X}_i} = (X_i,X^-_i)\,)$ be a double sequence of oriented $N$-dimensional rectangles of a metric space $Z$ and let $({\mathcal K}_i,\psi_i)_{i\in {\mathbb Z}}\,,$
with ${\mathcal K}_i \subseteq X_i\,,$ be a sequence such that
$$({\mathcal K}_i,\psi_i): {\widetilde{X}_i}  \stretchx {\widetilde{X}_{i+1}}\,,\quad\forall\, i\in {\mathbb Z}.$$
Let us denote by
$X^{i}_{\ell}$ and $X^{i}_{r}$ the two components of $X_i^-\,.$
Then the following conclusions hold:
\begin{itemize}
\item There is a sequence $(w_k)_{k\in {\mathbb Z}}$ such that $w_k\in {\mathcal K}_k$
and $\psi_k(w_k) = w_{k+1}\,,$ for all ${k\in {\mathbb Z}}\,;$
\item For each $j\in {\mathbb Z}$ there exists a compact connected set ${\mathcal C}_j\subseteq {\mathcal K}_j$
which cuts the arcs between $X^{j}_{\ell}$ and $X^{j}_{r}$ in $X_j$ and such that, for every $w\in {\mathcal C}_j\,,$
there is a sequence $(y_i)_{i\geq j}$ with $y_j = w$ and
$$y_i \in {\mathcal K}_i\,,\;\; \psi_{i}(y_i) = y_{i+1}\,,\;\forall\, i\geq j.$$
The dimension of ${\mathcal C}_j$ at any point is at least $N-1.$
Moreover, $\mbox{\rm dim}({\mathcal C}_j\cap \vartheta X_j)$ $\geq N - 2$ and
$\pi: h_j^{-1}({\mathcal C}_j)\cap \partial I^N \to {\mathbb R}^{N-1}\setminus\{0\}$
is essential (where $\pi$ is defined as in \eqref{eq-prt} for $p_i=\tfrac 1 2,\,\forall\, i$ and $h_j$ is the homeomorphism defining $X_j$);
\item If there are integers $h$ and $l,$ with $h < l,$ such that ${\widetilde{X}_h} = {\widetilde{X}_{l}}\,,$
then there exists a finite sequence $(z_i)_{h\leq i\leq l-1}\,,$ with $z_i\in {\mathcal K}_i$ and
$\psi_i(z_i) = z_{i+1}$ for each $i=h,\dots,l-1,$ such that $z_{l} = z_h\,,$ that is,
$z_h$ is a fixed point of $\psi_{l-1}\circ\dots\circ\psi_{h}\,.$
\end{itemize}
\end{theorem}
\begin{proof}
The thesis follows by steps analogous to the ones in the proof of Theorem \ref{th-comp}, using Theorem
\ref{th-fpn} in place of Theorem \ref{th-fp} and recalling Definition \ref{def-cut}. The estimates on the
dimension of the continuum ${\mathcal C}_j$ come from Theorem \ref{th-invh}.
\end{proof}

\noindent
As a final remark, we observe that, instead of maps expansive just along one direction, it is possible to deal with the more general case of functions defined on $N$-dimensional rectangles (or homeomorphic images of them) with $u$ unstable directions (along which an expansion occurs) and $s=N-u$ stable directions (along which the system is compressive). With this respect, many works are available in the literature, like, for instance, \cite{BaCs-07,BaCs-08,GiRo-03,Ka-00,PoSzMI-01,ZgGi-04}.
In particular, we recall that a similar framework was analyzed in \cite{PiZa-05}, where the author and Zanolin introduced a ``deformation relation'' among generalized $N$-dimensional rectangles with some similarities to the stretching property in Definition \ref{def-sapn}. Indeed also that relation could be viewed (cf. \cite[Remark 3.4]{PiZa-05}) as an higher-dimensional counterpart of the planar stretching property in Definition \ref{def-sap} from \cite{PaZa-04a, PaZa-04b} and moreover it was suitable for the detection of fixed points, periodic points and chaotic dynamics, too. However, its definition required some assumptions on the topological degree of the map involved, while the conditions in Definition \ref{def-sapn} do not need any sophisticated topological tool.
Therefore, it is in this perspective of
simplicity that the above results from \cite{PiZa-07} have to be considered.

\chapter[Chaotic dynamics]{Chaotic dynamics for continuous maps}\label{ch-ch}

In the present chapter we show how to apply the method of
``stretching along the paths'' from Chapter \ref{ch-me} to prove the presence of chaotic dynamics in discrete-time dynamical systems. In order to do that, we first set our framework among more classical ones available in the literature, concerning the concepts of ``covering'' and ``crossing''. Indeed, this preliminary introductory discussion is meant to give an outline of the reasons that led various authors to deal with similar settings, as well as to quickly present the topics developed along Sections \ref{sec-de}--\ref{sec-ltm}, where the undefined terms will be rigorously explained.\\
Generally speaking, in the investigation of a dynamical system defined by a map $f$
on a metric space, it can be difficult
or impossible to find a direct proof of the presence of chaos, according to one or the other of the several, more or less equivalent, existing definitions. Hence, a canonical strategy is
to establish a conjugacy or semi-conjugacy
between $f,$ perhaps restricted to a suitable (positively) invariant subset of its domain,
and the shift map $\sigma:(s_i)_{i\in{\mathbb I}}\mapsto (s_{i+1})_{i\in{\mathbb I}}$ on
$\{0,\dots,m-1\}^{\mathbb I},$ for ${\mathbb I}={\mathbb N}$ or ${\mathbb I}={\mathbb Z},$ i.e. on the space of unilateral or bilateral sequences of $m\geq 2$ symbols,
which displays many chaotic features, like transitivity, sensitivity, positive entropy, etc. In such indirect manner, one can conclude that
$f$ possesses all the properties of $\sigma$ that are preserved by the conjugacy/semi-conjugacy
relation. In particular, the known fact (cf. Section \ref{sec-sd})
that $h_{\rm top}(f^n)=nh_{\rm top}(f)$ for $n\geq 1,$
where $h_{\rm top}(f)$ denotes the topological entropy of $f,$ guarantees that if a power
of a map is conjugate or semi-conjugate to $\sigma,$ then such map has positive entropy.
Accordingly, to show the existence of chaos (positive entropy) of a map $f$ it is
sufficient to look at any of its iterates.\\
A crucial role in this strategy is played by the study of a
topological property for maps which is called  ``crossing'' and,
broadly speaking, refers to the way in which the image of suitable sets under the
iterates of the maps
intersect the  original sets. In particular, the kind of
crossing known as  ``horseshoe'' has turned out to be a fundamental tool
of analysis to prove the existence of chaotic dynamics in a
well-defined manner. The name and the idea of horseshoe derive
from the celebrated article by Smale \cite{Sm-65} that provided a
mathematically rigorous and geometrically suggestive proof of the
existence of chaos for a special planar map. \\
More precisely (see \cite{De-89, Mo-73, Wi-88}
for the technical details),
in the horseshoe model \index{horseshoe, Smale} a square $S$ is first shrunk
uniformly in one direction and expanded in the other one.
Subsequently,  the elongated rectangle obtained in the previous step
is bent along the original square in order to cross it twice.
The resulting map $F$ is a diffeomorphism with $F$ and
$F^{-1}$ transforming the vertical and horizontal lines of the
square into similar lines crossing the domain.
In Smale's construction, the set
${\mathcal I}_S:=\{q\in S: F^k(q)\in S, \;\forall\, k\in {\mathbb Z}\},$
consisting of the points which remain in the square under all the
iterates of $F$ in both forward and backward time, is a compact
invariant set for $F$ which contains as a dense subset the
periodic points of $F$ and such that $F$ is sensitive on initial conditions and topologically
transitive on ${\mathcal I}_S.$ Actually, $F$ acts on ${\mathcal I}_S$ like the
shift map $\sigma$ on two symbols since $F\restriction_{{\mathcal I}_S}$ and $\sigma$ are conjugate.  As mentioned above, the conjugacy, which is given by a
homeomorphism $\pi:{\mathcal I}_S\to \{0,1\}^{\mathbb Z},$ with $\pi\circ
F=\sigma\circ \pi,$ allows to transfer to $F\restriction_{{\mathcal I}_S}$ the
well-known dynamical properties of the Bernoulli shift.\\
In the applications to concrete dynamical systems,
the presence of a complex behaviour for a given map $F$ can be verified by proving
the existence of a horseshoe structure either for the map itself
or for one of its iterates. This led some authors (see
\cite{BuWe-95}) to define a \textit{horseshoe} like a set
$\Lambda$ that is invariant under $F^{\,n_0},$ for some $n_0\geq
1,$ with the property that $F^{\,n_0}\restriction_{\Lambda}$ is topologically
conjugate to $\sigma$ (or, more generally, to a nontrivial
subshift of finite type). As pointed out by Burns and Weiss in
\cite{BuWe-95}, the difficulty in finding a horseshoe lies in
showing that the map $\pi: \Lambda\to \{0,1\}^{\mathbb Z}$ is injective. Such step,
in turns, often requires the verification of some assumptions on
$F,$ like being a diffeomorphism satisfying suitable hyperbolicity
conditions, which either are not fulfilled or are hard to check
for a given map.
Hence, more general and less stringent definitions of horseshoe have been
suggested to reproduce some geometrical features typical of
the Smale horseshoe while discarding the hyperbolicity conditions
associated with it, which are difficult or impossible to prove in practical cases.
This led to the study of the so-called ``topological horseshoes'' \index{horseshoe, topological}
\cite{BuWe-95,KeYo-01,ZgGi-04}.\\
The core of such field of research consists in providing an adequate
notion of crossing property for higher dimensional
dynamical  systems, so that a  map $F$ (or one of its iterates) satisfying certain geometrical
conditions is proved to be semi-conjugate to a full shift on $m\ge 2$ symbols.
This is enough to conclude that $F$ displays chaotic dynamics in the
sense that $F$ has, for instance, positive topological entropy.
More generally, one could show that $F$ is semi-conjugate
to a nontrivial subshift of finite type, by employing some tools from symbolic dynamics
\cite{Ki-98,LiMa-95}, as explained in Section \ref{sec-sd}.\\
In regard to the one-dimensional case, a classical example of horseshoe-type crossing property relies on the definition of ``covering''\index{covering relation}. We recall that,
according to \cite{BlGu-80}, given a continuous map $f: {\mathbb R}\to {\mathbb R}$
and two intervals $I,J\subseteq
{\mathbb R},$ we say that
\textit{$I$ $f$-covers $J$} if there exists a subinterval
$I_0\subseteq I$ such that $f(I_0) = J.$
We also say that
\textit{$I$ $f$-covers $J$ $m$ times} if there exist $m\ge 2$
subintervals $I_0\,,\dots, I_{m-1}\subseteq I,$ with pairwise
disjoint interiors, such that $f(I_{k}) = J,$ for $k=0,\dots,m-1.$
Special significance from the point of view of complex dynamics is assigned to the case in which $I$ $f$-covers $I$ $m$ times.
Different authors have investigated possible variants of the above definitions,
dealing with the framework in which there are $m\geq 2$
intervals $I_{0}\,,\dots,I_{m-1}$ such that
$I_i$ $f$-covers $I_j,$ for each $i,j\in\{0,\dots,m-1\},$ and some disjointness condition is imposed on them. For instance, the case $m=2$ with $I_0\,,$ $I_1$ open disjoint intervals
was taken by Glendinning in \cite{Gl-94} as a definition of
``horseshoe'' for $f,$ while Block and Coppel in \cite{BlCo-86,BlCo-92} define ``turbulence'' \index{turbulence} for $f$ the similar framework when $I_0$ and $I_1$ are compact intervals with disjoint interiors and ``strict turbulence''\index{turbulence, strict} when $I_0,\,I_1$ are compact and disjoint.
In such settings one obtains the typical features associated to the
concept of chaos (like, e.g., existence of periodic points of each period,
semi-conjugacy to the Bernoulli shift for the map $f$ or for some of its iterates and thus positive topological entropy,
ergodicity with respect to some invariant measure, sensitivity to initial conditions, transitivity). Besides the works already quoted,
see \cite{AuKi-01,BlTe-03,LaYo-77,Ru-87}.\\
An elementary introduction to these ideas is provided  by the
analysis of the logistic map $F:[0,1]\to\mathbb R,$ $F(x)=\mu x(1-x),$
with $\mu$ a positive real parameter.
It is easy to see that for
$\mu\ge 4$ the interval  $[0,1]$  $F$-covers $[0,1]$ twice,
whereas for
$\mu < 4$ the existence of multiple coverings can be established for some iterate of $F,$
as shown in Section \ref{sec-sd}.
The concepts  of covering and horseshoe can also be used to revisit some classical results in the theory of
one-dimensional unimodal maps, such as Li-Yorke celebrated Theorem \cite{LiYo-75}. For instance, \cite[Theorem 11.13]{Gl-94} shows that,
whenever a unimodal map $f$ has a point of period three, then its second iterate $f^2$ has a horseshoe.
However, it is in the higher-dimensional frameworks that the approach described above shows its full power.\\
The most natural way to extend the notion of multiple $f$-covering to the abstract setting
of a continuous self-map $f$ of a metric space $X$
may be that of assuming the existence of $m\ge 2$ compact sets $K_0\,,\dots,K_{m-1}\subseteq X$
such that $f(K_i)\supseteq \bigcup_{j=0}^{m-1} K_{j},\,\forall\, i=0,\dots,m-1.$
Moreover, in order to avoid trivial conclusions, one usually requires some disjointness conditions on the sets
$K_i$'s, such as
$K_i\cap K_j=\emptyset,
\;\forall\, i\ne j\, \left(\mbox{or, more generally, }\; \bigcap_{i=0}^{m-1} K_i=\emptyset\,\right).$
Blokh and Teoh \cite{BlTe-03} describe such framework  by saying that
$(f, K_0\,,\dots,K_{m-1})$ form a $m$-horseshoe (or a weak $m$-horseshoe, respectively).
Under these conditions it is possible to show the existence of some chaotic behaviour for the discrete
dynamical system generated by $f$ (see \cite{LaYo-77,Ru-87}). In particular, in \cite{BlTe-03}
it is proved that if a power of $f$ admits a weak $m$-horseshoe, then the topological entropy of $f$
is positive and there exists a set $B$ and a power $g$ of $f$ such that $B$ is $g$-invariant
and $g\restriction_{B}$ is semi-conjugate to the Bernoulli shift on $\{0,\dots,m-1\}^{\mathbb N}.$\\
In spite of the generality of the setting in which such results can be obtained
and their effectiveness for the one-dimensional framework, when
one tries to apply the theory to specific higher dimensional mathematical models
arising in applications, it may be expedient to follow
Burns and Weiss' suggestion \cite{BuWe-95} and replace
the previous covering relation with a weaker condition
of the form
$f(K_i) \;\, \mbox{``\,goes across\,''}\, \bigcup_{j=0}^{m-1} K_{j},\,\forall\, i=0,\dots,m-1,$
which does not require the map $f$ to be surjective on the $K_i$'s.\\
In our approach the expression ``\,goes across\,'' (called also ``Markov property'' \cite[p.291]{HaKa-03}) has to be understood
in relation to the stretching along the paths effect of a map $f$ from Definition \ref{def-sap} or Definition \ref{def-sapn} \footnote{By the similarity between the results obtained in Chapter \ref{ch-me} for
the planar and the $N$-dimensional settings, for the sake of generality, we will generally work in the higher
dimensional framework, maybe not indicating the dimension, when no confusion may occur.}. More precisely,
we assume that each path $\gamma$ in our generalized ($N$-dimensional) rectangle ${\mathcal R},$ joining the two sides
of ${\mathcal R}^-,$ intercepts every $K_i$ and is then expanded to a path $f\circ\gamma$
which crosses all the $K_j$'s. Further details can be found in Section \ref{sec-sd}.\\
Different characterizations of the concept of
``crossing'' have been suggested by various authors in order to
establish the presence of complex dynamics for continuous maps
in higher dimensional spaces (see, for instance, \cite{KeKoYo-01,KeYo-01,MiMr-95a,Sr-00,Zg-96,Zg-01,ZgGi-04} and the references therein). In this respect, the already discussed approach by Kennedy, Ko\c cak and Yorke in \cite{KeKoYo-01} and
Kennedy and Yorke in \cite{KeYo-01} is perhaps the most general, both with regard to the spaces considered and the restrictions on the maps involved.
However, as we have seen, the great generality of such setting comes with a price, in the sense that the existence of periodic points is not ensured. Other authors (for example in \cite{MiMr-95a,Sr-00,SrWo-97,WoZg-00,Zg-96,Zg-01,ZgGi-04}) have developed theories of topological horseshoes more focused on the search of
fixed and periodic points for maps defined on subsets of the $N$-dimensional
Euclidean space. The tools employed
in these and related works range from the Conley index \cite{MiMr-95a}
to the Lefschetz fixed point theory \cite{Sr-00}
and the topological degree \cite{ZgGi-04}.\\
Our framework may be looked at as an
intermediate point of view between the theory of topological horseshoes
developed by Kennedy and Yorke in \cite{KeYo-01} and the above-mentioned
works based on some more or less sophisticated fixed point index theories.
Indeed, compared to \cite{KeYo-01}, the method of stretching along the paths,
thanks to its specialized framework, allows one to obtain sharper results concerning the
existence of periodic points. On the other hand, our approach,
although mathematically rigorous, avoids the use of more advanced topological theories
and it is relatively easy to apply to specific models arising
in applications.\\
In more details, in Section \ref{sec-de} we describe the chaotic features that we are able to obtain with our method. Some tools from symbolic dynamics are introduced in Section \ref{sec-sd}, in order to collate our achievements on chaotic dynamics with other related results descending from classical approaches. Such discussion on the various notions of chaos is pursued further in Section \ref{sec-var}.
An alternative geometrical context for the applicability of the stretching along the paths method is presented in Section \ref{sec-ltm}, where we deal with the so-called ``Linked Twist Maps''.

\section{Definitions and main results}\label{sec-de}
In the literature various different notions of complex dynamics can be found, so that one could say ``as many authors, as many definitions of chaos'' \cite{BlGl-02}. However, some common features are shared by several of these definitions, such as the unpredictability of the future behaviour of the system under consideration. In particular, one of the most natural notions of chaos is related to the possibility of realizing a generic coin-flipping experiment \cite{Sm-98}. The definition of
chaos that we choose is adapted from that considered by
Kirchgraber and Stoffer in \cite{KiSt-89} under the name of
\textit{chaos in the coin-tossing sense}\index{chaos! coin-tossing}. In fact, exactly like in
\cite{KiSt-89}, our definition concerns the
possibility of reproducing, via the iterates of a given map
$\psi,$ any coin-flipping sequence
of two symbols. On the other hand, our definition, when compared to \cite{KiSt-89},
is enhanced by the possibility of realizing periodic sequences of
two symbols by means of periodic points of $\psi.$
More formally, we have:

\begin{definition}\label{def-ch}
\rm{ Let $X$ be a metric space,  $\psi: X\supseteq D_{\psi} \to X$
be a map and let ${\mathcal D}\subseteq D_{\psi}\,.$ We say that
\textit{$\psi$ induces chaotic dynamics on two symbols \index{chaotic dynamics on two symbols} on the set
${\mathcal D}$} if there exist two nonempty disjoint
compact sets
$${\mathcal K}_0,\, {\mathcal K}_1\subseteq {\mathcal D},$$
such that, for each two-sided sequence $(s_i)_{i\in {\mathbb Z}}
\in  \{0,1\}^{\mathbb Z},$ there exists a corresponding sequence
$(w_i)_{i\in {\mathbb Z}}\in {\mathcal D}^{\mathbb Z}$ such that
\begin{equation}\label{eq-ch}
w_i \,\in\, {\mathcal K}_{s_i}\;\;\mbox{ and }\;\, w_{i+1} =
\psi(w_i),\;\; \forall\, i\in {\mathbb Z},
\end{equation}
and, whenever $(s_i)_{i\in {\mathbb Z}}$ is a $k$-periodic
sequence (that is, $s_{i+k} = s_i\,,\forall i\in {\mathbb Z}$) for
some $k\geq 1,$ there exists a corresponding $k$-periodic sequence
$(w_i)_{i\in {\mathbb Z}}\in {\mathcal D}^{\mathbb Z}$ satisfying
\eqref{eq-ch}. When we want to emphasize the role of the sets ${\mathcal
K}_j$'s, we also say that \textit{$\psi$ induces chaotic dynamics
on two symbols on the set ${\mathcal D}$ relatively to ${\mathcal
K}_0$ and ${\mathcal K}_1$}.}
\end{definition}

\noindent
To get a feel of such definition,
we can imagine to associate the name $\mbox{`` head ''}= H$
to the set ${\mathcal K}_0$ and the name $\mbox{`` tail ''}= T$ to ${\mathcal K}_1\,.$
If we consider any sequence of symbols
$$(s_i)_{i\in{\mathbb Z}}\in \{0,1\}^{\mathbb Z} \cong \{H,T\}^{\mathbb Z}$$
so that, for each $i,$
$s_i$ is either `` head '' or `` tail '', then we have the same itinerary of heads and tails
realized through the map $\psi.$ Namely, there exists a sequence $(w_i)_{i\in{\mathbb Z}}$
of points of the metric space $X$ which is a full orbit for $\psi,$ i.e.
$$w_{i+1} = \psi(w_i),\forall\, i\in{\mathbb Z},$$
and such that $w_i\in {\mathcal K}_0$ or $w_i\in {\mathcal K}_1,$ according to the fact that
the $i$-th term of the sequence $(s_i)_i$ is `` head '' or `` tail ''.
Moreover, as already remarked, our definition extends that in \cite{KiSt-89} in the sense that any periodic
sequence of heads and tails can be realized by suitable points which are periodic points for $\psi.$
For instance, there exists a fixed point of $\psi$ in the set ${\mathcal K}_1$
corresponding to the constant sequence of symbols $s_i= \mbox{`` tail ''},\,\forall\,i\in{\mathbb Z}.$
There is also a point $w\in {\mathcal K}_0$ of period three with $\psi(w) \in {\mathcal K}_0$
and $\psi^2(w)\in {\mathcal K}_1\,,$ corresponding to the periodic sequence
$\dots HHT\,\,HHT\,\,HHT \dots,$ and so on.

\smallskip

\noindent
We stress that we have presented our definition of chaos with reference to just two symbols in order
to make the connection with the coin-flipping process clearer. On the other hand, it is possible
to define \textit{chaotic dynamics on $m$ symbols}\index{chaotic dynamics on $m$ symbols} for an arbitrary integer $m\ge 2$ in a completely
analogous manner, when there are $m$ pairwise disjoint compact sets
$\mathcal K_0,\dots,\mathcal K_{m-1}$ acting as $\mathcal K_0$ and $\mathcal K_1$ in Definition \ref{def-ch}. More precisely we have:
\begin{definition}\label{def-chm}
\rm{ Let $X$ be a metric space,  $\psi: X\supseteq D_{\psi} \to X$
be a map and let ${\mathcal D}\subseteq D_{\psi}\,.$ Let also $m\ge 2$ be an integer.
We say that \textit{$\psi$ induces chaotic dynamics on $m$ symbols on the set
${\mathcal D}$} if there exist $m$ nonempty pairwise disjoint
compact sets
$${\mathcal K}_0,\dots, {\mathcal K}_{m-1}\subseteq {\mathcal D},$$
such that, for each two-sided sequence $(s_i)_{i\in {\mathbb Z}}
\in  \{0,\dots,m-1\}^{\mathbb Z},$ there exists a corresponding sequence
$(w_i)_{i\in {\mathbb Z}}\in {\mathcal D}^{\mathbb Z}$ such that
\begin{equation}\label{eq-chm}
w_i \,\in\, {\mathcal K}_{s_i}\;\;\mbox{ and }\;\, w_{i+1} =
\psi(w_i),\;\; \forall\, i\in {\mathbb Z}
\end{equation}
and, whenever $(s_i)_{i\in {\mathbb Z}}$ is a $k$-periodic
sequence (that is, $s_{i+k} = s_i\,,\forall i\in {\mathbb Z}$) for
some $k\geq 1,$ there exists a corresponding $k$-periodic sequence
$(w_i)_{i\in {\mathbb Z}}\in {\mathcal D}^{\mathbb Z}$ satisfying
\eqref{eq-chm}. When we want to emphasize the role of the sets ${\mathcal
K}_j$'s, we also say that \textit{$\psi$ induces chaotic dynamics
on $m$ symbols on the set ${\mathcal D}$ relatively to ${\mathcal
K}_0,\dots,{\mathcal K}_{m-1}$}.}
\end{definition}

\noindent
Such variant of Definition \ref{def-ch} will be of particular importance when dealing with Linked Twist Maps (see Section \ref{sec-ltm}). Moreover,
the latter generalization agrees with the kind of chaotic behaviour detected
in various papers, where the dynamical systems are investigated
using some topological tools related to fixed point theory
\cite{SrWo-97,WoZg-00}. Hence, in order to collate our notion of chaos to other ones available in the literature, from now on we will consider Definition \ref{def-chm} instead of Definition \ref{def-ch}.\\
In regard to the results in \cite{KeKoYo-01,KeYo-01}, we recall that the authors therein also
deal with chaotic dynamics meant as the possibility of realizing any sequence of symbols. The difference, as
it should be clear from the discussion in Section \ref{sec-sap}, resides in the possibility of realizing periodic coin-flipping
sequences through periodic itineraries. Nonetheless, the kind of chaos detected in \cite{KeKoYo-01,KeYo-01} for a map $f$ satisfying the ``\,horseshoe hypotheses $\Omega$\,'' in \cite{KeYo-01} with crossing number $m\ge 2$ allows to
prove the existence of a compact $f$-invariant set $Q_I$ (i.e. $f(Q_I)=Q_I$),
such that $f\restriction_{Q_I}$ is semi-conjugate to the one-sided shift on $m$ symbols.
Furthermore, in \cite[Lemma 4]{KeKoYo-01} (Chaos Lemma), the existence of a smaller compact invariant set $Q_{*}\subseteq Q_I$ is obtained,
on which the map $f$ is sensitive and such that each forward itinerary on $m$ symbols
is realized by the $f$-itinerary generated by some point of $Q_{*}.$\\
Before showing in Theorem \ref{th-cons} what we are able to prove in this direction, let us see which is the link between the
method of stretching along the paths from Chapter \ref{ch-me} and the notion of chaotic dynamics in the sense
of Definition \ref{def-chm}:

\begin{theorem}\label{th-ch}
Let ${\widetilde{\mathcal R}}:= ({\mathcal
R},{\mathcal R}^-)$ be an oriented ($N$-dimensional) rectangle of a metric space
$X$ and let $\mathcal D\subseteq \mathcal R\cap D_{\psi}\,,$
with $D_{\psi}$ the domain of a map
$\psi: X \supseteq D_{\psi}\to X.$ If
${\mathcal K_0},\dots,{\mathcal K_{m-1}}$ are $m\ge 2$ pairwise disjoint compact
sets contained in ${\mathcal D}$ and
\begin{equation*}
({\mathcal K}_i,\psi): {\widetilde{\mathcal R}} \stretchx {\widetilde{\mathcal R}}, \mbox{ for } i=0,\dots,m-1,
\end{equation*}
then $\psi$ induces chaotic dynamics on $m$ symbols on the set
${\mathcal D}$ relatively to ${\mathcal K}_0,\dots,
{\mathcal K}_{m-1}.$
\end{theorem}
\begin{proof}
Recalling Definition \ref{def-chm}, the thesis is just a reformulation of the second and the forth conclusions in Theorem \ref{th-per} (or in Theorem \ref{th-pern}, in the case of an $N$-dimensional setting)\footnote{Recalling also the third conclusion in Theorem \ref{th-per} (or in Theorem \ref{th-pern}), we could complement Theorem \ref{th-ch} with the further information about the existence, for any forward sequence on $m$ symbols, of a corresponding continuum joining the lower and the upper sides of $\mathcal R$ and consisting of the points whose forward $\psi$-itinerary realizes the given sequence. However, for the sake of simplicity and since in the applications we don't use this fact, we have decided to omit it. The same observation also applies to the results obtained in Section \ref{sec-ltm} in relation to the framework of the Linked Twist Maps.}.
\end{proof}

\noindent
The same conclusions on chaotic dynamics could be obtained in the framework of Theorem \ref{th-fpt}, when assuming condition \eqref{eq-intt}. As discussed in \cite{PiZa-05}, such remark looks useful in view of possible applications of our method to the detection of chaos via computer-assisted proofs.

\smallskip

\noindent
When we enter the setting of Definition \ref{def-chm},
many interesting properties for the map $\psi$
can be proved. They are gathered in Theorem \ref{th-cons} below. The precise
explanation of some concepts (like topological entropy, sensitivity, transitivity, etc.) will be given in Sections \ref{sec-sd}--\ref{sec-var}, where the reader can find a more detailed discussion about various classical notions of chaos considered in the literature.

\begin{theorem}\label{th-cons}
Let $\psi$ be a  map inducing chaotic dynamics on $m\ge 2$ symbols on a
set ${\mathcal D}\subseteq X$ and which is continuous on
$$
\mathcal K:=\bigcup_{i=0}^{m-1}\mathcal K_i\subseteq {\mathcal D},
$$
where $\mathcal K_0,\dots,\mathcal K_{m-1},{\mathcal D}$ and $X$ are as
in Definition \ref{def-chm}. Setting
\begin{equation}\label{eq-lam}
{\mathcal I}_{\infty}:=\bigcap_{n=0}^{\infty}\psi^{-n}(\mathcal K),
\end{equation}
then there exists a nonempty compact set
$${\mathcal I}\subseteq {\mathcal I}_{\infty} \subseteq {\mathcal K},$$
on  which the following are fulfilled:

\begin{itemize}
\item[$(i)$] ${\mathcal I}$
is invariant for $\psi$ (that is, $\psi({\mathcal I}) = {\mathcal I}$);

\item[$(ii)$] $\psi\restriction_{\mathcal I}$ is semi-conjugate to the one-sided Bernoulli shift on $m$ symbols, i.e.
there exists a continuous map $\pi$ of ${\mathcal I}$ onto $\Sigma_m^+:=\{0,\dots,m-1\}^{\mathbb N},$ endowed with the distance
\begin{equation}\label{eq-dist}
\hat d(\textbf{s}', \textbf{s}'') := \sum_{i\in {\mathbb N}} \frac{|s'_i - s''_i|}{m^{i + 1}}\,,
\end{equation}
for $\textbf{s}'=(s'_i)_{i\in {\mathbb N}}$ and
$\textbf{s}''=(s''_i)_{i\in {\mathbb N}}\in \Sigma_m^+\,,$
such that the diagram
\begin{equation}\label{diag-1}
\begin{diagram}
\node{{\mathcal I}} \arrow{e,t}{\psi} \arrow{s,l}{\pi}
      \node{{\mathcal I}} \arrow{s,r}{\pi} \\
\node{\Sigma_m^+} \arrow{e,b}{\sigma}
   \node{\Sigma_m^+}
\end{diagram}
\end{equation}
commutes, i.e $\pi\circ\psi=\sigma\circ\pi,$ where $\sigma:\Sigma_m^+\to\Sigma_m^+$ is the Bernoulli
shift defined by $\sigma((s_i)_i):=(s_{i+1})_i,\,\forall
i\in\mathbb N\,;$

\item[$(iii)$] The set $\mathcal P$ of the periodic points of $\psi\restriction_{{\mathcal I}_{\infty}}$ is dense in ${\mathcal I}$
and the preimage $\pi^{-1}(\textbf{s})\subseteq {\mathcal I}$ of
every
$k$-periodic sequence $\textbf{s} = (s_i)_{i\in {\mathbb N}}\in \Sigma_m^+$
contains at least one $k$-periodic point.
\end{itemize}

Furthermore, from conclusion $(ii)$ it follows that:

\begin{itemize}
\item[$(iv)$] $$h_{\rm top}(\psi)\ge h_{\rm top}(\psi\restriction_{\mathcal I})\geq h_{\rm top}(\sigma) = \log(m),$$
where $h_{\rm top}$ is the topological entropy;

\item[$(v)$] There exists a compact invariant set $\Lambda\subseteq {\mathcal I}$ such that $\psi\vert_{\Lambda}$ is
semi-conjugate to the one-sided Bernoulli shift on $m$ symbols, topologically transitive and has sensitive dependence on initial conditions.
\end{itemize}
\end{theorem}
\begin{proof}
Let us begin by checking that the set ${\mathcal I}_{\infty}$ in \eqref{eq-lam} is compact and nonempty. By the continuity of the map $\psi$ on $\mathcal K,$ it follows that ${\mathcal I}_{\infty}$ is closed and, being contained in the compact set $\mathcal K,$ it is compact, too. The fact that ${\mathcal I}_{\infty}$ is nonempty follows from Definition \ref{def-chm} on chaotic dynamics, by observing that $z\in {\mathcal I}_{\infty} \Leftrightarrow \psi^n(z)\in\mathcal K\,,\forall n\ge 0.$ This remark also implies that $\psi({\mathcal I}_{\infty})\subseteq {\mathcal I}_{\infty}:$ indeed, it is straightforward to see that if $z\in {\mathcal I}_{\infty},$ then also $\psi(z)\in{\mathcal I}_{\infty}.$\\
Calling ${\mathcal P}$ the subset of ${\mathcal I}_{\infty}$ consisting of the periodic points of $\psi\restriction_{{\mathcal I}_{\infty}}\,,$
that is,
\begin{equation}\label{eq-per}
{\mathcal P}:= \left\{w\in {{\mathcal I}_{\infty}}: \exists \, k\in\mathbb N_0, \; \psi^{k}(w) = w\right\},
\end{equation}
we claim that $\psi(\mathcal P)=\mathcal P.$ Indeed, if $z\in \mathcal P,$ then there exists $l\in\mathbb N_0$ such that $\psi^l(z)=z.$ Hence, on the one hand, $\psi(z)=\psi(\psi^l(z))=\psi^{l+1}(z)=\psi^l(\psi(z))$ and thus $\psi(z)\in\mathcal P,$ too. This shows that $\psi(\mathcal P)\subseteq\mathcal P.$ Notice that, repeating the same argument, it is possible to prove that if $z\in \mathcal P,$ then $\psi^h(z)\in\mathcal P,$ for any $h\ge 1.$ On the other hand, if $\psi^l(z)=z,$ for some $l\in\mathbb N_0,$ then two possibilities can occur for $l,$ that is, $l=1$ or $l\ge 2.$ In the former case we get $\psi(z)=z$ and so $z\in \psi(\mathcal P),$ while in the latter we obtain $z=\psi^l(z)=\psi(\psi^{l-1}(z)).$ Hence, since $\psi^{l-1}(z)\in\mathcal P$ whenever $z\in\mathcal P,$ we find again $z\in \psi(\mathcal P).$ In any case we have proved that, if $z\in \mathcal P,$ then $z\in \psi(\mathcal P),$ i.e. $\mathcal P\subseteq \psi(\mathcal P).$ The claim is thus checked.\\
At this point we observe that, since $\mathcal P$ is contained in the compact set ${\mathcal I}_{\infty},$ also
\begin{equation}\label{eq-i}
\mathcal I:= \overline{\mathcal P}\subseteq {\mathcal I}_{\infty},
\end{equation}
and moreover $\mathcal I$ is compact, as it is closed in a compact set. From $\psi(\mathcal P)=\mathcal P,$ it follows that
$$\psi(\mathcal I)=\psi(\overline{\mathcal P})\supseteq\psi(\mathcal P)=\mathcal P.$$
But again, by the compactness of $\psi(\mathcal I),$ it holds that
$$\psi(\mathcal I)\supseteq\overline{\mathcal P}=\mathcal I.$$
Let us show that the reverse inclusion is also fulfilled for $\mathcal I,$ that is $\psi(\mathcal I)\subseteq\mathcal I.$ Indeed, since $\psi$ is continuous, we have
$$\psi(\mathcal I)=\psi(\overline{\mathcal P})\subseteq\overline{\psi(\mathcal P)}=\overline{\mathcal P}=\mathcal I.$$
Hence, the invariance of $\mathcal I$ is verified, in agreement with conclusion $(i).$\\
Let us consider now the diagram
\begin{equation*}
\begin{diagram}
\node{{{\mathcal I}_{\infty}}} \arrow{e,t}{\psi} \arrow{s,l}{\pi}
      \node{{{\mathcal I}_{\infty}}} \arrow{s,r}{\pi} \\
\node{\Sigma_m^+} \arrow{e,b}{\sigma}
   \node{\Sigma_m^+}
\end{diagram}
\end{equation*}
with $(\Sigma_m^+,\sigma)$ the Bernoulli system, and define the map $\pi:{\mathcal I}_{\infty}\to \Sigma_m^+$ by associating to any $w\in{\mathcal I}_{\infty}$ the sequence $(s_n)_{n\in\mathbb N}\in \Sigma_m^+$ such that $s_n=j$ if $\psi^n(w)\in\mathcal K_j,$ for $j=0,\dots,m-1.$ More formally, we notice that for any $w\in{\mathcal I}_{\infty},$ there exists a forward itinerary $(w_i)_{i\in{\mathbb N}}$ such that
$w_0=w$ and $\psi(w_{i}) = w_{i+1}\in {\mathcal K},$ for every $i\in {\mathbb N}.$
Hence the function $g_1:{{\mathcal I}_{\infty}}\to {\mathcal I}_{\infty}^{\mathbb N},$ which maps any $w\in {{\mathcal I}_{\infty}}$ into the one-sided sequence of points from the set ${{\mathcal I}_{\infty}}$
$$\textbf{s}_w:= (w_i)_{i\in {\mathbb N}},\;\; \mbox{where} \;\; w_i:= \psi^{i}(w),\quad\forall\, i\in {\mathbb N},$$
with the usual convention $\psi^0 = \Id_{{\mathcal I}_{\infty}}$ and $\psi^1 = \psi,$ is well-defined.
Since the sets ${\mathcal K}_0,\dots,{\mathcal K}_{m-1}$ are pairwise disjoint, for every term $w_i$ of $\textbf{s}_w$ there exists a
unique index
$$s_i = s_i(w_i),\;\;\mbox{with}\;\; s_i\in \{0,\dots,m-1\},$$
such that
$w_i \in {\mathcal K}_{s_i}\,.$
Therefore the map $g_2:{\mathcal I}_{\infty}^{\mathbb N}\to \Sigma_m^+,$
\begin{equation*}
g_2: \textbf{s}_w \mapsto (s_i)_{i\in {\mathbb N}}\in \Sigma_m^+\,
\end{equation*}
is also well-defined. Thus, by Definition \ref{def-chm} the map
\begin{equation*}
\pi:= g_2\circ g_1: {{\mathcal I}_{\infty}}\to \Sigma_m^+
\end{equation*}
is a surjection that makes the diagram \eqref{diag-1} commute and
the preimage through $\pi$ of any $k$-periodic sequence in
$\Sigma_m^+$ contains at least one $k$-periodic point of
${\mathcal I}_{\infty}.$
To check that $\pi$ is continuous, we
prove the continuity in a generic $\bar z\in{\mathcal I}_{\infty},$ by showing
that for any $\varepsilon>0,$ there exists $\delta>0$ such that
$\forall z\in{\mathcal I}_{\infty}$ with $d_X(z,\bar z)<\delta,$ then $\hat
d(\pi(z),\pi(\bar z))<\varepsilon,$ where $d_X$ is the distance on $X$ and $\hat d$ is the metric defined
in \eqref{eq-dist}. Let us fix $\varepsilon >0$ and let $n\in
\mathbb N$ such that $0<1/m^n <\varepsilon.$ We notice that it is sufficient to prove that
$(\pi(z))_i=(\pi(\bar z))_i,$ for any $i=0,\dots,n.$ Indeed, if this is the case, by the definition of $\hat d,$ it follows that $\hat d(\pi(z),\pi(\bar z))\le 1/m^n
<\varepsilon.$ \\
Since $\bar z\in{\mathcal I}_{\infty},$ there exists a sequence
$(s_0,\dots,s_n)\in \{0,\dots,m-1\}^{n+1}$ such that
$$\bar z\in \mathcal K_{s_0},\,\psi(\bar z)\in \mathcal K_{s_1},\dots, \psi^n(\bar z)\in \mathcal K_{s_n}.$$
By the pairwise disjointness of the sets ${\mathcal K}_i$'s, it holds that
$$\eta:=\min\{d_X({\mathcal K}_i, {\mathcal K}_{j}):i,j=0,\dots,m-1\}>0.$$
Hence, for any $z\in {\mathcal I}_{\infty}$ with $d_X(z,\bar z)<\eta/2,$ it
follows that $z\in \mathcal K_{s_0},$ too.  By the continuity of
$\psi$ in $\bar z,$ there exists $\delta_1>0$ such that $\forall
z\in{\mathcal I}_{\infty}$ with $d_X(z,\bar z)<\delta_1,$ then
$d_X(\psi(z),\psi(\bar z))<\eta/2.$ But this means that $\psi(z)\in
\mathcal K_{s_1}.$ Analogously, by the continuity of $\psi^2$ in
$\bar z,$ there exists $\delta_2>0$ such that $\forall
z\in{\mathcal I}_{\infty}$ with $d_X(z,\bar z)<\delta_2,$ then
$d_X(\psi^2(z),\psi^2(\bar z))<\eta/2$ and thus $\psi^2(z)\in
\mathcal K_{s_2},$ for any such $z.$ Proceeding in such way until
the $n$-th iterate of $\psi$ and setting
$$\delta:=\min\left\{\frac{\eta}{2},\delta_1,\dots,\delta_n\right\}\,,$$
we find that, for any $z\in{\mathcal I}_{\infty}$ with $d_X(z,\bar z)<\delta,$ it holds that
$$z\in \mathcal K_{s_0},\,\psi(z)\in \mathcal K_{s_1},\dots, \psi^n(z)\in \mathcal K_{s_n},$$
exactly as for $\bar z.$ But this means that $(\pi(z))_i=(\pi(\bar z))_i,$ for any $i=0,\dots,n,$ and hence
$\hat d(\pi(z),\pi(\bar z))\le 1/m^n <\varepsilon.$ The continuity of $\pi$ is thus proved.\\
Considering in diagram \eqref{diag-1} the restriction of $\psi$ to $\mathcal P\subseteq {\mathcal I}_{\infty},$ we find the commutative diagram
\begin{equation*}
\begin{diagram}
\node{{\mathcal P}} \arrow{e,t}{\psi} \arrow{s,l}{\pi}
      \node{{\mathcal P}} \arrow{s,r}{\pi} \\
\node{{\mathcal P}_m^+} \arrow{e,b}{\sigma}
   \node{{\mathcal P}_m^+}
\end{diagram}
\end{equation*}
where ${\mathcal P}_m^+\subseteq \Sigma_m^+$ is the set of the
periodic sequences of $m$ symbols.  Notice that $\pi(\mathcal
P)={\mathcal P}_m^+$ thanks to Definition \ref{def-chm}. Recalling
the well-known fact that ${\mathcal P}_m^+$ is dense in
$\Sigma_m^+$ \cite[Proposition 1.9.1]{KaHa-95}, by the continuity of $\pi,$ it follows that
$$\pi(\mathcal I)=\pi(\overline{\mathcal P})\subseteq \overline{{\mathcal P}_m^+}=\Sigma_m^+.$$
On the other hand $\pi(\mathcal I)$ is a compact set containing $\pi(\mathcal P)={\mathcal P}_m^+,$ and hence
$$\pi(\mathcal I)\supseteq\overline{{\mathcal P}_m^+}=\Sigma_m^+.$$
Therefore, we can conclude that $\pi(\mathcal I)=\Sigma_m^+$ and the diagram
\begin{equation*}
\begin{diagram}
\node{{\mathcal I}} \arrow{e,t}{\psi} \arrow{s,l}{\pi}
      \node{{\mathcal I}} \arrow{s,r}{\pi} \\
\node{\Sigma_m^+} \arrow{e,b}{\sigma}
   \node{\Sigma_m^+}
\end{diagram}
\end{equation*}
still commutes. Moreover, the preimage through
$\pi$ of any $k$-periodic sequence in $\Sigma_m^+$ contains at least
one $k$-periodic point of $\mathcal I,$ as $\mathcal P\subseteq\mathcal I.$ Conclusions $(ii)$ and $(iii)$ are thus proved.\\
Assertion $(iv),$ about the positive topological entropy, comes from property $(ii)$ about the semi-conjugacy to the Bernoulli shift. For a proof, see \cite[Theorem 7.12]{Wa-82}.\\
Finally, conclusion $(v),$ about the existence of the compact invariant set
$\Lambda \subseteq {\mathcal I},$ follows by applying Theorem \ref{th-ay} with the positions $(X, f)=(\mathcal I,\psi\restriction_{\mathcal I})$ and $(Y, g)=(\Sigma_m^+,\sigma).$
\end{proof}

\noindent
We point out that results similar to Theorem \ref{th-cons}
are almost known in the literature (see e.g. \cite[Theorem 3]{Sr-00} for a related statement), maybe with exception of conclusion $(v).$
However, for the sake of
completeness, we have decided to prove it in full details. Notice
that, in frameworks more general than ours, the set ${\mathcal I}_{\infty}$ in
\eqref{eq-lam}, and a fortiori $\mathcal P$ in \eqref{eq-per}, can
be empty and thus it is not possible to define $\mathcal I$ as in
\eqref{eq-i} in order to have an invariant set.
In our case we can do like that since, by Definition \ref{def-chm},
the sequences of symbols are realized by points in ${\mathcal I}_{\infty}$ and the periodic sequences of symbols are reproduced by points in $\mathcal P.$
We also remark that the density of the periodic points of $\psi$ in $\mathcal I$ does not imply that the periodic points are dense in the smaller set $\Lambda,$ on which $\psi$ is transitive and sensitive. Therefore, we cannot conclude the presence of Devaney chaos (cf. Definition \ref{def-dev}) for $\psi$ on $\Lambda.$

\begin{remark}\label{rem-inj}
{\rm{If, in addition to the hypotheses of Theorem \ref{th-cons},
the map $\psi$ is also injective on $\mathcal K,$ then one can deal with bi-infinite sequences of $m$ symbols instead of forward ones and prove the stronger
property for $\psi$ of semi-conjugacy to the two-sided Bernoulli shift
$\sigma:\Sigma_m:=\{0,\dots,m-1\}^{\mathbb Z}\to\Sigma_m\,,$ defined as
$\sigma((s_i)_i):=(s_{i+1})_i\,,\,\forall i\in\mathbb Z\,.$ Indeed, it
is sufficient to replace the set ${\mathcal I}_{\infty}$ in \eqref{eq-lam} with
$${\mathcal I}_{\infty}:=\bigcap_{n=-\infty}^{\infty}\psi^{-n}(\mathcal K)$$ and consider
$\mathcal I$ as in \eqref{eq-i}, in order to obtain, via
similar steps, the analogue of the conclusions of Theorem \ref{th-cons} with respect to two-sided sequences, rather than
one-sided sequences. The precise statement reads as follows:

\begin{theorem}\label{th-inj}
Let $\psi$ be a  map inducing chaotic dynamics on $m\ge 2$ symbols on a
set ${\mathcal D}\subseteq X$ and which is continuous and injective on
$$
\mathcal K:=\bigcup_{i=0}^{m-1}\mathcal K_i\subseteq {\mathcal D},
$$
where $\mathcal K_0,\dots,\mathcal K_{m-1},{\mathcal D}$ and $X$ are as
in Definition \ref{def-chm}. Setting
\begin{equation*}
{\mathcal I}_{\infty}:=\bigcap_{n=-\infty}^{\infty}\psi^{-n}(\mathcal K),
\end{equation*}
then there exists a nonempty compact set
$${\mathcal I}\subseteq {\mathcal I}_{\infty} \subseteq {\mathcal K},$$
on  which the following are fulfilled:

\begin{itemize}
\item[$(i)$] ${\mathcal I}$
is invariant for $\psi$ (that is, $\psi({\mathcal I}) = {\mathcal I}$);

\item[$(ii)$] $\psi\restriction_{\mathcal I}$ is semi-conjugate to the two-sided Bernoulli shift on $m$ symbols, i.e.
there exists a continuous map $\pi$ of ${\mathcal I}$ onto $\Sigma_m:=\{0,\dots,m-1\}^{\mathbb Z},$ endowed with the distance
\begin{equation*}
\hat d(\textbf{s}', \textbf{s}'') := \sum_{i\in {\mathbb Z}} \frac{|s'_i - s''_i|}{m^{|i| + 1}}\,,
\end{equation*}
for $\textbf{s}'=(s'_i)_{i\in {\mathbb Z}}$ and
$\textbf{s}''=(s''_i)_{i\in {\mathbb Z}}\in \Sigma_m\,,$
such that the diagram
\begin{equation*}
\begin{diagram}
\node{{\mathcal I}} \arrow{e,t}{\psi} \arrow{s,l}{\pi}
      \node{{\mathcal I}} \arrow{s,r}{\pi} \\
\node{\Sigma_m} \arrow{e,b}{\sigma}
   \node{\Sigma_m}
\end{diagram}
\end{equation*}
commutes, i.e $\pi\circ\psi=\sigma\circ\pi,$ where $\sigma:\Sigma_m\to\Sigma_m$ is the Bernoulli
shift defined by $\sigma((s_i)_i):=(s_{i+1})_i,\,\forall
i\in\mathbb Z\,;$

\item[$(iii)$] The set $\mathcal P$ of the periodic points of $\psi\restriction_{{\mathcal I}_{\infty}}$ is dense in ${\mathcal I}$
and the preimage $\pi^{-1}(\textbf{s})\subseteq {\mathcal I}$ of
every
$k$-periodic sequence $\textbf{s} = (s_i)_{i\in {\mathbb N}}\in \Sigma_m$
contains at least one $k$-periodic point.
\end{itemize}

Furthermore, from conclusion $(ii)$ it follows that:

\begin{itemize}
\item[$(iv)$] $$h_{\rm top}(\psi)\ge h_{\rm top}(\psi\restriction_{\mathcal I})\geq h_{\rm top}(\sigma) = \log(m),$$
where $h_{\rm top}$ is the topological entropy;

\item[$(v)$] There exists a compact invariant set $\Lambda\subseteq {\mathcal I}$ such that $\psi\vert_{\Lambda}$ is
semi-conjugate to the two-sided Bernoulli shift on $m$ symbols, topologically transitive and has sensitive dependence on initial conditions.
\end{itemize}
\end{theorem}

\noindent
For further details see \cite{PiZa-08}, where some alternative approaches are expounded, such as that in \cite{LWSr-02}.\\
The present observation applies, for instance, when dealing with the original Smale horseshoe map, introduced at the beginning of Chapter \ref{ch-ch}. Another field of applicability is that of the ODEs with periodic coefficients. Namely, in such a case, if one looks for periodic solutions, the corresponding map $\psi$ is the homeomorphism called Poincar\'e map associated to the system. Some examples will be presented in Chapter \ref{ch-ode}.
}}

\hfill$\lhd$\\
\end{remark}

\noindent
To conclude the discussion about the results in \cite{KeKoYo-01}, we notice that, thanks to the fact that the kind of chaos in Definition \ref{def-chm} is stricter than the one detected in \cite{KeKoYo-01,KeYo-01}, we are able to obtain a result (Theorem \ref{th-cons}) that covers the Chaos Lemma (\cite[Lemma 4]{KeKoYo-01}), adding further information. In fact, the set $\mathcal I$ in Theorem \ref{th-cons} plays the role of $Q_I$ in \cite{KeKoYo-01} and
the subsets $\Lambda\subseteq \mathcal I$ and $Q_{*}\subseteq Q_I$ share some common features, like the invariance and the sensitivity of the maps defined on them, but we have in addition the density of the periodic points on $\mathcal I,$ while, as already observed, in \cite{KeKoYo-01, KeYo-01} the existence of periodic points is not ensured at all. Also the transitivity is missing in the statement of the Chaos Lemma (\cite[Lemma 4]{KeKoYo-01}), even if, after a look at its proof, it is evident that such property holds in that framework, too. Indeed the invariant set $Q_{*}$ is the $\omega$-limit set of a certain point $x^*$ of $Q_I$ and thus the orbit of $x^*$ is dense in $Q_{*}.$ On the other hand, by the invariance of $Q_{*},$ this is enough to infer the transitivity on $Q_{*}.$ As we shall see in Section \ref{sec-var}, the same argument is employed in the proof of Theorem \ref{th-ay}, where the reader can find the missing definitions and details.

\medskip

\noindent
Finally, a natural question that can arise is whether our stretching condition in Theorem \ref{th-ch} is strong enough to imply a conjugacy to the Bernoulli shift, rather than just a semi-conjugacy as stated in Theorem \ref{th-cons}. The answer in general is negative. However, it is possible to add further assumptions in order to get it. A first step in this direction is represented by some recent works by Zgliczy\'nski and collaborators \cite{CaZg->,Zg-09}, where they combine the theory of  covering relations from \cite{ZgGi-04} with certain cone conditions to establish the properties implied by hyperbolicity, such as the existence of stable and unstable manifolds. Moreover, in their framework, every periodic sequence of symbols is realized by the itinerary generated by a unique point: in symbols, using the notation introduced in the statement of Theorem \ref{th-cons}, this means that the preimage $\pi^{-1}(\textbf{s})\subseteq {\mathcal I}$ of
every $k$-periodic sequence $\textbf{s} = (s_i)_{i\in {\mathbb N}}\in \Sigma_m^+$
contains exactly one $k$-periodic point.

\section{Symbolic dynamics}\label{sec-sd}
In this section we furnish some tools from symbolic dynamics that will be useful to investigate more deeply the relationship between the notion of chaos in the sense of Definition \ref{def-chm} and other ones widely considered in the literature. In order to accomplish such task, we need to recall a few basic definitions: some of them have already been introduced in Section \ref{sec-de}, but we prefer to gather them here for the sake of consistency.

\smallskip

\noindent
Given an integer $m\ge 2,$ we denote by
$\Sigma_m:= \{0,\dots,m-1\}^{\mathbb Z}$
the set
of the two-sided sequences of $m$ symbols and by
$\Sigma_m^+:=\{0,\dots,m-1\}^{\mathbb N}$ the set of one-sided sequences of $m$ symbols. These compact spaces are usually endowed with the distance
\begin{equation}\label{eq-sd}
\hat d(\textbf{s}', \textbf{s}'') := \sum_{i\in {\mathbb I}} \frac{|s'_i - s''_i|}{m^{|i| + 1}}\,,\quad
\mbox{ for }\; \textbf{s}'=(s'_i)_{i\in {\mathbb I}}\,,\;
\textbf{s}''=(s''_i)_{i\in {\mathbb I}}\,,
\end{equation}
where ${\mathbb I} = {\mathbb Z}$ or ${\mathbb I} = {\mathbb N},$ respectively.
The metric in \eqref{eq-sd} could be replaced with
$$\tilde d(\textbf{s}', \textbf{s}'') := \sum_{i\in {\mathbb I}} \frac{d(s'_i, s''_i)}{m^{|i| + 1}}\,,\quad
\mbox{ for }\; \textbf{s}'=(s'_i)_{i\in {\mathbb I}}\,,\;
\textbf{s}''=(s''_i)_{i\in {\mathbb I}}\,,$$
where $d(\cdot\,,\cdot)$ is the discrete distance on $\{0,\dots, m-1\},$ that is,
$d(s'_i, s''_i)=0$ for $s'_i=s''_i$ and $d(s'_i, s''_i)=1$ for
$s'_i\ne s''_i.$ The significance of this second choice reveals when one needs to look at the elements from $\{0,\dots, m-1\}$ as symbols instead of numbers.\\
On such spaces we define the \textit{one-sided Bernoulli shift}\index{Bernoulli shift, one-sided} $\sigma: \Sigma_m^+\to \Sigma_m^+$ and the \textit{two-sided Bernoulli shift} $\sigma: \Sigma_m\to \Sigma_m$ \index{Bernoulli shift, two-sided} on $m$ symbols as $\sigma((s_i)_i):=(s_{i+1})_i,$ $\forall
i\in\mathbb I,$ for $\mathbb I=\mathbb N$ or $\mathbb I=\mathbb Z,$ respectively. Both maps are continuous and the two-sided shift is a homeomorphism.\\
A precious tool for the detection of complex dynamics is the \textit{topological entropy}\index{entropy, topological} and indeed its positivity is generally considered as one of the trademarks of chaos. Such object can be introduced for any continuous self-map $f$ of a compact topological space $X$ and we indicate it with the symbol $h_{\rm top}(f).$ Its original definition due to Adler, Konheim and McAndrew \cite{AdKoMA-65} is based on the open coverings. More precisely, for an open cover $\alpha$ of $X,$ we define the \textit{entropy of $\alpha$} as $H(\alpha):=\log N(\alpha),$ where $N(\alpha)$ is the minimal number of elements in a finite subcover of $\alpha.$ Given two open covers $\alpha$ and $\beta$ of $X,$ we define their \textit{join} $\alpha\vee\beta$ as the open cover of $X$ made by all sets of the form $A\cap B,$ with $A\in\alpha$ and $B\in\beta.$ Similarly one can define the join $\vee_{i=1}^n \alpha_i$ of any finite collection of open covers of $X.$ If $\alpha$ is an open cover of $X$ and $f:X\to X$ a continuous map, we denote by $f^{-1}\alpha$ the open cover consisting of all sets $f^{-1}(A),$ with $A\in\alpha.$ By $\vee_{i=0}^{n-1} f^{-i}\alpha$ we mean $\alpha\vee f^{-1}\alpha\vee\dots\vee f^{-n+1}\alpha.$
Finally, we have:
$$h_{\rm top}(f):=\sup_{\alpha}\left(\lim_{n\to\infty}\frac{1}{n}\left(H\left(\vee_{i=0}^{n-1} f^{-i}\alpha\right)\right)\right),$$
where $\alpha$ ranges over all open covers of $X.$\\
Among the several properties of the topological entropy, we recall just the ones that are useful in view of the subsequent discussion.\\
In regard to the (one-sided or two-sided) Bernoulli shift $\sigma$ on $m$ symbols, it holds that
$$h_{\rm top}(\sigma) = \log(m).$$
Given a continuous self-map $f$ of a compact topological space $X$ and a \textit{invariant}\index{invariant set} (resp. \textit{positively invariant})\index{invariant set, positively} subset $\mathcal I\subseteq X,$
i.e such that $f(\mathcal I)=\mathcal I$ (resp. $f(\mathcal I)\subseteq\mathcal I$), then
\begin{equation}\label{eq-gr}
h_{\rm top}(f)\ge h_{\rm top}(f\restriction_{\mathcal I}).
\end{equation}
Denoting by $f^n$ the $n$-th iterate of the continuous self-map $f$ of a compact topological space $X,$ we have
\begin{equation}\label{eq-it}
h_{\rm top}(f^n)=nh_{\rm top}(f),\,\forall n\geq 1.
\end{equation}
Given two continuous self-maps $f:X\to X$ and $g:Y\to Y$ of the compact topological spaces $X$ and $Y$ and a continuous onto map $\phi:X\to Y$ that makes the diagram
\begin{equation}\label{diag-comm}
\begin{diagram}
\node{{X}} \arrow{e,t}{f} \arrow{s,l}{\phi}
      \node{{X}} \arrow{s,r}{\phi} \\
\node{Y} \arrow{e,b}{g}
   \node{Y}
\end{diagram}
\end{equation}
\textit{commute}\index{commuting diagram}, i.e. such that $\phi\circ f=g\circ\phi,$ then it holds that
$$h_{\rm top}(f)\ge h_{\rm top}(g).$$
If $\phi$ is also injective, the above inequality is indeed an equality.\\
When the diagram in \eqref{diag-comm} commutes, we say that $f$ and $g$ are \textit{topologically semi-conjugate} and that $\phi$ is a \textit{semi-conjugacy} between them. If $\phi$ is also one-to-one, then $f$ and $g$ are called \textit{topologically conjugate} and $\phi$ is named \textit{conjugacy}\index{conjugacy, semi-conjugacy}.\\
Thus, when for a continuous self-map $f$ of a compact topological space $X$ and a (positively) invariant subset $\mathcal I\subseteq X$ it holds that
$f\restriction_{\mathcal I}$ is semi-conjugate to the (one-sided or two-sided) Bernoulli shift $\sigma$ on $m$ symbols, then
\begin{equation}\label{eq-hch}
h_{\rm top} (f) \ge h_{\rm top}(f\restriction_{\mathcal I})\geq h_{\rm top}(\sigma) = \log(m).
\end{equation}
If $f\restriction_{\mathcal I}$ is conjugate to $\sigma,$ then the second inequality is indeed an equality.\\
We notice that, although the topological entropy can be defined for continuous self-maps of topological spaces, we confine ourselves to the case of metric spaces. More precisely, when dealing with chaotic dynamics, we will consider \textit{dynamical systems}\index{dynamical system}, i.e. couples $(X,f),$ where $X$ is a compact metric space and $f:X\to X$ is continuous and surjective.\\
For further features of $h_{\rm top}$ and additional details, see \cite{AdKoMA-65,KaHa-95,Wa-82}.
With reference to the case of compact metric spaces, alternative definitions of entropy can be found in \cite{Bo-71,Di-70}.
\medskip

\noindent
Let us begin our excursus on complex dynamics by returning on the concepts of ``covering'' and ``crossing'' introduced at the beginning of Chapter \ref{ch-ch} and fundamental in the theory of symbolic dynamics. \\
For the reader's convenience, we recall that, given a continuous mapping
$f: {\mathbb R}\to {\mathbb R}$ and two intervals $I,J\subseteq {\mathbb R},$
we say that \textit{$I$ $f$-covers $J$}\index{covering relation} if $f(I)\supseteq J,$ or equivalently, if there exists a subinterval $I_0\subseteq I$ such that $f(I_0) = J.$ We also say that
\textit{$I$ $f$-covers $J$ $m$ times} if there are $m\ge 2$
subintervals $I_0\,,\dots, I_{m-1}\subseteq I,$ with pairwise
disjoint interiors, such that $f(I_{k}) = J,$ for $k=0,\dots,m-1$ (see \cite{BlGu-80}).
In particular, we will focus on the case in which $I=J$ and there exist $m$ compact intervals $I_{0}\,,\dots,I_{m-1}\subseteq I,$ with pairwise disjoint interiors, such that
$I_i$ $f$-covers $I_j,$ for some (maybe all) $i,j\in\{0,\dots,m-1\}.$\\
The natural extension of the concept of $f$-covering to the setting of a continuous self-map $f$ of a generic metric space $X$ can be obtained by replacing intervals with compact subsets of $X$ in the previous definitions. Indeed, one can assume the existence of $m$ pairwise disjoint compact sets $C_{0}\,,\dots,C_{m-1}\subseteq X$ such that
$f(C_i)\supseteq C_j\,,$ for some $i,j\in\{0,\dots,m-1\}.$ On the other hand,
as suggested in \cite{BuWe-95}, instead of the above covering relation, one could deal with a weaker condition of the form
$f(C_i) \;\, \mbox{``\,goes across\,''}\, C_{j},$
which does not require the map $f$ to be surjective on the $C_j$'s. For example, in our approach, we interpret the expression ``\,goes across\,''
in terms of the stretching along the paths effect from Definition \ref{def-sap} or Definition \ref{def-sapn}, that is, $f({C}_i)$ ``\,goes across\,'' ${C}_j$ means for us that $f:\widetilde{C}_i\stretchx \widetilde{C}_j\,,$ where $\widetilde{C}_i$ and $\widetilde{C}_j$ are (N-dimensional) oriented rectangles of $X.$ In this respect, an interesting framework is the one described in Definition \ref{def-cn} for $\widetilde A=\widetilde B=\widetilde{\mathcal R},$ or in the statement of Theorem \ref{th-ch} with $X=\mathbb R^2.$ \footnote{The case of sets homeomorphic to $[0,1]^2$ but contained in a generic metric space $X$ could be considered as well. However, since in the applications the space $X$ is the Euclidean plane and the generalized rectangles are compact regions bounded by graphs of continuous functions, for simplicity we confine ourselves to the planar setting.} In such situation, there exist $m\ge 2$ pairwise disjoint compact subsets ${\mathcal K}_0,\dots,{\mathcal K}_{m-1}$ of a generalized planar rectangle $\mathcal R$ for which $(\mathcal K_i,\psi):\widetilde{\mathcal R}\stretchx \widetilde{\mathcal R}$ holds, where $\widetilde{\mathcal R}=(\mathcal R,\mathcal R^-).$ Then, in view of Remark \ref{rem-mp}, it is possible to find $m$ pairwise disjoint vertical slabs $\widetilde{\mathcal R}_i$ of $\widetilde{\mathcal R}$ (cf. Definition \ref{def-hv}), with $\mathcal R_i\supseteq {\mathcal K}_i,$ for $i=0,\dots,m-1,$
which satisfy the relation $\psi:\widetilde{\mathcal R}_i\stretchx \widetilde{\mathcal R}_j,\,\forall\, i,j\in \{0,\dots,m-1\}.$ \\
The pretext for considering the above covering and crossing relations comes from the previously observed fact that, whenever a map $\psi$ induces chaotic dynamics in the sense of Definition \ref{def-chm}, then its entropy is positive (cf. conclusion $(iv)$ in Theorem \ref{th-cons}) and this happens, for instance, when the stretching condition in Theorem \ref{th-ch} is fulfilled.
Actually, we are going to show that a positive entropy can be obtained even under weaker assumptions.\\
For clarity's sake, we first explain the idea in the one-dimensional setting, in order to simplify the comprehension of the generic $N$-dimensional case (for $N\ge 2$). Hence, let us suppose that $f:I\to I$ is a continuous self-mapping of a compact interval $I\subset\mathbb R$ and assume there exists $a\in I$ such that, calling $b=f(a),\,c=f(b)=f^2(a)$ and $d=f(c)=f^2(b)=f^3(a),$ it holds that
\begin{equation}\label{eq-ly}
d\le a< b< c \quad \mbox{or} \quad  d\ge a> b> c\,.
\end{equation}
In particular, if $d=a$ then the map $f$ has in $I$ a point of period three. This is in fact the case mentioned in the title of the well-known paper \cite{LiYo-75}. Notice that, setting
\begin{equation}\label{eq-ints}
I_0:=[\min{\{a,b\}},\max{\{a,b\}}] \quad \mbox{and} \quad I_1:=[\min{\{b,c\}},\max{\{b,c\}}],
\end{equation}
it follows that $I_0\, f$-covers $I_1$ and $I_1\, f$-covers $I_0,$ as well as $I_1\, f$-covers $I_1.$ Under condition \eqref{eq-ly}, the authors in \cite{LiYo-75} can prove the presence in $I$ of periodic points of any period and also the existence of an uncountable scrambled set. We will return on this concept in Definition \ref{def-ly}. Before that, we want to show the positivity of the topological entropy of $f$ in the just described framework. To such end, we need some tools from the theory of symbolic dynamics \cite{Ki-98,LiMa-95}.

\smallskip

\noindent
Given a continuous map $g:J\to J$ defined on a compact interval $J\subset\mathbb R$ and $n\ge 2$ closed subintervals $J_0,\dots,J_{n-1}\subseteq J,$ with pairwise disjoint interiors, we associate to the dynamical system $(J,g)$ the $n\times n$ \textit{transition matrix}\index{matrix! transition} $T=T(i,k),$ for $i,k=0,\dots,n-1,$ defined as
\begin{equation}\label{eq-tm}
T(i,k)= \left\{ \begin{array}{ll}
1 & \textrm{ if $J_{i}\,\, g$-covers $J_{k}\,,$}\\
0 & \textrm{ else}.
\end{array} \right.
\end{equation}
Moreover, let
\begin{equation}\label{eq-sht}
\Sigma^+_T:=\{(s_i)_{i\in \mathbb N}\in\Sigma_n^+:T(s_k,s_{k+1})=1,\,\forall k\in\mathbb N\}
\end{equation}
be the \textit{space of the admissible sequences for} $T$ and
\begin{equation}\label{eq-sft}
\sigma_T:\Sigma^+_T\to\Sigma^+_T,\quad \sigma_T:=\sigma\restriction_{\Sigma^+_T}
\end{equation}
be the \textit{subshift of finite type for the matrix}\index{subshift of finite type} $T.$ In particular, according to \cite {BuWe-95}, we say that such subshift of finite type is \textit{nontrivial}\index{subshift of finite type, nontrivial} if $T$ is an \textit{irreducible} matrix\index{matrix! irreducible} (that is, for every couple of integers $i,\,k\in \{0,\dots, n-1\}$ there exists a positive integer $l$ such that $T^l(i,k)>0$), which is not a permutation on $n$ symbols\footnote{Notice that, when $T$ is the matrix whose entries are all $1,$ then $\Sigma^+_T=\Sigma_n^+$ and $\sigma_T$ coincides with the classical Bernoulli shift $\sigma,$ also known as \textit{full shift}\index{full shift} \cite{Ki-98}.}. The space $\Sigma^+_T$ inherits the metric from $\Sigma_n^+$ and, with this choice, $\sigma_T$ is continuous.
Furthermore, it is possible to prove (cf. \cite[Observation 1.4.2]{Ki-98}) that, when $T$ is irreducible, $h_{\rm top}(\sigma_T)= \log(\overline\lambda),$ where $\overline\lambda$ is the largest real eigenvalue of $T$ in absolute value, also called \textit{Perron eigenvalue of T}.
Therefore, if there exists a (positively) invariant set ${\mathcal I}^*\subseteq J$ such that $g\restriction_{{\mathcal I}^*}$ is semi-conjugate to $\sigma_T,$ it holds that
\begin{equation}\label{eq-irr}
h_{\rm top}(g)\ge h_{\rm top}(g\restriction_{{\mathcal I}^*})\ge h_{\rm top}(\sigma_T) = \log(\overline\lambda).
\end{equation}
In the case under investigation, we can take
\begin{equation}\label{eq-invt}
{\mathcal I}^*:=\bigcap_{i=0}^{+\infty}g^{-i}(J)
\end{equation}
and define the semi-conjugacy $\pi:{{\mathcal I}^*}\to \Sigma^+_T$ between $g\restriction_{{\mathcal I}^*}$ and $\sigma_T$ as
\begin{equation}\label{eq-pi}
(\pi(x))_i=k \quad \mbox{iff} \quad g^i(x)\in J_k\,, \quad \mbox{for} \quad i\in\mathbb N.
\end{equation}
We recall that, according to \cite[Lemma 1.2]{BuWe-95}, for the Perron eigenvalue it holds that $\overline\lambda$ is strictly greater than $1$ unless $T$ is a permutation matrix (otherwise, $\overline\lambda=1$). Thus, the topological entropy of any nontrivial subshift of finite type is positive.

\smallskip

\noindent
We have now available all the ingredients to deal with a map $f$ with the properties expressed in \eqref{eq-ly}. The transition matrix for $f$ associated to the intervals $I_0$ and $I_1$ in \eqref{eq-ints} is

\begin{equation}\label{mat-f}
T_f=\left( \begin{array}{cc}
0 & 1\\
1   & 1
\end{array} \right).
\end{equation}

\noindent
Defining the (positively) invariant set ${\mathcal I}^*$ for $f$ as in \eqref{eq-invt}, it is possible to obtain a semi-conjugacy as in \eqref{eq-pi} between $f\restriction_{{\mathcal I}^*}$ and the subshift of finite type $\sigma_{T_f}$ for the matrix $T_f.$ Thus, we can conclude that
\begin{equation}\label{eq-gold}
h_{\rm top}(f)\ge h_{\rm top}(\sigma_{T_f})= \log\left(\frac{1+\sqrt{5}}{2}\right),
\end{equation}
since the Perron eigenvalue of $T_f$ is $\overline\lambda=\frac{1+\sqrt{5}}{2}.$ This quantity coincides with the \textit{golden mean ratio} and the map $\sigma_{T_f}$ is also known under the name of \textit{golden mean shift}.

\medskip

\noindent
We stress that more general situations can be handled in a similar manner. For instance,
a further case that can be considered is when for a continuous map $g:J\to J,$ some multiple coverings among the subintervals $J_0,\dots,J_{n-1}$ of $J$ occur. In such a framework, the transition matrix $T$ is replaced by the $n\times n$ \textit{adjacency matrix}\index{matrix! adjacency} $A=A(i,k),$ for $i,k=0,\dots,n-1,$ defined as
\begin{equation}\label{eq-am}
A(i,k)= \left\{ \begin{array}{ll}
l & \textrm{ if $J_{i}\,\, g$-covers $J_{k}\,\,l$ times\,,}\\
0 & \textrm{ if $J_{i}$ does not $g$-cover $J_{k}$}\,,
\end{array} \right.
\end{equation}
where $l$ is a positive integer\footnote{If $l=1,$ condition ``$J_{i}\,\, g$-covers $J_{k}\,\,l$ times'' simply reads as ``$J_{i}\,\, g$-covers $J_{k}$''.}.
The entries of $A$ are now nonnegative integers, possibly different from $0,1.$
Also in this case, one can associate a suitable subshift of finite type to $A.$ To understand how to act, it is necessary to recall some notions from graph theory \cite{LiMa-95,Ro-99}.\\ A \textit{graph} $G$ consists of a finite set $\mathcal V=\mathcal V(G)$ of \textit{vertices} and of a finite set of \textit{edges} $\mathcal E=\mathcal E(G)$ among them. Every edge $e\in\mathcal E(G)$ starts at a vertex denoted by $i(e)$ (\textit{initial state}) and ends at a vertex denoted by $t(e)$ (\textit{terminal state}), which can possibly coincide with $i(e).$ Of course, there may be several edges between the same initial and terminal states. In the special case that $A=A(i,k)$ is an adjacency matrix, we call \textit{graph of} $A$ the graph $G=G(A),$ with vertex set $\mathcal V(G)=\{0,\dots,n-1\}$ and with $A(i,k)$ edges with initial state in $i$ and terminal state in $k,$ for $i,k=0,\dots,n-1.$ The cardinality of  the set of the edges is $|\mathcal E(G)|=\sum_{i,k\in\{0,\dots,n-1\}}A(i,k).$ We name $N$ this quantity, so that $\mathcal E(G)=\{e_1,\dots,e_N\}$ is the set of the edges. \\
The shift $\sigma_T$ in \eqref{eq-sft} for the transition matrix $T$ in \eqref{eq-tm} is called  a \textit{vertex subshift}. We want now to define an \textit{edge subshift} for the adjacency matrix $A$ in \eqref{eq-am}. In order to realize it, we construct the transition matrix $T'=T'(i,k)$ on $\mathcal E(G),$ that is, the $N\times N$ matrix with entries
\begin{equation*}
T'(i,k)= \left\{ \begin{array}{ll}
1 & \textrm{ if $t(e_i)=i(e_k),$ }\\
0 & \textrm{ else}.
\end{array} \right.
\end{equation*}
The desired edge subshift for $A$ corresponds to the vertex subshift for $T',$ $\sigma_{T'}:\Sigma^+_{T'}\to\Sigma^+_{T'}\,,\,\sigma_{T'}:=\sigma\restriction_{\Sigma^+_{T'}}\,,$  where the definition of $\Sigma^+_{T'}$ is analogous to the one of $\Sigma^+_{T}$ in \eqref{eq-sht} \footnote{Observe that, when the entries of $A$ are already from $\{0,1\},$ it is still possible to construct an associated transition matrix $T'$ as described above, that will be possibly different from $A$ (also the dimensions $n$ and $N$ won't coincide in general). On the other hand, it is not difficult to show that the vertex subshift $\sigma_A$ and the corresponding edge subshift $\sigma_{T'}$ are conjugate. Indeed, the conjugacy $h:\Sigma^+_{T'}\to\Sigma^+_A$ can be defined as $h((e_j)_{j\in \mathbb N})=(v_j)_{j\in \mathbb N},$ where $v_j=i(e_j),\,\forall j\in\mathbb N,$ with $v_j\in\mathcal V$ and $e_j\in\mathcal E.$}.\\
The edge subshifts can thus be seen as a counterpart of the vertex subshifts in relation to matrices with nonnegative integer entries, possibly different from $0,1.$ An example of matrices of such kind is represented by the powers of a given transition matrix and it is particularly important to handle them when dealing with the iterates of a given function.\\
Notice that, both in the case of edge subshifts and in the case of vertex subshifts, we have confined ourselves to the one-sided sequences of symbols, since we do not assume $g$ to be one-to-one on its domain. Analogous results hold true for bi-infinite sequences under the additional hypothesis of injectivity for $g$ (cf. Remark \ref{rem-inj}).

\smallskip

\noindent
If, in place of self-maps on intervals, one deals with functions
$f:X\to X$ defined on a generic metric space $X,$ and considers, instead of the one-dimensional covering relation for intervals, some covering or crossing relation among pairwise disjoint compact subsets (or at least compact subsets with pairwise disjoint interiors, in the case of Markov partitions \cite{Ro-04}) of $X,$ then, to such framework, it is possible to associate a transition matrix or an adjacency matrix exactly as in \eqref{eq-tm} and \eqref{eq-am}, respectively, with intervals replaced by compact sets. \\
In particular, when the hypotheses of Theorem \ref{th-ch} are fulfilled with $X=\mathbb R^2$ and $m=2,$
we have already observed that there exist two vertical slabs $\widetilde{\mathcal R}_0$ and $\widetilde{\mathcal R}_1$ of $\widetilde{\mathcal R},$ with $\mathcal R_0\supseteq \mathcal K_0$ and $\mathcal R_1\supseteq \mathcal K_1,$ which satisfy $\psi:\widetilde{\mathcal R}_i\stretchx\widetilde{\mathcal R}_j,$ for $i,j=0,1.$
Thus, the transition matrix for $\psi$ associated to $\widetilde{\mathcal R}_0$ and $\widetilde{\mathcal R}_1$ is

\begin{equation}\label{mat-psi}
T_{\psi}=\left( \begin{array}{cc}
1 & 1\\
1   & 1
\end{array} \right).
\end{equation}

\noindent
Analogous conclusions can be drawn when restricting our stretching along the paths relation to the one-dimensional setting. In such a
case, a pair ${\widetilde I}=(I,I^-),$ where $I=[a,b]$ is a
compact interval and $I^- = \{a,b\}$ is the set of its extreme
points, may be seen as a degenerate kind of oriented rectangle.
Accordingly, the stretching property $({\mathcal K},\psi): {\widetilde I}\stretchx {\widetilde I}$
is equivalent to the fact that ${\mathcal K}$ contains a compact interval $I_0$ such that $\psi(I_0) = I.$ This shows that, in the one-dimensional case,
our stretching property reduces to the classical covering relation discussed above. Moreover, noticing that, in the just described framework, we have $(I_0,\psi): {\widetilde I}\stretchx {\widetilde I},$ then, if $I$ contains two disjoint compact subintervals $I_0$ and $I_1$
such that $\psi(I_i) = I,$ for $i=0,1,$ we enter the setting of
Theorem \ref{th-ch}, with ${\widetilde I}$ playing the role of ${\widetilde{\mathcal R}}$
and with ${\mathcal K}_i := I_i\,.$

\bigskip

\noindent
A look at the matrices $T_f$ in \eqref{mat-f} and $T_{\psi}$ in \eqref{mat-psi} shows that the assumptions of Theorem \ref{th-ch}, when restricted to the one-dimensional setting, are stronger than the conditions in \eqref{eq-ly} \footnote{At first sight, this conclusion may seem rather obvious, since a map $\psi$ as in Theorem \ref{th-ch} has periodic points of all periods and thus, in particular, also a point of period three. However, things are not so simple. Indeed, as a consequence of the \v{S}arkowskii Theorem \cite{St-77}, any continuous self-map $f$ defined on a compact interval and possessing a point of period three has periodic points of each period.
Moreover, the conditions in \eqref{eq-ly} do not necessarily imply the existence of a point of period three, being in fact more general. This is the reason that led us to introduce some elements from symbolic dynamics in order to compare our framework to the one in \cite{LiYo-75}.}. Recalling \eqref{eq-gold}, this verifies our claim that a positive topological entropy can be obtained under hypotheses weaker than the ones in Theorem \ref{th-ch}. On the other hand, with respect to \cite{LiYo-75}, the assumptions in Theorem \ref{th-ch} allow to obtain sharper consequences from a dynamical point of view (cf. Section \ref{sec-var}).

\medskip

\noindent
Before abandoning the one-dimensional setting, let us employ the just introduced tools from symbolic dynamics to study the behaviour of the logistic map\index{logistic map}
\begin{equation}\label{eq-log}
F:[0,1]\to\mathbb R,\,\,
F(x):=\mu x(1-x),
\end{equation}
when the parameter $\mu$ ranges in $[0,+\infty].$\\
It is easy to see that for
$\mu > 4$ the interval $[0,1]$  $F$-covers $[0,1]$ twice. Indeed, the maximum of $F$ is attained when $x=1/2$ and $F(1/2)=\mu/4,$ so that $F(1/2)\ge 1$ for $\mu \ge 4.$ Since $F(0)=F(1)=0,$ setting $\alpha:=F^{-1}(\{1\})\cap ([0,1/2])$ and $\beta:=F^{-1}(\{1\})\cap ([1/2,1]),$ we have $({I_i},F): {\widetilde I}\stretchx {\widetilde I},\,i=0,1,$ where $\widetilde I:=([0,1],\{0,1\}),$ $I_0:=[0,\alpha]$ and $I_1:=[\beta,1].$ \footnote{Actually, $[0,1]$  $F$-covers $[0,1]$ twice also for the ``watershed'' value $\mu =4.$ However, in this case $I_0$ and $I_1$ are not disjoint and thus Theorem \ref{th-ch} does not apply directly.}
Unfortunately, in such a case, $F$ does not map $[0,1]$ into itself and,
for almost every initial point,
the iterates $F^{n}$ of the map limit to $-\infty$ as $n\to
\infty.$ This generates problems with numerical simulations: indeed, although Theorems \ref{th-ch} and \ref{th-cons} ensure the existence of a chaotic invariant set in $[0,1],$ such region is not visible on the computer screen, because almost all points eventually leave the domain through the iterates of $F.$

\begin{figure}[ht]
\centering
\includegraphics[scale=0.25]{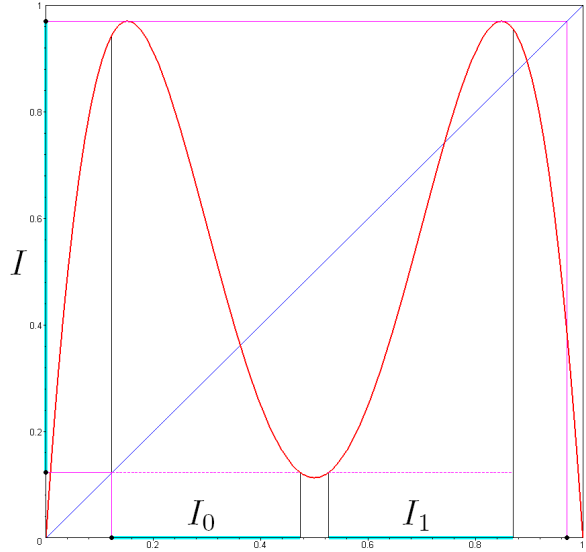}
\caption{\footnotesize {The graph of the second iterate
of the logistic map $F$ in \eqref{eq-log} for $\mu=3.88.$
}}
\label{fig-log}
\end{figure}

\noindent
When $\mu \le 4,$ the existence of multiple coverings can be established for some iterate of $F$ through geometric arguments analogous to the ones employed for the case $\mu > 4,$ that allow to conclude that the conditions of Theorem \ref{th-ch} are verified for sufficiently large values of the parameter $\mu \le 4.$
A corresponding example is illustrated in Figure
\ref{fig-log},
which depicts the second iterate of
$F$ for $\mu=3.88.$
The interval $I\subseteq [0,1],$ drawn with a thicker line on the $y$-axis, is oriented by choosing as $I^-$ its extreme points. For
$\widetilde I= (I, I^-)$ and calling $I_0$ and $I_1$ the two subintervals of $I$
highlighted on the $x$-axis, it is immediate to verify that $F^2(I_i)=I,$ for $i=0,1,$ and hence $(I_i, F^2):\widetilde I\stretchx\widetilde I,\,i=0,1.$
Therefore, the map $F^2$ induces
chaotic dynamics on two symbols relatively to $I_0,\,I_1$ and
possesses all the chaotic properties listed in Theorem \ref{th-cons}. In particular, from this something can still be deduced for $F.$ Indeed, by the semi-conjugacy between $F^2$ restricted to a suitable invariant set and the one-sided shift $\sigma$ on two symbols (cf. Theorem \ref{th-cons}, conclusions $(ii)$ and $(v)$), it holds that
$$h_{\rm top} \left(F^2\right) \ge \log(2)$$
and, since by \eqref{eq-it} we have $h_{\rm top} \left(F^2\right)=2 h_{\rm top} (F),$ it follows
$$h_{\rm top} (F) \ge \log(\sqrt 2)>0.$$

\bigskip

\noindent
We conclude this section by reconsidering Definition \ref{def-chm} and Theorem \ref{th-ch} in the light of the above results from symbolic dynamics. \\
Indeed, when a map satisfies the hypotheses of Theorem \ref{th-ch} with $X=\mathbb R^2$ and for some $m\ge 2,$ then the associated transition matrix has all elements equal to 1 (see \eqref{mat-psi} for the case $m=2$). On the other hand, one could face the more general situation in which the transition matrix $T$ associated to the system is irreducible, but possibly contains also some 0's \cite{LWSr-02}. The topological entropy can then be estimated as in \eqref{eq-irr} and one could say that the map $\psi$ induces chaotic dynamics when the topological entropy of the corresponding subshift of finite type $\sigma_T$ is positive. As already mentioned, this happens, for example, when $\sigma_T$ is nontrivial \cite{BuWe-95}. \\
According to \cite{KaHa-95,Ki-98}, if $T$ is irreducible, the map $\sigma_T$ is Devaney chaotic (cf. Definition \ref{def-dev}) on $\Sigma^+_T$ and thus it displays transitivity and sensitivity with respect to initial conditions and the set of periodic points of $\sigma_T$ is dense in $\Sigma^+_T.$ However, in order to proceed as in the proof of Theorem \ref{th-ch} and transfer the chaotic features of $\sigma_T$ to $\psi$ (via Theorem \ref{th-ay}), we need the periodic sequences of $\sigma_T$ to be realized by periodic points of $\psi.$ In view of Theorem \ref{th-ch}, this happens, for instance, when our stretching relation is fulfilled with respect to some oriented rectangles, but maybe not all the possible coverings are realized, so that the conclusions of Theorem \ref{th-cons} hold with the full shift $\sigma$ on $\Sigma^+_m$ replaced by $\sigma_T$ on $\Sigma^+_T.$\\
At last we notice that a concrete framework in which it is possible to obtain chaotic dynamics in the sense of Definition \ref{def-chm} is when there exist $m\ge 2$ pairwise disjoint subsets $C_i,\,i=0,\dots,m-1,$ of a metric space $X,$ each with the Fixed Point Property (FPP) \footnote{We say that a topological space $W$ has FPP if every continuous self-map of $W$ has a fixed point in $W.$} and such that there exists a homeomorphism $\varphi$ defined on $\bigcup_{i=0}^{m-1}C_i,$ with $\varphi(C_i)\supseteq C_j,\forall\, i,j=0,\dots,m-1.$ Then, since $\varphi(C_i)\supseteq C_i$ is equivalent to $\varphi^{-1}(C_i)\subseteq C_i,$  by the continuity of $\varphi^{-1}:C_i\to C_i\,,$ it follows that there exists $x^*\in C_i$ with $\varphi^{-1}(x^*)=x^*,$ i.e. $\varphi(x^*)=x^*.$ Thus, the presence of at least a fixed point is ensured in each $C_i,\,i=0,\dots,m-1.$ Working with the iterates of the map $\varphi,$ the existence of periodic points in the $C_i$'s can be shown in an analogous manner. To check that the periodic sequences on $m$ symbols are realized by periodic itineraries, it is then sufficient to follow the same steps as in the proof of the second and the forth conclusions in Theorem \ref{th-per}. Due to the invertibility of $\varphi,$ we observe that it is possible to deal with two-sided sequences (cf. Remark \ref{rem-inj}).
In view of the previous discussion, one could also consider the more general case in which $\varphi(C_i)\supseteq C_j,$ for \textit{some} $i,j=0,\dots,m-1,$ but the associated transition matrix $T$ is still irreducible. Obviously, in this setting, only the admissible (one-sided or two-sided) sequences
can be realized.

\section{On various notions of chaos}\label{sec-var}

Let us pursue further the discussion on chaotic dynamics started in Sections \ref{sec-de}--\ref{sec-sd}. In particular, we will try to show the mutual relationships among some of the most classical definitions of chaos (such as the ones by Li-Yorke, Devaney, etc.), considering also the notion of chaotic dynamics in the sense of Definition \ref{def-chm}.\\
Firstly, we would like to conclude the comparison between our results in Section \ref{sec-de} and the achievements in \cite{LiYo-75}.
As we have seen, the presence of a point of period three for a continuous self-map $f$ defined on a compact interval, or more generally the conditions \eqref{eq-ly} from \cite{LiYo-75}, ensure the positivity of the topological entropy for $f.$ Thanks to some tools from symbolic dynamics, we have also realized that our assumptions in Theorem \ref{th-ch} are stronger than the ones in \eqref{eq-ly}. Now we are going to show that the
consequences of Theorem \ref{th-ch}, listed in Theorem \ref{th-cons}, are strictly sharper than the conclusions in \cite{LiYo-75}.\\
The main result obtained by Li and Yorke in \cite{LiYo-75} is often recalled as a particular case of the \v{S}arkowskii Theorem \cite{St-77} but, actually, the authors in \cite{LiYo-75} proved much more than the presence of periodic points of any period for an interval map possessing a point of period three. Indeed, the existence of an uncountable \textit{scrambled} set (cf. Definition \ref{def-ly}) was obtained for such a map, which was called ``chaotic'' in \cite{LiYo-75} for the first time in the literature, even if the precise corresponding definition of chaos (now known as \textit{Li-Yorke chaos}) was not given there.

\begin{definition}\label{def-ly}
\rm
{Let $(X,d_X)$ be a metric space and $f:X\to X$ be a continuous map.
We say that $S\subseteq X$ is a \textit{scrambled set for $f$}\index{scrambled set}
if for any $x,y\in S,$ with $x\ne y,$ it holds that
$$ \liminf_{n\to\infty}d_X(f^n(x), f^n(y))=0 \quad \mbox{and} \quad \limsup_{n\to\infty}d_X(f^n(x), f^n(y))>0.$$
If the set $S$ is uncountable, we say that $f$ is \textit{chaotic in the sense of Li-Yorke}\index{chaos! Li-Yorke}.
}
\end{definition}

\noindent
We remark that according to \cite{LiYo-75} the scrambled set $S$ should satisfy an
extra assumption, i.e.
$$\limsup_{n\to\infty}d_X(f^n(x),f^n(p))>0,$$
for any $x\in S$ and for any periodic point $p\in X.$
However, in \cite{AuKi-01} this condition has been proved to be redundant
in any compact metric space and therefore it is usually omitted.\\
We also point out that the original framework in \cite{LiYo-75} was
one-dimensional. The subsequent extension to generic metric spaces
is due to different authors (see e.g. \cite{AuKi-01, BlGl-02,
HuYe-02, Ko-04}) that have collated the concept of chaos from \cite{LiYo-75}
to other ones available in the literature. We will try to present some of
these connections in the next pages.\\
We warn the reader that, in what follows, the term chaotic will be referred without distinction to a dynamical system, meant as a couple $(X,f),$
where $X$ is a compact metric space and $f:X\to X$ is continuous and surjective, as well as only to the map $f$ defining it. In this respect, with a slight abuse of terminology, some properties of the map $f$ (such as transitivity, sensitivity, etc.) will be transferred to the system $(X,f).$
When we need to specify the distance $d_X$ on $X,$ we will also write $(X,f, d_X)$ in place of $(X,f).$
Since the map $f$ has to be onto, if we are in the framework described in Theorem \ref{th-cons}, the dynamical system we usually consider is given by $(\mathcal I, \psi\restriction_{\mathcal I}),$ where $\mathcal I$ is the invariant set in \eqref{eq-i}.

\smallskip

\noindent
In order to understand the relationship between the kind of chaos expressed in Definition \ref{def-chm} and the Li-Yorke chaos, a key role is played by the topological entropy. Indeed, as we have seen in Theorem \ref{th-cons}, conclusion $(iv),$ thanks to the
semi-conjugacy with the Bernoulli shift, the topological entropy of $\psi$ is positive
in the setting described in Definition \ref{def-chm}. On the other hand, in
\cite[Theorem 2.3]{BlGl-02} it is established that any dynamical system with
positive topological entropy admits an uncountable scrambled set
and therefore it is chaotic in the sense of Li-Yorke. Hence, we can
conclude that our notion of chaos is stronger than the one in
Definition \ref{def-ly}, since any system chaotic according to
Definition \ref{def-chm} is also Li-Yorke chaotic, while the vice versa does not
hold in general. In fact, there exist maps Li-Yorke chaotic but with zero
topological entropy: for an example on the unit interval, see
\cite{Sm-86}.

\smallskip

\noindent
After the discussion on the concepts introduced in \cite{LiYo-75} and
on the results obtained therein,
the next feature related to the presence of chaos we take
into consideration is the \textit{sensitivity with respect to initial data}.
This is one of the three requirements in Devaney's definition of
chaos, together with the \textit{topological transitivity} and the \textit{density
of periodic points}, and it is maybe the most intuitive among them.
Indeed, it is pretty natural to associate the idea of chaos to
a certain unpredictability of the forward behaviour of the system under consideration.
The sensitivity with respect to initial data expresses exactly
such concept: no matter how close two points start, there exists
an instant in the future in which they are at a given positive
distance. However, although its intuitiveness, the sensitivity on initial
conditions has been proved
to be redundant in Devaney's definition of chaos in
any infinite metric space \cite{BaBr-92} \footnote{We recall that, in the special case of intervals,
also the density of periodic points is superfluous according to
\cite{VeBu-94} and therefore Devaney chaos coincides with transitivity in
the one-dimensional framework.}. \\
In view of the subsequent
treatment, we present the complete definition of \textit{Devaney chaos} for
the reader's convenience, even if at first we will focus only on the
third condition.
\begin{definition}\label{def-dev}
\rm{Given a metric space $(X,d_X)$ and a continuous function $f:X\to X,$ we say that $f$ is \textit{chaotic in the sense of Devaney} \index{chaos! Devaney}if:
\begin{itemize}
\item  $f$ is \textit{topologically transitive}\index{transitivity, topological}, i.e. for any
couple of nonempty open subsets $U,\,V\subseteq X$  there exists an
integer $n\ge 1$ such that $U\cap f^n(V)\ne\emptyset\,;$ \item The
set of the periodic points for $f$ is dense in $X;$ \item $f$ is
\textit{sensitive with respect to initial data} (or $f$ displays
\textit{sensitive dependence on initial conditions}) \index{sensitive dependence}\textit{on} $X,$ i.e.
there exists $\delta > 0$ such that for any $x\in X$ there is a
sequence $(x_i)_{i\in\mathbb N}$ of points in $X$ such that
$x_i\to x$ when $i\to\infty$ and
for each $i\in\mathbb N$ there exists a positive integer
$m_i$ with $d_X(f^{m_i}(x_i), f^{m_i}(x))\ge\delta\,.$
\end{itemize}
}
\end{definition}

\noindent
We stress that some authors call Devaney chaotic maps for
which the three above conditions hold true only with respect to a
compact positively invariant subset of the domain \cite{AuKi-01, Li-93}.
We also remark that sometimes a map $f$ on $X$ is named topologically transitive if there exists a dense orbit for $f$ in $X$ \cite{KoSn-97}. In the case of compact metric spaces without
isolated points, the two definitions turn out to be equivalent, but in general
they are independent. The precise relationship between such
notions can be found in \cite[Proposition (1.1)]{Si-92}: in any metric space without isolated points the existence of a dense orbit implies the topological
transitivity; vice versa, in separable and second category metric
spaces, the topological transitivity implies the
existence of a dense orbit. According to \cite[Lemma 3]{AuYo-80}, another case in which the two definitions of transitivity coincide is when $f$ is onto and this happens, for instance, in any dynamical system. Such fact will find an application in Theorem \ref{th-ay}.\\
About the sensitivity, we notice that it can be equivalently defined using neighborhoods instead of sequences: given a metric space $(X,d_X),$
the continuous map $f:X\to X$ is said to display sensitive dependence on initial conditions if there exists
$\delta > 0$ such that, for any $x\in X$ and for every open set $X\supseteq O_x\ni x,$ there exist $y\in O_x$ and
a positive integer $m$ with $d_X(f^m(x), f^m(y))\ge\delta$ \cite{Si-92}.
We also recall that it is possible to give a pointwise version of the definition of sensitivity, using neighborhoods or sequences. As regards this latter characterization, any $x\in X$ admitting a sequence $(x_i)_{i\in\mathbb N}$ of points of $X$ with $x_i\to x$ for $i\to\infty$ and a sequence $(m_i)_{i\in\mathbb N}$
of positive integers with $d_X(f^{m_i}(x_i), f^{m_i}(x))\ge\delta,\,\forall i\in\mathbb N$ and for some $\delta >0,$ is sometimes called \textit{$\delta$-unstable}
(or simply \textit{unstable})\index{unstable point} \cite{AuYo-80}. For an equivalent definition of unstable point based on neighborhoods, see \cite{BlGl-02}, where it is also observed that, in general, the instability of $f$ at every point of $X$ does not
imply $f$ to be sensitive on $X.$ Indeed there could be no positive $\delta$ such that all points of $X$ are $\delta$-unstable. However, the previous inference is true if the map is e.g. transitive \cite{BlGl-02}, such as in Theorem \ref{th-ay}, where we use the concept of sensitivity in both its equivalent versions and also the definition of unstable point based on sequences.

\medskip

\noindent The positivity of the topological entropy and the
sensitivity on initial conditions are someway related, since both are
signals of a certain instability of the system. The
topological entropy is however a ``locally detectable'' feature, in the sense that, according to \eqref{eq-gr}, it is
sufficient to find a (positively) invariant subset of the domain where it is
positive in order to infer its positivity on the whole domain.
Therefore, in general, we cannot expect the system to be sensitive
at each point if the entropy is positive. Adding a global property, such as transitivity, then this
implication holds true. Indeed in \cite{BlGl-02} it is argued that any
transitive map with positive topological entropy displays
sensitivity with respect to initial data. On the other hand,
if we are content with the presence of sensitivity only on an invariant subset of
the domain, then some authors have obtained corresponding results by considering,
instead of the positivity of the entropy, the stronger property of chaos in the sense of
coin-tossing (cf. Section \ref{sec-de}),
or at least the semi-conjugacy to the Bernoulli shift $\sigma$ for the map $f$ defining
the dynamical system or one of its iterates. For example,
the already cited \cite[Lemma 4]{KeKoYo-01} (Chaos Lemma), under hypotheses for a map $f$ similar
to the ones for $\psi$ in Theorem \ref{th-ch},
establishes the existence of a compact $f$-invariant set $Q_{*},$ on which $f$ is sensitive and such that each forward itinerary on $m$ symbols
is realized by the itinerary generated by some point of $Q_{*}$ (here $m$ is the crossing number introduced in Subsection \ref{sub-sap}).
A similar result is mentioned, without proof, by Aulbach and
Kieninger in \cite{AuKi-01} and it asserts that if a continuous self-map $f$ of a compact
metric space $X$ is \textit{chaotic in the sense of
Block and Coppel}\index{chaos! Block-Coppel}, that is, if there exist an iterate
$f^k$ (with $k\ge 1$) of $f$ and a compact subset $Y\subseteq X$
positively $f^k$-invariant, such that $f^k\restriction_Y$ is
semi-conjugate to the one-sided Bernoulli shift on two symbols \footnote{We observe that, although the definition of Block-Coppel chaos is presented in \cite{AuKi-01} with respect to the shift on two symbols, in all the related results stated below, one could deal with the more general case in which $m\ge 2$ symbols are involved. Indeed, the situation is exactly the same such as for Definitions \ref{def-ch} and \ref{def-chm}, where the difference is only a formal matter.}, then $f\restriction_Z$ is Block-Coppel
chaotic, transitive and sensitive on $Z,$ where $Z$ is a suitable compact positively $f$-invariant subset of $X.$
As suggested in \cite{AuKi-01}, such
conclusions rely on \cite[Theorem 3]{AuYo-80}. Since the proof of this latter
result is only sketched in \cite{AuYo-80}, we furnish the
details in Theorem \ref{th-ay} below for the sake of completeness. Notice that the original requirement on the existence of a dense orbit in the statement of \cite[Theorem 3]{AuYo-80} has been here replaced with the one of transitivity from Definition \ref{def-dev}, since such notions coincide in every dynamical system (cf. \cite[Lemma 3]{AuYo-80}). Moreover, we have replaced the pointwise instability in \cite[Theorem 3]{AuYo-80} with the (generally) stronger property of sensitivity: indeed, as mentioned before, such concepts turn out to be equivalent in any transitive system \cite{BlGl-02}
\footnote{Actually, the sensitivity on initial conditions is equivalent to the instability of a point with dense orbit \cite{AuYo-80}, as it will emerge from the proof of Theorem \ref{th-ay}.}. Finally, we stress that the set $X_0$ in the statement of \cite[Theorem 3]{AuYo-80} is claimed to be only positively invariant. However, after a look at its proof, since by construction $X_0$ is the $\omega$-limit set of a certain point of $X,$ it turns out to be invariant by the compactness of $X$ (cf. \cite[Theorem 5.5]{Wa-82}).

\begin{theorem}[Auslander-Yorke]\label{th-ay}
Let $(X, f, d_X)$ and $(Y,g, d_Y)$ be dynamical
systems and let $\pi:X\to Y$ be a continuous and surjective map
such that $\pi\circ f=g\circ\pi.$ Assume that $g$ is sensitive and transitive on $Y.$ Then there exists a closed
$f$-invariant subset $X_0\subseteq X$
such that $\pi(X_0)=Y$ and such that $f$ is sensitive and transitive on $X_0.$
\end{theorem}
\begin{proof}
First of all we use Zorn Lemma to show the existence of
a closed positively $f$-invariant  subset $X_0\subseteq X$ such
that $\pi(X_0)=Y$ and $X_0$ is minimal with respect to these
properties. Indeed, let us consider a chain $\mathcal
C=\{X_i\}_{i\in \mathcal F},$ whose elements are closed positively
$f$-invariant subsets of $X$ and $\pi(X_i)=Y,\,\forall i\in
\mathcal F.$ Then we have to check that $X_{\infty}:=\bigcap_{i\in
\mathcal F} X_i$ is still closed, positively $f$-invariant and
$\pi(X_{\infty})=Y.$ Notice that $X_{\infty}$ is nonempty by the
strong Cantor property, valid for compact spaces (in fact,
equivalent to compactness in any topological space)\footnote{We
recall that a topological space $X$ has the strong Cantor
property if, for any family $\{C_{\alpha}\}_{\alpha\in \mathcal
I}$ of closed subsets of X with the finite intersection property
(i.e. such that, for every finite subset $\mathcal J\subseteq
\mathcal I,$ it holds that $\bigcap_{\alpha\in \mathcal J}
C_{\alpha}\ne\emptyset$), we have $\bigcap_{\alpha\in \mathcal I}
C_{\alpha}\ne\emptyset.$}. Since $X_{\infty}$ is the intersection of closed
sets, it is closed, too. The positive $f$-invariance of $X_{\infty}$ follows easily by observing that
$$f(X_{\infty})\subseteq \bigcap_{i\in \mathcal F}f(X_i)\subseteq \bigcap_{i\in \mathcal F}X_i=X_{\infty}.$$
In order to verify that $\pi(X_{\infty})=Y,$ we have to check two
inclusions. Since $\pi(X_i)=Y,\,\forall i\in \mathcal F,$ then
$\pi(X_{\infty})\subseteq\bigcap_{i\in \mathcal
F}\pi(X_i)= \bigcap_{i\in \mathcal F}Y=Y.$ For the reverse
inclusion, let us fix $y\in Y$ and consider the compact sets
$C_i:=\pi^{-1}(y)\cap X_i,$ $i\in\mathcal F.$ Notice that they are
nonempty as $\pi(X_i)=Y,\,\forall i\in \mathcal F.$ Moreover the
family $\{C_{i}\}_{i\in \mathcal F}$ has the finite intersection
property: indeed, recalling that $\mathcal C$ is a chain, given
$C_{i_1},\dots,C_{i_k},$ with $i_1,\dots,i_k \in\mathcal F,$ it holds that
$$\bigcap_{j=1}^k C_{i_j}=C_{i^*}, \quad \mbox{for some} \quad   i^*=i_1,\dots,i_k\,$$
and thus such intersection is nonempty. By the strong Cantor property, there exists
$$ x\in \bigcap_{i\in \mathcal F}C_i=\bigcap_{i\in \mathcal F}\left(\pi^{-1}(y)\cap X_i\right)=\pi^{-1}(y)\cap
\left(\bigcap_{i\in \mathcal F}X_i\right)=\pi^{-1}(y)\cap X_{\infty},$$ that
is, there exists $ x\in X_{\infty}$  such that $\pi( x)=y.$ Hence,
$Y\subseteq \pi(X_{\infty}).$
The existence of the minimal set $X_0$ is then ensured by Zorn Lemma.\\
Notice that, since $\pi(X_0)=Y$ and $f(X_0)\subseteq X_0,$ the
relation $\pi\restriction_{X_0}\circ f\restriction_{X_0}=g\circ\pi\restriction_{X_0}$ holds. Indeed, for any $x_0\in X_0\subseteq X,$ we have $\pi(f\restriction_{X_0}(x_0))=\pi(f(x_0))=g(\pi(x_0))=g(\pi\restriction_{X_0}(x_0)).$
On the other hand, since $f(x_0)\in X_0,$ then $\pi(f\restriction_{X_0}(x_0))=\pi\restriction_{X_0}(f\restriction_{X_0}(x_0)),$ from which $\pi\restriction_{X_0}(f\restriction_{X_0}(x_0))=g(\pi\restriction_{X_0}(x_0)).$ The validity of $\pi\restriction_{X_0}\circ f\restriction_{X_0}=g\circ\pi\restriction_{X_0}$ is thus checked.\\
By the transitivity of the onto map $g$ on $Y,$ there exists
$y^*\in Y$ with dense $g$-orbit, that is,
$\overline{\gamma(y^*)}=Y,$ for
$\gamma(y^*):=\{g^n(y^*):n\in\mathbb N\}.$ Let $x_0^*\in X_0$ be
such that $\pi(x_0^*)=y^*.$ We will show that
$$\pi(\omega(x_0^*))=\omega(y^*)=Y\,,$$
where, for a dynamical system $(Z,l)$ with $z\in Z,$ we have
$$\omega(z):=\{x\in Z: \exists\, n_j\nearrow\infty \mbox{ with } l^{n_j}(z)\to x\}.$$
According to \cite[Theorem 5.5]{Wa-82}, it holds that $\omega(z),$ called \textit{$\omega$-limit set of} \index{$\omega$-limit set} $z,$ is closed nonempty and invariant by the compactness of $Z.$
Let us start with the verification of
$\omega(y^*)=Y.$ By the surjectivity of $g$ on $Y,$ there exists
$\tilde y\in Y$ such that $g(\tilde y)=y^*.$ Since the $g$-orbit
of $y^*$ is dense in $Y$ and recalling that
$\overline{\gamma(y^*)}=\gamma(y^*)\cup \omega(y^*),$ there are
two possibilities for $\tilde y,$ that is, $\tilde y\in
\gamma(y^*)$ or $\tilde y\in \omega(y^*).$ In the former case, we
find $\tilde y=g^k(y^*),$ for some $k\ge 0.$ Since $g(\tilde y)=y^*,$ this means that $g^{k+1}(y^*)=y^*$ and thus $y^*$ is a periodic
point, from which
$Y=\overline{\gamma(y^*)}=\gamma(y^*)=\omega(y^*)$ is a finite
set. On the other hand, if $\tilde y\in\omega(y^*),$ it follows
that $g(\tilde y)=y^*\in \omega(y^*),$ by the invariance of $\omega(y^*).$ Recalling that it is also closed, then $Y=\overline{\gamma(y^*)}\subseteq
\omega(y^*).$ Therefore we find $\omega(y^*)=Y$ also in the
latter case.\\
As regards the equality
$\pi(\omega(x_0^*))=\omega(y^*),$ at first we observe that, since
$\pi\restriction_{X_0}\circ f\restriction_{X_0}=g\circ\pi\restriction_{X_0},$ then it follows that
$\pi\restriction_{X_0}\circ (f\restriction_{X_0})^m=g^m\circ\pi\restriction_{X_0},\\\forall m\ge
2.$ Let us start with the inclusion
$\pi(\omega(x_0^*))\subseteq\omega(y^*).$ If $x^*\in
\omega(x_0^*),$ then there exists a sequence $n_j\nearrow \infty$ such
that $f^{n_j}(x_0^*)\to x^*.$ By the continuity of the map $\pi,$
it holds that $\pi(f^{n_j}(x_0^*))\to \pi(x^*)$ for $j\to \infty,$ but
$\pi(f^{n_j}(x_0^*))=g^{n_j}(\pi(x_0^*))=g^{n_j}(y^*),$ i.e. $\lim_{j\to\infty}g^{n_j}(y^*)=\pi(x^*)$
and thus $\pi(x^*)\in\omega(y^*).$ For the reverse inclusion, we have to check that any $\bar y\in\omega (y^*)$
can be written as $\bar y=\pi(\bar x),$ with $\bar x\in \omega(x_0^*).$ Indeed, if $\bar y\in\omega (y^*),$
then there exists a sequence $m_k\nearrow \infty$ such that $g^{m_k}(y^*)\to \bar y.$ Therefore $\pi(f^{m_k}(x_0^*))=g^{m_k}(\pi(x_0^*))=g^{m_k}(y^*)\to \bar y$ as $k\to \infty.$
Since $(f^{m_k}(x_0^*))_{k\in\mathbb N}$ is a sequence contained in the compact set $X_0,$
there exist a subsequence $m_{k_j}\nearrow\infty$ of $(m_k)_{k\in\mathbb N}$ and $\bar x\in X_0$
such that $f^{m_{k_j}}(x_0^*)\to \bar x$ for $j\to\infty.$ This means that $\bar x\in \omega(x_0^*).$ Notice that also for
the subsequence $(m_{k_j})_{j\in\mathbb N}$ it holds that $\pi(f^{m_{k_j}}(x_0^*))=g^{m_{k_j}}(\pi(x_0^*))=g^{m_{k_j}}(y^*)\to \bar y$ as $j\to\infty.$ Hence, passing to the limit,
we find $\pi(\bar x)=\lim_{j\to \infty}g^{m_{k_j}}(\pi(x_0^*))=\bar y,$ from which one obtains $\bar y=\pi(\bar x),$
with $\bar x\in \omega(x_0^*).$ The equality $\pi(\omega(x_0^*))=\omega(y^*)$ is thus established.\\
Then, since $\pi(\omega(x_0^*))=Y$ and recalling that $\omega(x_0^*)$ is a closed
$f$-invariant subset of $X_0,$ by the minimality of $X_0$ it follows that $\omega(x_0^*)=X_0.$ By the invariance of $\omega(x_0^*),$ then $f(\omega(x_0^*))=\omega(x_0^*)$ and so $X_0$ is invariant, too. \\
Moreover, from $\overline{\gamma(x_0^*)}=\gamma(x_0^*)\cup \omega(x_0^*),$ we find that the $f$-orbit of $x_0^*$ is dense in $X_0.$ In fact, by the minimality of $X_0,$ the $f$-orbit of any point of $X_0$ is dense in $X_0$ \footnote{A subset $M$ of the dynamical system $(Z,l)$ is called \textit{minimal} if it is closed nonempty positively $l$-invariant and contains no proper closed nonempty positively $l$-invariant subsets. An equivalent definition is that $M$ is minimal if each of its points has dense $l$-orbit in $M.$}. \\
Let us check now that $f$ is sensitive on $X_0.$ By the sensitivity
of $g$ on $Y,$ there exists $\varepsilon>0$ such that for every
$y\in Y$ with dense $g$-orbit (actually, for any point of $Y$),
there exist a sequence $(y_j)_{j\in\mathbb N}$ of points of $Y$
and a sequence $(n_j)_{j\in\mathbb N}$ of positive integers such
that $y_j\to y,$ but
$d_Y(g^{n_j}(y),g^{n_j}(y_j))>\varepsilon,$ $\forall j\in\mathbb
N.$ Let $(x_j)_{j\in\mathbb N}$ be a sequence in $X_0$ such that $\pi(x_j)=y_j,\,\forall
j\in\mathbb N,$ and let $(x_{j_k})_{k\in\mathbb N}$ be a converging
subsequence of $(x_j)_{j\in\mathbb N},$ with $x_{j_k}\to \overline{x_0},$ for
some $\overline{x_0}\in X_0.$ By the continuity of $\pi$ it follows that
$\pi(\overline{x_0})=y.$ We want to show that $\overline{x_0}$ is unstable, i.e that there exists $\bar\delta>0$
such that
$d_X(f^{n_{j_k}}(\overline{x_0}),f^{n_{j_k}}(x_{j_k}))>\bar\delta,$ $\forall
k\in\mathbb N,$ where $(n_{j_k})_{k\in\mathbb N}$ is the subsequence of $(n_j)_{j\in\mathbb N}$ corresponding to $(y_{j_k})_{k\in\mathbb N}.$ Indeed, if by contradiction there exists a subsequence of $(n_{j_k}),$ that by notational convenience we still denote by $(n_{j_k}),$ such that
$d_X(f^{n_{j_k}}(\overline{x_0}),f^{n_{j_k}}(x_{j_k}))\to 0$ for
$k\to\infty,$ by the continuity of $\pi,$ we would find
$0<\varepsilon<d_Y(g^{n_{j_k}}(y),g^{n_{j_k}}(y_{j_k}))=d_Y(g^{n_{j_k}}(\pi(\overline{x_0})),
g^{n_{j_k}}(\pi(x_{j_k})))=$ $d_Y(\pi(f^{n_{j_k}}(\overline{x_0})),\pi(f^{n_{j_k}}(x_{j_k})))\to
0.$ Therefore
$d_X(f^{n_{j_k}}(\overline{x_0}),f^{n_{j_k}}(x_{j_k}))>
\bar\delta,$ $\forall k\in\mathbb N,$ for some $\bar\delta>0,$ and $\overline{x_0}\in X_0$ is an
unstable point. Moreover, since $X_0$ is minimal, the $f$-orbit of $\overline{x_0}$ is dense in $X_0.$ Let us prove that these two facts about $\overline{x_0}$ are sufficient to conclude that there exists a $\tilde\delta>0$ such that any point of $X_0$ is $\tilde\delta$-unstable.\\
By the continuity of
$f,$ since $\overline{x_0}$ is $\bar\delta$-unstable, then also $f(\overline{x_0})$ is and
hence the whole $f$-orbit $\gamma(\overline{x_0})$ is $\bar\delta$-unstable.
Calling $\mathcal U_{\bar\delta}\subseteq X_0$ the set of
$\bar\delta$-unstable points of $X_0,$ we show that
$\overline{\mathcal U_{\bar\delta}}\subseteq \mathcal U_{{\bar\delta}/2}.$
If this is true, since $X_0=\overline{\mathcal U_{\bar\delta}},$
the thesis is achieved for $\tilde\delta:={\bar\delta}/2.$ So let
$x\in \overline{\mathcal U_{\bar\delta}}.$ Then there exists a
sequence $(x_i)_{i\in\mathbb N}$ in ${\mathcal U_{\bar\delta}}$ with
$x_i\to x.$ For any such $x_i$ and for every open set $X_0\supseteq
O_i\ni x_i,$ there exist $\widetilde{x_i}\in O_i$ and a positive
integer $m_i$  with $d_X(f^{m_i}(x_i),
f^{m_i}(\widetilde{x_i}))\ge\bar\delta.$ In particular we can choose
$O_i:=B(x_i,1/i)\cap X_0,$ so that
$d_X(x_i,\widetilde{x_i})<1/i\to 0,$ as $i\to\infty.$ Now there
are two alternatives: if $d_X(f^{m_i}(x_i),
f^{m_i}(x))\ge{\bar\delta}/2,$ then $x$ is ${\bar\delta}/2$-unstable and we
are done. Otherwise it holds that $d_X(f^{m_i}(x),
f^{m_i}(\widetilde{x_i}))\ge{\bar\delta}/2.$ Indeed, if this would fail,
then $d_X(f^{m_i}(x_i), f^{m_i}(\widetilde{x_i}))\le
d_X(f^{m_i}(x_i), f^{m_i}(x))+d_X(f^{m_i}(x),
f^{m_i}(\widetilde{x_i}))<\bar\delta,$ a contradiction. Since $\widetilde{x_i}\to x,$ it follows also in this latter case that
$x\in \mathcal U_{{\bar\delta}/2}.$ The proof is complete.
\end{proof}

\noindent
By the similarity between our point of view and the frameworks
in \cite{AuKi-01,KeKoYo-01}, it is not difficult to
show that both the above quoted results on chaotic invariant
sets obtained therein still hold in the setting of Definition
\ref{def-chm}. Therefore the sensitivity is for us
ensured at least on some subset of the domain. More precisely, as
discussed in Section \ref{sec-de}, since our
notion of chaos in Definition \ref{def-chm} is stricter than the one considered in \cite{KeKoYo-01,KeYo-01}, we can prove the existence of
the chaotic invariant set $Q_{*}$ mentioned above, exactly as in
\cite[Lemma 4]{KeKoYo-01}. On the other hand, with reference to
\cite{AuKi-01} and \cite{AuYo-80}, we notice that any map chaotic according to Definition
\ref{def-ch} is also chaotic in the sense of Block-Coppel, thanks
to Theorem \ref{th-cons}, conclusion $(ii).$
Thus, following the suggestion in \cite{AuKi-01},
we have used Theorem \ref{th-ay} to get the existence of the
set $\Lambda$ in Theorem \ref{th-cons}, conclusion $(v).$
More precisely, we point out that in Theorem \ref{th-cons}, conclusion $(ii),$
a property stronger than the notion of chaos in the sense of Block-Coppel is obtained. Indeed,
the semi-conjugacy with the Bernoulli shift is established for the map $\psi$ itself and
not for one of its iterates. The same remark applies to Theorem \ref{th-cons}, conclusion $(v),$ where again a sharper feature than the Block-Coppel chaos is deduced.\\
Pursuing further this discussion on the Block-Coppel chaos, we notice that every system $(X, f)$ Block-Coppel chaotic has positive topological entropy. Indeed, by the postulated semi-conjugacy between an iterate $f^k$ (with $k\ge 1$) of the map $f,$
restricted to a suitable positively invariant subset of the domain, and the one-sided
Bernoulli shift on two symbols, by \eqref{eq-it} and \eqref{eq-hch} it follows that
$k h_{\rm top}(f)=h_{\rm top}(f^k)\ge \log(2),$ from which $h_{\rm top}(f)\ge \log(2)/k>0.$ Thus, by the previously quoted
\cite[Theorem 2.3]{BlGl-02}, any such system is also Li-Yorke chaotic. However, we observe that the Block-Coppel chaos
is strictly weaker than chaos in the sense of coin-tossing and, a fortiori,
also than the concept in Definition \ref{def-ch}: indeed, according to
\cite[Remark 3.2]{AuKi-01}, there exist systems $(X,f)$ such that
$f^2$ restricted to some $f$-invariant subset of $X$ is
semi-conjugate to the one-sided Bernoulli shift on two symbols, while such
property does not hold for $f.$

\smallskip
\noindent
Returning back to the analysis on the relationship between the topological entropy and the sensitivity on initial conditions,
we recall that, except for the special case of continuous self-mappings of compact intervals,
in general the sensitivity does not imply a positive entropy \cite{Ko-04}. Actually, even more can be said.
Indeed, if instead of the sensitivity alone, we take into account the definition of Devaney chaos in its completeness,
then it is possible to prove that the Devaney chaoticity and the positivity of topological entropy are independent,
as none of the two implies the other \cite{BaSn-03, KwMi-05}. In particular this means that, in generic metric spaces,
Devaney chaos does not imply chaos in the sense of Definition \ref{def-chm} (because, otherwise, the topological entropy would be positive in
any Devaney chaotic system). Vice versa it is not clear if our notion of chaos implies the one by Devaney.
Indeed, on the one hand, in \cite{AuKi-01} it is presented a dynamical system Block-Coppel chaotic,
but not chaotic according to Devaney, since it has no periodic points. On the other hand, we have already noticed
that our notion of chaos is strictly stronger than the one in the Block-Coppel sense.\\
When confining ourselves to the one-dimensional case, most of the definitions of chaos are known to be equivalent,
while it is in the higher dimensional setting that the relationship among them becomes more involved. Indeed, in \cite{Li-93}
it was proved that for a continuous self-mapping $f$ of a compact interval to have positive topological entropy
is equivalent to be chaotic in the sense of Devaney on some closed (positively) invariant subset of the domain. The positivity of the topological
entropy is also equivalent to the fact that some iterate of $f$ is turbulent (or even strictly turbulent, cf. Section \ref{sec-sd} for the corresponding definitions) or to the chaoticity in the Block-Coppel sense \cite{AuKi-01}. Such equivalence, however,
does not extend to chaos in the sense of Li-Yorke: indeed, as already pointed out, there exist interval
maps Li-Yorke chaotic, but with zero topological entropy \cite{Sm-86}.\\
When moving to more general frameworks,
the previous discussion should suggest that several links among the various notions of chaos get lost.
Nonetheless something can still be said. In addition to
the facts already expounded (e.g., chaos according to Definition \ref{def-ch} $\Rightarrow$ chaos in the sense of coin-tossing $\Rightarrow$ Block-Coppel chaos $\Rightarrow
h_{\rm top}>0\Rightarrow$ Li-Yorke chaos), we mention that Devaney
chaos implies Li-Yorke chaos in any compact metric space:
actually, in order to prove
this implication, the hypothesis on the density of periodic points
in Definition \ref{def-dev} could be replaced with the weaker
condition that at least one periodic point does exist
\cite{HuYe-02}. On the other hand this weaker requirement
is necessary, because transitivity and sensitivity alone are not
sufficient to imply  Li-Yorke chaos and vice versa,
as shown in \cite{BlGl-02}. We recall that a dynamical system that
is both sensitive and transitive is sometimes named
\textit{Auslander-Yorke chaotic}\index{chaos! Auslander-Yorke} \cite{AuYo-80,BlGl-02}. Therefore
we can rephrase the previous sentence by saying that the concepts
of Li-Yorke chaos and Auslander-Yorke chaos are independent.

\medskip

\noindent
Of course, the above treatment is just meant to give an idea of the intricate network of connections among some of the most well-known definitions of chaos. A useful and almost complete survey on the existing relationships can be found in \cite{Mo-04}.

\section{Linked Twist Maps}\label{sec-ltm}

In this last section we present a different geometrical context where it is possible to apply our method of ``stretching along the paths'' from Section \ref{sec-sap} and the corresponding results on chaotic dynamics from Section \ref{sec-de}. More precisely we will be concerned with the study of the so-called \textit{Linked Twist Maps}\index{Linked Twist Maps} (for short, LTMs). In such framework, instead of considering a single map that expands the arcs along a domain homeomorphic to a rectangle, one deals with a geometric configuration characterized by
the alternation of two planar homeomorphisms (or diffeomorphisms) which twist two
circular annuli (or two families of them \cite{Pr-86}) intersecting in two disjoint generalized rectangles $\mathcal A$ and $\mathcal B.$ Each annulus is turned onto itself by a homeomorphism
which leaves the boundaries of the annulus invariant.
Both the maps act in their domain so that a twist effect is produced.
This happens, for instance, when the angular speed is monotone with respect
to the radius. Considering the composition of the two movements
in the common regions, we obtain a resulting function which is what we call a ``linked twist map'' (see \cite{Wi-99, WiOt-04} for a detailed description of the geometry of the domain
of a LTM). Such kind of maps furnish a geometrical setting for the existence of Smale horseshoes: in fact, under certain conditions,
it is possible to prove the presence of a Smale horseshoe inside ${\mathcal A}$
and $\mathcal B$ \cite{De-78}.
Usual assumptions on the twist mappings require, among others, their smoothness,
monotonicity of the angular speed with respect to the radial coordinate and preservation of the Lebesgue measure. On the other hand, since our approach is purely topological, we just need a twist condition on the boundary (cf. Example \ref{ex-ltm1}).\\
In the past decades a growing interest has concerned LTMs. In the 80s they were studied from a theoretical point of view by
Devaney \cite{De-78}, Burton and Easton \cite{BuEa-80} and Przytycki \cite{Pr-83, Pr-86}
(just to cite a few contributions in this direction), proving some mathematical properties
like ergodicity, hyperbolicity and conjugacy to the Bernoulli shift.
However, as observed in \cite{De-78},
such maps naturally appear in various different applicative contexts, like for instance in mathematical models for particle motions in a magnetic field,
as well as in differential geometry, in the study of diffeomorphisms of surfaces. Special configurations
related to LTMs can also be found in the restricted three-body problem
\cite[pp.90--94]{Mo-73}.
In more recent years, thanks to the work of Ottino, Sturman and Wiggins, significant applications of LTMs have been performed in the area of
fluid mixing \cite{St-06, StOtWi-06, Wi-99, WiOt-04}.

\bigskip

\noindent
Our purpose is to adapt the general results from Sections \ref{sec-sap}--\ref{sec-de} to a geometrical framework which is connected to and generalizes (in a direction explained in Example \ref{ex-ltm2}) the case of the LTMs.
The corresponding  Theorems \ref{th-ltma} and \ref{th-ltmb}
below will be applied to some nonlinear ODEs with periodic coefficients in Chapter \ref{ch-ode}.

\begin{theorem}\label{th-ltma}
Let $X$ be a metric space and
assume that $\varphi: X\supseteq D_{\varphi}\to X$ and
$\psi: X\supseteq D_{\psi}\to X$ are continuous maps defined on the sets $D_{\varphi}$ and $D_{\psi},$ respectively. Let also
${\widetilde{\mathcal A}} := ({\mathcal A},{\mathcal A}^-)$ and
${\widetilde{\mathcal B}} := ({\mathcal B},{\mathcal B}^-)$ be oriented rectangles of $X.$
Suppose that the following conditions are satisfied:
\begin{itemize}
\item[$\quad (H_{\varphi})\;\;$] There are $m\ge 2$ pairwise disjoint compact sets
${\mathcal H}_0\,,\dots, {\mathcal H}_{m-1}\,\subseteq {\mathcal A}\cap D_{\varphi}$
such that
$\displaystyle{({\mathcal H}_i,\varphi): {\widetilde{\mathcal A}}
\stretchx\, {\widetilde{\mathcal B}}},$ for $i=0,\dots,m-1\,;$
\\
\item[$(H_{\psi})\;\;$] ${\mathcal B}\subseteq D_{\psi}$ and
$\psi\,: {\widetilde{\mathcal B}} \stretchx {\widetilde{\mathcal A}}\,.$
\end{itemize}
Then the map $\phi:=\psi\circ \varphi$ induces chaotic dynamics on $m$ symbols
in the set $\mathcal H\cap{\varphi}^{-1}(\mathcal B),$ where
${\mathcal H}:= \bigcup_{i=0}^{m-1} {\mathcal H}_i,$ and thus satisfies properties $(i)$-$(v)$ from Theorem \ref{th-cons} (or the stronger properties $(i)$-$(v)$ from Theorem \ref{th-inj}, if $\phi$ is also injective on $\mathcal H\cap{\varphi}^{-1}(\mathcal B)$ \footnote{This is for instance the case in Chapter \ref{ch-ode}, since there we deal with the Poincar\'e map, that is a homeomorphism.}).
\end{theorem}
\begin{proof}
We show that
\begin{equation}\label{eq-st}
({\mathcal H}_i\cap \varphi^{-1}(\mathcal B),\phi): {\widetilde{\mathcal A}}
\stretchx {\widetilde{\mathcal A}}\,,\;\;\forall\, i=0,\dots,m-1,
\end{equation}
from which the thesis about the chaotic dynamics is an immediate consequence of Theorem \ref{th-ch}.\\
To check condition \eqref{eq-st}, let us consider a path
$\gamma: [0,1]\to {\mathcal A}$ such that $\gamma(0)\in {\mathcal A}^-_{\ell}$
and $\gamma(1)\in {\mathcal A}^-_{r}$ (or with $\gamma(0)\in {\mathcal A}^-_{r}$
and $\gamma(1)\in {\mathcal A}^-_{\ell}$) and let us fix $i\in\{0,\dots,m-1\}.$
By $(H_{\varphi}),$ there exists a compact interval $[t',t'']\subseteq [0,1]$ such that
$\gamma(t)\in {\mathcal H}_i$ and $\varphi(\gamma(t))\in {\mathcal B},$ for every $t\in [t',t''],$ with $\varphi(\gamma(t'))$ and $\varphi(\gamma(t''))$ belonging to different
components of ${\mathcal B}^-.$
Define now
$$\omega: [t',t'']\to {\mathcal B},\quad \omega(t):= \varphi(\gamma(t)).$$
By $(H_{\psi})$ there is a compact interval $[s',s'']\subseteq [t',t'']$ such that
$\psi(\omega(t))\in {\mathcal A},$ for every $t\in [s',s''],$ with
$\psi(\omega(s'))$ and $\psi(\omega(s''))$ belonging to different components of ${\mathcal A}^-.$
\\
Rewriting all in terms of $\gamma,$ we have thus proved that
$$\gamma(t)\in {\mathcal H}_i\cap \varphi^{-1}(\mathcal B)\quad\mbox{and }\; \phi(\gamma(t))\in {\mathcal A},\quad\forall\, t\in [s',s''],$$
with
$\phi(\gamma(s'))$ and $\phi(\gamma(s''))$ belonging to different components of ${\mathcal A}^-.$ The continuity of the composite mapping $\phi=\psi\circ\varphi$ on ${\mathcal H}_i\cap \varphi^{-1}(\mathcal B)$ follows from the continuity of $\varphi$ on $D_{\varphi}\supseteq {\mathcal H}_i$ and from the continuity of $\psi$ on $D_{\psi}\supseteq \mathcal B.$
By the arbitrariness of the path $\gamma$ and of $i\in \{0,\dots,m-1\},$ the verification of \eqref{eq-st} is complete.
\end{proof}

\noindent
We remark that in condition $(H_{\varphi})$ we could consider the case in which there exists a set $\mathcal D\subseteq \mathcal A\cap D_{\varphi}$ with $\mathcal H=\bigcup_{i=0}^{m-1} {\mathcal H}_i\subseteq \mathcal D,$ in analogy with Theorem \ref{th-ch}. From the proof of Theorem \ref{th-ltma} it is also clear that the continuity of $\varphi$ on $D_{\varphi}$ and of $\psi$ on $D_{\psi}$ could be weakened requiring that $\varphi$ is continuous only on $\mathcal H$ and $\psi$ on $\mathcal B.$ For simplicity's sake and in view of the applications in Section \ref{sec-sb}, we have preferred to present the easier framework.\\
Moreover, as regards condition $(H_{\psi}),$ we could assume that the stretching condition therein holds with respect to a compact set $\mathcal K\subseteq {\mathcal B}\subseteq D_{\psi},$ that is, $(\mathcal K,\psi): {\widetilde{\mathcal B}} \stretchx {\widetilde{\mathcal A}}\,.$ Then, the chaotic dynamics would be localized in the set
$${\mathcal H}^*:=
\bigcup_{i=0}^{m-1} \left({\mathcal H}_i\cap \varphi^{-1}({\mathcal K})\right).$$
Such case is analyzed in the more general Theorem \ref{th-ltmb}.

\medskip

\noindent
We present now two simple examples for the application of Theorem \ref{th-ltma},
which are meant to show how this result is well-fit for studying LTMs.
Example \ref{ex-ltm1} is fairly classical as it concerns two overlapping annuli
subject to twist rotations, while Example \ref{ex-ltm2} describes the composition of a map which twists an annulus
with a longitudinal motion along a strip.
A similar geometric setting was already considered by Kennedy and Yorke in \cite{KeYo-97} in the framework of the theory of fluid mixing, studying planar functions obtained as composition of a squeezing map and a stirring rotation.\\
We stress that such examples are only of ``pedagogical'' nature and, in fact, the chaotic-like dynamics that we obtain could be also proved using different approaches already developed in various papers (like, for instance, \cite{MiMr-95a, Sr-00, Zg-96, ZgGi-04}).

\begin{example}\label{ex-ltm1}
{\rm{A classical kind of LTM is represented by the composition of two planar maps
$\varphi$ and $\psi$ which act as
twist rotations around two given points.
For instance, a possible choice is that of considering two continuous functions
expressed by means of complex variables as
\begin{equation}\label{eq-cphi}
\varphi(z):= - r + (z + r)\, e^{\imath(c_1 + d_1 |z + r|)}
\end{equation}
and
\begin{equation}\label{eq-cpsi}
\psi(z):= r + (z - r)\, e^{\imath(c_2 + d_2 |z - r|)},
\end{equation}
where $\imath$ is the imaginary unit, while $r > 0,$ $c_j$ $(j=1,2)$ and $d_j\ne 0$
are real coefficients. Such maps twist around
the centers $(-r,0)$
and $(r,0),$ respectively. We denote by $p_1$ and $p_2$
the inner and the outer radii for the annulus around $(-r,0)$
and by $q_1$ and $q_2$
the inner and the outer radii for the annulus around $(r,0).$\\
For a suitable choice of $r,$ of the radii $p_1 < p_2$ and
$q_1 < q_2\,,$ as well as of the
parameters determining $\varphi$ and $\psi\,,$ it is possible to apply
Theorem \ref{th-ltma} in order to obtain chaotic dynamics
(see Figures \ref{fig:04}--\ref{fig:05}).
}}
\end{example}

\begin{figure}[htbp]
\centering
\includegraphics[scale=0.22]{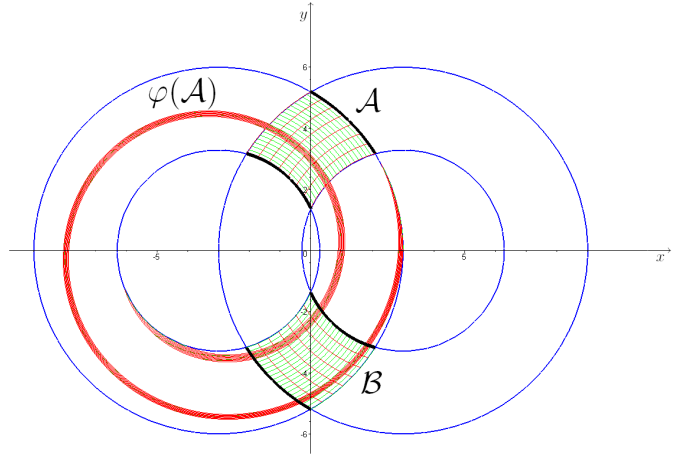}
\caption{\footnotesize {A pictorial comment to
condition $(H_{\varphi})$ of Theorem \ref{th-ltma},
with reference to Example \ref{ex-ltm1}. In order to determine the two annuli,
we have set $r=3,$
$p_1=q_1= 3.3$ and $p_2= q_2= 6$
(this choice is just to simplify the explanation,
since no special symmetry is needed).
For the map $\varphi$ in \eqref{eq-cphi} we have taken $c_1=-1.5$ and $d_1=3.3.$ We choose as generalized rectangle ${\mathcal A}$ in
Theorem \ref{th-ltma} the upper intersection of the two annuli
and select as components of ${\mathcal A}^-$ the intersections of ${\mathcal A}$ with
the inner and outer boundaries of the annulus at the left-hand side.
The set ${\mathcal B}$ is defined as the
the lower intersection of the two annuli and the two components
of ${\mathcal B}^-$ are the intersections of ${\mathcal B}$ with
the inner and outer boundaries of the annulus at the right-hand side.
The sets ${\mathcal A}^-$ and ${\mathcal B}^-$ are drawn with thicker lines.
The narrow strip spiralling inside the left annulus
is the image of ${\mathcal A}$ under $\varphi\,.$ Clearly, any path
in ${\mathcal A}$ joining
the two components of ${\mathcal A}^-$ is transformed by $\varphi$ onto a path
crossing ${\mathcal B}$ twice in the correct manner (that is, from a side of ${\mathcal B}^-$ to the other side).
}}
\label{fig:04}
\end{figure}

\begin{figure}[htbp]
\centering
\includegraphics[scale=0.22]{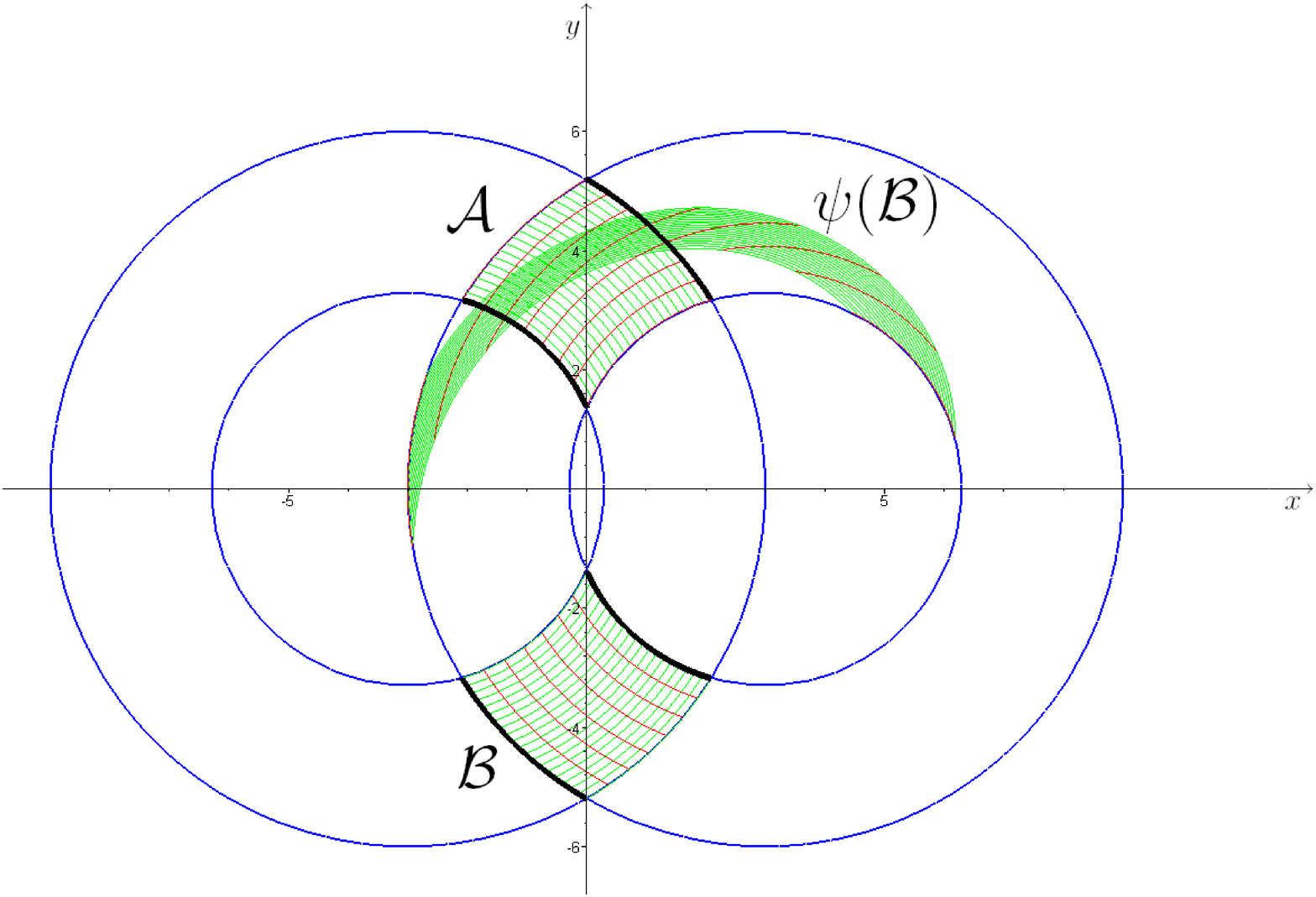}
\caption{\footnotesize {A pictorial comment to
condition $(H_{\psi})$ of Theorem \ref{th-ltma},
with reference to Example \ref{ex-ltm1}.
For $r, p_1, p_2, q_1, q_2$ as
in Figure \ref{fig:04}, we have also fixed the parameters of
$\psi$ in \eqref{eq-cpsi} by taking $c_2=0$ and $d_2=0.9.$ It is evident that
the image of ${\mathcal B}$ under
$\psi$ crosses ${\mathcal A}$ once in the
correct way.
}}
\label{fig:04bis}
\end{figure}

\begin{figure}[htbp]
\centering
\includegraphics[scale=0.3]{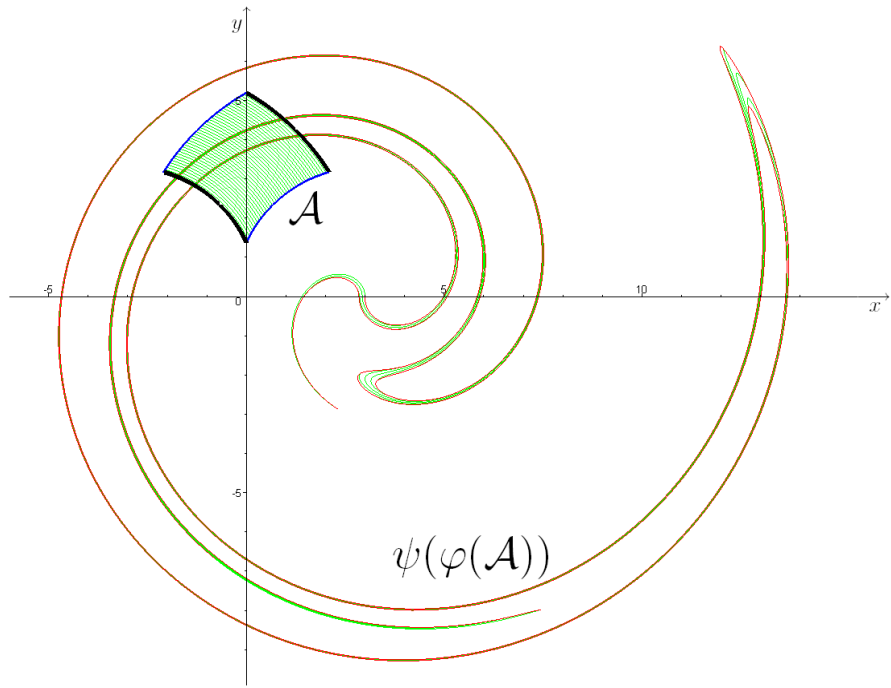}
\caption{\footnotesize {For $r, p_1, p_2, q_1, q_2,\varphi$ as in Figure \ref{fig:04} and $\psi$ as
in Figure \ref{fig:04bis}, we show
the image of ${\mathcal A}$ under the composite map
$\psi\circ \varphi\,.$ It is clear that $\psi(\varphi({\mathcal A}))$
crosses ${\mathcal A}$ twice in the
correct way.
}}
\label{fig:05}
\end{figure}

\clearpage

\begin{example}\label{ex-ltm2}
{\rm{
A nonstandard LTM is represented by the composition of two planar maps
$\varphi$ and $\psi,$ where $\varphi$ is a twist
rotation around a given point and $\psi$ produces a longitudinal motion along a strip
with different velocities on the upper and lower components of the boundary of the strip.\\
First we introduce, for any pair
of real numbers $a < b,$ the
real valued function
$$Pr_{[a,b]}(t):= \frac{1}{b-a}\, \min\{b - a,\max\{0,t - a\}\},$$
and then we set
$$f(t):= c_1 + d_1 Pr_{[p_1,p_2]}(t), \quad g(t):= c_2 + d_2 Pr_{[q_1,q_2]}(t),$$
where $c_j$ ($j=1,2$) and $d_j\ne 0$
are real coefficients, while
$$0 < p_1 < p_2\quad \mbox{and } - p_1 < q_1 < q_2 < p_1$$
are real parameters which determine the inner and the outer radii of the circular
annulus
$$C[p_1,p_2] := \{(x,y): \, p_1^2 \leq x^2 + y^2 \leq p_2^2\}$$
and the position of the strip
$$S[q_1,q_2] :=  \{(x,y): \, q_1 \leq y \leq q_2\}.$$
Finally we define two continuous maps
expressed by means of complex variables as
\begin{equation}\label{eq-cphi2}
\varphi(z):= z\, e^{\imath f(|z|)}
\end{equation}
(which twists around the origin) and, for $\Im(z)$ the imaginary part of $z,$
\begin{equation}\label{eq-cpsi2}
\psi(z):= z + g(\Im(z))
\end{equation}
(which shifts the points along the horizontal lines).
\\
For a suitable choice of the
parameters determining $\varphi$ and $\psi\,,$ it is possible to apply
Theorem \ref{th-ltma} in order to obtain chaotic dynamics
(see Figures \ref{fig:06}--\ref{fig:08}).
}}
\end{example}

\begin{figure}[htbp]
\centering
\includegraphics[scale=0.25]{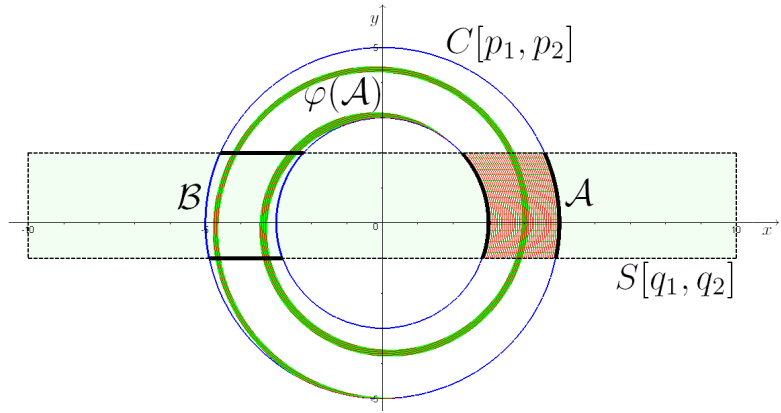}
\caption{\footnotesize {
A pictorial comment to
condition $(H_{\varphi})$ of Theorem \ref{th-ltma},
with reference to Example \ref{ex-ltm2}.
We have set $p_1= 3,$
$p_2= 5,$ $q_1= -1$ and $q_2= 2$ in order to determine the circular
annulus $C[p_1,p_2]$ centered at the origin and the strip $S[q_1,q_2]$
which goes across it.
For the map $\varphi$ in \eqref{eq-cphi2} we have taken
$c_1= 0.4\,\pi$ and $d_1= 3\,\pi.$
We choose as set ${\mathcal A}$ in
Theorem \ref{th-ltma} the right-hand side intersection of the annulus with
the strip
and select as components of ${\mathcal A}^-$
the intersections of ${\mathcal A}$ with
the inner and outer boundaries of the annulus.
The set ${\mathcal B}$ is defined as the
left-hand side intersection of the annulus with the strip
and the two components
of ${\mathcal B}^-$ are the intersections of ${\mathcal B}$ with
the lower and upper sides of the strip. The sets ${\mathcal A}^-$ and ${\mathcal B}^-$ are drawn with thicker lines.
The narrow band spiralling inside the annulus
is the image of ${\mathcal A}$ under $\varphi\,.$ Clearly, any path
in ${\mathcal A}$ joining
the two components of ${\mathcal A}^-$ is transformed by $\varphi$ onto a path
crossing ${\mathcal B}$ twice in the correct manner.
}}
\label{fig:06}
\end{figure}

\begin{figure}[htbp]
\centering
\includegraphics[scale=0.3]{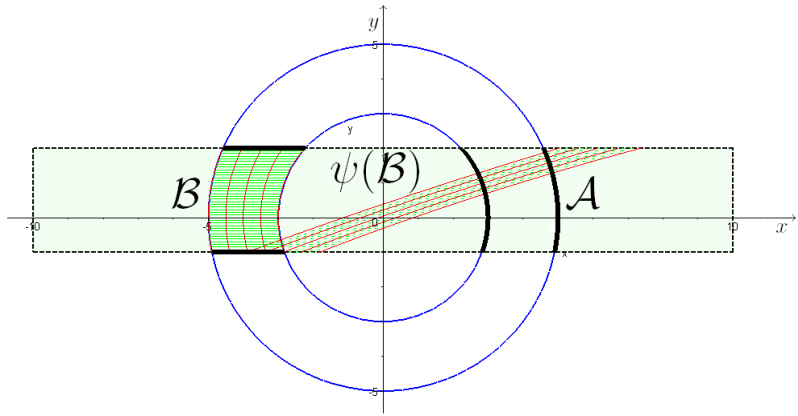}
\caption{\footnotesize {A pictorial comment to
condition $(H_{\psi})$ of Theorem \ref{th-ltma},
with reference to Example \ref{ex-ltm2}.
The sets $C[p_1,p_2]$ and $S[q_1,q_2]$ and
the oriented rectangles
${\widetilde{\mathcal A}}$ and
${\widetilde{\mathcal B}}$
are as in Figure \ref{fig:06}.
For the map $\psi$ in \eqref{eq-cpsi2}, we have chosen
$c_2=1$ and $d_2= 8.6.$
The parallelogram-shaped narrow figure inside the strip $S[q_1,q_2]$
is the image of ${\mathcal B}$ under $\psi\,.$ Any path
in ${\mathcal B}$ joining
the two components of ${\mathcal B}^-$ is transformed by $\psi$ onto a path
crossing ${\mathcal A}$ once in the right way.
}}
\label{fig:07}
\end{figure}

\begin{figure}[htbp]
\centering
\includegraphics[scale=0.22]{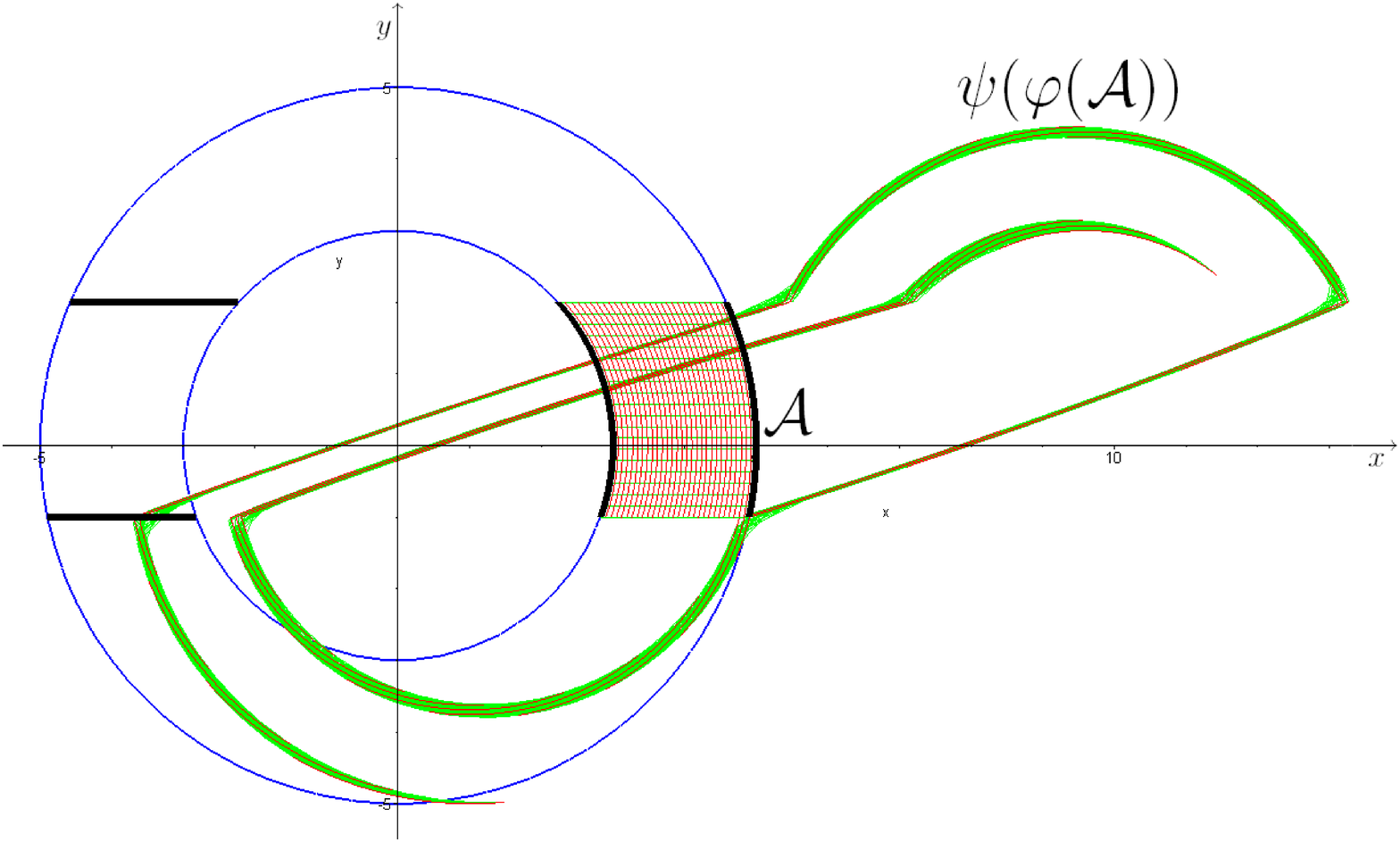}
\caption{\footnotesize{For $p_1, p_2, q_1, q_2$ and $\varphi$ as
in Figure \ref{fig:06} and $\psi$ as
in Figure \ref{fig:07}, we show
the image of ${\mathcal A}$ under the composite map
$\psi\circ \varphi\,.$ It is evident that $\psi(\varphi({\mathcal A}))$
crosses ${\mathcal A}$ twice in the
correct manner.
}}
\label{fig:08}
\end{figure}

\begin{figure}[htbp]
\centering
\includegraphics[scale=0.25]{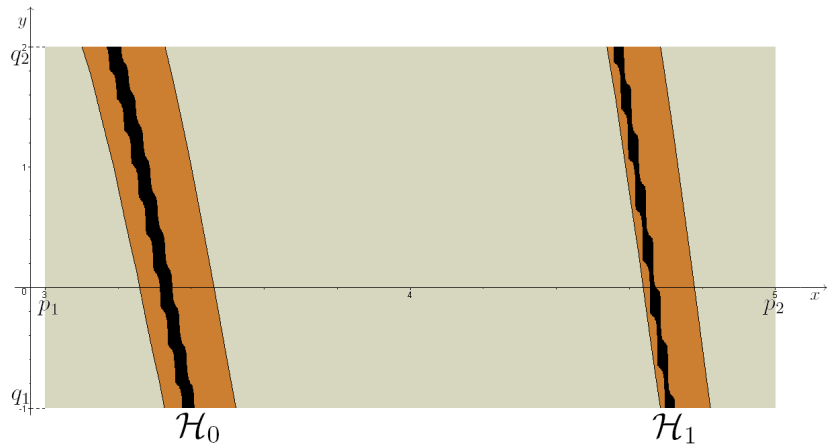}
\caption{\footnotesize {In the simple case of Example \ref{ex-ltm2},
for $\varphi$ and $\psi$ with the coefficients chosen
for drawing Figure \ref{fig:08}, we can determine the sets
${\mathcal H}_0$ and ${\mathcal H}_1$ from Theorem \ref{th-ltma}.
Indeed, inside the set ${\mathcal A},$ transformed in the
rectangle $[p_1,p_2]\times [q_1,q_2]$ by a suitable change of coordinates, we have
found two subsets (drawn with a darker color) whose image under
$\varphi$ is contained in ${\mathcal B}.$ Inside these subsets, we have
localized two smaller darker domains such that their images under $\psi\circ \varphi$
are contained again in ${\mathcal A}.$ Clearly, any path contained
in the rectangle $[p_1,p_2]\times [q_1,q_2]$ and joining the left and
the right sides contains two subpaths (inside the smaller darker subsets) which are stretched by the composite mapping onto paths connecting again the left and the right sides of the rectangle. Such smaller subsets can thus be taken as ${\mathcal H}_0$ and ${\mathcal H}_1.$
Up to a homeomorphism between ${\mathcal A}$ and $[p_1,p_2]\times [q_1,q_2],$
this is precisely what happens in $({\mathcal A},{\mathcal A}^-)$
for $\psi\circ \varphi\,.$
}}
\label{fig:09}
\end{figure}

\clearpage

\noindent
With a straightforward modification in the proof of Theorem \ref{th-ltma}, one could check
that $(H_{\varphi})$ and $(H_{\psi})$ imply the
existence of chaotic dynamics also for $\varphi\circ\psi.$ Actually, both Theorem \ref{th-ltma}
and such variant of it can be obtained as corollaries of the following more general result,
whose proof is only sketched, by the similarity with the verification of Theorem \ref{th-ltma}.

\begin{theorem}\label{th-ltmb}
Let $X$ be a metric space and
assume that $\varphi: X\supseteq D_{\varphi}\to X$ and
$\psi: X\supseteq D_{\psi}\to X$ are continuous maps defined on the sets $D_{\varphi}$ and $D_{\psi},$ respectively. Let also
${\widetilde{\mathcal A}} := ({\mathcal A},{\mathcal A}^-)$ and
${\widetilde{\mathcal B}} := ({\mathcal B},{\mathcal B}^-)$ be oriented rectangles of $X.$
Suppose that the following conditions are satisfied:
\begin{itemize}
\item[$(H_{\varphi})\;\;$] There are $m\geq 1$ pairwise disjoint compact sets
${\mathcal H}_0\,,\dots, {\mathcal H}_{m-1}\,\subseteq {\mathcal A}\cap D_{\varphi}$
such that
$\displaystyle{({\mathcal H}_i,\varphi): {\widetilde{\mathcal A}}
\stretchx\, {\widetilde{\mathcal B}}},$ for $i=0,\dots,m-1\,;$
\\
\item[$(H_{\psi})\;\;$] There exist $l\geq 1$ pairwise disjoint compact sets
${\mathcal K}_0\,,\dots, {\mathcal K}_{l-1}\,\subseteq {\mathcal B}\cap D_{\psi}$
such that
$\displaystyle{({\mathcal K}_i,\psi): {\widetilde{\mathcal B}}
\stretchx\, {\widetilde{\mathcal A}}},$ for $i=0,\dots,l-1\,.$
\end{itemize}
If at least one between $m$ and $l$ is greater or equal than $2,$
then the composite map $\phi:=\psi\circ \varphi$ induces chaotic dynamics on $m \times l$ symbols
in the set
$${\mathcal H}^*:=
\bigcup_{i,j} {\mathcal H}'_{i,j}\,,\,\mbox{with }\;
{\mathcal H}'_{i,j}\,:= {\mathcal H}_i\cap \varphi^{-1}({\mathcal K}_j)\,, i=0,\dots,m-1,\, j=0,\dots,l-1,$$ and thus satisfies properties $(i)$-$(v)$ from Theorem \ref{th-cons}.
\end{theorem}
\begin{proof}
In order to get the thesis, it suffices to show that
$$({\mathcal H}'_{i,j},\phi): {\widetilde{\mathcal A}}
\stretchx {\widetilde{\mathcal A}}\,,\;\;\forall\, i=0,\dots,m-1\, \mbox{ and }\,j=0,\dots,l-1.$$
Such condition can be checked by steps analogous to the ones in Theorem \ref{th-ltma}.
The details are omitted since they are straightforward.
\end{proof}

\noindent
As we shall see, this latter result will find an application in Section \ref{sec-pp} when dealing with a periodic version of the Volterra predator-prey model.

\part{Applications}

\chapter{Applications to difference equations}\label{ch-de}
In order to present possible applications of the ``stretching along the paths'' method introduced in Part \ref{pa-ec}, we start with some discrete-time economic models, for which we show the presence of a chaotic behaviour in a rigorous manner, without the need of computer simulations. In fact, when studying problems
arising in different areas of research, there is plenty of
numerical evidence of complex dynamics \cite{AgEl-04,CsGa-06,FJJB-08,GaZg-01,He-76,KrSu-04,LiLi-04,Ll-95,TrCo-80,YaLi-05,Zg-97}, but, if we except the
simpler one-dimensional case, few rigorous proofs of the
existence of chaos in one or the other of its main
characterizations can be found in the literature.\\
In more details, in Section \ref{sec-olg} we deal with some models of the ``Overlapping Generations'' class (henceforth OLG). This name comes from the fact that one
considers an economy with a population divided into ``young'' and
``old''.
In order to make the comprehension of our techniques easier, we start with the analysis of some one-dimensional cases in Subsection \ref{sub-1d} and we move to planar systems in Subsection \ref{sub-2d}.\\
On the other hand, in Section \ref{sec-dg} we consider a bidimensional model of a ``Duopoly Game'' (from now on DG) where the players have heterogeneous expectations. By
``duopoly'', economists denote a market form characterized by
the presence of two firms. The term ``game''
refers to the fact that the players - in our case the firms - make
their decisions, or ``moves'', reacting to each other actual or
expected moves, following a more or less sophisticated strategy.
In the present framework, we deal with a dynamic game where
moves are repeated in time, at discrete, uniform intervals.
Finally, the expression ``heterogeneous expectations'' simply means that, in the absence of precise information about its competitor's future decisions, each firm is assumed to employ a different strategy.\\
In the next sections we will mainly focus on the mathematical aspects of the treatment, limiting the explanation of the economic interpretation to what is barely necessary to the understanding of our main argument.

\section{OLG Models}\label{sec-olg}
\subsection{Unidimensional examples}\label{sub-1d}

Just to fix ideas, we start by presenting a
one-dimensional application of the Stretching Along the Paths method to a basic pure exchange overlapping generations (OLG)
model. In this simple case, we do not expect to establish any new result
but the exercise is intended to introduce the reader to our techniques in a clear and intuitive manner.\\
The model deals with
an economy with a constant population living two periods of unit
time length and, at each moment in time, being divided into two equally
numerous classes of persons, labeled respectively ``young'' and
``old''. Because individuals in each class are assumed to be
otherwise identical, we shall describe the situation in terms of ``the
young (old) representative agent''. There is no production but constant,
nonnegative endowments of the unique, perishable consumption good are
distributed at the beginning of each period to young and old.\\
For each generation, the young representative agent maximizes a
concave utility function over a two-period life, subject to a
budget constraint. We assume that the market for the consumption
good is always in equilibrium (demand = supply) and that agents'
expectations are fulfilled (perfect foresight).\\
Let the notation be as follows:
$$
\begin{array}{llll}
c_t\,\,\geq 0 : \mbox{young agent's consumption at time $t$}& \\
g_t\,\,\geq 0 : \mbox{old agent's consumption at time $t$}& \\
w_0\geq 0: \mbox{young agent's endowment}&\\
w_1\geq 0: \mbox{old agent's endowment}&\\
\rho_t\,\,> 0 : \mbox{interest factor at time $t$}& \\
 \qquad\quad\,\,=\mbox{exchange rate between present and future consumption}\\
 U(c_t, g_{t+1})=u_1(c_t)+u_2(g_{t+1}): \mbox{utility functions (the
same for all agents).}
\end{array}
$$
A mathematical formulation of such problem can be written as
$$
\begin{array}{llll}
&{\displaystyle{\max_{c_t, g_{t+1}}^{} \{u_1(c_t)+ u_2(g_{t+1})\}}}\\
&\qquad\mbox{s.t.}\ g_{t+1}\leq w_1+\rho_t(w_0-c_t),\\ &\qquad
c_t,\, g_{t+1}\geq 0,\,\, \forall t\geq 0
\end{array}
\eqno(\mbox{P1})
$$
and the market-clearing condition is
\begin{equation}\label{eq-mcc}
c_t+g_t=w_0+w_1,\,\, \forall t\geq 0.
\end{equation}
From the first order conditions of $(\mbox{P1})$ and
from \eqref{eq-mcc}, we deduce that the optimal choice for the young
agent's current saving (dis-saving), $(w_0-c_t),$ equal to the old agent's
current dis-saving (saving), $(g_t-w_1),$ must satisfy the equation
\begin{equation}\label{eq-sadi}
H(c_{t+1}, c_t)={\mathcal U}(c_{t+1})+{\mathcal V}(c_t)=0,
\end{equation}
where ${\mathcal U}(c)=u_2'(c)(w_1-g)$ and ${\mathcal V}(c)=u_1'(c)(c-w_0).$\\
Whether or not from \eqref{eq-sadi} we can derive a difference equation moving
forward in time depends on whether the function $\mathcal U$ is
invertible. To visualize the situation, let us adopt a
quasilinear-quadratic pair of utility functions following \cite{BeDa-82}, namely
\begin{equation*}
u_1(c)=ac-(b/2)c^{2}\,,\quad u_2(g)=g.
\end{equation*}
If, for simplicity's sake, we put $a=b=\mu$ and, rescaling endowments,
$w_0=0,$ we obtain the difference equation
\begin{equation}\label{eq-de}
c_{t+1}=F(c_t)=\mu c_t(1-c_t),
\end{equation}
i.e. the much-studied ``logistic map''. Starting from an arbitrary
initial condition $c_0,$ equation \eqref{eq-de} determines sequences of
young agents' consumption forward in time. Through the equilibrium
condition \eqref{eq-mcc}, it will determine old agents'
consumption as well. According to \cite{Ga-73},
this is the ``classical case'', where the young agents dis-save
and borrow and the old ones save and lend\footnote{Notice that,
whereas $w_0$ is unimportant here, $w_1$ must be chosen in such a
way that, whenever $c_t\geq 0,$ so is $g_t,$ i.e. $w_1\geq
c_{\mbox{max}}=F_{\mu}(1/2).$}. Alternatively, using the pair of
functions
$$
u_1(c)=-\mu\, e^{-c}\,, \quad u_2(g)=g,
$$
we get the difference equation
\begin{equation}\label{eq-OLG1dim}
c_{t+1}=F(c_t)=\mu\, c_t\, e^{-c_t}.
\end{equation}
The logistic map has already been analyzed in Section \ref{sec-sd}
and thus we deal with it only briefly here. As we have seen, if we consider the second iterate
of the function $F$ in \eqref{eq-de}, by a simple geometric
argument it can be established that the conditions of Theorem
\ref{th-ch} are verified for sufficiently large values of the
parameter $\mu \leq 4.$ For instance, in Figure \ref{fig-log} we
visualized $F^2$ for $\mu=3.88:$ with reference to that picture,
using the terminology from Theorem \ref{th-ch}, we showed that
$F^2$ induces chaotic dynamics on two symbols and
consequently possesses all the chaotic properties listed in
Theorem \ref{th-cons}. \\
We recall that the proof that
$F:[0,1]\to {\mathbb R}$ satisfies the hypotheses of Theorem \ref{th-ch} and generates
chaotic dynamics on two symbols is trivial for $\mu >4.$ Unfortunately, in this
case almost all points of $[0,1]$ limit to $-\infty$
through the iterates of $F.$

\begin{figure}[ht]\label{fig-3}
\centering
\includegraphics[scale=0.23]{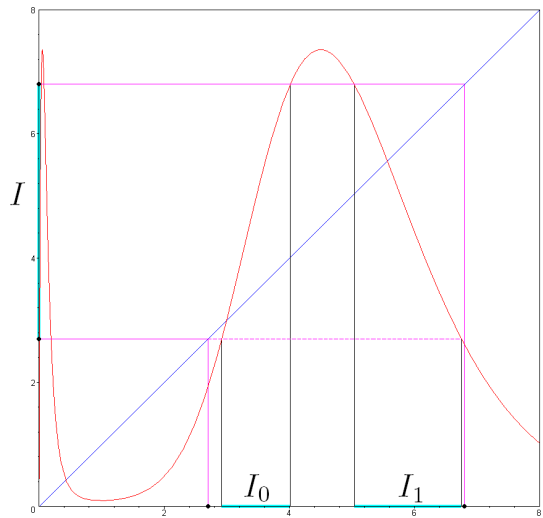}
\caption{\footnotesize{Partial representation of the graph of the second iterate
of the map $F$ in \eqref{eq-OLG1dim}, with $\mu=20.$
}}
\label{fig-olg1d}
\end{figure}

\noindent An entirely analogous line of reasoning can be followed
for the map $F$ from \eqref{eq-OLG1dim}, for which we draw the graph of the second iterate in Figure \ref{fig-olg1d} with $\mu=20.$
We enter the framework of Theorem \ref{th-ch} with the positions
$$\widetilde{\mathcal R}=\widetilde I= (I, I^-),\quad {\mathcal
K_0}=I_0\quad \mbox{ and } \quad{\mathcal K_1}=I_1\,,$$ where $I, I_0$ and $I_1$ are the
intervals highlighted on the coordinate axes in Figure \ref{fig-olg1d}, so that
we can promptly verify that
$$ ({I}_i,F^2): {\widetilde{I}} \stretchx {\widetilde{I}}, \quad i=0,1.$$
Hence, the map $F^2$ in \eqref{eq-OLG1dim} induces
chaotic dynamics on two symbols relatively to $I_0$ and $I_1$ and, consequently, has the chaotic features described in Theorem \ref{th-cons}.

\subsection{A planar model}\label{sub-2d}
We continue our treatment with a bidimensional OLG model with production, such as that
discussed in \cite{Me-92} and \cite{Re-86}.\\
Economically, there are two basic differences between this framework
and the one discussed in the previous subsection. The first one is that there are no
endowments. Instead, in each period $t,$ a certain amount of the
single homogeneous output $x$ (call it ``harvest corn'') is
produced  by means of inputs of current labor $l$ and capital
(``seed corn'') $k$ is saved and invested in the previous period. The
stock of capital depreciates entirely during one period of time.
The second difference concerns the fact that only young people work, earning a unit wage $w,$ and save, while only
old people consume. The supply of labor at each time $t$ is
determined as a solution of a problem of intertemporal constrained
maximization of the young agent's utility function, whose
arguments are labor (at time $t$) and consumption (at time
$t+1$) \footnote{To give the problem a beginning, an exception must
be made for the first generation, which is assumed to be born old
at $t=1$ and to be endowed with a given amount of ``harvest
corn'', entirely consumed in that period.}.\\
To simplify the analysis, we assume a linear, fixed coefficient
technology, i.e., formally, we adopt the production function

\begin{equation}\label{eq-pf1}
x_t=\min\{a\,l_t,b\,k_{t-1}\},
\end{equation}
where the constants $a$ and $b$ denote the output/labor and the
output/capital coefficients, respectively, and
\begin{equation}\label{eq-pf2}
k_t=x_t - c_t\,.
\end{equation}
In order to ensure that the economy is ``viable'', we shall suppose that
$b>1,$ i.e. that the fixed requirement of ``seed corn'' per unit of
``harvest corn'' is less than one. Moreover, to economize
in the use of parameters and without loss of generality, we shall
choose the unit of measure of labor so that $a=1.$\\
Under these hypotheses, the young agent's problem can be formalized
as follows:
$$
\begin{array}{llll}
&{\displaystyle{\max_{c_{t+1},l_t}\{u(c_{t+1})-v(l_t)\}}}\\
& \qquad\mbox{s.t.}\ c_{t+1} \le \rho_{t+1}(w_t\, l_t-c_t),\\
& \qquad c_t,\, l_t,\, k_t\, \ge 0,\,\, \forall t\ge 0,
\end{array}
\eqno(\mbox{P2})
$$
where $u$ and $-v$ are the two separable components of the utility
function with
$$u'(c)>0\,,\,\, u''(c)\leq 0 \quad \mbox{and} \quad v'(l)>0\,,\,\, v''(l)>0.$$
From the first order conditions of problem $(\mbox{P2}),$ we get
\begin{equation}\label{eq-u}
\mathcal U(c_{t+1}) = \mathcal V(l_t),
\end{equation}
where $\mathcal U(c) = u'(c)\,c$ and $\mathcal V(l)=v'(l)\,l.$
Moreover, from \eqref{eq-pf1}--\eqref{eq-pf2}, we have
\begin{equation}\label{eq-l}
l_{t+1}=b(l_t-c_t).
\end{equation}
Considering that, under the postulated assumptions, $\mathcal
V(l)$ is always invertible, whether conditions
\eqref{eq-u}--\eqref{eq-l} can actually define forward
discrete-time dynamics of the state variables $(c,l)$ or not
depends entirely on the properties of $\mathcal U(c)$
and therefore on those of $u(c).$\\
If $\mathcal U(c)$ is invertible,
then simple calculations lead to the difference equation
\begin{equation}\label{eq-crra}
(c_{t+1}, l_{t+1})=F(c_t, l_t)=[\,\mathcal U^{-1}(\mathcal V(l_t)),
b(l_t-c_t)].
\end{equation}
Although for this case there is a rigorous proof of the
existence of periodic solution in \cite{Re-86}, the most
interesting dynamics are found when $\mathcal U(c)$ is not
invertible. We thus concentrate on the latter situation and, in particular, we
assume the ``constant absolute risk aversion (CARA)''
utility function
$$
u(c) = - \mu\, e^{-c},
$$
where $\mu$ is a positive parameter, whence
$$
\mathcal U(c)=u'(c)c=\mu\, c\,e^{-c}.
$$
We also adopt the labor (dis)utility function
$$
v(l)=\frac{1}{\beta}\,l^{\beta},
$$
with $\beta>1.$ Now $\mathcal U(c)$ has a ``one-hump'' form and consequently it does not
admit a global inverse $\mathcal U^{-1}(c).$ Hence, we cannot define forward dynamics as we did with
\eqref{eq-crra}. Since $\mathcal V(l)$ is invertible, however, we
could write the system of equations
\begin{equation}\label{eq-sys}
(c_t, l_t)=F(c_{t+1}, l_{t+1})=\left\{\begin{array}{ll}
                                g_0(c_{t+1})-\frac{1}{b}l_{t+1}\\
                                g_0(c_{t+1}), \end{array} \right.
\end{equation}
where
\begin{equation}\label{eq-g0}
g_0(c):={\mathcal V}^{-1}[\,{\mathcal U}(c)].
\end{equation}
Given the initial conditions for $l$ and $c,$ system \eqref{eq-sys} determines
the dynamics of the variables $(c, l)$ ``in the past''.

\smallskip

\noindent
In what follows, we apply the stretching along the paths method to prove that, for
certain parameter configurations, the dynamical system
\eqref{eq-sys} (actually, a generalization of it) exhibits chaos. We will then suggest how to use Inverse Limit Theory \cite{MeRa-07} to show that chaotic dynamics backward in time imply chaotic dynamics forward in time.\\
In order to make reference to the treatment of Section
\ref{sec-sap} easier and to deal with a
more general case,
we rewrite system \eqref{eq-sys},
defining the continuous planar map $F:(\mathbb R^+)^2\to\mathbb R^2$ of the first quadrant as
\begin{equation}\label{eq-olg}
F(x,y)=(F_1(x,y),F_2(x,y)):=\Bigl(g(x)-\frac{y}{b},g(x)\Bigr),
\end{equation}
where $b > 1$ is a constant and $g(x)$ is a continuous real-valued
function assuming nonnegative values for $x\in [0,+\infty)$ and
vanishing in $0.$ Moreover, we suppose that $g$ has a positive
maximum value $M$ which is achieved at a unique $\bar x\in\,(0,
+\infty).$ Notice that the required properties for $g$ are
satisfied by any unimodal map with $g(0)=0$ and therefore they do
not necessarily restrict the choice
to the map $g_0$ in \eqref{eq-g0}.\\
In Theorem \ref{th-1} below, we prove that the stretching along the paths property for
the map $F$ in \eqref{eq-olg} is satisfied when taking as generalized rectangle a
member of the family of trapezoidal regions described analytically
by
\begin{equation}\label{eq-dom1}
\mathcal R=\mathcal R(K, M):=\left\{(x,y)\in\mathbb R^2: 0\le x\le K,\; x\le y\le M\right\},
\end{equation}
with $M=g(\bar x)$ the maximum of the map $g$ as discussed above
and $K$ a suitable parameter. The stretching property will be checked
under the technical hypothesis that there exists $K>\bar x$ such
that
\begin{equation}\label{eq-gen}
b g(K)\le K < M\Bigl(1-\frac{1}{b}\Bigr).
\end{equation}
Notice that, in order for \eqref{eq-gen} to hold for some
value of $K > \bar{x},$ we must have
$$M = g(\bar{x}) > \frac{b \bar{x}}{b -1}\,.$$

\begin{figure}[ht]\label{fig-4}
\centering
\includegraphics[scale=0.21]{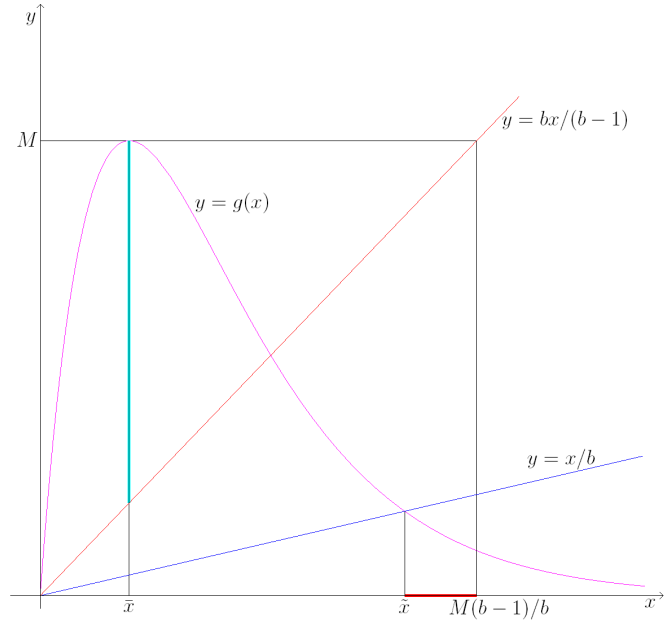}
\caption{\footnotesize{A visual representation of conditions \eqref{eq-gen}.
}}
\label{fig-fc}
\end{figure}

\noindent
Figure \ref{fig-fc} provides a graphical illustration of conditions
\eqref{eq-gen}. Let $\tilde x$ be the maximal solution of the
equation $x=bg(x).$ If the maximum $M$ of the map $g$ lies on
the thicker vertical half-line, that is, if $M=g(\bar x)>b\bar
x/(b-1),$ and so $\tilde x>\bar x>0,$ then any point in the
segment $[\tilde x,M(b-1)/b)$ on the $x$-axis
can be taken as the right edge $K$ for the trapezoid $\mathcal R.$\\
In the special case of system \eqref{eq-sys} with the choice
\eqref{eq-g0}, the function $g$ has the form
\begin{equation}\label{eq-g}
g(x)=g_0(x):=(\mu x\exp(-x))^{1/\beta},
\end{equation}
with $\mu>0$ and $\beta>1.$ Thus $\bar x=1$ and $M=\left(
\frac{\mu}{e}\right)^{1/\beta}.$ It follows that, for any given value of
$b>1,$ the condition $M=g(\bar x)>b\bar x/(b-1)$ is satisfied for
sufficiently large values of $\mu,$ i.e. when the unimodal curve of $g$ is
sufficiently steep. Also \eqref{eq-gen} holds true for every $\mu$ large
enough. For example, if we fix $b=2,\,\beta=1.3,$ the conditions of
Theorem \ref{th-1} are fulfilled for $\mu=80,$ $b=2,\,\beta=1.3$ (to which
there corresponds $M=g(1)\approx 13.48474370$) and $K=6.$ These are indeed the parameters used in Figures \ref{fig-5}--\ref{fig-7}.

\smallskip

\noindent
The generalized rectangle in \eqref{eq-dom1} can be oriented by setting
\begin{equation}\label{eq-lr}
\mathcal R^-_{\ell}:= \{0\}\times [0,M] \quad  \mbox{ and } \quad \mathcal R^-_{r}:=\{K\}\times [K,M]\,,
\end{equation}
and, according to our notation, their union will be denoted by $\mathcal R^-.$

\smallskip

\noindent
Our main result for system \eqref{eq-sys} (and, more generally, for the map in \eqref{eq-olg}) is the following:

\begin{theorem}\label{th-1}
For the map $F$ defined in \eqref{eq-olg} and for any generalized rectangle ${\mathcal R}=\mathcal R(K, M)$ belonging to the family described in \eqref{eq-dom1}, with
\begin{equation}\label{eq-par}
M=g(\bar x),\quad b g(K)\le K < M\Bigl(1-\frac{1}{b}\Bigr),\quad
K>\bar x,
\end{equation}
and oriented as in \eqref{eq-lr}, there are two disjoint compact subsets ${\mathcal K}_0={\mathcal K}_0(\mathcal R)$ and ${\mathcal K}_1={\mathcal K}_1(\mathcal R)$ of $\mathcal R$ such that
\begin{equation}\label{eq-k}
({\mathcal K}_i,F): {\widetilde{\mathcal R}} \stretchx {\widetilde{\mathcal R}}, \mbox{ for } i=0,1.
\end{equation}
Hence, the map $F$ induces chaotic dynamics on two symbols on
${\mathcal R}$ relatively to $\mathcal K_0$ and $\mathcal K_1$ and
has all the properties listed in Theorem \ref{th-cons}.
\end{theorem}
\begin{proof}
We are going to show that the conditions on the parameters
$K$ and $M$ in \eqref{eq-par} are sufficient to guarantee that
the image under the map $F$ of any path $\gamma=(\gamma_1,\gamma_2):[0,1]\to\mathcal R,$ joining in ${\mathcal R}$
the sides $\mathcal R^-_{\ell}$ and $\mathcal R^-_{r}$ defined in \eqref{eq-lr}, satisfies the following:
\begin{itemize}
\item [$(C1)$] $F_1(\gamma(0))\le 0\,;$
\item [$(C2)$] $F_1(\gamma(1))\le 0\,;$
\item [$(C3)$] $\exists \,\,\bar t\in (0,1): F_1(\gamma(\bar t))> K\,;$
\item [$(C4)$] $F_2(\gamma(t))\subseteq
[F_1(\gamma(t)),M],\,\forall t\in [0,1].$
\end{itemize}
The first three requirements together imply an expansion along the
$x$-axis which is accompanied by a folding.
Indeed, the image $F\circ
\gamma$ of any path $\gamma$ joining ${\mathcal R}^-_{\ell}$ and
${\mathcal R}^-_r$ crosses a first time the trapezoid ${\mathcal
R}\,,$ for $t\in (0,\bar t),$ and then crosses ${\mathcal R}$ back
again, for $t\in(\bar t,1).$ Condition $(C4)$
implies a contraction along the $y$-axis. \\
In order to simplify
the exposition, we will replace $(C3)$ with the stronger condition
\begin{itemize}
\item [$(C3^{'})$] ${{F_1(\bar x,y)>K,\,\forall\, y\in [\bar x,M]\,,}}$
\end{itemize}
i.e. we ask that the inequality in $(C3)$ holds for any $\bar t\in
(0,1)$ such that $\gamma_1(\bar t)=\bar x.$ Notice that, since
$K>\bar x,$
the vertical segment $\{\bar x\}\times [\bar x, M]$ is contained in $\mathcal R.$ \\
Let us now define

$$\mathcal R_0:=\left\{(x,y)\in\mathbb R^2: 0\le x\le \bar x,\; x\le y\le M\right\},$$
$$\mathcal R_1:=\left\{(x,y)\in\mathbb R^2: \bar x\le x\le K,\; x\le y\le M\right\}$$
and
$$\mathcal K_0:=\mathcal R_0\cap F(\mathcal R)\,,\qquad \mathcal K_1:=\mathcal R_1\cap F(\mathcal R)\,.$$
Then, we claim that conditions $(C1),(C2),(C3^{'})$ and
$(C4)$ imply the stretching property in
\eqref{eq-k}. \\
First of all consider that, by condition $(C3^{'}),$ the sets $\mathcal K_0$
and $\mathcal K_1$ are disjoint because $\{\bar x\}\times [\bar x,
M]$ is mapped by $F$ outside $\mathcal R.$
Furthermore $(C1),\,(C2)$ and $(C3^{'})$ ensure that, for every
path $\gamma: [0,1]\to {\mathcal R}$ with $\gamma(0)\in
{\mathcal R}^-_{\ell}$ and $\gamma(1)\in {\mathcal R}^-_{r}$ (or
$\gamma(0)\in {\mathcal R}^-_{r}$ and $\gamma(1)\in {\mathcal
R}^-_{\ell}$), there exist two disjoint subintervals
$[t_0',t_0''],\, [t_1',t_1'']\subseteq [0,1]$ such that
$$\gamma(t)\in {\mathcal K_0},\,\, \forall\, t\in [t_0',t_0''],\quad\gamma(t)\in {\mathcal K_1},\,\,\forall\, t\in [t_1',t_1'']\,,$$
with $F(\gamma(t_0'))$ and
$F(\gamma(t_0'')),$ as well as $F(\gamma(t_1'))$ and $F(\gamma(t_1'')),$ belonging to different components of ${\mathcal R}^-.$
Finally, from $(C4)$ it follows that $F(\gamma(t))\in {\mathcal R},$ \\
$\forall\, t\in [t_0',t_0'']\cup [t_1',t_1''].$ Recalling Definition \ref{def-sap}, we can conclude that, under the postulated conditions, \eqref{eq-k}
is verified and thus the last part of the statement is a straightforward consequence of Theorem \ref{th-ch}. Observe that $F$ is continuous on ${\mathcal K}_0\cup {\mathcal K}_1$ because it is continuous on $(\mathbb R^+)^2.$\\
Next, we want to prove that \eqref{eq-par} implies conditions
$(C1)$--$(C3')$ and $(C4).$ For this purpose,  let us consider any path $\gamma=(\gamma_1,\gamma_2):[0,1]\to\mathcal R$
such that $\gamma(0)\in\mathcal R^-_{\ell}$ and
$\gamma(1)\in\mathcal R^-_{r}$ (the case in which $\gamma(0)\in
{\mathcal R}^-_{r}$ and $\gamma(1)\in
{\mathcal R}^-_{\ell}$ can be treated analogously).\\
Since, in the present situation,
$$F_1(\gamma(0))=-\frac{\gamma_2(0)}{b}\le 0,$$
$(C1)$ is automatically satisfied.\\
For the validity of $(C2),$ we need to establish that
$$F_1(\gamma(1))=g(K)-\frac{\gamma_2(1)}{b}\le 0,$$
but, since $\gamma(1)\in {\mathcal R}^-_{r},$ it follows that
${\displaystyle{\frac{\gamma_2(1)}{b}=\frac{ K}{b}}}$
and therefore $(C2)$ holds if
\begin{equation*}\label{eq-cond1}
F_1(\gamma(1))=g(K)-\frac{K}{b}\le 0,
\end{equation*}
that is, if
\begin{equation}\label{eq-cond1n}
b g(K)\le K.
\end{equation}
To prove $(C3^{'}),$ we must show that
\begin{equation*}
g(\bar x)-\frac{\gamma_2(\bar t)}{b}> K,
\end{equation*}
$\forall \, \bar t\in (0,1)$ with $\gamma_1(\bar t)=\bar x.$
Recalling that $M=g(\bar x)$ is the maximum value for the map $g$
and noticing that $$M-\frac{\gamma_2(\bar t)}{b}>
M\left(1-\frac{1}{b}\right),$$
$(C3')$ is fulfilled if
\begin{equation}\label{eq-cond2}
M\left(1-\frac{1}{b}\right)> K.
\end{equation}
Finally, for $(C4)$ we have to consider two inequalities.
First, since both $b$ and $\gamma_2(t),\,t\in [0,1],$ are
nonnegative, the inequality $F_2(\gamma(t))\ge F_1(\gamma(t)),$
$\forall t\in [0,1],$ or equivalently
$$g(\gamma_1(t))\ge g(\gamma_1(t))-\frac{\gamma_2(t)}{b},\,\forall t\in [0,1],$$
is satisfied. Second, the inequality $F_2(\gamma(t))\le M,\,\forall t\in [0,1],$ or equivalently
$$g(\gamma_1(t))\le M,\,\forall t\in [0,1],$$
is always fulfilled by the choice of $M.$ \\
Summing up, in view of \eqref{eq-cond1n} and \eqref{eq-cond2}, the assumptions
sufficient to imply $(C1),\,(C2),\,(C3^{'})$ and $(C4)$ are
\begin{equation*}
\left\{
\begin{array}{ll}
g(K)-\frac{K}{b}\le 0\\
M\left(1-\frac{1}{b}\right)> K,\\
\end{array}
\right.
\end{equation*}
which are fulfilled for any positive $K$ such that
$$b g(K)\le K< M\Bigl(1-\frac{1}{b}\Bigr),$$
as postulated in \eqref{eq-par}. The proof is
complete.
\end{proof}

\noindent
Figures \ref{fig-5}--\ref{fig-7} provide a visual representation of the stretching along the paths property for system \eqref{eq-sys}.\\

\begin{figure}[ht]
\centering
\includegraphics[width=6in,height=4.5in]{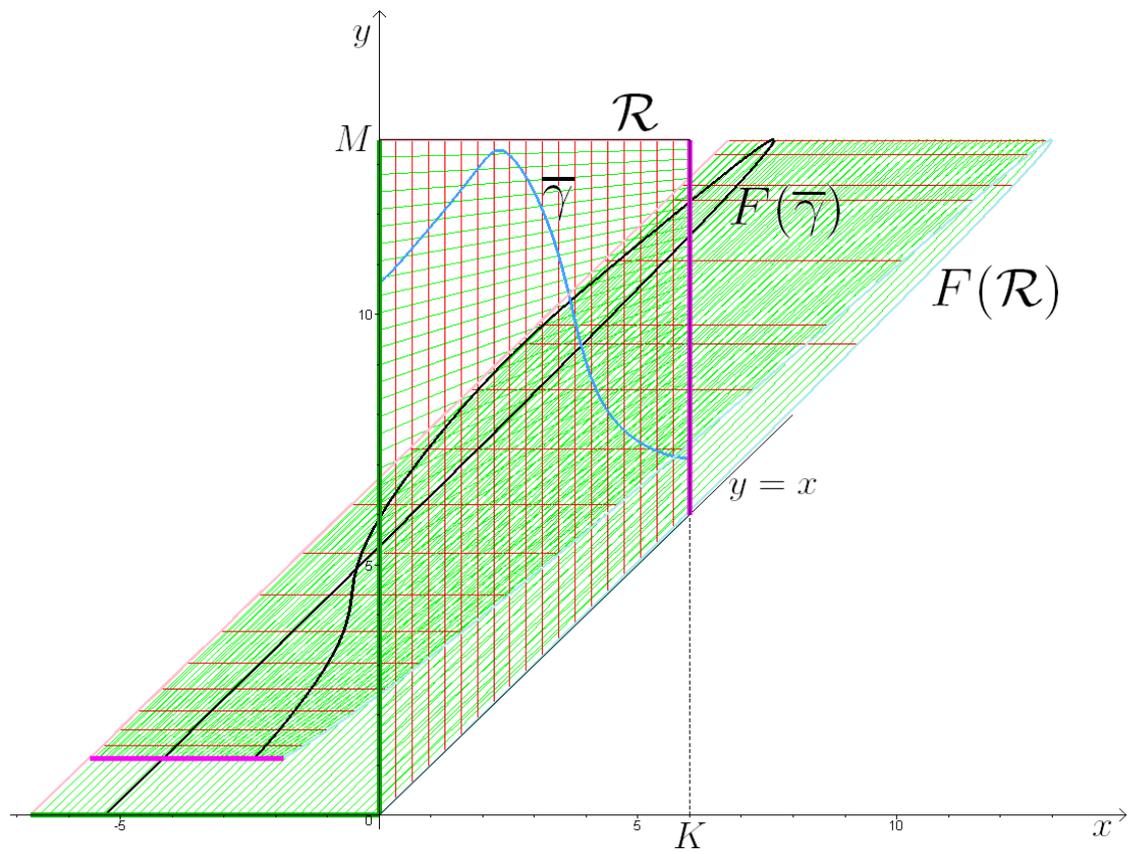}
\caption{\footnotesize{The trapezoid ${\mathcal R},$ oriented in the usual
left-right manner, and its image under the map $F$ in
\eqref{eq-olg}, with $g$ as in \eqref{eq-g}. An arbitrary path
$\gamma$ joining in $\mathcal R$ the two components of the
boundary set $\mathcal R^-$ is transformed by $F$ so that its
image intersects ${\mathcal R}$ twice.
}}
\label{fig-5}
\end{figure}

\clearpage

\begin{figure}[ht]
\centering
\includegraphics[scale=0.35]{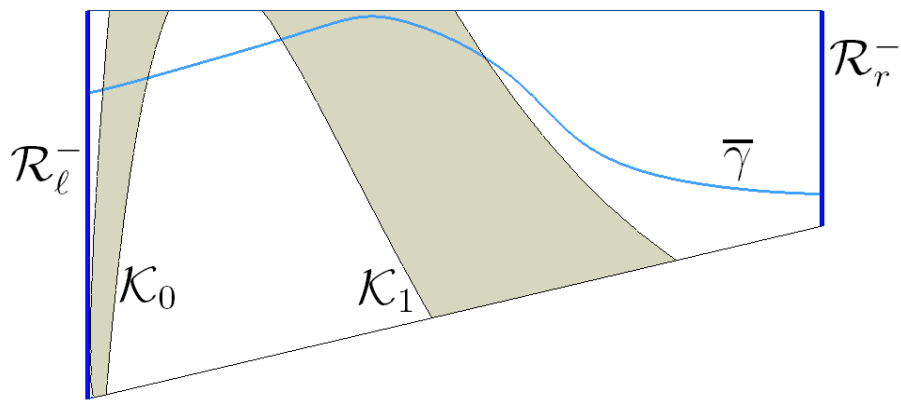}
\caption{\footnotesize{A zoom (with a magnification in the $x$-direction) of the
generalized rectangle ${\mathcal R}$ from Figure \ref{fig-5}, showing
the intersections between the path $\gamma$ and the two compact
sets $\mathcal K_0$ and $\mathcal K_1,$ which are drawn here with a darker color.
}}
\label{fig-6}
\end{figure}

\begin{figure}[ht]
\centering
\includegraphics[scale=0.35]{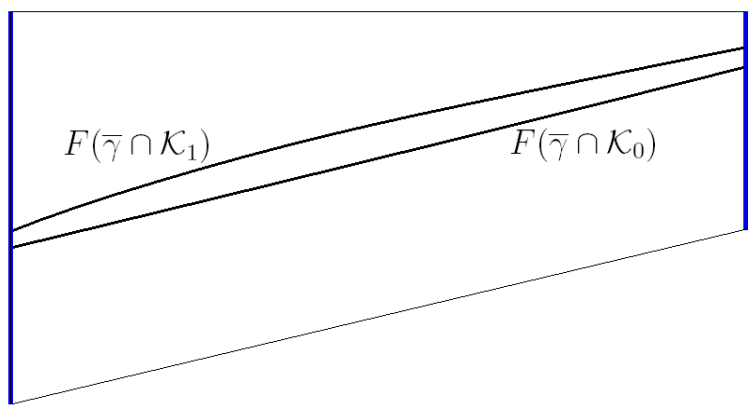}
\caption{\footnotesize{An illustration of the components $F(\overline{\gamma}\cap\mathcal K_0)$
and $F(\overline{\gamma}\cap\mathcal K_1)$ of the image of the path $\gamma$
under the map $F.$
}}
\label{fig-7}
\end{figure}

\clearpage

\noindent
As we have already observed while introducing Theorem \ref{th-1}, for the particular choice of $g=g_0$ adopted in \eqref{eq-g}, condition \eqref{eq-par} is
always satisfied when $\mu$ is sufficiently large.

\smallskip

\noindent
The inequalities in \eqref{eq-par}
define
geometrical conditions depending on the choice of the generalized
rectangle ${\mathcal R}$ (i.e. the trapezoid in \eqref{eq-dom1}).\\
The question naturally arises how a different
${\mathcal R}$ would affect the corresponding conditions on the system
parameters. A systematic investigation of this question would lead
us afar. Hence, we limit ourselves to the analysis of an
interesting alternative choice for
$\mathcal R,$ specifically related to the crucial function $g=g_0$
in \eqref{eq-g}.\\
The new $\mathcal R,$ depicted in Figure \ref{fig-8}, is a subset of the
region below the graph
of the map $x\mapsto bg(x),$ whose
definition is
\begin{equation}\label{eq-new}
\mathcal R=\mathcal R(\nu, d):=\left\{(x,y)\in\mathbb R^2: 0\le y\le b
g(x),\, x\le y\le\nu x+d\right\},
\end{equation}
with $\nu$ and $d$ suitably fixed parameters. Such domain
can be oriented by taking as $\mathcal R^-_{\ell}$ and $\mathcal
R^-_{r}$ the left and the right components of the intersection between the graph of
$y=bg(x)$ and $\mathcal R,$ respectively.
An illustration of this choice for $\mathcal R$ with the parameters
$\mu=14.5,\,b=2,\,\beta=1/0.95,\,\nu=0.6$ and $d \approx
3.662313254$ is furnished by Figure \ref{fig-8}. Moreover, Figure \ref{fig-9} shows  that
the image under $F$ of an arbitrary path $\gamma$ joining $\mathcal R^-_{\ell}$ and $\mathcal
R^-_{r}$ in $\mathcal R$ intersects twice
$\mathcal R$ and suggests a strong similarity between this case
and the one discussed in Theorem \ref{th-1}.
In order to avoid tedious repetitions, however, we will not
provide here a proof that the stretching property and all its
implications indeed obtain for ${\mathcal R}$ in \eqref{eq-new}.

\begin{figure}[ht]
\centering
\includegraphics[scale=0.19]{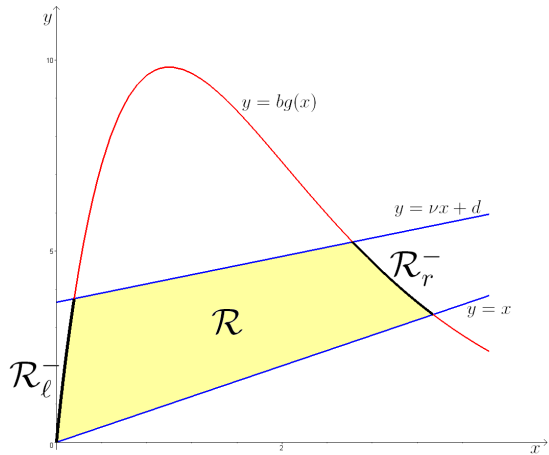}
\caption{\footnotesize{A different choice of the oriented rectangle
$\widetilde{\mathcal R}$
 for system \eqref{eq-sys}
with the two components of $\mathcal R^-$ indicated with  thicker
lines.
}}
\label{fig-8}
\end{figure}

\begin{figure}[ht]
\centering
\includegraphics[width=3.7in,height=2.8in]{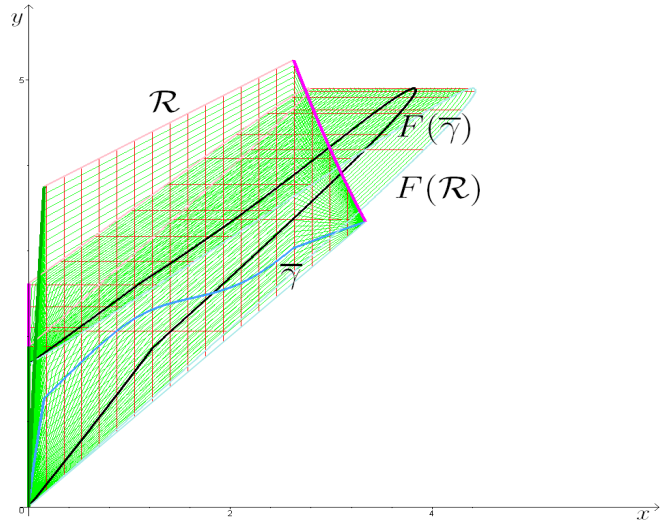}
\caption{\footnotesize{The generalized rectangle $\mathcal R$ represented in
Figure \ref{fig-8} together with an arbitrary path $\gamma$ joining the two
components of the boundary set $\mathcal R^-$ and their images
under $F.$
}}
\label{fig-9}
\end{figure}

\smallskip

\noindent
To conclude, we have thus shown that, for a suitable choice of the generalized rectangle $\mathcal R,$ the dynamics of the
map $F$ controlling our OLG model are therein chaotic in the sense of Theorems \ref{th-ch} and \ref{th-cons}.
The economic interpretation of this result, however, is
complicated by the fact that the iterations of $F$ move
\textit{backward in time}, while economic agents care about the
future not the past. Moreover, in the case under investigation,
$F$ is non-invertible and  therefore the study of the
\textit{forward in time} dynamics implicitly defined by $F$ is not trivial. Inverse Limit Theory \index{Inverse Limit Theory} provides
the ideal mathematical tool to solve such problem, as it allows, under certain hypotheses, to transfer the chaotic features of the map $F$ to a forward-moving map $\sigma$ \footnote{We have chosen to name this map as the classical Bernoulli shift (cf. Section \ref{sec-sd}), as it acts as a (one-sided) shift on a suitable space of sequences. The application-oriented
reader can find a brief introduction to the subject and basic
bibliography in \cite{MeRa-07}.}. The asymptotic
behaviour of the iterations $\sigma^n$ of $\sigma,$ in the limit for
$n\rightarrow \infty,$ can be taken as a representation of the
long-run dynamics of the system. The details of this process are described in \cite{MePiZa->}. We also mention that Inverse Limit Theory allows to deal with another question related to the chaotic dynamics of the model above. Indeed, so far we have only discussed the \textit{existence} of chaotic sets, not
their \textit{attractiveness}. Here, we would like to add the
following brief observation. We have proved
the existence of backward moving chaos for values of  parameters
of the map $F$ for which the numerical simulations ``explode'' and
the chaotic set $\Lambda$ from Theorem \ref{th-cons} is therefore not observable on the
computer screen. This fact suggests that $\Lambda$ is a repeller
for the map $F.$ Hence, the corresponding inverse limit set is a
natural candidate for an attractor, in one or the other of the various existing definitions, of the forward moving map
$\sigma.$ The cautious formulation of such statement is motivated
by the fact that, in general, $F$ may have many (even infinitely
many) repellers, while only one, or some of the corresponding
subsets of the inverse limit space are  attractors for $\sigma.$ This
question is discussed at great length in \cite{MeRa-07}, where
general results are found for the case in which the bonding maps
are unimodal functions of the interval.

\section{Duopoly Games}\label{sec-dg}

We now apply our stretching along the paths method to detect chaotic dynamics for a planar economic model belonging to the class of ``duopoly games'', drawn from  Agiza and Elsadany's paper \cite{AgEl-04}.\\
The economy consists of two firms producing an identical commodity
at a constant unit cost, not necessarily
equal for the two firms. The commodity is sold in a single market
at a price which depends on total output through a given ``inverse
demand function'', known to both firms. The goal of each firm is
the maximization of profits, i.e. the difference between revenue
and cost. The problem of any firm is to decide at each time $t$
how much to produce at time $(t+1)$ on the basis of the limited
information available and, in particular, the  expectations about
its competitor's future decision.\\
Different hypotheses about demand and cost functions and
firms' expectations lead to different formulations of the
problem. For the model discussed here the following notation and
assumptions are postulated:\vskip .5cm

\noindent 1. \bf Notation   \rm \vskip .25cm

$x_t:$ output of Firm 1 at time $t\,;$

$y_t:$ output of Firm 2 at time $t\,;$

$p:$ unit price of the single commodity\,. \vskip .5cm

\noindent \bf  2. Inverse demand function \rm \vskip .25cm

$p=a - b(x+y),$ where $a$ and $b$ are positive constants\,. \vskip .5cm

\noindent \bf 3. Technology \rm \vskip .25cm

The unit cost of production for firm $i$  is equal to  $c_i, \,i=1,2,$
where $c_1, c_2$ are positive constants, not necessarily equal. \vskip .5cm

\clearpage

\noindent \bf 4. Expectations \rm \vskip .25cm

In the presence of incomplete information concerning its competitor's
future decisions (and therefore about future prices), each firm is assumed
to use a different strategy, namely: \vskip .25cm

\noindent \bf Firm 1: bounded rationality. \rm At each time $t,$
Firm 1 changes its output  by an amount proportional to the marginal profit (= derivative of
profit calculated at the known values of output and price at time $t$).
\vskip .25cm

\noindent \bf Firm 2: adaptive expectations. \rm At each time $t,$
Firm 2 changes its output  by an amount proportional to the difference between  the  previous output $y_t$ and the ``na\"\i ve expectations value'' (calculated by
maximizing profits on the assumption that both firms  will keep output unchanged).

\medskip

\noindent
The result of these assumptions is a system of two difference equations in
the variables $x$ and $y,$ as follows:

\begin{equation*}
\left\{
\begin{array}{ll}
x_{t+1}= x_t(1+\alpha a -\alpha b y_t -\alpha c_1 -2\alpha b x_t)\\
y_{t+1}= (1-\nu)y_t+\frac{\nu}{2b}(a-c_2-b x_t),\\
\end{array}
\right. \eqno{(DG)}
\end{equation*}
where $\alpha$ is a positive parameter denoting the speed of Firm 1's adjustment to changes in profit, $\nu\in [0,1]$ is Firm 2's rate of adaptation and $a,\,b,\,c_1,\,c_2$ are positive constants.\\
In \cite{AgEl-04}, Agiza and Elsadany discuss the
equilibrium solutions of system $(DG)$ and their stability and
provide numerical evidence of the existence of chaotic dynamics.
Here we integrate those authors' analysis,
rigorously proving that, for certain parameter configurations,
system $(DG)$ exhibits chaotic behaviour in the precise sense
discussed in Section \ref{sec-de}. Notice, however, that we only prove \textit{existence} of invariant, chaotic sets, not their \textit{attractiveness}. \\
In order to apply our method
to the study of system $(DG),$ it is expedient to represent it
in the form of a continuous map $F=(F_1,F_2):(\mathbb R^+)^2\to\mathbb R^2,$ with components
\begin{equation}\label{eq-dg}
\begin{array}{ll}
F_1(x,y):=x(1+\alpha a-\alpha b y-\alpha c_1-2\alpha b x),\\
F_2(x,y):=(1-\nu)y+\frac{\nu}{2b}(a-c_2-b x).
\end{array}
\end{equation}
In Theorem \ref{th-2} below we show that the stretching property for the map $F$ is satisfied when taking a generalized rectangle $\mathcal R=\mathcal R(P,Q)$ from the family of rectangular trapezia of the first quadrant described analytically by
\begin{equation}\label{eq-tr}
\mathcal R:=\left\{(x,y)\in\mathbb R^2: P\le y\le Q,\; 0\le x\le \frac{a-c_1+1/\alpha}{2b}-\frac{y}{2}\right\},
\end{equation}
with $a,\,b,\,c_1$ and $\alpha$ as in $(DG)$ and $P,\,Q$ satisfying the restrictions
\begin{equation}\label{eq-pq}
0\le P:=\frac{a+c_1-2c_2-1/\alpha}{3b}<Q:=\frac{a-c_2}{2b}< \frac{a-c_1+1/\alpha}{b}\,.
\end{equation}
Notice that, for this choice, the set $\mathcal R$ is nonempty.\\
A visual illustration of the above conditions is furnished by Figure \ref{fig-10}, where unexplained notation is as in the proof of Theorem
\ref{th-2}. In analogy to \cite{AgEl-04}, in Figures \ref{fig-10}--\ref{fig-12} we have fixed the parameters as
$a=10,\, b=0.5,\, c_1=3,$ $c_2=5,\,\alpha=26/27$ and $\nu=0.5.$\\
\begin{figure}[ht]
\centering
\includegraphics[scale=0.2]{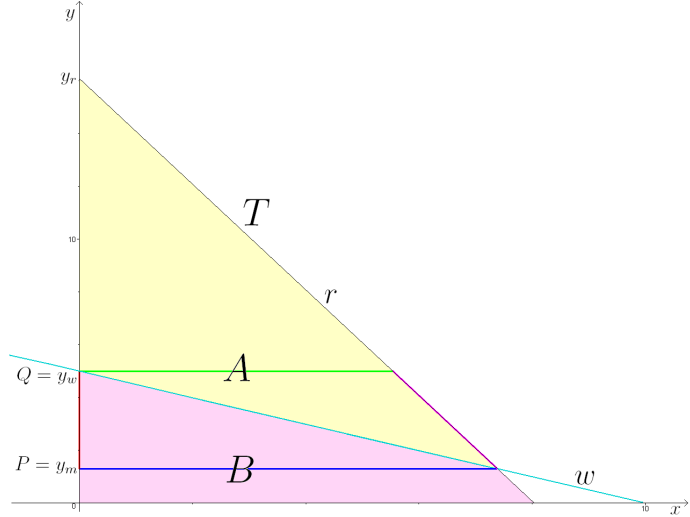}
\caption{\footnotesize{Geometric representation of some conditions in Theorem
\ref{th-2}.  The triangular region $T$ is delimited by the
coordinate axes and by the straight line $r.$ The ``watershed'' line
$w$ separates $T$ into the two regions $A$ (yellow, in the colored version) and $B$
(pink). The values of $P$ and $Q$ are defined by the intersections
of $w$ with $r$ and the $y$-axis, respectively.
}}
\label{fig-10}
\end{figure}

\smallskip

\noindent
A trapezoid ${\mathcal R}$ as in \eqref{eq-tr} can be oriented
by setting
\begin{equation}\label{eq-ori}
\mathcal R^-_{\ell}:= \{0\}\times [P,Q],\,\,
\mathcal
R^-_{r}:=\left\{\left(\frac{a-c_1+1/\alpha}{2b}-\frac{y}{2},y\right):y\in
[P,Q]\right\}.
\end{equation}
The choice of ${\mathcal R}^- ={\mathcal R}^-_{\ell} \cup
{\mathcal R}^-_r$ is depicted in Figure \ref{fig-11}.

\begin{figure}[ht]
\centering
\includegraphics[scale=0.2]{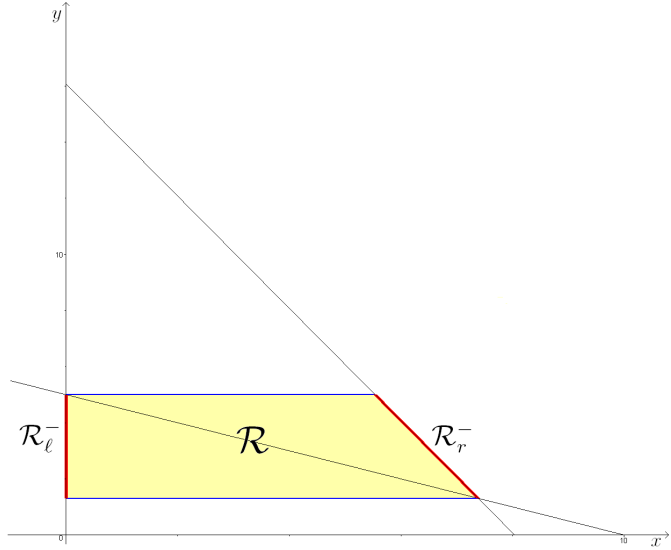}
\caption{\footnotesize{ A suitable generalized rectangle $\mathcal
R$ for system $(DG),$ according to conditions  \eqref{eq-tr} and
\eqref{eq-pq}. It is oriented in the usual left-right manner, by
taking as $[\cdot]^{-}$-set the two ``vertical'' segments
$\mathcal R^-_{\ell}$ and $\mathcal R^-_{r}$ defined in
\eqref{eq-ori}.
}}
\label{fig-11}
\end{figure}

\smallskip

\noindent Our result on system $(DG)$ can be stated as follows:
\begin{theorem}\label{th-2}
If  the parameters of the map $F$ defined in \eqref{eq-dg} satisfy the conditions
\begin{equation}\label{eq-cond}
a-2c_1+c_2 > \frac{26}{3\alpha}\,,\qquad
a+c_1-2c_2-\frac{1}{\alpha}\ge 0\,,
\end{equation}
then, for any generalized rectangle ${\mathcal R}=\mathcal R(P, Q)$ belonging to the family described in \eqref{eq-tr}, with
\begin{equation}\label{eq-ine}
P:=\frac{a+c_1-2c_2-1/\alpha}{3b} \quad \mbox{and} \quad Q:=\frac{a-c_2}{2b}\,,
\end{equation}
and oriented as in \eqref{eq-ori}, there exist two disjoint compact subsets ${\mathcal K}_0={\mathcal K}_0(\mathcal R)$ and ${\mathcal K}_1={\mathcal K}_1(\mathcal R)$ of $\mathcal R$ such that
\begin{equation}\label{eq-ks}
({\mathcal K}_i,F): {\widetilde{\mathcal R}} \stretchx {\widetilde{\mathcal R}}, \mbox{ for } i=0,1.
\end{equation}
Hence, the map $F$ induces chaotic dynamics on two symbols in
$\widetilde{\mathcal R}$ relatively to $\mathcal K_0$ and
$\mathcal K_1$ and has all the properties listed in Theorem
\ref{th-cons}.
\end{theorem}
\begin{proof}
We are going to show that, when choosing $P$ and $Q$ as in
\eqref{eq-ine}, conditions \eqref{eq-cond} are sufficient to
guarantee that the image under the map $F$ of any path
$\gamma=(\gamma_1,\gamma_2):[0,1]\to\mathcal R,$ joining in
${\mathcal R}$ the sides $\mathcal R^-_{\ell}$ and $\mathcal
R^-_{r}$ defined in \eqref{eq-ori}, satisfies the following:
\begin{itemize}
\item [$(C1)$] $F_1(\gamma(0))\le 0\,;$
\item [$(C2)$] $F_1(\gamma(1))\le 0\,;$
\item [$(C3)$] $\exists \,\, t^*\in \, (0,1): F_1(\gamma(t^*))>x_M:=\max\{x:(x,y)\in \mathcal R^-_{r}\}\,;$
\item [$(C4)$] $F_2(\gamma(t))\subseteq
[P,Q],\,\forall t\in [0,1].$
\end{itemize}
Broadly speaking, conditions $(C1)$--$(C3)$ describe an expansion
with folding along the $x$-coordinate. In fact, the image under $F$ of
any path $\gamma$ joining ${\mathcal R}^-_{\ell}$ and
${\mathcal R}^-_r$ in $\mathcal R$ crosses a first time the trapezoid ${\mathcal
R}\,,$ for $t\in (0,t^*),$ and then crosses ${\mathcal R}$ back
again, for $t\in(t^*,1).$ Condition $(C4)$ implies a
contraction along the $y$-coordinate.\\
In order to simplify the exposition, we replace $(C3)$ with
the stronger condition
\begin{itemize}
\item [$(C3^{'})$] ${\displaystyle{F_1\left(\frac{a-c_1+1/\alpha}{4b}-\frac{y}{4},y\right)>x_M,\,\forall y\in [P,Q],}}$
\end{itemize}
i.e. we require that the inequality in $(C3)$ holds for any
$t^*\in (0,1)$ such that
$\gamma(t^*)=(\frac{a-c_1+1/\alpha}{4b}-\frac{y}{4},y),$ for some
$y\in [P,Q].$ \\
Recalling that $\mathcal
R^-_{r}:=\left\{\left(\frac{a-c_1+1/\alpha}{2b}-\frac{y}{2},y\right):y\in
[P,Q]\right\},$ taking $P,Q$ as in \eqref{eq-pq} and considering
that
$0<\frac{a-c_1+1/\alpha}{4b}-\frac{y}{4}<\frac{a-c_1+1/\alpha}{2b}-\frac{y}{2},\,\forall
y\in [P,Q],$ we can conclude that the segment
\begin{equation}\label{eq-spm}
S:=\left\{\left(\frac{a-c_1+1/\alpha}{4b}-\frac{y}{4},y\right):y\in [P,Q]\right\}
\end{equation}
is contained in $\mathcal R.$ A graphical illustration of this fact is provided by Figure \ref{fig-12}.\\
Defining
$$\mathcal R_0:=\left\{(x,y)\in\mathbb R^2: 0\le x\le \frac{a-c_1+1/\alpha}{4b}-\frac{y}{4}\,,\; P\le y\le Q\right\},$$
\begin{figure}[ht]
\centering
\includegraphics[scale=0.23]{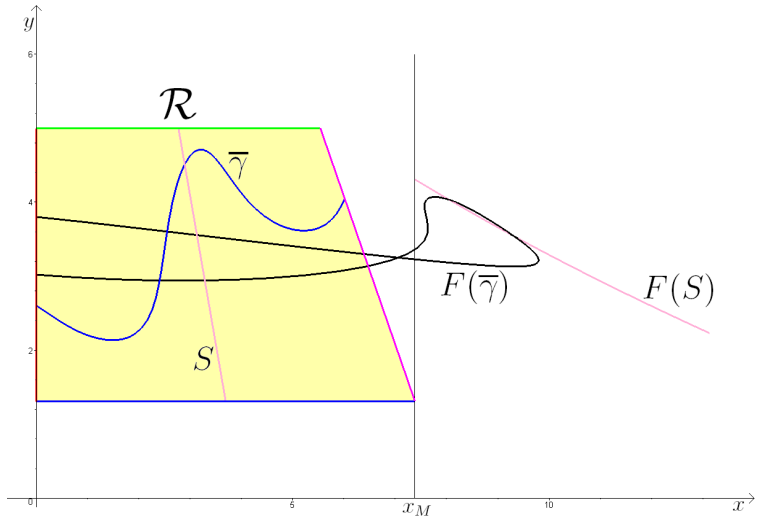}
\caption{\footnotesize{ With reference to the trapezoid $\mathcal R$ in Figure \ref{fig-11},
reproduced here at a different scale, we show that the image under
$F$ of an arbitrary path $\gamma$ joining in $\mathcal R$ the two
components of the boundary set $\mathcal R^-$ intersects twice
$\mathcal R.$ This is due to the fact that the vertical sides
$\mathcal R^-_{\ell}$ and $\mathcal R^-_{r}$ are mapped by $F$
into the $y$-axis and the segment $S,$ joining the middle points of
the bases of the trapezoid, is mapped by $F$ beyond the vertical
line $y=x_M,$ in conformity with condition $(C3^{'}).$
}}
\label{fig-12}
\end{figure}
$$\mathcal R_1:=\left\{(x,y)\in\mathbb R^2: \frac{a-c_1+1/\alpha}{4b}-\frac{y}{4}\le x\le \frac{a-c_1+1/\alpha}{2b}-\frac{y}{2}\,, P\le y\le Q\right\}$$
and
$$\mathcal K_0:=\mathcal R_0\cap F(\mathcal R)\,, \qquad  \mathcal K_1:=\mathcal R_1
\cap F(\mathcal R),$$
we claim that $(C1),(C2),(C3^{'})$ and
$(C4)$ together imply \eqref{eq-ks}. \\
Firstly, notice that $\mathcal K_0$
and $\mathcal K_1$ are disjoint because, by
condition $(C3^{'}),$ the segment $S$ in
\eqref{eq-spm} is mapped by $F$ outside $\mathcal R.$ Furthermore, $(C1),\,(C2)$ and
$(C3^{'})$ imply that, for any path $\gamma: [0,1]\to {\mathcal R}$ with
$\gamma(0)\in {\mathcal R}^-_{\ell}$ and $\gamma(1)\in
{\mathcal R}^-_{r}$ (or $\gamma(0)\in {\mathcal R}^-_{r}$ and
$\gamma(1)\in {\mathcal R}^-_{\ell}$), there exist two disjoint
subintervals $[t_0',t_0''],\, [t_1',t_1'']\subseteq [0,1]$ such
that
$$\gamma(t)\in {\mathcal K_0},\,\,
\forall\, t\in [t_0',t_0''],\quad \gamma(t)\in {\mathcal K_1},\,\,
\forall\, t\in [t_1',t_1''],$$
with $F(\gamma(t_0'))$ and
$F(\gamma(t_0'')),$ as well as $F(\gamma(t_1'))$ and
$F(\gamma(t_1'')),$ belonging to different components of ${\mathcal R}^-.$ Finally from $(C4)$ it follows that $F(\gamma(t))\in {\mathcal R},$
$\forall\, t\in [t_0',t_0'']\cup [t_1',t_1''].$ Recalling Definition \ref{def-sap}, the stretching condition \eqref{eq-ks} is achieved and our claim is thus checked.
The last part of the statement is now an immediate consequence of Theorem \ref{th-ch}.
Observe that $F$ is continuous on ${\mathcal K}_0\cup {\mathcal K}_1,$ as it is continuous on $(\mathbb R^+)^2.$

\smallskip

\noindent
The next step is to verify that a choice of the parameters as in \eqref{eq-cond}
implies conditions $(C1),(C2),(C3^{'})$ and
$(C4).$ In so doing, we will prove that the inequalities in $(C1)$ and $(C2)$ are indeed equalities.\\
First of all, we have to check  that the set $\mathcal R(P,Q)$
defined in \eqref{eq-tr} and \eqref{eq-pq} is contained in the
region $\left\{(x,y)\in (\mathbb R^+)^2:F(x,y)\in (\mathbb R^+)^2\right\}$
\footnote{Such condition guarantees that the invariant chaotic set
discussed in Theorem \ref{th-cons}
lies entirely in the first quadrant and therefore makes economic sense for the application in question.}.\\
For this purpose, let us first consider  $F_1.$ Setting
\begin{equation*}
M:= 1+\alpha a-\alpha b y-\alpha c_1 \quad \mbox{ and } \quad N:=2\alpha b\, ,
\end{equation*}
for any given value of $y$ it is possible to rewrite $F_1$ as the continuous, one-dimensional map
\begin{equation*}
\phi:\mathbb R_+\to \mathbb R,\quad \phi(x):=x(M-Nx),
\end{equation*}
which can be looked at as  a parabola intersecting  the $x$-axis in $0$ and $M/N$ and
lying in the first quadrant in between. Notice  that, under the postulated assumptions on the model parameters,  $N$ is positive, while $M=M(y)> 0$ if $y< (a-c_1+1/\alpha)/b.$ This inequality and the requirement that  $y$  is nonnegative define  the restriction
\begin{equation*}
0\le y < y_r:=\frac{a-c_1+1/\alpha}{b}\,.
\end{equation*}
If such condition is satisfied, $M/N$ is positive and we can restrict our map $\phi$ to the interval $[0,M/N],$ a geometric configuration analogous to that of the logistic map discussed in Section \ref{sec-sd}. In terms of the planar coordinates $(x,y),$ restricting $\phi$ to $[0,M/N]$ means that we consider only the values of $x$ such that
$$0\le x\le \frac{M}{N}=\frac{1+\alpha a-\alpha b y-\alpha c_1}{2\alpha b}=\frac{a-c_1+1/\alpha}{2b}-\frac{y}{2}\,.$$
The conditions on the variables $x$ and $y$ determine a triangular region $T$ in the first quadrant, lying between the coordinate axes and the straight line $x=\frac{a-c_1+1/\alpha}{2b}-\frac{y}{2},$ or equivalently
\begin{equation*}
y=\frac{a-c_1+1/\alpha}{b}-2x\,,
\end{equation*}
a straight line that we will denote by  $r$  (cf. Figure \ref{fig-10}). In
view of the above discussion, the point $(0,y_r)$ does not belong
to $T.$ Observe that the trapezoid $\mathcal R$ in \eqref{eq-tr}
lies is this region for any $0\le P<Q< y_r$ \footnote{The case
$P=Q$ has to be excluded, because otherwise the generalized
rectangle $\mathcal R$ would reduce to a segment and it could not
be homeomorphic to the unit square of $\mathbb R^2.$ For the same
reason we will impose strict inequalities for the $y$-coordinates also in \eqref{eq-y0}.
Similar considerations led us to require $M/N>0,$ instead of
$M/N\ge 0.$} and thus $F_1(x,y)$ is nonnegative on $\mathcal R.$
Furthermore, since $\phi(0)=\phi(M/N)=0,$ for any path
$\gamma=(\gamma_1,\gamma_2):[0,1]\to\mathbb R^2$ with
$\gamma_1(0)=0$ and $\gamma(1)$ belonging to the straight line
$r$ it follows that $F_1(\gamma(0))=0=F_1(\gamma(1)).$ Since
$\mathcal R^-_{\ell}$ is contained in the $y$-axis and $\mathcal
R^-_{r}$ in the line $r,$
this implies that $(C1)$ and $(C2)$ are fulfilled for every choice of $P$ and $Q$ in $[0,y_r)$ (with $P<Q$).\\
In order to understand when $F_2(x,y)\ge 0,$ let us consider the stronger condition
$F_2(x,y)=(1-\nu)y+\frac{\nu}{2b}(a-c_2-b x)\ge y,$ which
is  satisfied if the point $(x,y)$ lies below
the straight line denoted by $w$ in Figure \ref{fig-10} and  defined by the
equation

\begin{equation}\label{eq-ws}
y=\frac{a-c_2}{2b}-\frac{x}{2}\,.
\end{equation}
Recalling that we are confined to the triangular region $T,$ in order to have an intersection between $w$ and the two segments
\begin{equation}\label{eq-se}
\{0\}\times[0,y_r) \quad \mbox{ and } \quad  \left\{\left(\frac{a-c_1+1/\alpha}{2b}-\frac{y}{2},y\right):y\in [0,y_r)\right\},
\end{equation}
we have to impose the following conditions on the $y$-coordinate
\begin{equation}\label{eq-y0}
y_r=\frac{a-c_1+1/\alpha}{b}> y_w:=\frac{a-c_2}{2b}>y_m\ge 0\,,
\end{equation}
where
\begin{equation}\label{eq-m}
(x_m,y_m):=\left(\frac{a-2 c_1+c_2+2/\alpha}{3b}\,,\frac{a+c_1-2c_2-1/\alpha}{3b}\right)
\end{equation}
is the point of intersection of $r$ and $w$  and $(0,y_w)$ is the point of intersection between $w$ and the $y$-axis.
Conditions \eqref{eq-y0} yield
\begin{equation*}
a-2 c_1+c_2+\frac{2}{\alpha}> 0
\end{equation*}
and
\begin{equation}\label{eq-h3}
a+c_1-2 c_2-\frac{1}{\alpha}\ge 0.
\end{equation}
As we will see below, the intersection between $w$ and the segments in \eqref{eq-se} is fundamental for the validity of condition $(C4).$ Calling
$$\mathcal R^{+}_{d}:=\left[0,\frac{a-c_1+1/\alpha}{2b}-\frac{P}{2}\right]\times \{P\}$$  and
$$\mathcal R^{+}_{u}:=\left[0,\frac{a-c_1+1/\alpha}{2b}-\frac{Q}{2}\right]\times \{Q\}$$
the lower and upper sides of the trapezoid $\mathcal R,$ respectively, we notice that, for $P,\, Q$ in $[0,y_r),$ the arcs $\mathcal R^+_{d}$ and $\mathcal R^+_{u}$ are horizontal segments contained in $T.$
Moreover, denoting by $\partial\mathcal R$ the boundary of $\mathcal R$
(that in our case coincides with the contour $\vartheta{\mathcal R}$), it holds that
$$\partial\mathcal R=\mathcal R^{+}_{d}\cup\mathcal R^{-}_{r}\cup\mathcal R^{+}_{u}\cup\mathcal R^{-}_{\ell}.$$
Choosing $\mathcal R^{+}_{d}$ and $\mathcal R^{+}_{u}$ as particular paths $\gamma=(\gamma_1,\gamma_2):[0,1]\to\mathcal R,$ joining in ${\mathcal R}$ the sides $\mathcal R^-_{\ell}$ and $\mathcal R^-_{r},$ in order to have condition $(C4)$ fulfilled, we need
$$F_2(x,P)\ge P,\,\,\quad\forall x\in \left[0,\frac{a-c_1+1/\alpha}{2b}-\frac{P}{2}\right],$$
as well as
$$F_2(x,Q)\le Q,\,\,\quad\forall x\in \left[0,\frac{a-c_1+1/\alpha}{2b}-\frac{Q}{2}\right].$$
But it is possible only if $\mathcal R^+_{d}$ lies below the ``watershed'' line $w$ in \eqref{eq-ws} and for $\mathcal R^+_{u}$ above it: more precisely, in view also of \eqref{eq-y0},
\begin{equation}\label{eq-y}
y_r>Q\ge y_w> y_m\ge P\ge 0\,.
\end{equation}
We will prove that such conditions are sufficient for the validity of $(C4)\,.$\\
Before that, however, we return for a while to the problem of the
nonnegativity of $F_2$ in  $\mathcal R.$ For this purpose,  let
us denote by  $A$ the subset of $T$ above the line $w$ and by $B$
its complement  in $T$ (see Figure \ref{fig-10}). Since $F_2(x,y)\ge y\ge 0$
in $B,$ here there are no further conditions to impose in order to
have $F_2$ nonnegative. On the other hand, as regards the region
$A,$ recalling the notation from \eqref{eq-m}, we require that

\begin{equation}\label{eq-min}
\begin{array}{lll}
{\displaystyle{\min_{(x,y)\in A}F_2(x,y)}}&=&{{(1-\nu)y_m+\frac{\nu}{2b}(a-c_2-b x_m)=}}\\{}\\
&=&{{(1-\nu)\Bigl(\frac{a+c_1-2c_2-1/\alpha}{3b}\Bigr)+\nu\Bigl(\frac{a+c_1-2c_2-1/\alpha}{3b}\Bigr)=}}\\{}\\
&=&{{\frac{a+c_1-2c_2-1/\alpha}{3b}\ge 0}}\,,
\end{array}
\end{equation}
which follows from \eqref{eq-h3} and therefore it does not add any new restriction on the parameters.\\
Having settled the question of the sign of $F_2$ on $\mathcal R,$ we return to $(C4)$ and
notice that it  is equivalent to
\begin{equation*}
\max_{(x,y)\in\mathcal R}F_2(x,y)\le Q \quad \mbox{ and } \quad
\min_{(x,y)\in\mathcal R}F_2(x,y)\ge P\,.
\end{equation*}
Since $F_2(x,y)\le y$ on $A$ and $F_2(x,y)\ge y$ on $B,$ it suffices to ask that
\begin{equation*}
\max_{(x,y)\in B\cap \mathcal R}F_2(x,y)\le Q\quad \mbox{ and }
\quad \min_{(x,y)\in A\cap\mathcal R}F_2(x,y)\ge P\,.
\end{equation*}
From the previous discussion, it follows that $\mathcal R^{+}_{u}$ must lie in $A$ and  $\mathcal R^{+}_{d}$ in $B,$ respectively, and thus  both $A\cap \mathcal R$ and $B\cap \mathcal R$ are nonempty. Actually, instead of the above inequalities, we will investigate the stricter conditions
\begin{equation*}
\max_{(x,y)\in B}F_2(x,y)\le Q \quad \mbox{ and }
\quad\min_{(x,y)\in A}F_2(x,y)\ge P\,,
\end{equation*}
which, as we shall see, do not introduce  any new restriction on the model parameters.\\
As concerns $\max_{(x,y)\in B}F_2(x,y)\le Q,$ recalling the expression of $F_2$ from \eqref{eq-dg}, we have
$$\max_{(x,y)\in B}F_2(x,y)=(1-\nu)y_w+\frac{\nu}{2b}(a-c_2)=(1-\nu)\frac{a-c_2}{2b}+\frac{\nu}{2b}(a-c_2)=\frac{a-c_2}{2b}\,,$$
from which it follows that
$$y_w\left(=\frac{a-c_2}{2b}\right)\le Q,$$
as in \eqref{eq-y}. As concerns  $\min_{(x,y)\in A}F_2(x,y)\ge P,$
from \eqref{eq-min} we obtain
\begin{equation*}
\min_{(x,y)\in A}F_2(x,y)=y_m=\frac{a+c_1-2 c_2-1/\alpha}{3b}\ge
P\,,
\end{equation*}
again as in \eqref{eq-y}.\\
Combining all the restrictions on the construction of $\mathcal R$ found so far, we can write
$$
0\le x\le\frac{a-c_1+1/\alpha}{2b}-\frac{y}{2},\,\;\;\forall y\in [P,Q],
$$
where
$$0\le P\le y_m<y_w\le Q<y_r\,.$$
Such conditions are satisfied under the following restrictions on the model parameters:
\begin{equation}\label{eq-2}
a-2 c_1+c_2+\frac{2}{\alpha}> 0 \quad \mbox{and} \quad a+c_1-2 c_2-\frac{1}{\alpha}\ge 0.
\end{equation}
At this point, we focus on  the remaining problem, that is, the validity of $(C3^{'}),$ and we start by checking that the choice
$$P=y_m:=\frac{a+c_1-2c_2-1/\alpha}{3b} \quad \mbox{and} \quad Q=y_w:=\frac{a-c_2}{2b}$$
is  in some sense optimal  among those allowed by \eqref{eq-y}. To such aim, we will leave for the moment $P,\,Q\in [0,y_r)$
unspecified, as in \eqref{eq-y}.
Recalling that
$$\mathcal R^-_{r}:=\left\{\left(\frac{a-c_1+1/\alpha}{2b}-\frac{y}{2},y\right):y\in [P,Q]\right\},$$ we can write
$$x_M:=\max\{x:(x,y)\in \mathcal R^-_{r}\}=\frac{a-c_1+1/\alpha}{2b}-\frac{P}{2}\,.$$
Clearly, the smaller $x_M,$ the more likely is $(C3^{'})$ to hold. Since $0\le P \le y_m,$ the optimal choice for $P$ is $P=y_m,$ as claimed. As concerns  $Q,$ let us consider it as a parameter $\bar y$ varying in $[y_w,y_r).$ An easy calculation shows that the maximum of the function $F_1$ on the horizontal segment $[0,\frac{a-c_1+1/\alpha}{2b}-\frac{\bar y}{2}]\times\{\bar y\}$ is attained in correspondence of its middle point $\bar x,$ that is for
$$\bar x:=\frac{a-c_1+1/\alpha}{4b}-\frac{\bar y}{4}\,.$$
Evaluating the function $F_1$ in $(\bar x,\bar y),$ we find
$$F_1(\bar x,\bar y)= \frac{\alpha\,(a-b\bar y- c_1+1/\alpha)^2}{8b}\,.$$
To fulfill condition  $(C3^{'}),$ we must have $F_1(\bar x,\bar y)>x_M,$ for any $\bar y\in [P,Q].$ Hence, $Q$ should be chosen as small as possible, i.e. $Q=y_w.$ \\
At this point, having fixed $P$ and $Q,$ from  the above considerations we gather that $(C3^{'})$ is satisfied if, for any $\bar y\in [P,Q],$
\begin{eqnarray*}
\max\left\{F_1(x,\bar y): x\in
\left[0,\frac{a-c_1+1/\alpha}{2b}-\frac{\bar y}{2}\right]\right\}&=
F_1(\bar x,\bar y)>x_M&=\\
&=\frac{a-c_1+1/\alpha}{2b}-\frac{y_m}{2}&=x_m\,.
\end{eqnarray*}
This requirement yields
$$\frac{\alpha\,(a-b y_w- c_1+1/\alpha)^2}{8b}=\frac{\alpha\,(a-2 c_1+c_2+2/\alpha)^2}{32\,b}>\frac{a-2 c_1+c_2+2/\alpha}{3b}$$
and, in view of \eqref{eq-2},
\begin{equation}\label{eq-last}
a-2 c_1+c_2>\frac{26}{3\alpha}\,,
\end{equation}
which corresponds to the first condition in \eqref{eq-cond}. Notice that \eqref{eq-last} implies the first inequality in \eqref{eq-2}.
Thus we conclude that, from \eqref{eq-last} and the second
condition in \eqref{eq-2} (which are precisely \eqref{eq-cond}),
the validity of $(C1),(C2),(C3^{'})$ and
$(C4)$ follows. \\
The proof is complete.
\end{proof}

\chapter{Examples from the ODEs}\label{ch-ode}
In the present chapter we apply the results from Part \ref{pa-ec} to some nonlinear ODE models with periodic coefficients. In particular we deal with planar systems of the first order or with second order scalar equations: since any second order equation can be written as a system of two equations of the first order, in this introductory discussion we will generally refer to planar systems.\\
The systems we consider are studied through a combination of a careful but elementary phase-plane analysis with the results on chaotic dynamics for LTMs from Section \ref{sec-ltm}. We recall that similar strategy has already been employed in \cite{PaZa-09,ZaZa->} with respect to second order equations.
More precisely, we prove in a rigorous way (i.e. without the need of computer
assistance, differently from several works on related topics \cite{BaCs-08,BaGe-01,GaZg-98,GaGe-07,Lo-63,MiMr-95b,PoRaVi-07,Wi-03,YaLi-06,Zg-97}) the presence of infinitely many periodic solutions for our systems, as well as of a chaotic behaviour in the sense of Definition \ref{def-chm} for the associated Poincar{\'e} map $\Psi.$ In fact, a classical approach (see \cite{Kr-68}) to show the existence of periodic solutions (harmonics or subharmonics) of non-autonomous differential systems like
\begin{equation}\label{eq-na}
\dot\zeta = f(t,\zeta),
\end{equation}
where $f: {\mathbb R} \times {{\mathbb R}}^N\to {{\mathbb R}}^N$ is a continuous vector field which is
$T$-periodic in the time-variable, that is, $f(t + T,z) = f(t,z),\,\forall (t,z) \in {\mathbb R} \times {{\mathbb R}}^N,$
under the assumption of uniqueness of the solutions for the Cauchy problems,
is based on the search of the fixed points
for the Poincar\'e map\index{Poincar\'e map} $\Psi=\Psi_T$ or for its iterates, where
$$\Psi: z\mapsto \zeta(t_0 + T;t_0,z)$$
and $\zeta(\cdot\,; t_0,z)$ is the solution of \eqref{eq-na} satisfying the initial condition $\zeta(t_0) = z\in {{\mathbb R}}^N.$
As a consequence of the fundamental theory of ODEs, it turns out
that $\Psi$ is a homeomorphism of its domain
(which is an open subset of ${\mathbb R}^N)$ onto its image.
Applying our Stretching Along the Paths method to $\Psi,$ we are led back to work with discrete dynamical systems. The difference with the examples from Chapter \ref{ch-de} is that, since the Poincar\'{e} map is a homeomorphism, it is possible to prove a semi-conjugacy to the two-sided Bernoulli shift (cf. Remark \ref{rem-inj}), while we recall that the controlling functions of the models in Chapter \ref{ch-de} are not injective.\\
We notice that the kind of chaos in Definition \ref{def-chm}, when considered in relation to the case of the Poincar\'{e} map,
looks similar to other ones detected in the literature on complex dynamics
for ODEs with periodic coefficients (see, for instance, \cite{CaDaPa-02, SrWo-97}).\\
In more details, in Section \ref{sec-pp} we furnish an application of the results on LTMs from Section \ref{sec-ltm} to a modified version of Volterra predator-prey model, in which a periodic harvesting is included \cite{PiZa-08}. Indeed, when the seasons with fishing alternate with the ones without harvesting in a periodic fashion, it is possible to prove the presence of chaotic features for the system, provided that we have the freedom to tune the switching times between the two regimes.
Analogous conclusions could be drawn for those time-periodic planar Kolmogorov systems \cite{Ko-36}
$$x'= X(t,x,y),\quad y'= Y(t,x,y),$$
which possess dynamical features similar to the ones of Volterra model. \\
On the other hand, in Section \ref{sec-sb} we employ the tools from Section \ref{sec-ltm} to show a complex behaviour for a nonlinear second order scalar equation of the form
\begin{equation*}
x'' + f(t,x) = 0,
\end{equation*}
with $f: {\mathbb R}\times {\mathbb R}\to {\mathbb R}$ a map $T$-periodic in the
$t$-variable and satisfying the Carath\'{e}odory assumptions.
In particular, we will deal with a simplified version of the
Lazer-McKenna suspension bridges model \cite{PaPiZa-08, PaZa-08}, belonging to the class of the
periodically perturbed Duffing equations
\begin{equation}\label{eq-pd}
x'' + g(x) = p(t),
\end{equation}
where $g: {\mathbb R}\to {\mathbb R}$ is locally Lipschitz
and $p(\cdot) : {\mathbb R} \to {\mathbb R}$ is a locally
integrable function such that $p(t + T) = p(t),$ for almost every
$t\in {\mathbb R}.$
However, we stress that our approach is in principle applicable also to more general
equations, including examples in which a term depending on $x'$ is present
in \eqref{eq-pd}.\\
As a final remark, we recall that the detection of chaotic dynamics for ODEs is the ``natural'' field of applicability for the stretching along the paths method.
Indeed its development has been
motivated by the analysis in \cite{PaZa-00} of the Poincar\'{e} map associated to
the nonlinear scalar Hill's type second order equation
\begin{equation}\label{eq-hill}
x'' + a(t) g(x) =0,
\end{equation}
in the case that $a(t)$ is a sign-changing continuous weight and
$g: {\mathbb R}\to {\mathbb R}$ is a function with superlinear growth at infinity.
Looking in the phase-plane at the solutions $(x(t),x'(t))$ of \eqref{eq-hill}, a stretching property along the paths was detected.
Namely, two topological planar rectangles
were found such that
every path joining two opposite sides of any one of the two rectangles contained a subpath which was expanded by the flow
across the same rectangle or the other one.
This was the key lemma in \cite{PaZa-00} in order to prove the
presence of a complicated oscillatory behaviour for the solutions of \eqref{eq-hill}.\\

\section{Predator-Prey Model}\label{sec-pp}
\subsection{The effects of a periodic harvesting}\label{sub-ph}

The classical Volterra predator-prey model\index{Volterra predator-prey model} concerns the first order planar differential system
\begin{equation*}
\left\{
\begin{array}{ll}
x'= x( a - by)\\
y'= y(-c + dx),\\
\end{array}
\right.
\eqno{(E_0)}
\end{equation*}
where
$a, \;b, \;c, \;d \, > 0$
are constant coefficients. The study of system $(E_0)$ is confined to the open first quadrant
$({\mathbb R}^+_0)^2$
of the plane, since
$x(t) > 0$ and $y(t) > 0$ represent the size (number of individuals or density) of the
prey and the predator populations, respectively.
Such model was proposed by Vito Volterra in 1926 answering D'Ancona's
question about the percentage of selachians and food fish caught in the northern Adriatic
Sea during a range of years covering the period of the World War I
(see \cite{Br-93,Ma-04} for a more detailed historical account).
\\
System $(E_0)$ is conservative and its phase-portrait is that of a global center at the
point
\begin{equation*}
P_0\,:= \left( \frac{c}{d},\frac{a}{b}\right),
\end{equation*}
surrounded by periodic orbits (run in the counterclockwise sense), which are the level lines of the first integral
\begin{equation*}
{\mathcal E}_0(x,y):= d x - c \log x + b y - a \log y,
\end{equation*}
that we call ``energy'' in analogy to mechanical systems.
The choice of the sign in the definition of the first integral implies that
${\mathcal E}_0(x,y)$ achieves a strict absolute minimum at the point $P_0\,.$
\\
According to Volterra's analysis of $(E_0),$ the average of a periodic solution
$(x(t),y(t)),$ evaluated over a time-interval corresponding to its natural period,
coincides with the coordinates of the point $P_0\,.$

\smallskip

\noindent
In order to include the effects of fishing in the model,
one can suppose that, during the harvesting time, both the prey and
the predator populations are
reduced at a rate proportional to the size of the population itself. This
assumption leads to the new system
\begin{equation*}
\left\{
\begin{array}{ll}
x'= x( a_{\mu} - by)\\
y'= y(-c_{\mu} + dx),\\
\end{array}
\right.
\eqno{(E_{\mu})}
\end{equation*}
where
$$a_{\mu}:= a - \mu\quad\mbox{and} \quad c_{\mu}:= c + \mu$$
are the modified growth coefficients which take into account the fishing rates $-\mu x(t)$ and
$- \mu y(t),$ respectively. The parameter $\mu$ is assumed to be positive but small enough
$(\mu < a$) in order to prevent the extinction of the populations. System $(E_{\mu})$
has the same form like $(E_0)$ and hence its phase-portrait is that of a global center
at
\begin{equation*}
P_{\mu}:= \left( \frac{c + \mu}{d},\frac{a - \mu}{b}\right).
\end{equation*}
The periodic orbits surrounding $P_{\mu}$ are the level lines of the first integral
\begin{equation*}
{\mathcal E}_{\mu}(x,y):= d x - c_{\mu} \log x + b y - a_{\mu} \log y.
\end{equation*}
The coordinates of $P_{\mu}$ coincide with the average values of the prey and the predator populations
under the effect of fishing (see Figure
\ref{fig-v1}). A comparison between the coordinates
of $P_0$ and $P_{\mu}$  motivates the
conclusion (\textit{Volterra's principle})
that a moderate harvesting has a favorable effect for the prey population
\cite{Br-93}.

\begin{figure}[ht]
\centering
\includegraphics[scale=0.22]{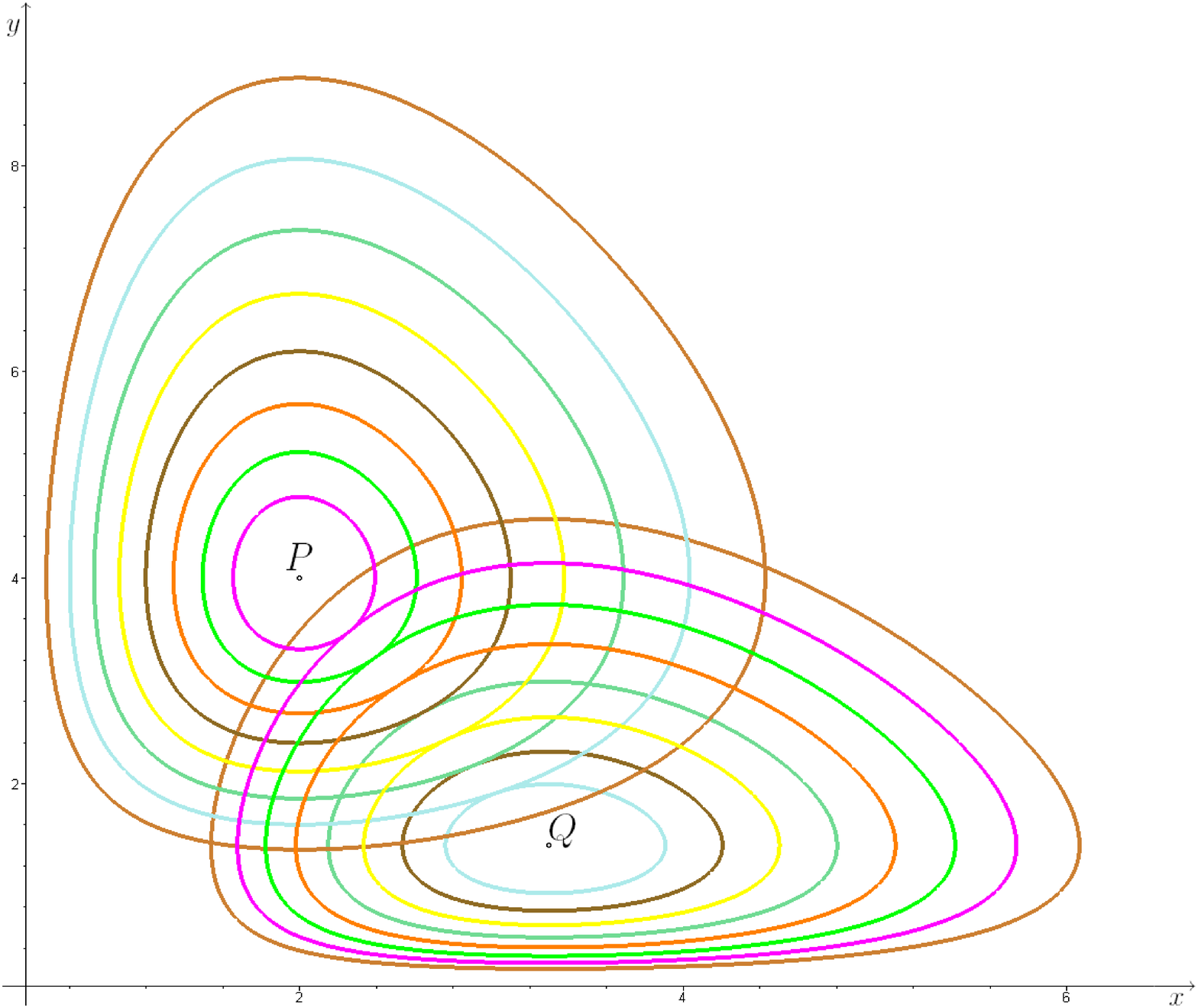}
\caption{\footnotesize {In this picture we show some periodic orbits of
the Volterra system $(E_0)$ with center at $P=P_0,$
as well as of the
perturbed system $(E_{\mu})$ with center at $Q= P_{\mu}\,$
(for a certain $\mu \in \,]0,a[\,$).
}}
\label{fig-v1}
\end{figure}

\smallskip

\noindent
If we are interested in incorporating the consequences of a cyclic environment in the Volterra original model $(E_0),$ we can assume a seasonal effect on the coefficients, which leads us to consider a system of the form
\begin{equation*}
\left\{
\begin{array}{ll}
x'= x( a(t) - b(t) y)\\
y'= y(-c(t) + d(t) x),\\
\end{array}
\right.
\eqno{(E)}
\end{equation*}
where $a(\cdot), b(\cdot), c(\cdot), d(\cdot): {\mathbb R}\to {\mathbb R}$
are periodic functions with a common period $T > 0.$
In such a framework, it is natural to look for harmonic\index{solution, harmonic}  (i.e. $T$-periodic)
or $m$-th order subharmonic\index{solution, subharmonic} (i.e. $mT$-periodic, for some integer $m\geq 2,$ with $mT$ the minimal period in the set $\{jT:j=1,2,\dots\}$) solutions with
range in the open first quadrant (\textit{positive solutions}).\\
The past forty years have witnessed a growing attention towards such kind of models and several results have been obtained about the existence, multiplicity and stability of periodic solutions for Lotka-Volterra type predator-prey systems with periodic coefficients \cite{AmOr-94,BuHu-91,BuFr-81,Cu-77,DiHuZa-95,DiZa-96,Ha-82,HaMa-91,LGOrTi-96,Ta-87}.
Further references are available in \cite{PiZa-08}, where the reader can also find historical details about the fortune of this model as well as some related results from \cite{DiHuZa-95,DiZa-96}.

\smallskip

\noindent
Let us come back for a moment to the original Volterra system with constant coefficients and
suppose that the interaction between the two populations
is governed by system $(E_0)$ for a certain period
of the season (corresponding to a time-interval of length $r_0$)
and by system $(E_{\mu})$ for the rest of the season (corresponding to a time-interval of length $r_{\mu}$).
Assume also that such alternation between $(E_0)$ and $(E_{\mu})$ occurs in a
periodic fashion, so that
$$T:=r_0 + r_{\mu}$$
is the period of the season. In other terms, first we consider system $(E_0)$ for $t\in [0,r_0[\,.$ Next we switch
to system $(E_{\mu})$ at time $r_0$ and assume that $(E_{\mu})$
rules the dynamics for $t\in [r_0,T[\,.$ Finally, we suppose that
we switch back to system $(E_0)$ at time $t=T$ and repeat the cycle with $T$-periodicity.\\
Such two-state alternating behaviour can be
equivalently described in terms of system $(E),$ by assuming
$$
a(t)= {\hat{a}}_{\mu}(t):= \left\{
\begin{array}{llll}
a\,\quad &\mbox{for } \, 0\leq t < r_0\,,\\
a - \mu\,\quad &\mbox{for } \, r_0\leq t < T\,,
\end{array}
\right.
$$
$$
c(t)={\hat{c}}_{\mu}(t):= \left\{
\begin{array}{llll}
c\,\quad &\mbox{for } \, 0\leq t < r_0\,,\\
c + \mu\,\quad &\mbox{for } \, r_0\leq t < T\,,
\end{array}
\right.
$$
as well as
$$b(t)\equiv b,\;\; d(t)\equiv d,$$
with $a,b,c,d$ positive constants and $\mu$ a parameter with $0 < \mu < a.$
Hence we can consider the system
\begin{equation*}
\left\{
\begin{array}{ll}
x'= x( {\hat{a}}_{\mu}(t) - b y)\\
y'= y(-{\hat{c}}_{\mu}(t) + d x),\\
\end{array}
\right.
\eqno{(E^*)}
\end{equation*}
where the piecewise constant functions ${\hat{a}}_{\mu}$ and ${\hat{c}}_{\mu}$ are supposed to be
extended to the whole real line by $T$-periodicity.

\smallskip

\noindent
It is our aim now to prove that $(E^*)$ generates chaotic dynamics in the sense of Definition \ref{def-chm}. To this end, as explained at the beginning of Chapter \ref{ch-ode}, we apply the results on LTMs from Section \ref{sec-ltm} to the Poincar\'{e} map
$$\Psi: ({\mathbb R}^+_0)^2 \to ({\mathbb R}^+_0)^2,\quad
\Psi(z):= \zeta(T,z),$$
where $\zeta(\cdot,z) = (x(\cdot,z),y(\cdot,z))$ is the solution of system $(E^*)$
starting from $z = (x_0,y_0)\in ({\mathbb R}^+_0)^2$ at the time $t=0.$
As a consequence, we will have ensured all the chaotic features listed in Theorem \ref{th-inj} (like, for instance,
a semi-conjugacy to the two-sided Bernoulli shift and thus positive topological entropy, sensitivity with respect to initial conditions, topological transitivity, a compact invariant set containing a dense subset of periodic points).

\begin{figure}[ht]
\centering
\includegraphics[scale=0.17]{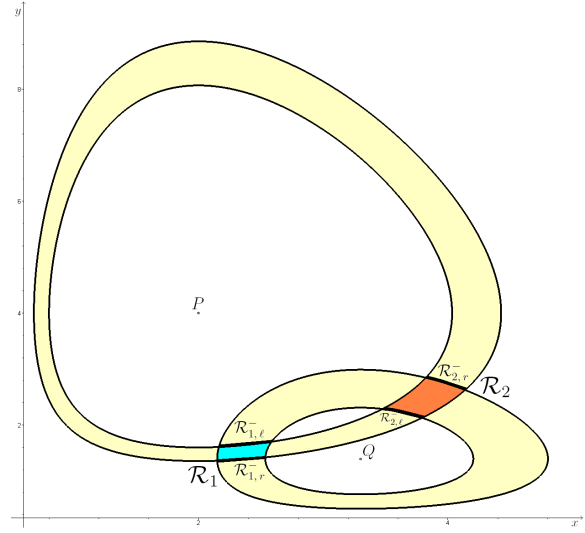}
\caption{\footnotesize {The two annular regions $\mathcal A_P$ and
$\mathcal A_Q$ (centered at $P$ and $Q,$ respectively) are linked
together. We have drawn with a darker color the two
rectangular sets $\mathcal R_1$ and $\mathcal R_2$ where they
meet. The boundary sets ${\mathcal R}_{1}^-={\mathcal R}^-_{1,\,{\ell}}\cup{\mathcal R}^-_{1,\,{r}}$ and ${\mathcal R}_{2}^-={\mathcal R}^-_{2,\,{\ell}}\cup{\mathcal R}^-_{2,\,{r}}$ have been indicated with a thicker line.
}}
\label{fig-v2}
\end{figure}

\noindent
With the aid of Figure
\ref{fig-v2}
we try to explain how to enter the setting of Theorem \ref{th-ltmb} for the switching system
$(E^*).$ \\
As a first step, we take two closed overlapping annuli
consisting of level lines of the first integrals associated to system
$(E_0)$ and $(E_{\mu}),$ respectively. In particular, the inner
and outer boundaries of each annulus are closed trajectories
surrounding the equilibrium point ($P=P_0$ for system $(E_0)$ and
$Q=P_{\mu}$ for system $(E_{\mu})$). Such annuli, that we call
from now on ${\mathcal A}_P$ and ${\mathcal A}_Q,$ intersect in
two compact disjoint sets ${\mathcal R}_1$ and
${\mathcal R}_2\,,$ which are generalized rectangles. The way in which we label the two regions (as ${\mathcal R}_1$/${\mathcal R}_2$) is completely arbitrary.
However, the choice of an order will effect some details in
the argument that we describe below. Whenever we enter a
framework like that visualized in Figure
\ref{fig-v2},
we say that the annuli
${\mathcal A}_P$ and ${\mathcal A}_Q$ are \textit{linked
together}\index{linked together annuli}. Technical conditions on the energy level lines defining
${\mathcal A}_P$ and ${\mathcal A}_Q\,,$ sufficient to guarantee the linking between them, are presented in Subsection \ref{sub-td}.\\
As a second step, we give an ``orientation'' to ${\mathcal R}_i$ (for $i=1,2$)
by selecting the boundary sets ${\mathcal R}_{i}^-={\mathcal R}^-_{i,\,{\ell}}\cup{\mathcal R}^-_{i,\,{r}}.$ In the specific example
of Figure
\ref{fig-v2},
we take as ${\mathcal R}_1^-$ the intersection of
${\mathcal R}_1$ with the inner and outer boundaries of ${\mathcal
A}_P$ and as ${\mathcal R}_2^-$ the intersection of ${\mathcal
R}_2$ with the inner and outer boundaries of ${\mathcal A}_Q\,.$
The way in which we name (as left/right) the two
components of ${\mathcal R}_i^-$ is inessential for the rest of
the discussion. Just to fix ideas, let us say that we choose as ${\mathcal R}^-_{1,\,{\ell}}$ the component of ${\mathcal R}_1^-$ which is closer to
$P$ and as ${\mathcal R}^-_{2,\,{\ell}}$ the component of ${\mathcal R}_2^-$ which is closer to $Q$ (of course, the ``right'' components ${\mathcal R}^-_{1,\,{r}}$ and ${\mathcal R}^-_{2,\,{r}}$ are the remaining ones).\\
As a third step, we observe that the Poincar\'{e} map associated to $(E^*)$ can be
decomposed as
$$\Psi= \Psi_{\mu}\circ \Psi_0\,,$$
where $\Psi_0$ is the Poincar\'{e} map of system $(E_0)$ on the time-interval $[0,r_0]$
and $\Psi_{\mu}$ is the Poincar\'{e} map for $(E_{\mu})$ on the time-interval $[0,r_{\mu}]= [0,T-r_0].$ Consider a path $\gamma: [0,1]\to {\mathcal R}_1$ with $\gamma(0)\in
{\mathcal R}^-_{1,\,{\ell}}$ and $\gamma(1)\in
{\mathcal R}^-_{1,\,{r}}\,.$ As we will see in Subsection \ref{sub-td},
the points of ${\mathcal R}^-_{1,\,{\ell}}$ move faster than those belonging to
${\mathcal R}^-_{1,\,{r}}$ under the action of system $(E_0).$ Hence, for a
choice of the first switching time $r_0$ large enough, it is possible to
make the path
$$[0,1]\ni s\mapsto \Psi_0(\gamma(s))$$
turn in a spiral-like fashion inside the annulus ${\mathcal A}_P$
and cross at least twice the rectangular region ${\mathcal R}_2$ from
${\mathcal R}^-_{2,\,{\ell}}$ to ${\mathcal R}^-_{2,\,{r}}\,.$
Thus, we can select two subintervals of $[0,1]$ such that $\Psi_0\circ \gamma$
restricted to each of them is a path contained in ${\mathcal R}_2$
and connecting the two components of ${\mathcal R}_2^-\,.$ We observe that the points of ${\mathcal R}^-_{2,\,{\ell}}$ move faster than those belonging to
${\mathcal R}^-_{2,\,{r}}$ under the action of system $(E_{\mu}).$
Therefore, we can repeat the same argument as above and conclude that, for a suitable
choice of $r_{\mu}=T-r_0$ sufficiently large, we can transform, via $\Psi_{\mu}\,,$ any path in ${\mathcal R}_2\,$
joining the two components of ${\mathcal R}_2^-\,$ onto a path which crosses at least once
${\mathcal R}_1$ from ${\mathcal R}^-_{1,\,{\ell}}$ to ${\mathcal R}^-_{1,\,{r}}\,.$\\
As a final step, we complete the proof about the existence of chaotic-like dynamics
by applying Theorem \ref{th-ltma}. Actually, in order to obtain a more complex behaviour, in place of Theorem \ref{th-ltma} it is possible to employ the more general Theorem \ref{th-ltmb}.\\
In fact, our main result can be stated as follows.

\begin{theorem}\label{th-mr}
For any choice of positive constants $a,b,c,d, \mu$ with $\mu < a$
and for every pair $({\mathcal A}_P,{\mathcal A}_Q)$ of linked together annuli,
the following conclusion holds:
\\
For every integer $m\geq 2,$
there exist two positive constants $\alpha$ and $\beta$ such that,
for each
$$r_0 > \alpha\quad\mbox{ and }\quad r_{\mu} > \beta,$$
the Poincar\'{e} map associated to system $(E^*)$ induces chaotic
dynamics on $m$ symbols in ${\mathcal R}_1$ and ${\mathcal
R}_2\,.$
\end{theorem}

\noindent
As remarked in \cite{PaPiZa-08}, if we consider Definition \ref{def-chm} and its consequences in the context of concrete examples of ODEs (for instance when $\psi$ turns out to be the Poincar\'{e} map), condition \eqref{eq-chm} may sometimes be interpreted in terms of the oscillatory behaviour of the solutions. Such situation occurred in \cite{PaZa-04a, PaZa-08} and takes place also for system $(E^*).$
Indeed, as it will be clear from the proof of Theorem \ref{th-mr}, it is possible to draw
more precise conclusions in the statement of our main result.
Namely, the following additional properties can be obtained:

\smallskip

\noindent
\textit{
For every decomposition of the integer $m\geq 2$ as
$$m= m_1\, m_2\,,\quad\mbox{with } \; m_1\,, m_2\, \in {\mathbb N}\,,$$
there exist integers $\kappa_1\,, \kappa_2\, \geq 1\,\,
($with $\kappa_1 = \kappa_1(r_0,m_1)$ and $\kappa_2 = \kappa_2(r_{\mu},m_2)\,)$
such that, for each two-sided sequence of symbols
$${\bf{s}} = (s_i)_{i\in{\mathbb Z}}= (p_i,q_i)_{i\in{\mathbb Z}}\in
\{0,\dots,m_1-1\}^{{\mathbb Z}}\times\{0,\dots,m_2-1\}^{{\mathbb Z}}\,,$$
there exists a solution
$$\zeta_{\bf{s}}(\cdot) = \bigl( x_{\bf{s}}(\cdot),y_{\bf{s}}(\cdot) \bigr)$$
of $(E^*)$ with
$\zeta_{\bf{s}}(0)\in {\mathcal R}_1$ such that $\zeta_{\bf{s}}(t)$ crosses  ${\mathcal R}_2$
exactly $\kappa_1 + p_i+1$ times for $t\in \,]iT,r_0 + iT [$ and crosses
${\mathcal R}_1$ exactly $\kappa_2 + q_i+1$ times for $t\in \,]r_0 + iT,(i+1)T[\,.$
Moreover, if $(s_i)_{i\in{\mathbb Z}}= (p_i,q_i)_{i\in{\mathbb Z}}$ is a periodic sequence,
that is, $s_{i + k} = s_i\,,$ for some $k\geq 1,$ then $\zeta_{\bf{s}}(t + k T) = \zeta_{\bf{s}}(t),$
$\forall\,t\in {\mathbb R}.$
\\
In particular, taking $m = m_1\geq 2$ and $m_2 =1,$ we obtain the dynamics on the set
of $m$ symbols
$\{0,\dots,m-1\} \equiv \{0,\dots,m-1\}\times\{0\}$ as in the statement of Theorem \ref{th-mr}.
}
\smallskip

\noindent
The solutions $(x(t),y(t))$ are meant in the
Carath\'{e}odory sense, i.e. $(x(t),y(t))$ is absolutely continuous and satisfies system $(E^*)$ for almost every $t \in {\mathbb R}.$
Of course, such solutions are of class $C^1$ if the coefficients are continuous.\\
The constants $\alpha$ and $\beta$ in Theorem \ref{th-mr}, representing the lower bounds for $r_0$ and $r_{\mu},$ can be estimated in terms of $m_1$ and $m_2$ and other geometric parameters, such as the fundamental periods of the orbits bounding the linked annuli
(see \eqref{eq-al} and \eqref{eq-be}\,).

\medskip

\noindent
We end this introductory discussion with a few observations about our main result.\\
First of all we notice that, according to Theorem \ref{th-mr},
there is an abundance of chaotic regimes for system $(E^*),$
provided that the time-interval lengths $r_0$ and $r_{\mu}$ (and,
consequently, the period $T$) are sufficiently large. More precisely, we
are able to prove the existence of chaotic invariant sets inside
each intersection of two annular regions linked together. One
could conjecture the presence of Smale horseshoes contained in
such intersections, like in the classical case of linked twist
maps with circular domains \cite{De-78}, even if we recall that our purely topological approach just requires to check a twist
hypothesis on the boundary, without the need of verifying any
hyperbolicity condition. This does not prevent the
possibility of a further deeper analysis using more complex
computations. \\
We also stress that our result is stable with
respect to small perturbations of the coefficients. In fact,
as it will emerge from the proof, whenever $r_0 > \alpha$ and
$r_{\mu} > \beta$ are chosen so as to achieve the conclusion of
Theorem \ref{th-mr}, it follows that there exists a constant
$\varepsilon > 0$ such that Theorem \ref{th-mr} applies to
equation $(E)$ too, provided that
$$
\int_0^T |a(t) - {\hat{a}}_{\mu}(t)|\,dt < \varepsilon,\quad
\int_0^T |c(t) - {\hat{c}}_{\mu}(t)|\,dt < \varepsilon,
$$
$$
\int_0^T |b(t) - b|\,dt < \varepsilon,\quad
\int_0^T |d(t) - d|\,dt < \varepsilon.
$$
Here the $T$-periodic coefficients may be in $L^1([0,T])$ or even continuous or smooth
functions, possibly of class $C^{\infty}.$\\
A final remark concerns the fact that, in our model, we have assumed that
the harvesting period starts and ends for both the species at the same moment.
With this respect, one could face a more general situation
in which some phase-shift between the two harvesting intervals occurs.
Such cases have been already explored in some biological models,
mostly from a numerical
point of view: see \cite{Na-86} for an example on competing species
and \cite{RiMu-93} for a predator-prey system.
If we assume a phase-shift in the periodic coefficients, that is, if we
consider
$$a(t):= {\hat{a}}_{\mu}(t - \theta_1)\quad\mbox{ and } \quad c(t):= {\hat{c}}_{\mu}(t - \theta_2),$$
for some $0<\theta_1,\theta_2<T,$
and we also suppose that the length $r_0$ of the time-intervals without harvesting
may differ for the two species (say $r_0 = r_a\in \, ]0,T[$ in the definition of
${\hat{a}}_{\mu}$ and $r_0 = r_c\in \, ]0,T[$ in the definition of
${\hat{c}}_{\mu}$), then
the geometry of our problem turns out to be a combination of linked twist maps on two, three or four
annuli, which are mutually linked together.
In this manner, we increase the possibility of chaotic configurations, provided that the system
is subject to the different regimes for sufficiently long time. For a pictorial comment, see Figure
\ref{fig-v3},
where all the possible links among four annuli are realized.

\begin{figure}[ht]
\centering
\includegraphics[scale=0.22]{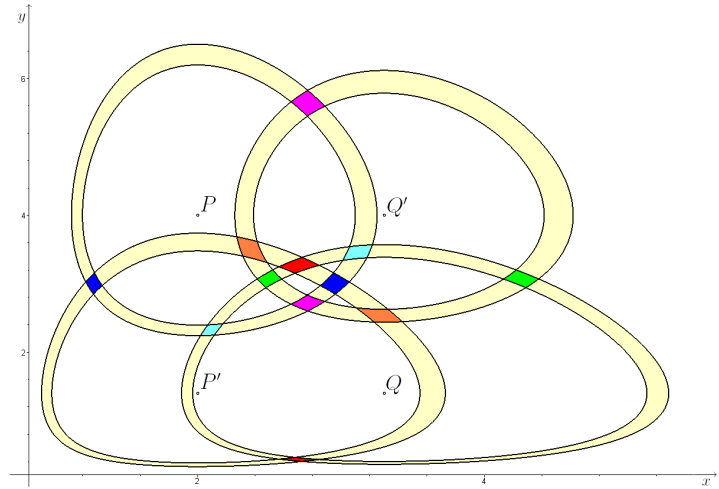}
\caption{\footnotesize {We have depicted four linked annular regions bounded by
energy level lines corresponding to Volterra systems with centers at
$P=P_0,$ $P'= (c/d,a_{\mu}/b),$
$Q= P_{\mu}\,$ and $Q'=(c_{\mu}/d,a/b),$
by putting in evidence the regions of mutual intersection,
where it is possible to locate the chaotic invariant sets.
}}
\label{fig-v3}
\end{figure}

\subsection{Technical details and proofs}\label{sub-td}

Let us consider system $(E_0)$ and let
$$\ell > \chi_0:={\mathcal E}_0(P_0) = \min\{{\mathcal E}_0(x,y): \, x>0, y> 0\}.$$
The level line
$$\Gamma_0(\ell):=\left\{(x,y)\in ({\mathbb R}_0^+)^2\,:\, {\mathcal E}_0(x,y) = \ell\right\}$$
is a closed orbit (surrounding $P_0$) which is run counterclockwise,
completing one turn in a fundamental period that we denote by $\tau_0(\ell).$
According to classical results on the period of the Lotka-Volterra system
\cite{Ro-85,Wa-86}, the map
$$\tau_0\,: \,]\chi_0,+\infty[\,\to {\mathbb R}$$
is strictly increasing
with $\tau_{0}(+\infty) = +\infty$ and satisfies
$$\lim_{\ell\to \chi_0^+} \tau_0(\ell) = T_0\,:=\frac{2\pi}{\sqrt{a c}}\,.$$
Similarly, considering system $(E_{\mu})$ with $0 < \mu < a\,,$
we denote by $\tau_{\mu}(h)$ the minimal period associated to the orbit
$$\Gamma_{\mu}(h):=\left\{(x,y)\in ({\mathbb R}_0^+)^2\,:\, {\mathcal E}_{\mu}(x,y) = h\right\},$$
for
$$h > \chi_{\mu}:={\mathcal E}_{\mu}(P_{\mu}) = \min\{{\mathcal E}_{\mu}(x,y): \, x>0, y> 0\}.$$
Also in this case the map $h\mapsto \tau_{\mu}(h)$ is strictly increasing
with $\tau_{\mu}(+\infty) = +\infty$ and
$$\lim_{h\to \chi_{\mu}^+} \tau_{\mu}(h) = T_{\mu}\,:=\frac{2\pi}{\sqrt{a_{\mu} c_{\mu}}}\,.$$

\medskip

\noindent
Before giving the details for the proof of our main result, we
describe conditions on the energy level lines of  two annuli
$\mathcal A_P$ and $\mathcal A_Q,$ centered at $P=P_0= \left(
\frac{c}{d},\frac{a}{b}\right)$ and $Q=P_{\mu}= \left( \frac{c +
\mu}{d},\frac{a - \mu}{b}\right),$ respectively, sufficient to
ensure that they are linked together\index{linked together annuli}. With this respect, we have
to consider the intersections among the closed orbits around the
two equilibria and the straight line $r$ passing through the
points $P$ and $Q,$ whose equation is $by+dx-a-c=0.$ We introduce an
orientation on such line by defining an order ``$\,\preceq\,$'' among its points.
More precisely, we set $A\preceq B$ (resp. $A\prec B$) if
and only if $x_A \leq x_B$ (resp. $x_A < x_B$), where $A=(x_A,y_A),\, B=(x_B,y_B).$ In this manner,
the order on $r$ is that inherited from the oriented $x$-axis, by projecting the points of $r$
onto the abscissa.
Assume now we
have two closed orbits $\Gamma_0(\ell_1)$ and $\Gamma_0(\ell_2)$
for system $(E_0),$ with $\chi_0<\ell_1<\ell_2.$ Let us call the
intersection points among $r$ and such level lines
$P_{1,-}\,,P_{1,+}\,,$ with reference to $\ell_1,$ and $P_{2,-}\,,P_{2,+}\,,$ with reference to
$\ell_2,$ where
$$P_{2,-}\,\prec\, P_{1,-}\, \prec \,P \prec \,P_{1,+}\, \prec\, P_{2,+}\,.$$
Analogously, when we consider two orbits $\Gamma_{\mu}(h_1)$ and
$\Gamma_{\mu}(h_2)$ for system $(E_{\mu}),$ with
$\chi_{\mu}<h_1<h_2,$ we name the intersection points among $r$ and
these level lines $Q_{1,-}\,,Q_{1,+},$ with reference to $h_1,$ and
$Q_{2,-}\,,Q_{2,+},$ with reference to $h_2,$ where
$$Q_{2,-}\,\prec\, Q_{1,-}\, \prec\, Q \prec\, Q_{1,+}\, \prec \,Q_{2,+}\,.$$
Then the two annuli $\mathcal
A_P$ and $\mathcal A_Q$ turn out to be linked together if
$$P_{2,-}\, \prec\, P_{1,-}\, \preceq\, Q_{2,-}\, \prec\, Q_{1,-}\, \preceq \,P_{1,+}\,
\prec\, P_{2,+}\, \preceq \,Q_{1,+}\, \prec\, Q_{2,+}\,.$$

\smallskip

\noindent {\textit{Proof of Theorem \ref{th-mr}.}} Consistently
with the notation introduced above, we denote by $\Gamma_0(\ell),$ with $\ell\in
[\ell_1,\ell_2],$ for some $\chi_0<\ell_1<\ell_2,$ the level lines filling $\mathcal A_P,$ so that
$${\mathcal A}_P=\bigcup_{\ell_1\le\ell\le\ell_2}\Gamma_0(\ell).$$
Analogously,
we indicate the level lines
filling $\mathcal A_Q$  by $\Gamma_{\mu}(h),$ with $h\in [h_1,h_2],$ for some
$\chi_{\mu}<h_1<h_2,$ so that
we can write
$${\mathcal A}_Q=\bigcup_{h_1\le h\le h_2}\Gamma_{\mu}(h).$$

\begin{figure}[ht]
\centering
\includegraphics[scale=0.25]{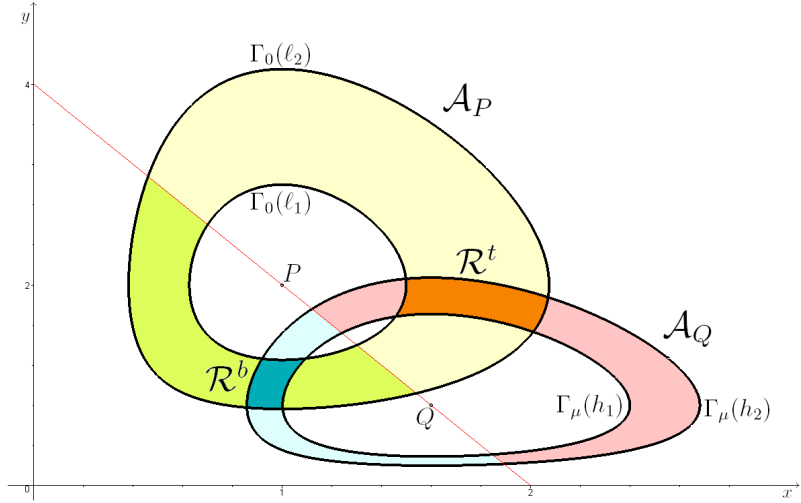}
\caption{\footnotesize{For the two linked annuli in the picture (one around
$P$ and the other around $Q$), we have drawn in
different colors the upper and lower
parts (with respect to the dashed line $r$), as well as the intersection regions ${\mathcal R}^t$
and ${\mathcal R}^b$ between them.
As a guideline for the proof, we recall that ${\mathcal A}_P$ is the annulus around $P,$ having as inner and outer
boundaries the energy level lines $\Gamma_0(\ell_1)$ and $\Gamma_0(\ell_2).$
Similarly, ${\mathcal A}_Q$ is the annulus around $Q,$ having as inner and outer
boundaries the energy level lines $\Gamma_{\mu}(h_1)$ and $\Gamma_{\mu}(h_2).$
}}
\label{fig-v4}
\end{figure}

\smallskip

\noindent
By construction, such annular regions turn out to be invariant for
the dynamical systems generated by $(E_0)$ and $(E_{\mu}),$
respectively. We consider now the two regions in which each
annulus is cut by the line $r$ (passing through $P$ and $Q$) and
we call such sets $\mathcal A_P^t,\,\mathcal A_P^b,\,\mathcal
A_Q^t$ and $\mathcal A_Q^b,$ in order to have ${\mathcal
A_P}=\mathcal A_P^t\cup \mathcal A_P^b$ and ${\mathcal
A_Q}=\mathcal A_Q^t\cup \mathcal A_Q^b,$ where the sets with superscript
$t$ are the ``upper'' ones and the sets with superscript $b$ are the
``lower'' ones, with respect to the line $r.$ We name
$\mathcal R^b$ the rectangular region in which $\mathcal A_P^b$
and $\mathcal A_Q^b$ meet and analogously we denote by
$\mathcal R^t$
the rectangular region belonging to the intersection between $\mathcal A_P^t$ and $\mathcal A_Q^t$ (see Figure
\ref{fig-v4}).

\smallskip

\noindent
Let
$$m_1\geq 2\,\quad\mbox{and}\quad m_2\geq 1$$
be two fixed integers. The case $m_1 = 1$ and $m_2\geq 2$
can be treated in a similar manner and therefore is omitted.\\
As a first step, we are interested in the solutions of system
$(E_0)$ starting from $\mathcal A_P^b$  and crossing $\mathcal
A_P^t$ at least $m_1$ times. After having performed the
rototranslation of the plane $\mathbb R^2$ that brings the origin
to the point $P$ and makes the $x$-axis coincide with the line
$r,$ whose equations are
\begin{equation*}
\left\{
\begin{array}{ll}
\tilde{x}= (x- \frac{c}{d}) \cos\vartheta+(y-\frac{a}{b}) \sin\vartheta\\
\tilde{y}= (\frac{c}{d}-x) \sin\vartheta+(y-\frac{a}{b}) \cos\vartheta ,\\
\end{array}
\right.
\end{equation*}
where $\vartheta:= \arctan(\frac{d}{b}),$ it is possible to use the
Pr\"ufer transformation and introduce generalized polar
coordinates, so that we can express the solution $\zeta(\cdot,z) =
(x(\cdot,z),y(\cdot,z))$ of system $(E_0)$ with initial point in
$z=(x_0,y_0)\in\mathcal A_P^b$ through the radial coordinate
$\rho(t,z)$ and the angular coordinate $\theta(t,z).$ Thus we
can assume that $\theta(0,z)\in [-\pi,0].$ For any $t\in
[0,r_0]$ and $z\in\mathcal A_P^b,$ let us introduce also the
\textit{rotation number}\index{rotation number}, that is the quantity
$$\rot_0(t,z):=\frac{\theta (t,z)-\theta (0,z)}{2\pi}\,,$$
that indicates the normalized angular displacement along the orbit
of system $(E_0)$ starting at $z,$ during the time-interval
$[0,t].$
The continuous dependence of the solutions from the initial data implies that the
function $(t,z)\mapsto \theta(t,z)$ and consequently
the map $(t,z)\mapsto \rot_0(t,z)$ are continuous. From the definition of rotation number
and the star-shapedness of the level lines of ${\mathcal E}_0$
with respect to the point $P,$ we have that for every $z\in \Gamma_0(\ell)$ the following
properties hold:

$$
\begin{array}{ll}
\forall \, j\in {\mathbb Z}\,:\;\; \rot_0(t,z) = j \, &\Longleftrightarrow  \, t = j \, \tau_0(\ell)\,,\\
\forall \, j\in {\mathbb Z}\,:\;\; j < \rot_0(t,z) < j+1 \, &\Longleftrightarrow \, j \,\tau_0(\ell) < t <
(j + 1)\, \tau_0(\ell)
\end{array}
$$
(if the annuli were not star-shaped, the inference `` $\Longleftarrow $ '' would still be true).
Although we have implicitly assumed that $\ell_1 \leq \ell \leq \ell_2\,,$ such properties hold
for every $\ell > \chi_0\,.$ \\
Observe that, thanks to the fact that the time-map $\tau_0$ is strictly
increasing, we know that $\tau_0(\ell_1)<\tau_0(\ell_2).$ We shall use this
condition to show that a twist property for the rotation number holds for sufficiently
large time-intervals. Indeed
we claim that, if we choose a switching time $r_0\ge\alpha,$ where
\begin{equation}\label{eq-al}
\alpha:=\frac{(m_1+ 3 + \tfrac{1}{2})\,\tau_0(\ell_1)\,\tau_0(\ell_2)}{\tau_0(\ell_2)-\tau_0(\ell_1)}\,,
\end{equation}
then, for any path $\gamma:[0,1]\to {\mathcal A_P},$ with
$\gamma(0)\in\Gamma_0(\ell_1)$ and $\gamma(1)\in\Gamma_0(\ell_2),$ the interval inclusion
\begin{equation}\label{eq-n*}
[\theta(r_0,\gamma(1)),\theta(r_0,\gamma(0))]\supseteq [2\pi
n^*,2\pi(n^*+m_1)-\pi]
\end{equation}
is fulfilled for some $n^*=n^*(r_0)\in {\mathbb N}.$
To check our claim, at first we notice that, for
a path $\gamma(s)$ as above, it holds that
$\rot_0(t,\gamma(0)) \geq \lfloor t/\tau_0(\ell_1)\rfloor$ and
$\rot_0(t,\gamma(1)) \leq \lceil t/\tau_0(\ell_2)\rceil,$ for every $t> 0,$
and so
$$\rot_0(t,\gamma(0)) - \rot_0(t,\gamma(1)) > t\, \frac{\tau_0(\ell_2) - \tau_0(\ell_1)}{\tau_0(\ell_1)\,\tau_0(\ell_2)}\, - 2,
\quad
\forall \, t> 0.$$
Hence, for $t\geq \alpha,$ with $\alpha$ defined as in \eqref{eq-al}, we obtain
$$\rot_0(t,\gamma(0))> m_1+1+ \tfrac{1}{2} + \rot_0(t,\gamma(1))\,,$$
which, in turns, implies
$$\theta(t,\gamma(0)) - \theta(t,\gamma(1)) > 2 \pi(m_1 + 1),\quad\forall\, t\geq \alpha.$$
Therefore, recalling the bound
$2\pi(\lceil t/\tau_0(\ell_2)\rceil  - \tfrac{3}{2} ) < \theta(t,\gamma(1)) \leq 2 \pi\lceil t/\tau_0(\ell_2)\rceil,$
the interval inclusion \eqref{eq-n*} is achieved for
\begin{equation*}
n^* = n^*(r_0):= \left\lceil \frac{r_0}{\tau_0(\ell_2)}\right\rceil.
\end{equation*}
This proves our claim.\\
By the continuity of the composite mapping $[0,1]\ni
s\mapsto \rot_0(r_0,\gamma(s)),$
it follows that
$$\{\theta
(r_0,\gamma(s)),\,s\in [0,1]\}\supseteq[2\pi
n^*,2\pi(n^*+m_1 -1)+\pi].$$
As a consequence, by the Bolzano Theorem, there exist $m_1$ pairwise
disjoint maximal intervals $[t_i{'},t_i{''}]\subseteq [0,1],$
for $i=0,\dots,m_1 - 1,$ such that
$$\theta(r_0,\gamma(s))\in [2\pi n^*+2\pi i,2\pi n^*+\pi+2\pi i],\,\forall s\in [t_i{'},t_i{''}],\, i=0,\dots,m_1 -1,$$
with $\theta(r_0,\gamma(t_i{'}))=2\pi n^*+2\pi i$ and
$\theta(r_0,\gamma(t_i{''}))=2\pi n^*+\pi+2\pi i.$ Setting
$$\mathcal R_1:=\mathcal R^b\quad\mbox{and} \quad \mathcal R_2:=\mathcal R^t,$$
we
orientate such rectangular regions by choosing
$${\mathcal
R}^-_{1,\,{\ell}}:=\mathcal R_1\cap\Gamma_0(\ell_1)\quad \mbox{and}\quad
{\mathcal R}^-_{1,\,{r}}:=\mathcal R_1\cap\Gamma_0(\ell_2),$$
as well as
$${\mathcal R}^-_{2,\,{\ell}}:=\mathcal
R_2\cap\Gamma_\mu(h_1)\quad \mbox{and} \quad {\mathcal
R}^-_{2,\,{r}}:=\mathcal R_2\cap\Gamma_\mu(h_2).$$
See the caption of Figure
\ref{fig-v4}
as a reminder for the corresponding sets.\\
Introducing at last the $m_1$ nonempty and pairwise disjoint compact sets
$$\mathcal H_i:=\left\{z\in\mathcal A_P^b: \theta(r_0,z)\in [2\pi n^*+2\pi
i,2\pi n^*+\pi+2\pi i]\right\},$$
for $i=0,\dots,m_1 - 1,$ we are ready to prove
that
\begin{equation}\label{eq-str1}
(\mathcal H_i,\Psi_0):\widetilde{\mathcal R}_1\stretchx\widetilde{\mathcal R}_2,\quad \forall\, i=0,\dots,m_1-1,
\end{equation}
where we recall that $\Psi_0$ is the Poincar\'e map associated to
system $(E_0).$ Indeed, let us take a path $\gamma:[0,1]\to\mathcal
R_1,$  with $\gamma(0)\in {\mathcal R}^-_{1,\,{\ell}}$ and
$\gamma(1)\in {\mathcal R}^-_{1,\,{r}}.$ For
$r_0\geq\alpha$ and fixing $i\in \{0,\dots,m_1-1\},$
there is a subinterval
$[t_i{'},t_i{''}]\subseteq[0,1],$ such that
$\gamma(t)\in\mathcal H_i$ and $\Psi_0(\gamma(t))\in \mathcal
A_P^t,\, \forall t\in [t_i{'},t_i{''}].$ Noting that
$\Gamma_\mu(\Psi_0(\gamma(t_i{'})))\le h_1$ and
$\Gamma_\mu(\Psi_0(\gamma(t_i{''})))\ge h_2,$ then
there is a subinterval
$[t_i^{*},t_i^{**}]\subseteq[t_i{'},t_i{''}],$ such that
$\Psi_0(\gamma(t))\in\mathcal R_2,\,\forall t\in
[t_i^{*},t_i^{**}],$ with $\Psi_0(\gamma(t_i^{*}))\in{\mathcal
R}^-_{2,\,{\ell}}$ and
$\Psi_0(\gamma(t_i^{**}))\in{\mathcal R}^-_{2,\,{r}}.$
Hence, condition \eqref{eq-str1} is verified.

\smallskip

\noindent
Let us turn to system $(E_\mu).$ This time we focus our
attention on the solutions of such system starting from $\mathcal
A_Q^t$  and crossing $\mathcal A_Q^b$ at least $m_2$ times.
Similarly as before, we assume to have performed a rototranslation
of the plane that makes the $x$-axis coincide with the line $r$
and that brings the origin to the point $Q,$ so that we can
express the solution $\zeta(\cdot,w)$ of system $(E_{\mu})$ with
starting point in $w\in\mathcal A_Q^t$ through polar coordinates
$(\tilde\rho,\tilde\theta).$ In particular, it holds that
$\tilde\theta(0,w)\in[0,\pi].$ For any $t\in [0,r_\mu]=[0,T-r_0]$
and $w\in\mathcal A_Q^t,$ the rotation number is now defined as
$$\rot_\mu(t,w):=\frac{\tilde\theta (t,w)-\tilde\theta (0,w)}{2\pi}.$$
Since the time-map $\tau_{\mu}$ is strictly increasing, it follows that
$\tau_\mu(h_1)<\tau_\mu(h_2).$ We claim that, choosing a switching time
$r_\mu\ge\beta,$ with
\begin{equation}\label{eq-be}
\beta:=\frac{(m_2+3 +\tfrac{1}{2})\,\tau_\mu(h_1)\,\tau_\mu(h_2)}{\tau_\mu(h_2)-\tau_\mu(h_1)}\,,
\end{equation}
then, for any path $\omega:[0,1]\to {\mathcal A_Q},$
with $\omega(0)\in\Gamma_\mu(h_1)$ and $\omega(1)\in\Gamma_\mu(h_2),$
the interval inclusion
\begin{equation}\label{eq-incl}
\left[\tilde\theta(r_\mu,\omega(1)),\tilde\theta(r_\mu,\omega(0))\right]\supseteq \left[\pi(
2 n^{**}+1),2\pi(n^{**}+m_2)\right]
\end{equation}
is satisfied for some $n^{**}=n^{**}(r_\mu)\in {\mathbb N}.$
The claim can be proved with arguments analogous to the ones employed above and thus its verification is omitted. The nonnegative integer $n^{**}$ has to be chosen as
\begin{equation*}
n^{**}:=\left\lceil \frac{r_\mu}{\tau_\mu(h_2)}\right\rceil.
\end{equation*}
By \eqref{eq-incl} and the continuity of the composite map $[0,1]\ni
s\mapsto \rot_\mu({r_\mu},\omega(s)),$  it follows that
$$\left\{\tilde\theta ({r_\mu},\omega(s)),\,s\in [0,1]\right\}\supseteq\left[2\pi
n^{**}+\pi,2\pi(n^{**}+m_2)\right].$$
As a consequence, Bolzano Theorem ensures the existence of $m_2$ pairwise
disjoint maximal intervals $[s_i{'},s_i{''}]\subseteq [0,1],$
for $i=0,\dots,m_2-1,$ such that
$$\tilde\theta({r_\mu},\omega(s))\in [2\pi n^{**}+\pi+2\pi i,2\pi n^{**}+2 \pi+2\pi i],\,\forall s\in [s_i{'},s_i{''}],\, i=0,\dots,m_2-1,$$
with $\tilde\theta(r_\mu,\omega(s_i{'}))=2\pi n^{**}+\pi+2\pi i$ and $\tilde\theta(r_\mu,\omega(s_i{''}))=2\pi n^{**}+2 \pi+2\pi i.$
\\
For $\widetilde{\mathcal R}_1$ and $\widetilde{\mathcal R}_2$ as above and introducing the $m_2$ nonempty,
compact and pairwise disjoint sets
$$\mathcal K_i:=\left\{w\in\mathcal A_Q^t: \tilde\theta(r_\mu,w)\in [2\pi n^{**}+\pi+2\pi i,2\pi n^{**}+2 \pi+2\pi i]\right\},$$
with $i=0,\dots,m_2-1,$
we are in position to check that
\begin{equation}\label{eq-str2}
(\mathcal K_i,\Psi_\mu):\widetilde{\mathcal R}_2\stretchx\widetilde{\mathcal R}_1,\forall\,i=0,\dots,m_2-1,
\end{equation}
where $\Psi_\mu$ is the Poincar\'e map associated to system $(E_\mu).$
Indeed, taking a path $\omega:[0,1]\to\mathcal R_2,$ with $\omega(0)\in {\mathcal R}^-_{2,\,{\ell}}$
and $\omega(1)\in {\mathcal R}^-_{2,\,{r}},$ for $r_\mu\ge\beta$ and for any $i\in \{0,\dots,m_2 -1\}$
fixed, there exists a subinterval $[s_i{'},s_i{''}]\subseteq[0,1],$ such that $\omega(t)\in\mathcal K_i$
and $\Psi_\mu(\omega(t))\in \mathcal A_Q^b,\, \forall t\in [s_i{'},s_i{''}].$
Since $\Gamma_0(\Psi_\mu(\omega(s_i{'})))\le\ell_1$ and $\Gamma_0(\Psi_\mu(\omega(s_i{''})))\ge\ell_2,$
there exists a subinterval $[s_i^{*},s_i^{**}]\subseteq[s_i{'},s_i{''}]$
such that $\Psi_\mu(\omega(t))\in\mathcal R_1,\,\forall t\in [s_i^{*},s_i^{**}],$
with $\Psi_\mu(\omega(s_i^{*}))\in{\mathcal R}^-_{1,\,{\ell}}$ and $\Psi_\mu(\omega(s_i^{**}))\in{\mathcal R}^-_{1,\,{r}}.$
Hence, condition \eqref{eq-str2} is proved.

\smallskip

\noindent
The stretching properties in \eqref{eq-str1} and \eqref{eq-str2} allow to apply
Theorem \ref{th-ltmb} and the thesis follows immediately.
\qed

\medskip
\noindent
We observe that in the proof we have
chosen as $\mathcal R_1$ the ``lower'' set $\mathcal R^b$ and as
$\mathcal R_2$ the ``upper'' set $\mathcal R^t.$ However, since the orbits of both systems $(E_0)$ and $(E_\mu)$
are closed, the same argument works (by slightly modifying some constants, if needed)
for
${\mathcal R}_1 = {\mathcal R}^t$ and
${\mathcal R}_2 = {\mathcal R}^b.$

\section{Suspension Bridges Model}\label{sec-sb}

In this section we illustrate a possible application of the results from Section \ref{sec-ltm} to a
periodically perturbed Duffing equation of the form
\begin{equation}\label{eq-sb}
x'' + k x^+ = p_{q,s}(t),
\end{equation}
where $k, q, s, r_q\,, r_s$
are positive real numbers and $p_{q,s}:{\mathbb R}\to {\mathbb R}$
is a periodic function of period
\begin{equation*}
T= T_{q,s}\,:= r_q + r_s\,,
\end{equation*}
defined on $[0,T[\,$ by
\begin{equation*}
p_{q,s}(t)\,:=\;
\left\{
\begin{array}{lll}
&q \quad &\mbox{ for }\; 0\leq t < r_q\,,\\
-{\!\!\!\!}&s \quad &\mbox{ for }\; r_q\leq t < r_q \, + r_s\,,\\
\end{array}
\right.
\end{equation*}
and then extended to the real line by $T$-periodicity.\\
Equation \eqref{eq-sb} can be regarded as a simplified version of the
Lazer-McKenna suspension bridges model \cite{LaMK-87, LaMK-90} and
it is also meaningful from the point of view of the study of ODEs
with ``jumping nonlinearities'' and the periodic Dancer-Fu\v{c}ik
spectrum. See \cite{Ma-07} for an interesting survey concerning recent
developments in such area.\\
The results that we are going to quote have been obtained in \cite{PaZa-08} and are recalled in \cite{PaPiZa-08}, where only the main arguments employed in the proofs are presented, while the technical details are omitted. This is the scheme that we follow in the treatment below.\\
The main achievement in \cite{PaZa-08} concerns the existence of infinitely many periodic solutions as well as the presence of complex dynamics for equation \eqref{eq-sb}. Here we provide a simplified version of the original \cite[Theorem 1.2]{PaZa-08}, where precise information about the oscillatory behaviour of the solutions was given, too.

\begin{theorem}\label{th-mrb}
For any choice of $k, q, s > 0$ and for every positive integer $m\geq 2,$
it is possible to find two intervals
$\,]a^{m}_q,b^{m}_q[\,$ and $\,]a^{m}_s,b^{m}_s[\,$ such that, for every
forcing term $p_{q,s}$ with $r_q \in\, ]a^{m}_q,b^{m}_q[\,$ and
$r_s \in\, ]a^{m}_s,b^{m}_s[\,,$ the Poincar\'{e} map
associated to \eqref{eq-sb} induces chaotic dynamics
on $m$ symbols in the plane.
\end{theorem}

\noindent
We stress that, according to \cite[Theorem 1.3]{PaZa-08}, the above result is stable with respect to small perturbations in the sense that, once we have proved the existence of chaotic-like dynamics in a certain range
of parameters for equation \eqref{eq-sb}, then the same conclusions still hold for
\begin{equation*}
x'' + c x' + k x^+ = p(t),
\end{equation*}
provided that $|c|$ and $\int_0^T |p(t) - p_{q,s}(t)|\,dt$ (with $T = r_q + r_s$)
are sufficiently small. See \cite{PaZa-08} for more details.

\smallskip

\noindent
We describe now the geometry associated to the phase-portrait of \eqref{eq-sb},
in order to understand how to apply Theorem \ref{th-ltma} from Section \ref{sec-ltm}.
Similar arguments could be also employed for equation
$$ x'' + g(x) = p(t),$$
with $g$ a monotone
nondecreasing function bounded from below and such that $g'(+\infty) = k > 0.$\\
With reference to \eqref{eq-sb}, first of all, we observe that the Poincar\'{e} map
$\Psi$ associated to the equivalent planar system
\begin{equation}\label{sys-pq}
\left\{
\begin{array}{ll}
\dot x = y\\
\dot y = -k x^+ + p_{q,s}(t)
\end{array}
\right.
\end{equation}
can be decomposed as
\begin{equation*}
\Psi= \Psi_{s}\circ \Psi_q\,,
\end{equation*}
where $\Psi_{q}$ and $\Psi_{s}$ are the Poincar\'{e} maps
\begin{equation*}
\Psi_{q}\,: z_0 \mapsto \zeta_{q}(r_q\,,z_0),\quad
\Psi_{s}\,: z_0 \mapsto \zeta_{s}(r_s\,,z_0)
\end{equation*}
along the time intervals $[0,r_q]$ and $[0,r_s]$
associated to the dynamical systems defined by
$$
\left\{
\begin{array}{ll}
\dot x = y\\
\dot y = -k x^+ + q
\end{array}
\right.
\leqno{(E_q)}
$$
and
$$
\left\{
\begin{array}{ll}
\dot x = y\\
\dot y = -k x^+ - s.
\end{array}
\right.
\leqno{(E_s)}
$$
Here $\zeta_{q}(\cdot\,,z_0)$ denotes the solution $(x(\cdot),y(\cdot))$ of $(E_{q})$ with
$(x(0),y(0))= z_0\,.$ Clearly, $\zeta_s(\cdot\,,z_0)$ has the same meaning with respect to
system $(E_{s}).$

\begin{figure}[h]
   \centering
   \includegraphics[scale=0.35]{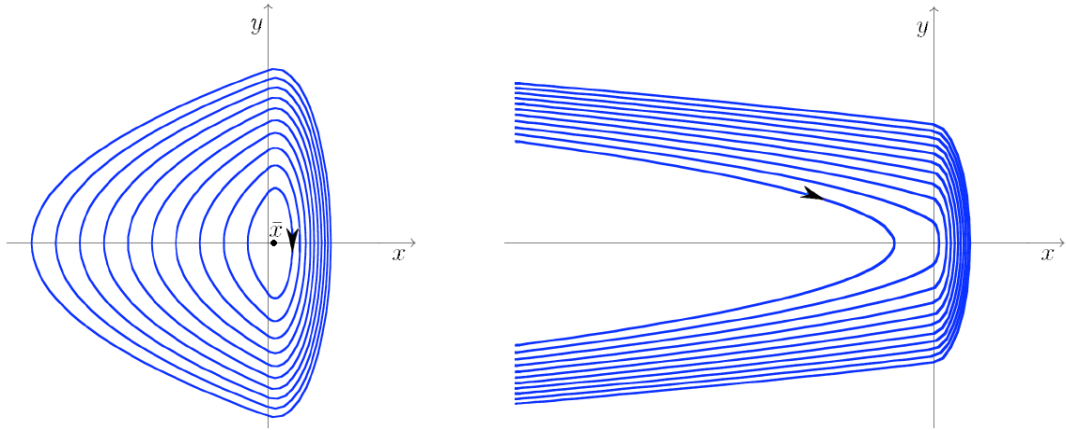}
   \caption{\footnotesize{Trajectories of system $(E_q)$  (left)
and system $(E_s)$ (right).
In this picture we
have chosen $k=10,$ $q=4$ and $s=0.5.$
}}
\label{fig:10}
\end{figure}

\noindent
The orbits of system $(E_q)$ correspond to the level lines
of the energy
$${\mathcal E}_{q}(x,y):= \frac 1 2 y^2 + \frac k 2 (x^+)^2 - q x.$$
Except for the equilibrium point $\bar x:=(\tfrac{q}{k},0),$
they are closed curves which are
run in the clockwise sense.
Moreover, their fundamental period $\tau_q( \cdot )$ is a monotonically nondecreasing function
with respect to the energy (in particular, strictly increasing on $[0,+\infty)$) and it limits to $+\infty$ as the energy tends to $+\infty.$\\
On the other hand,
the orbits of system $(E_s)$ correspond to the level lines
of the energy
$${\mathcal E}_{s}(x,y):= \frac 1 2 y^2 + \frac k 2 (x^+)^2 + s x$$
and they are run with $y'< 0.$

\noindent
Since our trajectories will switch from system $(E_q)$ to
system $(E_s)$ and vice versa,
in order to study the effect of the forcing term $p_{q,s}(t),$
we are led to overlap the energy level lines of the two systems.
Intersecting two level lines of
$(E_{q})$ with two level lines of $(E_{s}),$
we can determine two rectangular regions,
like in Figure \ref{fig:12}.

\begin{figure}[ht]
   \centering
   \includegraphics[scale=0.28]{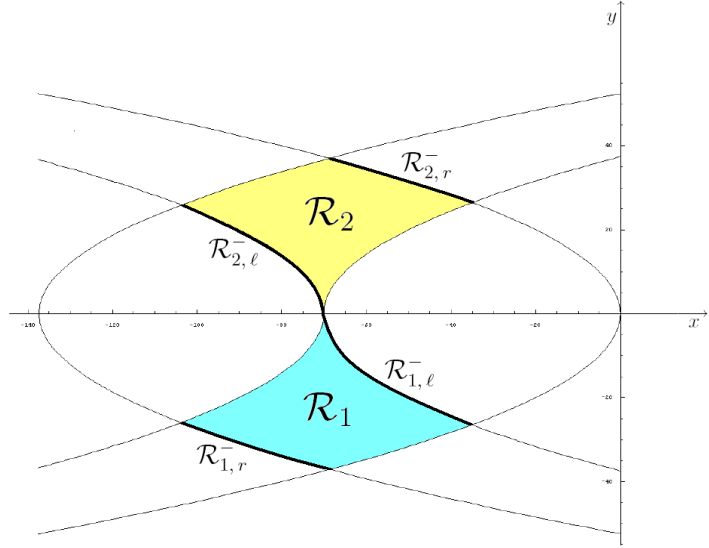}
   \caption{\footnotesize{Level lines of  system $(E_q)$ overlapped with the level lines of
system $(E_s)$ in the left half plane $x\leq 0.$ The
sets ${\mathcal R}_1$ and ${\mathcal R}_2$ are those bounded by the
four different level lines and contained in the lower half plane
$y\leq 0$ and in the upper half plane $y \geq 0,$ respectively. Orienting them by choosing as
${\mathcal R}_{1}^-={\mathcal R}^-_{1,\,{\ell}}\cup{\mathcal R}^-_{1,\,{r}}$ and ${\mathcal R}_{2}^-={\mathcal R}^-_{2,\,{\ell}}\cup{\mathcal R}^-_{2,\,{r}}$ the boundary sets drawn with a thicker line, we enter the setting of Theorem \ref{th-ltma} and thus the presence of chaotic dynamics follows.
}}
\label{fig:12}
\end{figure}

\noindent
To get periodic and chaotic solutions for system \eqref{sys-pq} and thus for equation
\eqref{eq-sb},
we employ the following elementary procedure.
We construct a set ${\mathcal R}_1$
in the third quadrant bounded by two level lines of $(E_{q})$ and two level lines of $(E_{s}).$
Since the period ${\tau}_q(\cdot)$ of the closed orbits of system $(E_q)$
is a strictly monotone increasing function on $[0,+\infty),$
we define the two components ${\mathcal R}^-_{1,\,{\ell}}$ and ${\mathcal R}^-_{1,\,{r}}$ of ${\mathcal R}^-_1$ as the intersections of ${\mathcal R}_1$ with the level lines of ${\mathcal E}_q\,,$ so that points on the inner boundary move faster than points on the outer boundary. Just to fix ideas, we choose as ${\mathcal R}^-_{1,\,{\ell}}$ the inner boundary and as ${\mathcal R}^-_{1,\,{r}}$ the outer boundary.
Hence, if we take a path $\gamma$ contained in the third quadrant and joining ${\mathcal R}^-_{1,\,{\ell}}$ to ${\mathcal R}^-_{1,\,{r}}\,,$
we find that, after a sufficiently long time $r_{q},$ its image under $\Psi_q$
turns out to be a spiral-like line crossing the second quadrant at least twice.
In particular, we can select two subpaths $\omega_0$ and $\omega_1$
of $\gamma,$ whose images under $\Psi_q$ cross the second quadrant.
Now we define a rectangular subset ${\mathcal R}_2$ in the second quadrant,
which is just the reflection of ${\mathcal R}_1$ with respect to the
$x$-axis, having as components of ${\mathcal R}^-_2$
the intersections of ${\mathcal R}_2$ with
two level lines of ${\mathcal E}_s.$ Selecting as ${\mathcal R}^-_{2,\,{\ell}}$ the intersection with the inner level line and as ${\mathcal R}^-_{2,\,{r}}$ the intersection with the outer one, it follows that $\Psi_q(\omega_0)$ and
$\Psi_q(\omega_1)$ admit subpaths contained in ${\mathcal R}_2$
and joining ${\mathcal R}^-_{2,\,{\ell}}$ to ${\mathcal R}^-_{2,\,{r}}\,.$
Switching to system $(E_s),$ it is possible to prove that the image of such subpaths under $\Psi_s$ crosses the set ${\mathcal R}_1$ from ${\mathcal R}^-_{1,\,\ell}$ to ${\mathcal R}^-_{1,\,r}\,.$ Then we can conclude by applying Theorem \ref{th-ltma}, provided we are free to choose the switching times $t_q$ and $t_s\,.$\\
The technical details for this argument can be found in \cite{PaZa-08}.

\smallskip

\noindent
Similar conclusions can be drawn for the second order ODE
\begin{equation*}
x'' + q(t) g(x) = 0,
\end{equation*}
with $g$ a nonlinear term and $q(\cdot)$ a periodic step function
(or a suitable small perturbation of a step function).

\smallskip

\noindent
The examples of ODEs we have presented suggest the possibility of proving
in an elementary but rigorous manner the presence of chaotic dynamics for a broad class of
second order nonlinear ODEs with periodic coefficients, which are piecewise
constant or at least not far from the piecewise constant case,
as long as we have the freedom to tune some switching time parameters within
a certain range depending on the coefficients governing the equations.
With this respect, our approach shows some resemblances with different techniques based on more sophisticated tools, like for instance on modifications of the methods by Shilnikov or Melnikov, in which one looks for a transverse homoclinic point \cite{BaPa-93,HaLi-86,Pa-86}. Indeed, also in such cases, it is possible to prove the presence of complex dynamics for differential equations, provided that the period of the time-dependent coefficients tends to infinity. See \cite{AmBa-98a,AmBa-98b,BeBo-99,dPFeTa-02}, where alternative techniques are employed. In \cite{HeZa-03} the differentiability hypothesis on the Melnikov function is replaced with sign-changing conditions and the proofs make use of the Poincar\'e-Miranda Theorem.\\
We conclude by observing that the same arguments employed above seem to work also for more general time-dependent coefficients.
However, a rigorous proof in such cases would require
a more delicate analysis or possibly the aid of computer assistance (as
in \cite{BaCs-08,GaZg-98,MiMr-95b,PoRaVi-07,Wi-03,YaLi-06,Zg-97}) and this is beyond the aims of the present thesis.

\backmatter

\end{document}